%% file: low-rank_tensors.tex
\documentclass[11pt,a4paper]{article}

\usepackage[margin=1.2in]{geometry}

\usepackage[utf8]{inputenc}
\usepackage{amsmath}
\usepackage{mathtools}
\usepackage{amsfonts}
\usepackage{amssymb}
\usepackage{dsfont}
\usepackage[]{algorithm}
\usepackage{algorithmic}
\usepackage{amsthm}
\usepackage{hyperref}
\usepackage[noabbrev,capitalize]{cleveref}
\usepackage[english]{babel}
\usepackage{comment}
\usepackage{tablefootnote}
\usepackage{array,multirow}
\usepackage{graphicx}
\usepackage{subcaption}
\usepackage[section]{placeins}
\usepackage{appendix}
\usepackage{color}
\usepackage[normalem]{ulem}

\include{defs}

\title{Efficiency of First-Order Methods for Low-Rank Tensor Recovery with the Tensor Nuclear Norm Under Strict Complementarity}
\author{Dan Garber \\ {\small Technion - Israel Institute of Technology}\\ {\small \texttt{dangar@technion.ac.il}}
\and
Atara Kaplan \\  {\small Technion - Israel Institute of Technology} \\ {\small  \texttt{ataragold@campus.technion.ac.il}}
}

\date{}

\begin{document} 

\maketitle

\begin{abstract}
We consider convex relaxations for recovering low-rank tensors based on constrained minimization over a ball induced by the tensor nuclear norm, recently introduced in \cite{tensor_tSVD}.
We build on a recent line of results that considered convex relaxations for the recovery of low-rank matrices and established that under  a strict complementarity condition (SC), both the convergence rate and per-iteration runtime of standard gradient methods may improve dramatically. We develop the appropriate strict complementarity condition for the tensor nuclear norm ball and obtain the following main results under this condition:
\begin{itemize}
\item
When the objective to minimize is of the form $f(\mX) = g(\mA\mX)+\langle{\mC,\mX}\rangle$ , where $g$ is strongly convex and $\mA$ is a linear map (e.g., least squares), a quadratic growth bound holds, which implies linear convergence rates for standard projected gradient methods, despite the fact that $f$ need not be strongly convex.
\item
For a smooth objective function, when initialized in certain proximity of an optimal solution which satisfies  SC, standard projected gradient methods only require SVD computations (for projecting onto the tensor nuclear norm ball) of rank that matches the \textit{tubal rank} of the optimal solution. In particular, when the tubal rank is constant, this implies nearly linear (in the size of the tensor) runtime per iteration, as opposed to super linear without further assumptions. Moreover, we establish a characterization of how increasing the rank of the SVD computations also increases the ball around the optimal solution in which we need to initialize. We also provide a practical and efficient procedure to verify that the low-rank SVD-based projections are indeed the exact projections. 
\item
For a nonsmooth objective function which admits a popular smooth saddle-point formulation, we derive similar results to the latter for the well known extragradient method.
\end{itemize} 
An additional contribution which may  be of independent interest, is the rigorous extension of many basic results regarding tensors of arbitrary order, which were previously obtained only for third-order tensors.
\end{abstract}

\section{Introduction}

Low-rank models for multi-dimensional arrays are extremely important in statistics, machine learning, and related areas. From a computational perspective, low-rank implies concise representation that allows for efficient storage and runtime implementations, which is crucial for high-dimensional settings. From a statistical perspective, low-rank often implies the ability  to recover multi-dimensional arrays from only noisy or partial information, under suitable assumptions, see for instance the seminal works on low-rank matrix completion \cite{candes2012exact} and robust principal component analysis  \cite{candes2011robust}. In the past two decades there has been numerous works on models involving low-rank matrices (2D-arrays), in terms of applications, statistical properties, and efficient optimization and learning, with too many references to mention. The focus of this work is on higher-order multi-dimensional arrays of low-rank, namely low-rank high-order tensors, which have gained significant interest in recent years. 


The question of how to define a low-rank tensor has been of great significance. Perhaps the most agreed upon definition of a tensor rank is the CP-rank \cite{kolda} which, in accordance with the matrix rank, is defined as the minimal number of rank-one tensors  necessary to sum to generate the tensor. Unfortunately, even the task of determining the CP-rank of a tensor is well known to be NP-hard \cite{haastad1989tensor}. An alternative is to use the Tucker rank \cite{kolda}, which for an order-d tensor $\mX\in\reals^{n_1\times\cdots\times n_d}$ is defined as $\rank_{\textnormal{tc}}(\mX)=\left(\rank(\X^{\lbrace1\rbrace}),\ldots,\rank(\X^{\lbrace d\rbrace})\right)$, where $\X^{\lbrace j\rbrace}\in\reals^{n_j\times\Pi_{i\neq j}n_i}$ is the mode-$j$ matricization of $\mX$, which is obtained by arranging the mode-$j$ fibers of $\mX$ as columns of a matrix. Thus, computing the Tucker rank of an order-d tensor merely requires computing the ranks of $d$ matrices. Since the nuclear norm for matrices is a well known convex surrogate for matrix rank, an accepted convex surrogate of the Tucker rank is thus the sum of nuclear norms (SNN) $\sum_{j=1}^d \Vert\X^{\lbrace j\rbrace}\Vert_*$, see for instance \cite{liu2012tensor, tensor1_krank, mu2014square, romera2013new, goldfarb2014robust}. However, as noted in \cite{romera2013new}, while the matrix nuclear norm is the convex envelope for the matrix rank over the ball of matrices with spectral norm at most 1, this correspondence does not hold anymore when considering the sum of nuclear norms $\sum_{j=1}^d \Vert\X^{\lbrace j\rbrace}\Vert_*$ and sum of ranks $\sum_{i=1}^d\rank(\X^{\{i\}})$.
Moreover, from an algorithmic point of view, efficient first-order methods for solving convex relaxations with the SNN are based on variable splitting methods that maintain each of the mode-$j$ matricization $\X^{\{j\}}$, $j=1,\dots,d$, separately \cite{liu2012tensor, tensor1_krank, mu2014square, romera2013new, goldfarb2014robust}. This has two important drawbacks. First, such splitting methods, e.g., ADMM \cite{tensor1_krank}, often suffer from slow convergence rates and are significantly more complex than first-order methods for models which could be solved efficiently without variable splitting or the use of Lagrangian methods. Second, since these methods require to maintain $d$ matrix variables, each of the same size (number of entries) as the original tensor, when the order of the tensor $d$ is not very small, these methods may run in inherently super-linear runtime, which may greatly limit the scale of problems to which they could be applied. 

In the last several years a new tensor product between two tensors called a t-product has been introduced \cite{firstTSVD, highOrderTSVD1}, which is denoted by $\mX*\mY$. Using the t-product, many concepts from matrix algebra can be extended to tensors such as a tensor t-SVD which decomposes a tensor as a product $\mX=\mU*\mS*\mV^{\top}$, where $\mU,\mV$ are orthogonal tensors and $\mS$ is a so-called f-diagonal tensor. Related to the t-SVD are two new tensor ranks, the tubal rank and the average rank. Indeed low tubal rank for instance implies that the tensor could be represented as a collection of low-rank matrices which implies both more efficient storage and more efficient computations with the tensor.
For the case of 3rd-order tensors, building on the t-product, the authors in \cite{tensor_tSVD} derived the corresponding tensor spectral norm and tensor nuclear norm (TNN) and have established that the TNN is the convex envelope of the average rank over the unit ball induced by the tensor spectral norm, in analogy with the matrix nuclear norm which is the convex envelope of the  matrix rank overt the unit spectral norm ball of matrices. These derivations led the authors in \cite{tensor_tSVD} to propose a convex relaxation for the problem of Tensor Robust Principal Component Analysis, based on loss minimization regularized with the TNN, and prove that under certain assumptions (akin to those in its celebrated matrix counterpart \cite{candes2011robust}), it exactly recovers a 3rd order tensor with low tubal rank, from its noisy observation. Similar results were obtained for the problem of exact completion of a 3rd-order tensor with low tubal rank from partial random observations in \cite{tensor_tSVD1}. In \cite{tensor_tSVD2, semerci2014tensor} the authors considered a similar approach for 3rd-order tensor completion and de-noising, and multienergy computed tomography. Interestingly, \cite{tensor_tSVD, tensor_tSVD1, tensor_tSVD2, semerci2014tensor} demonstrated empirically that convex relaxations based on the TNN could be superior in practice (in terms of accuracy) to those based on the SNN approach discussed before. For tensors of order greater than 3 however, the results and theory in the literature are very limited. In \cite{highOrderTSVD1} the authors define many of the high-order tensor concepts, yet the definitions are mostly recursive which, while providing an intuition for the mathematical concepts, is not always sufficient for using in further analyzes. In \cite{highOrderTSVD3, highOrderTSVD2} the authors define the TNN norm for higher order tensors and demonstrate the potential of their methods in numerical examples, however their paper lacks theoretical foundation. 

Inspired by the recent literature on convex relaxations for the recovery of low-rank tensors based on the TNN, throughout this paper we are interested in the following general minimization problem over the unit ball induced by the TNN in the space of order-$d$ (for arbitrary $d\geq 3$) real tensors: 
\begin{align} \label{mainModel}
& \min_{\Vert\mX\Vert_*\le1}f(\mX),
\end{align}
where $f:\reals^{n_1\times\cdots\times n_d} \rightarrow \reals$ is  convex and $\Vert\cdot\Vert_*$ is the TNN which we will be formally defined in the sequel. We will consider the both the case that $f$ is smooth and nonsmooth. 

As we shall see, projecting a tensor onto the unit TNN ball amounts to computing $\Pi_{i=3}^dn_i$ matrix SVDs of matrices of size $n_1\times n_2$. Each such SVD computation  requires in worst cast $\mathcal{O}(\min\lbrace n_1,n_2\rbrace\max\lbrace n_1,n_2\rbrace^2)$ runtime. In particular, this amounts to superlinear runtime in the size of the tensor --- $\Pi_{i=1}^dn_i$, which may greatly limit the application of gradient methods to solve Problem \eqref{mainModel}  even in moderate dimensions.  

Very recently, in a series of works, several authors considered the matrix version of Problem \eqref{mainModel} (or very close variants of it, such as minimization over the set of positive semidefinite matrices with bounded trace) in case a \textit{strict complementarity} (SC) condition holds (or certain relaxed notions of). It was established that SC provably and significantly improves the performance of gradient methods in two central aspects. First, for a very popular structure of the objective, namely when $f(\X) := g(\mA\X) = \langle{\X,\C}\rangle$, and when $g(\cdot)$ is strongly convex and $\mA$ is a linear map (e.g., a least squares problem with $f(\X) := \frac{1}{2}\Vert{\mA\X-\b}\Vert_2^2$), $f(\cdot)$ satisfies a quadratic growth bound over the unit nuclear norm of matrices \cite{ding2020kfw, SpectralFrankWolfe, strictComplementarity2, garberRankOne}. Such a quadratic growth bound is well known to imply, in case $f(\cdot)$ is also smooth, linear convergence rates for standard projected-gradient methods, see for instance \cite{nesterov_LCwithQG}, despite the fact that $f(\cdot)$ need not be strongly convex.  Second, it was established, that in a certain radius of an optimal solution which satisfies SC (the radius depends on the measure of SC, or certain relaxation of), the Euclidean projected gradient mapping admits rank that does not exceed that of the optimal solution. This immediately implies that, at least in certain proximity of an optimal solution, full-rank SVD computations, which are required for the computation of the projection on the the matrix nuclear norm ball and are computationally-prohibitive in high-dimensions, could be replaced with only low-rank SVDs, which could be carried out much more efficiently and lead to dramatic reduction in runtime \cite{garberNuclearNormMatrices, garberStochasticLowRank, ourExtragradient}.

The aim of this work is to study the possible extension of  the results in \cite{ding2020kfw,  garberNuclearNormMatrices, ourExtragradient} mentioned above for  convex optimization under SC over the unit nuclear norm
ball of matrices, to the significantly more challenging corresponding tensor optimization problem \eqref{mainModel}.

The main contributions of this work are as follows:
\begin{itemize}
\item
We derive the strict complementarity condition for the tensor optimization problem  \eqref{mainModel}, in case $f$ is differentiable, and motivate it, similarly to its matrix counterpart \cite{garberNuclearNormMatrices, SpectralFrankWolfe}, by demonstrating that it both holds generically, and that it is related to a certain notion of robustness of Problem  \eqref{mainModel} to miss-specification. See Section \ref{sec:SCdefSmooth}.
\item
For an objective function of the form $f(\mX) = g(\mA\mX)+\langle{\mX,\mC}\rangle$ with $g(\cdot)$ being strongly convex and $\mA$ being a linear map, we establish that under the SC condition, and assuming a unique optimal solution, Problem \eqref{mainModel} satisfies a quadratic growth bound, which under the additional assumption that $f(\cdot)$ is smooth, implies a linear convergence rate for standard gradient methods. See Section \ref{sec:QGtheorem}.

\item
Considering the case that $f(\cdot)$ is smooth, we prove that inside a ball centered at an optimal solution with tubal rank at most $r^*$ which satisfies SC, where the radius of the ball scales linearly with the measure of SC, the projected gradient mapping (i.e., the tensor obtained from a projected gradient step) always admits tubal rank at most $r^*$. This directly implies that  the projection step in projected gradient methods (including accelerated variants), could be implemented by computing $N=\Pi_{i=3}^{d}n_i$ thin SVD computations of $n_1\times n_2$ matrices, where only the top $r^*$ components in each such SVD are computed, as opposed to the worst case in which all components of each SVD may be required for the projection. Importantly, for constant $r^*$ this implies nearly linear (in the size of the tensor) per-iteration runtime for standard projected gradient methods  (in the proximity of the optimal solution), as opposed to super-linear without such assumptions. Moreover, we  provide a precise tradeoff which considers weaker versions of the SC condition, and allows to weaken the initialization requirement (i.e., increase the radius of the ball) in favor of increased tubal rank of the projected gradient mapping, i.e., increased complexity of computing the projection. On the practical side, while verifying if the iterates are indeed the aforementioned ball could be difficult, we show that for any given tensor, it could be easily verified if indeed the projection could be computed using only low-rank SVDs as discussed above, which is all that is needed to verify that the method indeed converges correctly (i.e., as when full-rank SVDs are used for the projection). See Section \ref{sec:lowTubalRankProjectionsSmooth}.
\item
We consider the case  that $f(\cdot)$ is nonsmooth but admits the popular structure $f(\mX) = h(\mX) + \max_{\y\in\mK}\y^{\top}(\mA\X-\b)$, where $h$ is smooth and convex, $\mA$ is a linear map, and $\mK$ is convex, compact, and projection-friendly. We derive the corresponding SC condition for the resulting saddle-point problem and obtain results similar to the previous item for the well-known Extragradient method, which is applicable to smooth convex-concave saddle-point problems. See Section \ref{sec:nonsmoothCase}.
\item
We provide some numerical experiments in support of our theoretical investigations. First, we demonstrate the plausibility of the SC condition for the tasks of low-rank tensor completion and low-rank tensor robust principal component analysis with synthetic data, in a meaningful setting in which indeed the TNN relaxation recovers the ground truth tensor with small error. Second, the experiments demonstrate our theoretical findings  regarding linear convergence rates under SC (for low-rank tensor completion), and that for both tasks, using very simple initializations, already from a very early stage of the run, the exact projection in the optimization algorithm could be computed using SVDs of rank that matches the tubal rank of the ground truth tensor, instead of using full-rank SVDs, as required in worst case.  See Section \ref{sec:experiments}.
\item
As an additional contribution, one which may be of independent interest, we rigorously extend many basic results regarding tensors and the t-product, that were previously rigorously obtained only for 3rd-order tensors, to arbitrary-order tensors. 
\end{itemize}

Table \ref{table:bottomLine} below present some concrete algorithmic implications of our work, demonstrating the substantial improvements in worst case complexity of several highly popular first-order methods for Problem \eqref{mainModel}, under the strict complementarity condition.

\begin{table*}[!htb]\renewcommand{\arraystretch}{1.3}
{\footnotesize
\begin{center}
  \begin{tabular}{| l | c | c | c | c | c | c | } \hline 
  & \multicolumn{3}{c|}{ without SC } & \multicolumn{3}{c|}{SC holds an for optimal solution with tubal rank at most $r$} \\ \cline{2-7}
   & conv. & SVD & memory &  conv. rate  & SVD rank & memory \\
  & rate & rank & & & (from warm-start init.) &  (from warm-start init.) \\ \hline
  \multicolumn{7}{|c|}{$f$ is $\beta$-smooth} \\ \hline
  PGD & $\beta/\varepsilon$  &  $n$ & $n^d$  & $\beta/\varepsilon$ &  $r$ & $rn^{d-1}$ \\ \hline
    \multicolumn{7}{|c|}{$f$ is $\beta$-smooth and there exists a unique optimal solution} \\ \hline
  AGD  & $\sqrt{\beta/\varepsilon}$ & $n$ & $n^d$  & $\sqrt{\beta/\varepsilon}$ & $r$ & $rn^{d-1}$ \\ \hline
    \multicolumn{7}{|c|}{$f$ is $\beta$-smooth of the form: $f(\mX)=g(\mA(\mX))+\langle\mC,\mX\rangle$ for strongly convex $g$, }
    \\ \multicolumn{7}{|c|}{ linear map $\mA$, and there exists a unique optimal solution} \\ \hline
  PGD & $\beta/\varepsilon$  &  $n$ & $n^d$  & $(\beta/\gamma)\log(\beta/\varepsilon)$ &  $r$ & $rn^{d-1}$ \\ \hline
  RFG  & $\sqrt{\beta/\varepsilon}$ & $n$ & $n^d$  & $\sqrt{\beta/\gamma}\log(1/\varepsilon)$ & $r$ & $rn^{d-1}$ \\ \hline
  \multicolumn{7}{|c|}{$f$ is nonsmooth of the form: $f(\mX) = h(\mX) + \max_{\y\in\mK}\y^{\top}(\mA(\mX)-\b)$ } 
     \\ \multicolumn{7}{|c|}{for smooth and convex $h$, linear map $\mA$, and convex and compact $\mK$} \\ \hline
  EG  & $\beta_{\textnormal{NS}}/\varepsilon$ & $n$ & $n^d$  & $\beta_{\textnormal{NS}}/\varepsilon$ & $r$ & $rn^{d-1}$ \\ \hline
   \end{tabular}
   \caption{Some algorithmic consequences of our results for the projected gradient descent method (PGD), Nesterov's accelerated gradient  method (AGD), Nesterov's restarted fast gradient method (RFG), and the projected extragradient decent method (EG). For simplicity we focus on tensors of equal dimensions, i.e., $n_i = n$ for all $i=1,\dots,d$. The convergence rate columns refer to the worst case number of iterations to guarantee $\varepsilon$ approximation error w.r.t. function value, the  SVD rank columns refer to the rank of the ($n\times n$ matrix) SVDs required in each iteration to compute the projection onto the unit TNN ball, and the memory columns refer to the memory required by the methods,  excluding gradient computations. The improved SVD rank and memory results  under SC are guaranteed to hold only when the methods are initialized sufficiently close to the optimal solution.   
   We omit all constants except for $\varepsilon$, the smoothness parameter $\beta$ or $\beta_{\textnormal{NS}}$ which is the effective smoothness parameter for the corresponding saddle-point problem, and the quadratic growth parameter $\gamma$, which is guaranteed to be positive. 
  }\label{table:bottomLine}
\end{center}
}
\vskip -0.2in
\end{table*}\renewcommand{\arraystretch}{1}

\subsection{Additional related work}

Our work concerns the convex problem \eqref{mainModel}, for which arguing about convergence to the optimal solution is straightforward, can handle quite general objective functions and is in particular free of any specific statistical model. In recent years there have been many works on recovering low-rank tensors (either with low CP rank, low tucker rank, or low tubal rank) which are based  on alternative nonconvex approaches. These could be divided into two groups. The first, concerns a specific task, e.g., tensor completion or tensor robust principal component analysis, under a specific statistical model, and provides convergence guarantees for the ground truth tensor which hold only in specific setting considered, see for instance   \cite{nonconvexTensors_cp1, nonconvexTensors_cp2,  nonconvexTensors_cp3, nonconvexTensors_cp4, nonconvexTensors_tucker2, nonconvexTensors_tucker1, pmlr-v162-qiu22d}.
The second group of works does not assume a specific statistical model, and consider either the use of nonconvex regularizers for low rank, e.g., \cite{zhang2018nonconvex, wang2021generalized, chen2020robust}, or consider working explicitly with a factorization of the low-rank tensor (either based on the Tucker decomposition or the t-SVD), which allows to always maintain only low-rank tensors throughout the run of the algorithm, e.g., \cite{zhou2017tensor, kressner2014low, goldfarb2014robust}. However, in all of these works, only convergence to critical points is established, or not at all.

\subsection{Organization}
For  ease of presentation all proofs are deferred to the appendix. For the main results, a sketch of the proof is provided in the main body of the paper.

\subsection{Notation}
We denote by $\reals$  the real numbers, and by $\complex$ denotes the complex numbers.  $\mathbb{S}^{n}$ denotes the space of real symmetric $n\times n$ matrices, and $\mathbb{H}^n$ denotes the space of Hermitian matrices of size $n\times n$.
We denote column vectors via lowercase boldface letters, e.g., $\x$, matrices via capital boldface letters, e.g.,  $\X$, and tensors via capital boldface calligraphic letters, e.g., $\mX$. We denote by $\I_n$ the identity matrix of size $n\times n$. For any matrix $\X\in\complex^{m\times n}$ we denote its conjugate-transpose as $\X^{\tc}\in\complex^{n\times m}$. For any tensor $\mX\in\complex^{n_1\times\cdots\times n_d}$ we denote by $\conj(\mX)$ the operator that takes the complex conjugate of each entry of $\mX$. For a matrix $\X\in\complex^{m\times n}$ we let $\sigma_i(\X)$ denote its $i$th largest singular value, and similarly, for a Hermitian matrix $\X\in\mathbb{H}^n$, we let $\lambda_i(\X)$ denote its $i$th largest (signed) eigenvalue. We let $\#\sigma_i(\X)$ denote the multiplicity of the $i$th largest singular value of $\X$. We denote by $[n]$ the set $[n]=\lbrace1,\ldots,n\rbrace$. 
The product $\cdot$ denotes standard matrix multiplication operation between two matrices. The product $\otimes$ denotes the Kronecker product between two matrices. For a tensor $\mX\in\complex^{n_1\times\cdots\times n_d}$ we denote its frontal slices as $\X^{(i_3,\ldots,i_d)}=\mX(:,:,i_3,\ldots,i_d)\in\complex^{n_1\times n_2}$, for any $i_3\in[n_3],\ldots,i_d\in[n_d]$. For a fixed tensor space $\reals^{n_1\times\dots\times{}n_d}$ or $\complex^{n_1\times\dots\times{}n_d}$, we let  $N:=n_3\cdots n_d$ to be the multiplication of all but the first and second dimension. The inner product between  two tensor $\mX,\mY\in\reals^{n_1\times\dots\times{}n_d}$ is defined as
$\langle\mX,\mY\rangle := \sum_{i_d=1}^{n_d}\cdots\sum_{i_1=1}^{n_1} \mX(i_1,\ldots,i_d)\mY(i_1,\ldots,i_d)$, and the Frobenius norm of a tensor is defined as $\Vert\mX\Vert_F=\sqrt{\sum_{i_d=1}^{n_d}\cdots\sum_{i_1=1}^{n_1} \mX(i_1,\ldots,i_d)^2}$. Finally, for a set $S$ we let $\#{}S$ denote its cardinality.


\section{Tensor Preliminaries}
\label{sec:tensorPreliminaries}

In this section we review the basic necessary definitions and results regarding high-order tensors that will be used throughout this paper. Along the way we also rigorously extend previous results that, to the best of our knowledge, were previously obtained only for 3rd-order tensors. Some of these results are based on \cite{highOrderTSVD1}. Additional relevant results have previously appeared in \cite{highOrderTSVD3, highOrderTSVD2}, however they lacked justifications and proofs.

We begin by defining the block circulant matrix which represents a tensor as a matrix while preserving some of the important structure of the tensor. Using the block circulant matrix we will define the t-product between two tensors. This will enable us to define the t-SVD factorization of a tensor. We will also define the Fourier transformation of a tensor and its important connections to the t-product and t-SVD. We will then present the two notions of tensor rank which are central to this work: the tensor average rank and the tensor tubal rank.

\begin{definition}[block-circulant matrix of $3$rd order tensors \cite{firstTSVD}]
Let $\mX\in\complex^{n_1\times n_2\times n_3}$. The block-circulant matrix of $\mX$ is defined as
\begin{align*}
\bcirc(\mX) := \left[\begin{array}{cccc}
\X^{(1)} & \X^{(n_3)} & \cdots & \X^{(2)}
\\ \X^{(2)} & \X^{(1)} & \cdots & \X^{(3)} 
\\ \vdots & \vdots & \ddots & \vdots
\\ \X^{(n_3)} & \X^{(n_3-1)} & \cdots & \X^{(1)} \end{array}\right]\in\complex^{n_1n_3\times n_2n_3}
\end{align*}
where $\X^{(i)}$, $i=1,\dots,n_3$, are the frontal slices of $\mX$.
\end{definition}

Moving to higher-order tensors, the block-circulant matrix generated from a tensor $\mX\in\complex^{n_1\times\cdots\times n_d}$ can be thought of as recursively unfolding each dimension of the tensor in a block circulant pattern. 
At the initial step the last dimension is unfolded in a block circulant pattern which returns a block-circulant tensor of size ${n_1n_d\times n_2n_d\times n_3\cdots\times n_{d-1}}$, where each block is a tensor of order $(d-1)$ and size $n_1\times\cdots\times{}n_{d-1}$.  At each successive level, the last dimension of each block created in the previous level is unfolded in the same way. At the base level, the $n_1\times n_2$ sized matrices generated from the first two dimensions of the tensor are placed in a block circulant pattern in each of the blocks of the previous level. 
Formally, the block-circulant matrix is defined as follows. 
\begin{definition}[block-circulant matrix \cite{Khoromskaia2017BlockCA, davis1979circulant}]
Let $\mX\in\complex^{n_1\times\cdots\times n_d}$. The block-circulant matrix of $\mX$ is defined as
\begin{align} \label{def:bcirc}
\bcirc(\mX) = \sum_{i_d=1}^{n_d}\cdots\sum_{i_3=1}^{n_3} \pi_{n_d}^{i_d-1} \otimes \cdots \otimes\pi_{n_3}^{i_3-1} \otimes \mX{(:,:,i_3,\ldots,i_d)}\in\complex^{n_1n_3\cdots n_d\times n_2n_3\cdots n_d},
\end{align}
where for any $n\ge1$, $\pi_n$ denotes the periodic downward shift permutation matrix which can be written as
\begin{align} \label{def:shiftPermutationMatrix}
\pi_n := \left[\begin{array}{cccccc}
0 & 0 & \cdots & 0 & 0 & 1
\\ 1 & 0 & \cdots & 0 & 0 & 0 
\\ 0 & 1 & \cdots & 0 & 0 & 0
\\ \vdots & \ddots & \ddots & \ddots & \vdots & \vdots
\\ 0 & 0 & \ddots & 1 & 0 & 0
\\ 0 & 0 & 0 & \cdots & 1 & 0
\end{array} \right] \in \reals^{n\times n}.
\end{align}
\end{definition}

Taking a power of the periodic downward shift permutation matrix permutes its columns. For $\pi_n\in \reals^{n\times n}$ the matrices $\pi_n^{0},\pi_n^{1},\ldots,\pi_n^{n-1}$ are all orthogonal to each other and their sum is the all ones matrix of size $n\times n$.

To better grasp the structure of the $\bcirc$ matrix, we demonstrate it for a $4$th order tensor $\mX\in\complex^{n_1\times n_2\times 3\times 3}$. For $i_3=3$, $i_4=1$ we will calculate 
\begin{align*}
& \pi_{3}^{0} \otimes \pi_{3}^{2} \otimes \mX{(:,:,3,1)} 
\\ & = 
\left[\begin{array}{ccc}
1 & 0 & 0
\\ 0 & 1 & 0 
\\ 0 & 0 & 1
\end{array} \right]
\otimes  
\left[\begin{array}{ccc}
0 & 1 & 0
\\ 0 & 0 & 1 
\\ 1 & 0 & 0
\end{array} \right]
\otimes \X^{(3,1)}
\\ &  = \left[
\begin{array}{c|c|c}
\begin{array}{ccc}
\mathbf{0} & \X^{(3,1)} & \mathbf{0}
\\ \mathbf{0} & \mathbf{0} & \X^{(3,1)}
\\ \X^{(3,1)} & \mathbf{0} & \mathbf{0}
\end{array}
& 
\begin{array}{c}
\mathbf{0}
\end{array}
& 
\begin{array}{c}
\mathbf{0}
\end{array}
\\ \hline
\begin{array}{c}
\mathbf{0}
\end{array}
& 
\begin{array}{ccc}
\mathbf{0} & \X^{(3,1)} & \mathbf{0}
\\ \mathbf{0} & \mathbf{0} & \X^{(3,1)}
\\ \X^{(3,1)} & \mathbf{0} & \mathbf{0}
\end{array}
&
\begin{array}{c}
\mathbf{0}
\end{array}
\\ \hline
\begin{array}{c}
\mathbf{0}
\end{array}
&
\begin{array}{c}
\mathbf{0}
\end{array}
&
\begin{array}{ccc}
\mathbf{0} & \X^{(3,1)} & \mathbf{0}
\\ \mathbf{0} & \mathbf{0} & \X^{(3,1)}
\\ \X^{(3,1)} & \mathbf{0} & \mathbf{0}
\end{array}
\end{array}
\right],
\end{align*}
where $\X^{(i,j)}$ denotes the ${i,j}^{th}$ frontal slice of $\mX$, i.e., $\X^{(i,j)}=\mX{(:,:,i,j)}$.
Summing all the indexes as in \eqref{def:bcirc} we obtain the full $\bcirc(\mX)$ matrix, which can be written as
\begin{align*}
& \bcirc(\mX)=
\\ & \left[
\begin{array}{c|c|c}
\begin{array}{ccc}
\X^{(1,1)} & \X^{(3,1)} & \X^{(2,1)}
\\ \X^{(2,1)} & \X^{(1,1)} & \X^{(3,1)}
\\ \X^{(3,1)} & \X^{(2,1)} & \X^{(1,1)}
\end{array}
& 
\begin{array}{ccc}
\X^{(1,3)} & \X^{(3,3)} & \X^{(2,3)}
\\ \X^{(2,3)} & \X^{(1,3)} & \X^{(3,3)}
\\ \X^{(3,3)} & \X^{(2,3)} & \X^{(1,3)}
\end{array}
& 
\begin{array}{ccc}
\X^{(1,2)} & \X^{(3,2)} & \X^{(2,2)}
\\ \X^{(2,2)} & \X^{(1,2)} & \X^{(3,2)}
\\ \X^{(3,2)} & \X^{(2,2)} & \X^{(1,2)}
\end{array}
\\ \hline
\begin{array}{ccc}
\X^{(1,2)} & \X^{(3,2)} & \X^{(2,2)}
\\ \X^{(2,2)} & \X^{(1,2)} & \X^{(3,2)}
\\ \X^{(3,2)} & \X^{(2,2)} & \X^{(1,2)}
\end{array}
& 
\begin{array}{ccc}
\X^{(1,1)} & \X^{(3,1)} & \X^{(2,1)}
\\ \X^{(2,1)} & \X^{(1,1)} & \X^{(3,1)}
\\ \X^{(3,1)} & \X^{(2,1)} & \X^{(1,1)}
\end{array}
&
\begin{array}{ccc}
\X^{(1,3)} & \X^{(3,3)} & \X^{(2,3)}
\\ \X^{(2,3)} & \X^{(1,3)} & \X^{(3,3)}
\\ \X^{(3,3)} & \X^{(2,3)} & \X^{(1,3)}
\end{array}
\\ \hline
\begin{array}{ccc}
\X^{(1,3)} & \X^{(3,3)} & \X^{(2,3)}
\\ \X^{(2,3)} & \X^{(1,3)} & \X^{(3,3)}
\\ \X^{(3,3)} & \X^{(2,3)} & \X^{(1,3)}
\end{array}
&
\begin{array}{ccc}
\X^{(1,2)} & \X^{(3,2)} & \X^{(2,2)}
\\ \X^{(2,2)} & \X^{(1,2)} & \X^{(3,2)}
\\ \X^{(3,2)} & \X^{(2,2)} & \X^{(1,2)}
\end{array}
&
\begin{array}{ccc}
\X^{(1,1)} & \X^{(3,1)} & \X^{(2,1)}
\\ \X^{(2,1)} & \X^{(1,1)} & \X^{(3,1)}
\\ \X^{(3,1)} & \X^{(2,1)} & \X^{(1,1)}
\end{array}
\end{array}
\right]
\in\complex^{9n_1\times 9n_2}.
\end{align*}

The t-product between two tensors is based on the $\fold(\cdot)$ and $\unfold(\cdot)$ operators. The $\unfold(\cdot)$ operator lays out all frontal slices of the tensor to create a vertical block vector, as defined in \cite{firstTSVD}. This is equal to the first block column of the block-circulant matrix generated from the tensor.
\begin{definition}[fold and unfold operators of a tensor \cite{unfold_def,firstTSVD}]
Let $\mX\in\complex^{n_1\times n_2\times \cdots\times n_d}$.
The $\unfold(\cdot)$ operator is defined as
$$\unfold(\mX)=\bcirc(\mX)\C_1\in\complex^{n_1n_3\cdots n_d\times n_2},$$ where $\C_1$ is the matrix
\begin{align*}
\C_1:=\left[\begin{array}{cc}
\I_{n_2} 
\\ \textnormal{\textbf{0}}
\end{array}\right]\in\reals^{n_2n_3\cdots n_d\times n_2}
\end{align*} 
and $\fold(\cdot)$ is the inverse operator such that
$$\fold(\unfold(\mX))=\mX.$$ 
\end{definition}

\begin{definition}[T-product \cite{firstTSVD}]
Let $\mX\in\complex^{n_1\times n_2\times \cdots\times n_d}$ and $\mY\in\complex^{n_2\times \ell\times n_3\times \cdots\times n_d}$. Then the T-product $\mX*\mY$ is defined as
\begin{align*}
\mX*\mY = \fold(\bcirc(\mX)\cdot\unfold(\mY)) \in \complex^{n_1\times \ell \times n_3\times\cdots\times n_d}.
\end{align*}
\end{definition}

\begin{definition}[identity tensor \cite{highOrderTSVD1}]
The identity tensor $\mI\in \reals^{n\times n\times n_3\times\cdots\times n_d}$ is the tensor such that $\mI(:,:,1,\ldots,1) = \I_n$, and all other entries are zero.
\end{definition}

We define the transpose and conjugate transpose operation in a recursive manner.
\begin{definition}[transpose of a tensor \cite{highOrderTSVD1}] 
The transpose of a $3$rd order tensor $\mX\in\reals^{n_1\times n_2\times n_3}$ is the tensor $\mX^{\top}\in\reals^{n_2\times n_1\times n_3}$ obtained by transposing each frontal slice $\mX(:,:,i_3)$ for $i_3\in[n_3]$ and then reversing the order of $\mX^{\top}(:,:,2)$ through $\mX^{\top}(:,:,n_3)$. 
\\The transpose of a tensor $\mX\in \reals^{n_1\times n_2\times n_3\cdots\times n_d}$ is the tensor $\mX^{\top}\in \reals^{n_2\times n_1\times n_3\cdots\times n_d}$ obtained by recursively tensor transposing each slice $\mX(:,\ldots,:,i_d)$ for $i_d\in[n_d]$ and then reversing the order of $\mX^{\top}(:,\ldots,:,2)$ through $\mX^{\top}(:,\ldots,:,n_d)$. 
\end{definition}
\begin{definition}[conjugate transpose of a tensor] 
The conjugate transpose of a $3$rd order tensor $\mX\in\complex^{n_1\times n_2\times n_3}$ is the tensor $\mX^{\tc}\in\complex^{n_2\times n_1\times n_3}$ obtained by conjugate transposing each frontal slice $\mX(:,:,i_3)$ for $i_3\in[n_3]$ and then reversing the order of $\mX^{\tc}(:,:,2)$ through $\mX^{\tc}(:,:,n_3)$ \cite{tensor_tSVD}. 
\\ The conjugate transpose of a tensor $\mX\in \complex^{n_1\times n_2\times n_3\cdots\times n_d}$ is the tensor $\mX^{\tc}\in \complex^{n_2\times n_1\times n_3\cdots\times n_d}$ obtained by recursively tensor conjugate transposing each slice $\mX(:,\ldots,:,i_d)$ for $i_d\in[n_d]$ and then reversing the order of $\mX^{\tc}(:,\ldots,:,2)$ through $\mX^{\tc}(:,\ldots,:,n_d)$.
\end{definition}

%

\begin{definition}[orthogonal tensor \cite{highOrderTSVD1}] 
A tensor $\mQ\in \complex^{n\times n\times n_3\times\cdots\times n_d}$ is orthogonal if $\mQ^{\tc}*\mQ=\mQ*\mQ^{\tc}=\mI$.
\end{definition}

\begin{definition}[f-diagonal tensor \cite{firstTSVD}]
A tensor $\mX\in \reals^{n_1\times\cdots\times n_d}$ is f-diagonal if $\mX(:,:,i_3,\ldots,i_d)$ is a diagonal matrix for all $i_3\in[n_3],\ldots,i_d\in[n_d]$.
\end{definition}

In the following lemma we generalize to higher-order tensors the result in \cite{tensor_tSVD}, which considered only $3$rd order tensors, that any real-valued tensor admits a t-SVD factorization. In \cite{highOrderTSVD1} the authors also derive a t-SVD factorization for higher-order tensor, however they do not show that the tensors $\mU,\mS,\mV$ in the factorization are real-valued as required. A proof of the lemma is given in \cref{sec:App:proofLemma51}.
\begin{lemma}[T-SVD] \label{lemma:tSVD}
Let $\mX\in\reals^{n_1\times \cdots\times n_d}$. Then, it can be factorized as
\begin{align*}
\mX = \mU*\mS*\mV^{\top},
\end{align*}
where $\mU\in\reals^{n_1\times n_1\times n_3\times \cdots\times n_d}$ and $\mV\in\reals^{n_2\times n_2\times n_3\times \cdots\times n_d}$ are orthogonal, and  $\mS\in\reals^{n_1\times \cdots\times n_d}$ is an f-diagonal tensor.
\end{lemma}


Computing the t-product explicitly using the block circulant matrices can be computationally expensive in terms of both runtime and memory. However, block circulant matrices have an important property that they can be block-diagonalized via the discrete Fourier transformation. Using the corresponding block diagonal matrix for the computations instead of the block circulant matrix, can significantly reduce the runtime and  memory requirements in many cases. In addition, the Fourier transform will enable us to constructively prove \cref{lemma:tSVD} by computing standard matrix SVDs in the Fourier domain, and thereby to derive a simple algorithm for computing the t-SVD of a tensor.

\begin{definition} [the discrete Fourier transformation of a tensor along the $j^{th}$ dimension and its inverse]  \label{def:fourierJdim}
Let $\mX\in\complex^{n_1\times \cdots\times n_d}$. For $j\in\lbrace3,\ldots,d\rbrace$ we define $\fft_j(\mX)\in\mathbb{C}^{n_1\times \cdots\times n_d}$ --- the discrete Fourier transformation of $\mX$ along the $j^{th}$ dimension, i.e., 
for every $i_1\in[n_1],\ldots,i_{j-1}\in[n_{j-1}],i_{j+1}\in[n_{j+1}],\ldots,i_d\in[n_d]$,
\begin{align*}
\left(\fft_j(\mX)\right)(i_1,i_2,\ldots,i_{j-1},:,i_{j+1},\ldots,i_d)=\F_{n_j} \cdot \mX(i_1,i_2,\ldots,i_{j-1},:,i_{j+1},\ldots,i_d),
\end{align*}
where
\begin{align*}
\F_{n_j} := \left[\begin{array}{ccccc}
1 & 1 & 1 & \cdots & 1
\\ 1 & \omega & \omega^2 & \cdots & \omega^{n_j-1} 
\\ \vdots & \vdots & \vdots & \ddots & \vdots
\\ 1 & \omega^{n_j-1} & \omega^{2(n_j-1)} & \cdots & \omega^{(n_j-1)(n_j-1)} \end{array}\right] \in\mathbb{C}^{n_j\times n_j},
\qquad\omega=\textrm{e}^{-\frac{2\pi i}{n_j}}.
\end{align*}
The inverse discrete Fourier transformation along the $j^{th}$ dimension of $\mX$, which we denote by $\ifft_j(\mX)\in\mathbb{C}^{n_1\times \cdots\times n_d}$, is defined as computing 
for every $i_1\in[n_1],\ldots,i_{j-1}\in[n_{j-1}],i_{j+1}\in[n_{j+1}],\ldots,i_d\in[n_d]$
\begin{align*}
\left(\ifft_j(\mX)\right)(i_1,i_2,\ldots,i_{j-1},:,i_{j+1},\ldots,i_d)=\F_{n_j}^{-1} \cdot \mX(i_1,i_2,\ldots,i_{j-1},:,i_{j+1},\ldots,i_d),
\end{align*}
where $\F_{n_j}^{-1}=\frac{1}{n_j}\F_{n_j}^{\tc}$.
\end{definition}

\begin{definition}[the discrete Fourier transformation of tensor and its inverse] \label{def:fft}
Let $\mX\in\complex^{n_1\times \cdots\times n_d}$. The Fourier transform of $\mX$ (along all but the first two dimensions) is defined as
\begin{align*}
\overbar{\mX}=\fft(\mX):=\fft_d(\cdots(\fft_4(\fft_3(\mX)))),
\end{align*}
and the inverse discrete Fourier transformation of $\mX$ (along all but the first two dimensions) is defined as
\begin{align*}
\ifft(\mX):=\ifft_3(\ifft_4(\cdots(\ifft_d(\mX)))).
\end{align*}
\end{definition}

\begin{definition}[block diagonal matrix of a $3$rd order tensor \cite{firstTSVD}]
Let $\mX\in\complex^{n_1\times n_2\times n_3}$. Then the matrix $\barX \in \complex^{n_1n_3\times n_2n_3}$ is the block diagonal matrix defined as
\begin{align*}
\barX= \bdiag(\overbar{\mX}) := \left[\begin{array}{cccc}
\barX^{(1)} &  &  & \\  & \barX^{(2)} &  &  
\\  &  & \ddots & 
\\  &  &  & \barX^{(n_3)} \end{array}\right],
\end{align*}
where $\barX^{(i)}$, $i=1,\dots,n_3$, are the frontal slices of $\overbar{\mX}$.
\end{definition}

The $\bdiag(\cdot)$ operator be easily extended to high-order tensors by placing all frontal slices of the tensor $\overbar{\mX}$ as blocks along a large block diagonal matrix in a colexicographical order. Some papers place the blocks in a different order, which is allowed as long as also the order of the product of the periodic downward shift permutation matrices in \eqref{def:bcirc} used to define the block-circulant matrix is set  accordingly.
\begin{definition}[block diagonal matrix of an order-d tensor]
Let $\mX\in\complex^{n_1\times\cdots\times n_d}$. The matrix $\barX \in \complex^{n_1n_3\cdots n_d\times n_2n_3\cdots n_d}$ is the block diagonal matrix defined as
\begin{align*}
& \barX = \bdiag(\overbar{\mX}) :=
\matbdiag\left(\left(\barX^{(i_3,\ldots,i_d)}\right)_{i_d\in[n_d]
, \ldots , i_3\in[n_3]}\right),
\end{align*}
where $\barX^{(i_3,\ldots,i_d)} = \overbar{\mX}(:,:,i_3,\ldots,i_d)$ denotes the frontal slice of $\overbar{\mX}$ for a set of indexes $i_3\in[n_3],\ldots,i_d\in[n_d]$, and $\matbdiag(\cdot)$ is an operator which maps a set of matrices to the diagonal blocks of a block diagonal matrix in a colexicographical order, i.e, for any $(i_3,k_3)\in[n_3]^2,\ldots,(i_d,k_d)\in[n_d]^2$, it places $\barX^{(i_3,\ldots,i_d)}$ before $\barX^{(k_3,\ldots,k_d)}$ if $i_j<k_j$ for the last index $j\in\lbrace3,\ldots,d\rbrace$ for which $i_j\not=k_j$.
\end{definition}

For example, the block diagonal matrix $\barX=\bdiag(\overbar{\mX})$ of a $4$th-order tensor $\mX\in\complex^{n_1\times n_2\times 3\times 3}$ is of the form
\begin{align*}
\barX = \left[\begin{array}{ccccccccc}
\barX^{(1,1)} &  &  & &  & & & & 
\\  & \barX^{(2,1)} &  & &  & & & &  
\\  &  & \barX^{(3,1)} &  &  & & & &  
\\  &  & & \barX^{(1,2)} &  & & & &  
\\  &  & &  & \barX^{(2,2)} & & & &  
\\  &  & &  &  & \barX^{(3,2)} & & &  
\\  &  & &  &  & & \barX^{(1,3)} & & 
\\  &  & &  &  & & & \barX^{(2,3)} &  
\\  &  & &  &  & & & & \barX^{(3,3)}
\end{array}\right].
\end{align*}

The following lemma (see also \cite{Khoromskaia2017BlockCA} which considered only on the case where $n_1=n_2$) connects between a tensor $\mX$ and the block diagonal matrix $\barX$ generated from the Fourier transform of $\mX$. It will be useful later on for proving that the t-product can be computed in the Fourier domain via standard matrix multiplications, which can be done more efficiently in some cases. The proof, which is an extension of a corresponding result derived only for $3$rd order tensors in \cite{tensor_tSVD}, is given in \cref{sec:App:proofLemma1}.
\begin{lemma} \label{lemma:bdiag_bcirc_equality}
Let $\mX\in\complex^{n_1\times \cdots\times n_d}$. 
Then, $\barX=\bdiag(\fft(\mX))$ if and only if
\begin{align*}
\bcirc(\mX) = (\F_{n_d}^{-1}\otimes\F_{n_{d-1}}^{-1}\otimes\cdots\otimes\F_{n_3}^{-1}\otimes\I_{n_1}) \cdot \barX \cdot (\F_{n_d}\otimes\F_{n_{d-1}}\otimes\cdots\otimes\F_{n_3}\otimes\I_{n_2}).
\end{align*}
\end{lemma}

In the following lemma we present the conjugate-complex symmetry condition of real-valued tensors. This condition will enable us to ensure the tensors we work with will all be real-valued as we desire. The proof is given in \cref{sec:App:proofLemma2}.  This is an extension to the conjugate-complex symmetry condition for a $3$rd order tensor $\mX\in\reals^{n_1\times n_2\times n_3}$, presented in \cite{tensor_tSVD}, which states that $\mX$ is real-valued if and only if 
$$\overbar{\mX}(:,:,1)\in\reals,\qquad \conj(\overbar{\mX}(:,:,i_3))=\overbar{\mX}(:,:,n_3-i_3+2),\ i_3\in\left\lbrace1,\ldots,\left\lceil(n_3+1)/2\right\rceil\right\rbrace.$$
\begin{lemma} \label{lemma:conjegateComplexSymmetry_realTensors}
A tensor $\mX\in\reals^{n_1\times\cdots\times n_d}$ is real-valued if and only if $\overbar{\mX}=\fft(\mX)$ satisfies the conjugate-complex symmetry condition
\begin{align} \label{eq:conjugateComplexSymmetryCondition}
\overbar{\mX}(:,:,i_3,\ldots,i_d) = \conj(\overbar{\mX}(:,:,i'_3,\ldots,i'_d))
\end{align}
for all $i_3\in\left\lbrace1,\ldots,\left\lceil\frac{n_3+1}{2}\right\rceil\right\rbrace,\ldots,i_d\in\left\lbrace1,\ldots,\left\lceil\frac{n_d+1}{2}\right\rceil\right\rbrace$, where for all $j\in\lbrace3,\ldots,d\rbrace$, 
\begin{align} \label{def:i_j_tag}
i'_j=\bigg\lbrace\begin{array}{ll}
1, & i_j=1
\\ n_j-i_j+2, & i_j\in\left\lbrace2,\ldots,\left\lceil\frac{n_j+1}{2}\right\rceil\right\rbrace
\end{array}.
\end{align}
\end{lemma}


The following lemma connects between the inner product and Frobenius norm for tensors, and the corresponding block diagonal matrices generated from their Fourier transforms. The proof is given in \cref{sec:App:proofLemma3}.
\begin{lemma} \label{lemma:orderPinnerProductEquivalence}
Let $\mX,\mY\in\complex^{n_1\times\cdots\times n_d}$. Then,
\begin{align*}
(i) \  \langle\mX,\mY\rangle  = \frac{1}{n_3\cdots n_d}\langle\barX,\barY\rangle,
\qquad\qquad (ii) \  \Vert\mX\Vert_F = \frac{1}{\sqrt{n_3\cdots n_d}}\Vert\barX\Vert_F.
\end{align*}
\end{lemma}

The following lemma was proved in \cite{unfold_def} for $3$rd order tensors and we extend it to higher order tensors. It establishes that the t-product between two tensors is equivalent to standard matrix multiplication in the Fourier domain. The proof is given in \cref{sec:App:proofLemma5}.
\begin{lemma} \label{lemma:tproduct_bdiagMult}
Let $\mX\in\complex^{n_1\times n_2\times \cdots\times n_d}$ and $\mY\in\complex^{n_2\times \ell\times n_3\times \cdots\times n_d}$ and denote $\barX=\bdiag(\overbar{\mX})$ and $\barY=\bdiag(\overbar{\mY})$. Then $\mZ=\mX*\mY$ if and only if $\barZ=\barX\barY$ where $\barZ=\bdiag(\overbar{\mZ})$.
\end{lemma}


Equipped with the Fourier transform of a tensor and the connections between the t-product and standard matrix multiplication in the Fourier domain, and the conjugate-complex symmetry condition, we can present in \cref{alg:tSVD} the algorithm for computing the t-SVD of a tensor. The algorithm generalizes the one given in \cite{tensor_tSVD} only for the case of 3rd order tensors. The derivation of the algorithm follows from the constructive proof of \cref{lemma:tSVD} which is given in \cref{sec:App:proofLemma51}.
\begin{algorithm}
	\caption{T-SVD}\label{alg:tSVD}
	\begin{algorithmic}
		\STATE \textbf{Input:} $\mX\in\reals^{n_1\times\cdots\times n_d}$
		\STATE Compute $\overbar{\mX}=\fft(\mX)$
         \FOR {$i_3\in[n_3],\ldots,i_d\in[n_d]$}
         \STATE Denote 
$ i'_j=\bigg\lbrace\begin{array}{ll}
1, & i_j=1
\\ n_j-i_j+2, & i_j\in\left\lbrace2,\ldots,\left\lceil\frac{n_j+1}{2}\right\rceil\right\rbrace
\end{array}$,
$j\in\lbrace3,\ldots,d\rbrace$
        		\IF {SVD of $\barX^{(i_3',\ldots,i_d')}$ was not computed}
        		\STATE Compute SVD of $\barX^{(i_3,\ldots,i_d)}$: $\barX^{(i_3,\ldots,i_d)}=\barU^{(i_3,\ldots,i_d)}\barS^{(i_3,\ldots,i_d)}{\barV^{(i_3,\ldots,i_d)}}{}^{\tc}$            
        		\ELSE
        		\STATE $\barX^{(i_3,\ldots,i_d)}=\conj(\barU^{(i_3',\ldots,i_d')})\barS^{(i_3',\ldots,i_d')}{\conj(\barV^{(i_3',\ldots,i_d')}}){}^{\tc}$
        		\ENDIF
        \ENDFOR
		\STATE Compute $\mS = \ifft(\overbar{\mS})$, $\mU = \ifft(\overbar{\mU})$, $\mV^{\top} = \ifft(\overbar{\mV}^{\tc})$
	\end{algorithmic}
\end{algorithm}

We now turn to present  two notions of tensor rank which will play a crucial part in this work.

\begin{definition}[tensor average rank \cite{tensor_tSVD}]
Let $\mX\in\reals^{n_1\times\cdots\times n_d}$. Then the tensor average rank is defined as
\begin{align*}
\rank_{\textnormal{a}}(\mX)=\frac{1}{N}\rank(\bcirc(\mX)).
\end{align*}
\end{definition}

Note that the matrices $(\F_{n_d}^{-1}\otimes\F_{n_{d-1}}^{-1}\otimes\cdots\otimes\F_{n_3}^{-1}\otimes\I_{n_1})$ and $(\F_{n_d}\otimes\F_{n_{d-1}}\otimes\cdots\otimes\F_{n_3}\otimes\I_{n_2})$ are full rank, since the matrices $\F_{n_3},\ldots,\F_{n_d}$ are all full rank, and a property of the Kronecker product is that for any two matrices $\A,\B$ it holds that $\rank(\A\otimes\B)=\rank(\A)\rank(\B)$. Therefore, by \cref{lemma:bdiag_bcirc_equality} we have that $\rank(\bcirc(\mX))=\rank(\barX)$, and thus, the tensor average rank is also  given by 
\begin{align} \label{eq:averageRankWithBar}
\rank_{\textnormal{a}}(\mX)=\frac{1}{N}\rank(\barX).
\end{align}


\begin{definition}[tensor tubal rank \cite{doi:10.1137/110837711, tensor_tSVD2, tensor_tSVD}]
Let $\mX\in\reals^{n_1\times\cdots\times n_d}$ and denote its t-SVD as $\mX = \mU*\mS*\mV^{\top}$. Then the tensor tubal rank is defined as
\begin{align*}
\rank_{\textnormal{t}}(\mX)= \#\lbrace i\ \vert\ \mS(i,i,:,\ldots,:)\not=0\rbrace = \max_{i_3\in[n_3],\ldots,i_d\in[n_d]}\rank\left(\barX^{(i_3,\ldots,i_d)}\right).
\end{align*}
\end{definition}


Note that the average rank of a tensor $\mX$ is the average number of nonzero singular values in the diagonal blocks of $\barX$, whereas the tubal rank is the maximal number of nonzero singular values in any of the diagonal blocks of $\barX$.

The following lemma shows the relationships between the tubal rank and the other types of tensor ranks, i.e., the average rank, CP-rank, and the ranks of the first two modes of the Tucker-rank. The computation of the tubal rank requires unfolding the Fourier transform of a tensor only along the $3$rd to $d$th dimensions, and so, the tubal rank can be compared only to the ranks of the first two modes of the Tucker-rank. The proof is given in \cref{sec:App:proofLemma13}.
\begin{lemma} \label{lemma:rankRelationships}
Let $\mX\in\reals^{n_1\times\cdots\times n_d}$. The following inequalities hold.
\begin{align*}
(i)\ & \rank_{\textnormal{a}}(\mX)\le\rank_{\textnormal{t}}(\mX)
\\ (ii)\ & \rank_{\textnormal{t}}(\mX)\le\rank_{\textnormal{cp}}(\mX)
\\ (iii)\ & \rank_{\textnormal{t}}(\mX)\le\min\left\lbrace\rank(\X^{\lbrace1\rbrace}),\rank(\X^{\lbrace2\rbrace})\right\rbrace,
\end{align*}
where $\rank_{\textnormal{cp}}(\mX)$ is the CP-rank of $\mX$, and $\X^{\lbrace j\rbrace}$ is the mode-$j$ matricization of $\mX$ \cite{kolda}.
\end{lemma}

The following definition of the skinny t-SVD is slightly different than the one in previous papers (e.g., \cite{tensor_tSVD}), which will be important for the proper derivation of the subdifferential set of the tensor nuclear norm in the sequel (see  \cref{lemma:subdifferentialSet}).

\begin{definition}[skinny t-SVD] \label{def:skinnytSVD}
Let $\mX\in\reals^{n_1\times \cdots\times n_d}$. For every $i_3\in[n_3],\ldots,i_d\in[n_d]$, denote $r_{i_3,\ldots,i_d}=\rank(\barX^{(i_3,\ldots,i_d)})$, and denote $r=\rank_{\textnormal{t}}(\mX)=\max_{i_3\in[n_3],\ldots,i_d\in[n_d]} r_{i_3,\ldots,i_d}$. Then, the skinny t-SVD of $\mX$ is defined as $\mX = \mU_r*\mS_r*{\mV_r}^{\top}$, where $\mU_r\in\reals^{n_1\times r\times n_3\times\cdots\times n_d}$, $\mS_r\in\reals^{r\times r\times n_3\times\cdots\times n_d}$, and $\mV_r\in\reals^{n_2\times r\times n_3\times\cdots\times n_d}$ are such that $\mS_r$ is a f-diagonal, tensor and for every $i_3\in[n_3],\ldots,i_d\in[n_d]$,  ${\barU_{r}^{(i_3,\ldots,i_d)}}^{\tc}{\barU_{r}^{(i_3,\ldots,i_d)}}=\diag(\I_{r_{i_3,\ldots,i_d}},\mathbf{0})$ and ${\barV_{r}^{(i_3,\ldots,i_d)}}^{\tc}{\barV_{r}^{(i_3,\ldots,i_d)}}=\diag(\I_{r_{i_3,\ldots,i_d}},\mathbf{0})$. 
In particular, given the t-SVD components $\mU,\mS,\mV$, 
 the tensors $\mU_r,\mS_r,{\mV_r}$ are such that $\barU_r,\barS_r,\barV_r^{\tc}$ are block diagonal matrices, where each one of their diagonal blocks, corresponding to some choice $i_3\in[n_3],\ldots,i_d\in[n_d]$, contains the leading $r_{i_3,\ldots,i_d}$ columns of the corresponding diagonal block in $\barU,\barS,\barV$, respectively.
\end{definition}

Similarly to the notion of a rank-$r$ SVD of a matrix, which considers only the leading $r$ components in the SVD of a given matrix, we now define the corresponding notion of a rank-$r$ t-SVD.

\begin{definition}[rank-$r$ t-SVD] \label{def:rankRtSVD}
Let $\mX\in\reals^{n_1\times \cdots\times n_d}$ and let $r\ge0$. The rank-$r$ t-SVD of $\mX$ is defined as $\mX_r = \mU_r*\mS_r*{\mV_r}^{\top}$, where $\mU_r\in\reals^{n_1\times r\times n_3\times\cdots\times n_d}$, $\mS_r\in\reals^{r\times r\times n_3\times\cdots\times n_d}$, and $\mV_r\in\reals^{n_2\times r\times n_3\times\cdots\times n_d}$ are such that given the t-SVD components $\mU,\mS,\mV$, the matrices $\barU_r,\barS_r,\barV_r$ are the block diagonal matrices whose each diagonal block contains the leading $r$ columns of the corresponding diagonal block of $\barU,\barS,\barV$, respectively.
\end{definition}

\subsection{The tensor nuclear norm}
\label{sec:TNN}

In this section we define the TNN, which was originally proposed for $3$rd order tensors in \cite{tensor_tSVD}. In \cite{highOrderTSVD2, highOrderTSVD3} the authors provide the definition for the TNN for high-order tensors, however they do not rigorously present the motivation for considering such a norm. To the best of our knowledge, we are the first to provide the full proof of the duality between the tensor spectral norm and the TNN for tensors of arbitrary order. We are also the first to define the full subdifferntial set of the nuclear norm for any order (in \cite{tensor_tSVD} the set the authors claimed to be the subdifferntial set of the  nuclear norm for  $3$rd order tensors is only a subset of the full subdifferntial set, as their definition of the skinny t-SVD is not broad enough).


We begin by defining the tensor spectral norm which can be viewed as  an operator norm of the t-product:
\begin{align*}
\Vert\mX\Vert_2 :&= \sup_{\substack{\mY\in\reals^{n_2\times1\times n_3\times\cdots\times n_d}\\ \Vert\mY\Vert_F\le1}}\Vert\mX*\mY\Vert_F = \sup_{\substack{\mY\in\reals^{n_2\times1\times n_3\times\cdots\times n_d}\\ \Vert\mY\Vert_F\le1}}\Vert\fold(\bcirc(\mX)\cdot\unfold(\mY))\Vert_F 
\\ & = \sup_{\substack{\mY\in\reals^{n_2\times1\times n_3\times\cdots\times n_d}\\ \Vert\mY\Vert_F\le1}}\Vert\bcirc(\mX)\cdot\unfold(\mY)\Vert_F = \Vert\bcirc(\mX)\Vert_2,
\end{align*}
where the last equality follows from the definition of the matrix spectral norm. This leads to the following definition of the tensor spectral norm.
\begin{definition}[tensor spectral norm]
Let $\mX\in\reals^{n_1\times\cdots\times n_d}$. The tensor spectral norm of $\mX$ is defined as
\begin{align*}
\Vert\mX\Vert_2 := \Vert\bcirc(\mX)\Vert_2=\Vert\barX\Vert_2.
\end{align*}
\end{definition}
The second equality in the definition above holds since, as we showed in \eqref{eq:inproofInnerproduct1} and \eqref{eq:inproofInnerproduct2}, $(\F^{-1}\otimes\I_{n_1})$ is unitary up to a constant of $1/N$ and $(\F\otimes\I_{n_2})$ is unitary up to a constant of $N$, where $\F:=\F_{n_d}\otimes\F_{n_{d-1}}\otimes\cdots\otimes\F_{n_3}$. Therefore, using the connection between $\bcirc(\mX)$ and $\barX$ in \cref{lemma:bdiag_bcirc_equality}, by the unitary invariant property of the spectral norm it follows that
$\Vert\bcirc(\mX)\Vert_2 = \Vert\barX\Vert_2$.

We now define the tensor nuclear norm, which as we will show is exactly the dual of the tensor spectral norm.
\begin{definition}[tensor nuclear norm]
Let $\mX=\mU*\mS*\mV^{\top}$ denote the skinny t-SVD of $\mX\in\reals^{n_1\times\cdots\times n_d}$. The tensor nuclear norm of $\mX$ is defined as
\begin{align*}
\Vert\mX\Vert_* := \langle\mS,\mI\rangle=\sum_{i=1}^{r}\mS(i,i,1,\ldots,1),
\end{align*}
where $r=\rank_{\textnormal{t}}(\mX)$.
\end{definition}

Using the connection between the tensor inner product and the inner product of the corresponding block diagonal matrices in the Fourier domain, as given in \cref{lemma:orderPinnerProductEquivalence}, and the relation $\barI=\bdiag(\overbar{\mI})=\I_{r n_3\cdots n_d}$,
 it can be seen that
\begin{align} \label{eq:nuclearNormEquivalence}
\Vert\mX\Vert_*=\langle\mS,\mI\rangle=\frac{1}{N}\langle\barS,\barI\rangle = \frac{1}{N}\trace(\barS)=\frac{1}{N}\Vert\barX\Vert_*.
\end{align}

%

The following lemma was proved in \cite{tensor_tSVD} for 3rd-order tensors. The proof for the general case is given in \cref{sec:App:proofLemma12}.
\begin{lemma} \label{lemma:dualNorms}
The tensor nuclear norm $\Vert\cdot\Vert_*$ is the dual norm of the spectral norm $\Vert\cdot\Vert_2$.
\end{lemma}

The following lemma is proved in \cite{tensor_tSVD} for 3rd order tensors. The proof for the general case is identical except for the change in the size of the block diagonal matrices.
\begin{lemma}
The convex envelope\footnote{The convex envelope of a set is the smallest convex set that contains it.} of the tensor average rank $\rank_{\textnormal{a}}(\mX)$ over the set $\lbrace\mX\in\reals^{n_1\times\cdots\times n_d}\ \vert\ \Vert\mX\Vert_2\le1\rbrace$ is the tensor nuclear norm $\Vert\mX\Vert_*$.
\end{lemma}

%
%
%
%

In the following lemma we define the full subdifferential set of the tensor nuclear norm. The proof is given in \cref{sec:App:proofLemma100}.

\begin{lemma} \label{lemma:subdifferentialSet}
Let $\mX\in\reals^{n_1\times\cdots\times n_d}$ and let $\mX=\mU*\mS*\mV^{\top}$ denote its skinny t-SVD. Then the subdifferential set of $\Vert\mX\Vert_*$ is 
\begin{align*}
\partial\Vert\mX\Vert_* = \lbrace\mU*\mV^{\top}+\mW\ \vert\ \mU^{\top}*\mW=\mathbf{0},\ \mW*\mV=\mathbf{0},\ \Vert\mW\Vert_2\le1\rbrace.
\end{align*}
\end{lemma}

Solving Problem \eqref{mainModel} using projected gradient methods requires computing projections onto the TNN ball. The projection of a real-valued tensor $\mX\in\reals^{n_1\times\cdots\times n_d}$ onto the TNN ball is similar to projecting a matrix onto the matrix nuclear norm ball, since it amounts to projecting the matrix $\barX=\bdiag(\overbar{\mX})$ onto a matrix nuclear norm ball. 
We recall that the projection a matrix $\X\in\complex^{m\times n}$ with a singular value decomposition $\X=\sum_{i=1}^{\min\lbrace m,n\rbrace}\sigma_i\u_i\v_i^{\tc}$ onto the matrix nuclear norm ball of radius $\tau$ takes the form:
\begin{align*}
\Pi_{\lbrace\Y\in\complex^{m\times n}~|~\Vert\Y\Vert_*\le\tau\rbrace}[\X] = \sum_{i=1}^{\min\lbrace m,n\rbrace}\max\lbrace0,\sigma_i-\sigma\rbrace\u_i\v_i^{\tc},
\end{align*}
where $\sigma\ge0$ satisfies  $\sum_{i=1}^{\min\lbrace m,n\rbrace}\max\lbrace0,\sigma_i-\sigma\rbrace=\tau$.

In \cref{alg:projectionTNN} we describe the procedure for computing the projection onto the TNN ball of radius $\tau$. The most expensive part of the computation is computing the t-SVD of the input tensor $\mX$, as described in Algorithm \ref{alg:tSVD}. 
\begin{algorithm}
	\caption{Projection onto the tensor nuclear norm ball of radius $\tau$}\label{alg:projectionTNN}
	\begin{algorithmic}	
		\STATE \textbf{Input:} $\mX\in\reals^{n_1\times\cdots\times n_d}$, $\tau\ge0$
		\STATE Compute the t-SVD $\mX=\mU*\mS*\mV^{\top}$
        \STATE Find $\sigma\ge0$ such that $\frac{1}{n_3\cdots n_d}\sum_{i=1}^{\rank_{\textnormal{t}}(\mX)}\sum_{i_3=1}^{n_3}\cdots\sum_{i_d=1}^{n_d} \max\lbrace0,\sigma_i(\barX^{(i_3,\ldots,i_d)})-\sigma\rbrace=\tau$
        \STATE Compute $\overbar{\mS}_{\tau}$ such that $\overbar{\mS}_{\tau}(i_1,\ldots,i_d)=\Bigg\lbrace\begin{array}{ll}0, & \textrm{if } i_1\not=i_2 \\ \max\lbrace0,\sigma_{i_1}(\barX^{(i_3,\ldots,i_d)})-\sigma\rbrace, & \textrm{if } i_1=i_2\end{array}$
        \STATE Return $\Pi_{\lbrace\Vert\mY\Vert_*\le\tau\rbrace}[\mX] = \mU*\ifft(\overbar{\mS}_{\tau})*\mV^{\top}$
	\end{algorithmic}
\end{algorithm}

\begin{lemma}[Projection onto the tensor nuclear norm ball]
\label{lemma:projectionOntoTNN}
Let $\mX\in\reals^{n_1\times\cdots\times n_d}$. The Euclidean projection of $\mX$ onto the TNN ball of radius $\tau\ge0$ can be computed by the steps described in \cref{alg:projectionTNN}.

Moreover, if $\rank_{\textnormal{t}}(\Pi_{\lbrace\Vert\mY\Vert_*\le\tau\rbrace}[\mX])  \leq r$, then the t-SVD computation in the algorithm can be replaced with the rank-$r$ t-SVD (Definition \ref{def:rankRtSVD}) and the summation from $1$ to $\rank_{\textnormal{t}}(\mX)$ in the computation of $\sigma$ can be replaced with a summation from $1$ to $r$.
\end{lemma}

The proof is given in \cref{sec:App:proofLemma11}.

\subsection{Complexity of tensor operations}
\label{sec:complexityAnalysis}

Since we are interested in the efficiency of applying projected-gradient methods for solving Problem \eqref{mainModel}, we now turn to discuss the complexity of computing the associated projected gradient mapping $\mX\rightarrow\Pi_{\lbrace\Vert\mY\Vert_*\le\tau\rbrace}[\mX-\eta\nabla{}f(\mX)]$ and related operations. 


As discussed above, given a tensor $\mX\in\reals^{n_1\times n_2\times\cdots\times n_d}$ and the corresponding gradient tensor $\nabla{}f(\mX)$, the most expensive step in computing the projection onto the unit TNN ball (as given in Algorithm \ref{alg:projectionTNN}), is the computation of the t-SVD of the tensor to project $\mX-\eta\nabla{}f(\mX)$. According to Algorithm \ref{alg:tSVD}, computing the t-SVD requires first computing the Fourier transformation of the tensor $\mX-\eta\nabla{}f(\mX)$, which takes $\mathcal{O}(Nn_1n_2(\log(n_3)+\cdots + \log(n_d)))$ runtime (the Fourier transformation along a dimension $j\in\lbrace3,\ldots,d\rbrace$ involves computing $n_1n_2N/n_j$ multiplications between $\F_{n_j}$ and some vector of length $n_j$, as in \cref{def:fourierJdim},  each takes $O(n_j\log(n_j))$ runtime, and this is computed along all dimensions $j\in\lbrace3,\ldots,d\rbrace$ to obtain the full transformation). This is followed by computing approximately half of the SVD of all $N$ frontal slices of $\overbar{\mX-\eta\nabla{}f(\mX)}$, which takes $\mathcal{O}(Nn_1n_2\min\lbrace n_1,n_2\rbrace)$ runtime (recall each such frontal slice is a $n_1\times n_2$ matrix). 
The last part in the t-SVD algorithm is to compute the inverse Fourier transform of the obtained SVD matrix components of the projection $\barU,\barS,\barV^{\tc}$, to obtain $\mU$,$\mS$,$\mV^{\top}$, which takes the same runtime as the Fourier transformation.

In case the projected gradient mapping $\Pi_{\lbrace\Vert\mY\Vert_*\le\tau\rbrace}[\mX-\eta\nabla{}f(\mX)]$ is known to be of low tubal rank  $r\ll\min\lbrace n_1,n_2\rbrace$, then by the second part of Lemma \ref{lemma:projectionOntoTNN}, it suffices to compute only the rank-$r$ t-SVD (\cref{def:rankRtSVD}) of the tensor $\mX-\eta\nabla{}f(\mX)$, which amounts to computing only the rank-$r$ matrix SVDs of the frontal slices of $\overbar{\mX-\eta\nabla{}f(\mX)}$. This is far more efficient when $r \ll \min\{n_1,n_2\}$, since each such matrix SVD can be computed in roughly $\mathcal{O}(n_1n_2r)$ runtime using fast iterative methods for matrix SVD, e.g., \cite{rankrSVD1,rankrSVD2,rankrSVD3}.  This runtime is  simplified since it omits the worst-case dependency of such fast iterative methods for matrix SVD on the condition number of the matrix and the desired accuracy, however such methods often converge very quickly in practice, and so the runtime is often dominated by the runtime of the associated matrix-matrix products which is $\mathcal{O}(n_1n_2r)$. 
Therefore, prior knowledge of low tubal rank can significantly reduce the runtime of computing the SVD of $\bdiag(\overbar{\mX-\eta\nabla{}f(\mX)})$ to roughly $\mathcal{O}(Nn_1n_2r)$, instead of the worst case $\mathcal{O}(Nn_1n_2\min\lbrace n_1,n_2\rbrace)$.  

Note the above discussions regarding runtime assumes batch computation of the SVDs of the frontal slices, which is the most expensive step in the computation of the projection. Additional obvious improvements in the overall runtime could be obtained by performing these decompositions in parallel.

The above discussion also motivates the following definition of the rank-$r$ truncated projection of a tensor onto the TNN ball, which upper-bounds the tubal rank of the projection by $r$.
\begin{definition}[rank-$r$ truncated projection onto the tensor nuclear norm ball] 
\label{def:lowRankProjection}
Let $\mX\in\reals^{n_1\times n_2\times\cdots\times n_d}$ and let $r\le\min\lbrace n_1,n_2\rbrace$. The Euclidean rank-r truncated projection of $\mX$ onto the TNN ball of radius $\tau\ge0$, is computed by performing the steps in \cref{alg:projectionTNN} with the modifications listed in the second part of Lemma \ref{lemma:projectionOntoTNN}, i.e., when only the first $r$ components in the SVD of each frontal slice of $\mX$ are computed.
\end{definition}

For convenience, Table \ref{table:operationComplexity} records some tensor operations of interest and their associated runtimes.

It is interesting to note that, as opposed to our discussion above on the benefit of prior knowledge of low tubal rank, prior knowledge of  low average rank does not seem to help in significantly reducing the runtime of computing the projection. The average rank  contains information about the  rank of all frontal slices combined, and not information on the rank of any specific frontal slice. Thus, it does not seem to help in reducing the computational cost of Algorithm \ref{alg:tSVD}, in particular due to the conjugate-complex symmetry conditions  \eqref{eq:conjugateComplexSymmetryCondition} that must be satisfied so that the resulting tensor is indeed real-valued.

\begin{table*}\renewcommand{\arraystretch}{1.3}
{\footnotesize
\begin{center}
  \begin{tabular}{| l | c | } \hline 
  operation &  runtime \\ \hline
  $\fft(\mX)$, $\ifft(\mX)$ & $n_1 n_2 N\sum_{i=3}^d\log(n_i)$ \\ \hline
  t-SVD of $\mX$ & $n_1 n_2 N (\sum_{i=3}^d\log(n_i)+ \min\lbrace n_1,n_2\rbrace)$ \\ \hline
  rank-$r$ t-SVD of $\mX$ 
   & $n_1 n_2 N (\sum_{i=3}^d\log(n_i)+ r)$   \\ \hline
    $\mX\rightarrow\Pi_{\lbrace\Vert\mY\Vert_*\le\tau\rbrace}[\mX-\eta\nabla{}f(\mX)]$ & $n_1 n_2 N (\sum_{i=3}^d\log(n_i)+ \min\lbrace n_1,n_2\rbrace)$ \\ \hline
   $\mX\rightarrow\Pi_{\lbrace\Vert\mY\Vert_*\le\tau\rbrace}[\mX-\eta\nabla{}f(\mX)]$
    & $n_1 n_2 N (\sum_{i=3}^d\log(n_i)+ r)$ \\  if $\rank_{\textnormal{t}}\left(\Pi_{\lbrace\Vert\mY\Vert_*\le\tau\rbrace}[\mX-\eta\nabla{}f(\mX)]\right)\le r$ & 
   \\ \hline
   \end{tabular}
  \caption{Runtimes of certain tensor operations for tensors in $\reals^{n_1\times\cdots\times{}n_d}$. For all runtimes we omit the universal constants and lower-order terms. The runtimes for computing t-SVDs and projected gradient mappings assume batch computation of matrix SVDs (as applied in Algorithm \ref{alg:tSVD}), and are given in simplified form which omits the dependency of iterative matrix SVD algorithms on the condition number and desired accuracy. For the projected gradient mapping  it is assumed that the gradient is  given.
  }\label{table:operationComplexity}
\end{center}
}
\vskip -0.2in
\end{table*}\renewcommand{\arraystretch}{1}


\section{Strict Complementarity for Smooth Problems}
\label{sec:SCsmooth}
In this section we present our main results for the case in which the objective function in Problem \eqref{mainModel} is smooth and strict complementarity holds. We begin by formally introducing and motivating the strict complementarity condition for Problem \eqref{mainModel}  in Section \ref{sec:SCdefSmooth}. Then, in Section \ref{sec:QGtheorem} we present our quadratic growth result which facilitates linear convergence rates for first-order methods, and in Section \ref{sec:lowTubalRankProjectionsSmooth} we present our result that establishes that under SC (or relaxed notions of), inside a certain ball around optimal solutions with low tubal rank, the projected gradient mapping always admits low tubal rank which, per the discussion in  Section \ref{sec:complexityAnalysis}, implies significantly improved runtimes for computing the projected-gradient mapping.

\subsection{Definition and motivation}
\label{sec:SCdefSmooth}

Strict complementarity for constrained optimization problems is a standard assumption in many settings of interest and has been thoroughly studied in recent years, e.g., \cite{nondegeneracy, drusvyatskiy2018error, strictComplementarity,strictComplementarity2}. 

\begin{definition}[strict complementarity \cite{strictComplementarity,strictComplementarity2}]
We say an optimal solution $\mX^*$ of Problem \eqref{mainModel} satisfies strict complementarity
if
\begin{align} \label{strictComplementarityDef}
\mathbf{0}\in \nabla{}f(\mX^*) + \ri(\mathcal{N}_{\lbrace\Vert\mX\Vert_*\le1\rbrace}(\mX^*))
\end{align}
\footnote{$\ri(\mathcal{S})$ denotes the relative interior of the set $\mathcal{S}$. The normal cone of a set $\mathcal{S}$ at a point $\x\in\mathcal{S}$ is $\mathcal{N}_{\mathcal{S}}(\x):=\lbrace\y~|~\langle\y,\z-\x\rangle\le0,\ \forall\z\in\mathcal{S}\rbrace$.} 
and the complementarity measure $\delta$ is defined as\begin{align} \label{complementarityMeasure}
\delta: & = \min\lbrace\langle\mZ-\mX^*,\nabla{}f(\mX^*)\rangle\ \vert\ \Vert\mZ\Vert_*\le1,\ \mU^{\top}*\mZ=\mathbf{0},\ \mZ*\mV=\mathbf{0}\rbrace.
\end{align}
\end{definition}
 
The normal cone of an atomic norm at some point $\x$ can also be written as the conic hull of its subdifferential set at $\x$, where a conic hull of a set is obtained  by taking nonnegative linear combinations of elements of the set. 
The relationship between the subdifferential set of the tensor nuclear norm, which we present in \cref{lemma:subdifferentialSet}, and the normal cone of the unit TNN ball gives an intuition for the connection between the definition of  strict complementarity in \eqref{strictComplementarityDef} and the complementarity measure in \eqref{complementarityMeasure}.

To present an equivalent condition for strict complementarity w.r.t. the unit TNN ball,  we first state the following lemma which connects between the SVDs of the frontal slices of the block diagonal matrices of the Fourier transforms of the optimal solution and its corresponding gradient direction. 
This lemma is an extension of a similar argument that holds for optimization over the matrix nuclear norm ball, given in Lemma 2 in \cite{garberNuclearNormMatrices}. 
The proof follows from the first-order optimality condition for Problem \eqref{mainModel} and is given in \cref{sec:App:proofLemma199}.

\begin{lemma} \label{lemma:svd_opt_grad_d_dim}
Let $\mX^*$ be an optimal solution to Problem \eqref{mainModel} and let $\overbar{\mX^*}$ denote its Fourier transform as defined in \cref{def:fft}. For each $i_3\in[n_3],\ldots,i_d\in[n_d]$, denote the SVD of the frontal slice ${\overbar{\X^*}}^{(i_3,\ldots,i_d)}$ as ${\overbar{\X^*}}^{(i_3,\ldots,i_d)}=\sum_{i=1}^{r_{i_3,\ldots,i_d}}\sigma_i^{(i_3,\ldots,i_d)}\u_i^{(i_3,\ldots,i_d)}{\v_i^{(i_3,\ldots,i_d)}}^{\tc}$, where $r_{i_3,\ldots,i_d}=\rank({\overbar{\X^*}}^{(i_3,\ldots,i_d)})$.
Then, each frontal slice of the Fourier transform of the gradient vector $\overbar{\nabla{}f(\mX^*)}^{(i_3,\ldots,i_d)}$ admits a SVD such that the set of pairs of vectors $\lbrace -\u_i^{(i_3,\ldots,i_d)}, \v_i^{(i_3,\ldots,i_d)}\rbrace_{i=1}^{r_{i_3,\ldots,i_d}}$ is a
set of top singular-vector pairs of $\overbar{\nabla{}f(\mX^*)}^{(i_3,\ldots,i_d)}$ which corresponds to the largest singular value $\sigma_1(\overbar{\nabla{}f(\mX^*)}^{(i_3,\ldots,i_d)})$. Furthermore, the top singular values of all nonzero slices are equal, that is,
$$\sigma_1(\overbar{\nabla{}f(\mX^*)})=\sigma_1(\overbar{\nabla{}f(\mX^*)}^{(i_3,\ldots,i_d)}),$$
for all $i_3\in[n_3],\ldots,i_d\in[n_d]$ such that $\sigma_1(\overbar{\nabla{}f(\mX^*)}^{(i_3,\ldots,i_d)})\not=0$.
\end{lemma}

The following lemma presents an easily computable equivalent condition for strict complementarity w.r.t. the unit TNN ball. The proof also establishes that \eqref{strictComplementarityDef} holds if and only if the complementarity measure \eqref{complementarityMeasure} satisfies that $\delta>0$. The proof is given in \ref{sec:App:proofLemma14}.

\begin{lemma} \label{lemma:SCequivalence}
Let $\mX^*\in\reals^{n_1\times\cdots\times n_d}$ be an optimal solution to Problem \eqref{mainModel} for which $\rank_{\textnormal{a}}(\mX^*)=r
<\min\lbrace n_1,n_2\rbrace$. $\mX^*$ satisfies the strict complementarity condition with some $\delta>0$ if and only if
\begin{align*}
\delta=\sigma_1(\overbar{\nabla{}f(\mX^*)}) - \sigma_{rN+1}(\overbar{\nabla{}f(\mX^*)}) >0.
\end{align*} 
\end{lemma}


The following two lemmas motivate the strict complementarity assumption. The first lemma shows that strict complementarity is necessary and sufficient for a certain notion of robustness of the average rank of optimal solutions to arbitrarily small perturbations in the radius of the TNN ball, to hold. That is, without strict complementarity, the low-rankness of optimal solutions is not robust, under the projected gradient mapping, to the slightest misspecification in the TNN radius of the convex relaxation \eqref{mainModel}. This lemma is analogues to Lemma 3 in \cite{garberNuclearNormMatrices} which considered optimization over the matrix nuclear norm ball. The proof is given in \cref{sec:App:proofLemma17}.

\begin{lemma} \label{lemma:motivateSCrobustness}
Let $\mX^*\in\reals^{n_1\times\cdots\times n_d}$ be an optimal solution to Problem \eqref{mainModel} such that $\nabla{}f(\mX^*)\not=0$ and let $\varepsilon\ge0$. Then, for any step-size $\eta>0$, it holds that
$$\rank_{\textnormal{a}}(\Pi_{\lbrace\Vert\mY\Vert_*\le 1+\varepsilon\rbrace}[\mX^*-\eta\nabla{}f(\mX^*)])>r$$
if and only if $\varepsilon>\eta\left(\sigma_1(\overbar{\nabla{}f(\mX^*)})-\sigma_{rN+1}(\overbar{\nabla{}f(\mX^*)})\right)$.
\end{lemma}

The next lemma establishes that for the family of functions $f(\mX)=g(\mX)+\langle\mC,\mX\rangle$, for almost any tensor $\mC$, strict complementarity holds. This result is analogues to Lemma 8 in \cite{SpectralFrankWolfe} (which was in turn inspired by \cite{nondegeneracy}), where the authors proved a similar result for optimization over the spectrahedron. The proof is given in \cref{sec:App:proofLemma18}.

\begin{lemma} \label{lemma:motivateSCalmostAll}
Assume $f(\mX) = g(\mX) + \langle\mC,\mX\rangle$. Then, for almost all $\mC\in\reals^{n_1\times\cdots\times n_d}$, Problem \eqref{mainModel} admits a unique minimizer which furthermore satisfies strict complementarity.
\end{lemma}

\subsection{Quadratic growth}
\label{sec:QGtheorem}

It is well known that the quadratic growth property is sufficient in many cases to achieve linear convergence rates for first-order methods  \cite{drusvyatskiy2018error, nesterov_LCwithQG}. 
In recent years there have been several works that showed that in a variety of settings strict complementarity implies quadratic growth, e.g., \cite{zhou2017unified, garberRankOne, SpectralFrankWolfe, ding2020kfw}. In this section we prove that under strict complementarity such a result can also be obtained for our Problem \eqref{mainModel}. Our proof is based on the ideas in \cite{ding2020kfw}, where a similar result was obtained for the matrix nuclear norm ball, by using a  dilation argument to rephrase the problem as optimization over a certain spectrahedron for which a quadratic growth result has been already established in \cite{SpectralFrankWolfe}. Our tensor setting however  is  substantially more involved as it requires dealing with complex matrices with the additional conjugate-complex symmetry conditions, as defined in \eqref{eq:conjugateComplexSymmetryCondition}, and so, the reduction to the real-valued spectrahedron setting (as in \cite{SpectralFrankWolfe}) is more challenging. 

Here we only give a sketch of the main ideas of the proof, and the full proof is given in \cref{sec:App:proofTheoremQG}.

\begin{theorem}[quadratic growth] \label{thm:QG}
Let $f(\mX)=g(\mA(\mX))+\langle\mC,\mX\rangle$, where $g$ is $\alpha$-strongly convex and $\mA:\reals^{n_1\times\cdots\times n_d}\rightarrow\reals^m$ is a linear map. Assume there exist a unique optimal solution $\mX^*\in\reals^{n_1\times\cdots\times n_d}$ to Problem \eqref{mainModel}, and that it satisfies strict complementarity. Then, there exists a constant $\gamma>0$ such that for every $\mX\in\lbrace\mX\in\reals^{n_1\times\cdots\times n_d}\ \vert\ \Vert\mX\Vert_*\le1\rbrace$, it holds that
\begin{align*}
f(\mX)-f(\mX^*) \ge \gamma\Vert\mX-\mX^*\Vert_F^2.
\end{align*}
\end{theorem}

\begin{remark}
\label{remark:QGparameter}
Our proof of \cref{thm:QG} establishes the existence of a positive constant for which quadratic growth holds. The constant $\gamma$ in \cref{thm:QG} can be written as
\begin{align*}
\gamma=\min\left\lbrace\frac{\delta}{8 N}
\left(1+\frac{2\sigma_{\max}^2(\mcP)}{\sigma_{\min}^2(\mcP_{V})}\right)^{-1}, \frac{\sigma_{\min}^2(\mcP_{V}) N\alpha}{8}\right\rbrace,
\end{align*}
and is unfortunately not very intuitive to understand. It can be of interest in future work to study this constant and establish whether there exists a simpler and more interpretable bound for $\gamma$. 
The parameter $\delta$ is the strict complementarity measure. The linear operator $\mcP$ upon the vector space $\mathbb{S}^{2(n_1+n_2) N}$ is dependent on the mapping $\mA$ and additional linear operators that ensure that a specific folding of a matrix in $\mathbb{S}^{2(n_1+n_2) N}$ into a tensor will return a real-valued tensor within the TNN ball.
Denote $\overbar{\X^*}=\overbar{\U^*}\overbar{\S^*}{\overbar{\V^*}}^{\tc}$ to be the SVD of $\overbar{\X^*}=\bdiag(\overbar{\mX^*})$ and denote 
$\widetilde{\V}_r:={2}^{-1/2}\left[\begin{array}{cc}\overbar{\U^*}^{\tc} & -\overbar{\V^*}^{\tc} \end{array}\right]^{\tc}$.
Then, we define the mapping $\mcP_{V}$ upon the vector space $\mathbb{S}^{2r}$ such that
\begin{align*}
\mcP_{V}(\S)=\mcP\left(\left[\begin{array}{cc} \real(\widetilde{\V}_r) & \im(\widetilde{\V}_r)
\\ \im(\widetilde{\V}_r) & -\real(\widetilde{\V}_r) \end{array}\right]\S\left[\begin{array}{cc} \real(\widetilde{\V}_r) & \im(\widetilde{\V}_r)
\\ \im(\widetilde{\V}_r) & -\real(\widetilde{\V}_r) \end{array}\right]^{\top}\right).
\end{align*}
Using these notations we denote the constants $\sigma_{\min}(\mcP_{V}) := \min_{\Vert\S\Vert_F=1}\Vert\mcP_{V}(\S)\Vert_2$  and $\sigma_{\max}(\mcP) := \max_{\Vert\S\Vert_F=1}\Vert\mcP(\S)\Vert_2$.%
\end{remark}

\begin{proof}[Proof sketch of \cref{thm:QG}]
We reformulate the problem of minimizing a function of the form $f(\mX)=g(\mA(\mX))+\langle\mC,\mX\rangle$ over the unit TNN ball, to minimization over the intersection of a real spectrahedron of higher dimension and radius $2N$, i.e., the set $\{\X\in\mbS^{2(n_1+n_2)N} ~|~\X\succeq 0, \trace(\X) = 2N\}$ and a certain linear subspace (which forces the constraint that a matrix in $\mbS^{2(n_1+n_2)N}$ could be transformed back into a real-valued tensor in $\reals^{n_1\times\dots\times{}n_d}$). We then establish that if the original problem has a unique solution so does the new problem.


Let $\mX\in\reals^{n_1\times\cdots\times n_d}$. Denoting the optimal solution of the equivalent problem over the constrained spectrahedron by $\X^*$ and by $\mcP$ a linear operator which is dependent on the mapping $\mA$, it can be showed using the uniqueness of $\X^*$ that there exists a matrix $\W$, which is dependent on $\mX$, such that 
\begin{align} \label{ineq:inProof99444}
\Vert\W-\X^*\Vert_F
& \le c\Vert\mcP(\W)-\mcP(\X^*)\Vert_2,
\end{align}
for some constant $c>0$. 
This follows since the uniqueness of $\X^*$ implies that for the set of matrices of similar  structure to $\W$ (the eigenvectors of $\W$ are related to those of $\X^*$ in a certain way), the mapping $\mcP$ is injective, which in turn lower-bounds the norm of this mapping over the set of such matrices, and so, $c$ is strictly positive.

Denote $\mW$ the tensor whose dilation is the matrix $\W$ mentioned above. Thanks to the specific structure of $\W$ (i.e., the fact that its eigenvectors are related to those of $\X^*$),  it can be shown that there exist a constant $\tilde{c}>0$ such that
\begin{align} \label{ineq:inProof993333}
\Vert\mW-\mX\Vert_F^2 \le \frac{\tilde{c}}{\delta}\langle\mX-\mX^*,\nabla{}f(\mX^*)\rangle.
\end{align}

By returning back to the original tensor problem and plugging-in both results  \eqref{ineq:inProof99444} and \eqref{ineq:inProof993333}, we obtain that 
\begin{align} \label{ineq:inProof995555}
\Vert\mX-\mX^*\Vert_F^2 & \le 2\Vert\mW-\mX\Vert_F^2+2\Vert\mW-\mX^*\Vert_F^2 \nonumber
\\ & \le \frac{1}{\gamma}
\left(\langle\nabla{}f(\mX^*),\mX-\mX^*\rangle + \frac{\alpha}{2}\Vert\mA(\mX)-\mA(\mX^*)\Vert_2^2\right) \nonumber
\\ & = \frac{1}{\gamma}
\left(\langle\mA^{\top}\nabla{}g(\mA\mX^*)+\mC,\mX-\mX^*\rangle + \frac{\alpha}{2}\Vert\mA(\mX)-\mA(\mX^*)\Vert_2^2\right),
\end{align}

for some $\gamma>0$ which is dependent on $\delta$, $c$, and $\tilde{c}$. The bound of the second term the RHS of \eqref{ineq:inProof995555} follows from \eqref{ineq:inProof99444} due to the relationship between $\mcP$ and $\mA$ and the connection between the norms of the tensors to the norms of their matrix dilations.

Using the special  structure of the objective function $f$ and standard strong convexity arguments to bound the RHS of \eqref{ineq:inProof995555}, it follows that  
\begin{align*}
f(\mX)-f(\mX^*) = g(\mA(\mX))-g(\mA(\mX^*))+\langle\mC,\mX-\mX^*\rangle \ge \gamma\Vert\mX-\mX^*\Vert_F^2.
\end{align*}

\end{proof}

\subsection{Low tubal rank of the projected gradient mapping near low-tubal rank minimizers}
\label{sec:lowTubalRankProjectionsSmooth}

In this section we prove that under strict complementarity (or even relaxed notions of), there exists a radius around low tubal rank optimal solutions within which the projected gradient mapping always admits low tubal rank. Thanks to this property we will be able to prove in \cref{sec:algorithmicConsequences} that highly popular projected gradient methods, when initialized in the proximity of a low tubal rank optimal solution, require only efficient low-rank (matrix) SVD computations (in order to project onto the unit TNN ball) throughout their run.  More precisely, we present a natural tradeoff that allows to increase the radius of the ball inside-which the projected gradient mapping admits low tubal rank (by considering more relaxed notions of the SC condition) in favor of increasing also the matrix rank of the SVDs required to compute the projection (which naturally increases the runtime required to compute the projection).



Motivation for the plausibility that an optimal solution to Problem \eqref{mainModel} will indeed have low tubal rank is given by \cref{lemma:rankRelationships} which states that the tubal rank is upper-bounded by the CP-rank. Indeed low CP-rank is a standard assumption in many low-rank tensor recovery problems. In addition,  \cite{tensor_tSVD, candes2011robust} established (under suitable assumptions) formal recovery results of low tubal rank tensors from noisy observations based on the TNN for the problems of tensor robust principal component analysis and low rank tensor completion, respectively. They also demonstrated these results empirically. In \cref{sec:experiments} we also bring empirical evidence that the convex relaxation \eqref{mainModel} indeed admits optimal solutions of low tubal rank for such tasks.

Towards obtaining our main result for this section, the following lemma establishes a sufficient and necessary condition on the singular values of the block diagonal matrix of the Fourier transformation of a tensor,  so that its projection onto the TNN ball of radius $\tau$ will have  low tubal rank.
The proof follows from the structure of the projection onto a TNN ball described in \cref{alg:projectionTNN} and is given in \cref{sec:App:proofLemma122}.

\begin{lemma} \label{lemma:iffLowRankfCondiotionOrderP}
Let $\mX\in\reals^{n_1\times\cdots\times n_d}$ and let $\mX=\mU*\mS*\mV^{\top}$ denote its t-SVD. For every $j\in\{1,\dots,\min\{n_1,n_2\}\}$, denote $\sigma_j^{\max}(\barX) = \max_{k_3\in[n_3],\ldots,k_d\in[n_d]}\sigma_{j}(\barX^{(k_3,\ldots,k_d)})$. Also, for every $j\in\{1,\dots,\min\{n_1,n_2\}\}$ and $i_3\in[n_3],\ldots,i_d\in[n_d]$, denote $\#\sigma_{>{}j}^{(i_3,\dots,i_d)}(\mX) =\#\left\lbrace i\ \bigg\vert\ \sigma_i(\barX^{(i_3,\ldots,i_d)})>\sigma_{j}^{\max}(\barX)\right\rbrace\le j-1$. Let $r$ such that $\min\lbrace n_1,n_2\rbrace>r\ge0$.
Then, $\rank_{\textnormal{t}}(\Pi_{\lbrace\Vert\mY\Vert_*\le\tau\rbrace}[\mX])\le r$ if and only if
\begin{align} \label{ineq:conditionOfLowTubalRank}
\frac{1}{N}\sum_{i_d=1}^{n_d}\cdots\sum_{i_3=1}^{n_3}\left(\sum_{i=1}^{\#\sigma_{>{}r+1}^{(i_3,\dots,i_d)}(\mX)}\sigma_{i}(\barX^{(i_3,\ldots,i_d)})-\#\sigma_{>{}r+1}^{(i_3,\dots,i_d)} (\mX)\cdot \sigma_{r+1}^{\max}(\barX)\right)\ge\tau.
\end{align}
\end{lemma}

The radius around an optimal solution $\mX^*$ within-which the projected gradient mapping will be guaranteed to be of low tubal rank, is dependent on the existence of a spectral gap in the block diagonal matrix of the Fourier transformation of the gradient at the optimal solution $\overbar{\nabla{}f(\mX^*)}$. 
As shown in \cref{lemma:SCequivalence}, the complementarity measure $\delta$ of an optimal solution $\mX^*$ indeed corresponds to the magnitude of the spectral gap between the largest singular value $\sigma_{1}(\overbar{\nabla{}f(\mX^*)})$ and the second largest singular value (which is strictly smaller) $\sigma_{rN+1}(\overbar{\nabla{}f(\mX^*)})$, where here $r$ denotes the average rank of $\mX^*$ (recall that according to Lemma \ref{lemma:svd_opt_grad_d_dim}, $\sigma_{1}(\overbar{\nabla{}f(\mX^*)}) = \sigma_{2}(\overbar{\nabla{}f(\mX^*)}) = \cdots = \sigma_{rN}(\overbar{\nabla{}f(\mX^*)})$). As we shall see, considering spectral gaps between $\sigma_{1}(\overbar{\nabla{}f(\mX^*)})$ and even lower singular values, will allow us to increase the radius in which the projected gradient mapping admits low tubal rank. This motivates the following definition.


\begin{definition}[generalized complementarity measure] \label{def:GSC}
Let $\mX^*$ be an optimal solution to Problem \eqref{mainModel}.
Then, for any $r$ such that $\min\lbrace n_1,n_2\rbrace>r\ge\max_{i_3\in[n_3],\ldots,i_d\in[n_d]}{\#\sigma_1(\overbar{\nabla{}f(\mX^*)}^{(i_3,\ldots,i_d)})}$, the generalized complementarity measure is defined as
\begin{align*}
\delta(r):=\sigma_1(\overbar{\nabla{}f(\mX^*)})-\max_{i_3\in[n_3],\ldots,i_d\in[n_d]}\sigma_{r+1}(\overbar{\nabla{}f(\mX^*)}^{(i_3,\ldots,i_d)}).
\end{align*}
\end{definition}

Since our focus is on Euclidean algorithms (by considering projections w.r.t. the Euclidean norm), our results regarding the low tubal rank of the projected gradient mapping in the proximity of an optimal solution will naturally be presented in terms of the Euclidean distance from the optimal solution. Nevertheless, in some cases it might be more appealing to measure this distance in spectral norm. Towards this we denote by $\beta_2$ the smoothness parameter of $f(\cdot)$ with respect to the spectral norm. That is, 
\begin{align*}
\forall \mX,\mY\in\{\mW\in\reals^{n_1\times\cdots\times n_d}~|~\Vert{\mW}\Vert_*\leq 1\}: \quad \Vert\nabla{}f(\mX)-\nabla{}f(\mY)\Vert_2\le\beta_2\Vert\mX-\mY\Vert_2.
\end{align*}
In the sequel we define the number of non-zero diagonal blocks of a block diagonal matrix $\barX\in\complex^{n_1N\times n_2N}$ as
\begin{align*}
\nnzb(\barX):=\#\lbrace(i_3,\dots,i_d)~|~\barX^{(i_3,\ldots,i_d)}\not=0\rbrace.
\end{align*}

We are now ready to present our main result for this section.
\begin{theorem} \label{lemma:RadiusForProjectedGradient_d_dim}
Assume $\nabla{}f$ is non-zero over the unit TNN ball and fix some optimal solution $\mX^*$ to Problem \eqref{mainModel}. Denote $ {\#\sigma_1}^{\max}:=\max_{i_3\in[n_3],\ldots,i_d\in[n_d]}\#\sigma_1(\overbar{\nabla{}f(\mX^*)}^{(i_3,\ldots,i_d)})$ and assume ${\#\sigma_1}^{\max}<\min\lbrace n_1,n_2\rbrace$.
Then, for any $\eta\ge0$, $\min\lbrace n_1,n_2\rbrace> r\ge{\#\sigma_1}^{\max}$, and $\mX\in\reals^{n_1\times{}\cdots\times{}n_d}$, if 
\begin{align} \label{ineq:FrobRadiusBound}
& \Vert\mX-\mX^*\Vert_F \le \nonumber
\\ &  \frac{\eta}{\sqrt{N}(1+\eta\beta)}\max\left\lbrace\frac{\delta(r)}{1+\frac{\sqrt{{\#\sigma_1}^{\max}\cdot\nnzb(\overbar{\nabla{}f(\mX^*)})}}{\#\sigma_1(\overbar{\nabla{}f(\mX^*)})}},\frac{\delta(r-{\#\sigma_1}^{\max}+1)}{\frac{1}{\sqrt{{\#\sigma_1}^{\max}}}+\frac{\sqrt{{\#\sigma_1}^{\max}\cdot\nnzb(\overbar{\nabla{}f(\mX^*)})}}{\#\sigma_1(\overbar{\nabla{}f(\mX^*)})}}\right\rbrace,
\end{align}
or  
\begin{align} \label{ineq:spectralRadiusBound}
\Vert\mX-\mX^*\Vert_2 \le \frac{\eta}{2(1+\eta\beta_2)}\delta(r),
\end{align}
then $\rank_{\textnormal{t}}(\Pi_{\lbrace\Vert\mY\Vert_*\le1\rbrace}[\mX-\eta\nabla{}f(\mX)])\le r$.
\end{theorem}

\begin{remark}
The Frobenius radius around the optimal solution in \eqref{ineq:FrobRadiusBound} behaves as $\Omega(\eta\delta(r)/\sqrt{N})$. This holds since the second term in the denominator (in both terms inside the maximum) satisfies that 
\begin{align*}
& \sqrt{{\#\sigma_1}^{\max}\cdot\nnzb(\overbar{\nabla{}f(\mX^*)})}\Big/\#\sigma_1(\overbar{\nabla{}f(\mX^*)})
\\ & =\sqrt{{\#\sigma_1}^{\max}/\#\sigma_1(\overbar{\nabla{}f(\mX^*)})}\sqrt{\nnzb(\overbar{\nabla{}f(\mX^*)})/\#\sigma_1(\overbar{\nabla{}f(\mX^*)})}\le1.
\end{align*}
Furthermore, it is important to note that the division by $\sqrt{N}$ arises from the definition of the Fourier transformation. In the Fourier matrix space, correspondingly to the matrix setting in \cite{garberNuclearNormMatrices}, the distance $\Vert\barX-\overbar{\X^*}\Vert_F$ is  bounded by a constant proportional to $\delta(r)$ without the division by $\sqrt{N}$.
\end{remark}

\begin{remark}
Note that there are situations where for a large enough value $r$ the second term in the RHS of \eqref{ineq:FrobRadiusBound} can be much larger than the first one. As an example, consider the case where all the frontal slices of $\overbar{\nabla{}f(\mX^*)}$ are nonzero, i.e., $\nnzb(\overbar{\nabla{}f(\mX^*)})=N$, and the multiplicities of the top singular value of all frontal slices are all equal, that is they are all equal to ${\#\sigma_1}^{\max}$. In this case, the second term in the RHS  of \eqref{ineq:FrobRadiusBound} grows with an extra factor of $\sqrt{{\#\sigma_1}^{\max}}$, which can be quite significant.
\end{remark}

\begin{remark}
The bound w.r.t. the spectral norm in \eqref{ineq:spectralRadiusBound} does not scale with $1/\sqrt{N}$, and thus might be considerably larger than its Frobenius counterpart in \eqref{ineq:FrobRadiusBound}. The Euclidean algorithms considered in this work naturally depend on Euclidean distances, which is also our focus when discussing concrete algorithmic results in \cref{sec:algorithmicConsequences}. Nevertheless, if these algorithms could be initialized in a way that  guarantees that all iterates remain within the spectral ball corresponding to \eqref{ineq:spectralRadiusBound} (e.g., by considering the appropriate level set of $f(\cdot)$), this may lead to significantly more relaxed initialization requirements, which may be easier to satisfy in practice.
\end{remark}

Here we only give a sketch of the main ideas in the proof of \cref{lemma:RadiusForProjectedGradient_d_dim}. The full proof is given in \cref{sec:App:proofTheoremSmoothRadius}. 

\begin{proof}[Proof sketch of \cref{lemma:RadiusForProjectedGradient_d_dim}]
From \cref{lemma:iffLowRankfCondiotionOrderP} it follows that for any tensor $\mP\in\reals^{n_1\times\cdots\times n_d}$ the condition $\rank_{\textnormal{t}}(\Pi_{\lbrace\Vert\mY\Vert_*\le1\rbrace}(\mP))\le r$ holds if and only if 
\begin{align} \label{ineq:conditionProofSketch}
\frac{1}{N}\sum_{i_d=1}^{n_d}\cdots\sum_{i_3=1}^{n_3}\left(\sum_{i=1}^{\#\sigma_{>{}r+1}^{(i_3,\dots,i_d)}(\mP)}\sigma_{i}(\barP^{(i_3,\ldots,i_d)})-\#\sigma_{>{}r+1}^{(i_3,\dots,i_d)}(\mP) \cdot \sigma_{r+1}^{\max}(\barP)\right)\ge1,
\end{align}
where $\barP=\bdiag(\overbar{\mP})$, and we denote $\sigma_{r+1}^{\max}(\barP) = \max_{k_3\in[n_3],\ldots,k_d\in[n_d]}\sigma_{r+1}(\barP^{(k_3,\ldots,k_d)})$ and $\#\sigma_{>r+1}^{(i_3,\dots,i_d)}(\mP) =\left|\left\lbrace i\ \bigg\vert\ \sigma_i(\barP^{(i_3,\ldots,i_d)})>\sigma_{r+1}^{\max}(\barP)\right\rbrace\right|$.

Denote the tensor $\mP^* = \mX^* - \eta\nabla{}f(\mX^*)$ for some optimal solution $\mX^*$, and note that for all $i_3\in[n_3],\ldots,i_d\in[n_d]$,  ${\overbar{\P^*}}^{(i_3,\ldots,i_d)}:= {\overbar{\X^*}}^{(i_3,\ldots,i_d)} - \eta\overbar{\nabla{}f(\mX^*)}^{(i_3,\ldots,i_d)}$. Invoking \cref{lemma:svd_opt_grad_d_dim}, we have that for any $i_3\in[n_3],\ldots,i_d\in[n_d]$,
\begin{align} \label{eq:svOpt_d_dim1_sketch}
\forall i\le\rank({\overbar{\X^*}}^{(i_3,\ldots,i_d)}):\quad & \sigma_i({\overbar{\P^*}}^{(i_3,\ldots,i_d)}) = \sigma_i({\overbar{\X^*}}^{(i_3,\ldots,i_d)}) + \eta \sigma_1(\overbar{\nabla{}f(\mX^*)}), \nonumber
\\ \forall i>\rank({\overbar{\X^*}}^{(i_3,\ldots,i_d)}):\quad & \sigma_i({\overbar{\P^*}}^{(i_3,\ldots,i_d)}) = \eta \sigma_i(\overbar{\nabla{}f(\mX^*)}^{(i_3,\ldots,i_d)}).
\end{align}

Therefore, since $\#\sigma_{>{}r+1}^{(i_3,\dots,i_d)}(\mP^*)\le r$, it follows that
\begin{align} \label{eq:PositiveSlackAtOpt}
& \frac{1}{N}\sum_{i_d=1}^{n_d}\cdots\sum_{i_3=1}^{n_3}\left(\sum_{i=1}^{\#\sigma_{>{}r+1}^{(i_3,\dots,i_d)}(\mP^*)}\sigma_{i}({\overbar{\P^*}}^{(i_3,\ldots,i_d)})-\#\sigma_{>{}r+1}^{(i_3,\dots,i_d)}(\mP^*)\cdot\sigma_{r+1}^{\max}(\overbar{\P^*})\right) \nonumber
\\ & \underset{(a)}{=} \frac{1}{N}\sum_{i_d=1}^{n_d}\cdots\sum_{i_3=1}^{n_3}\left(\sum_{i=1}^{\#\sigma_{>{}r+1}^{(i_3,\dots,i_d)}(\mP^*)}\left(\sigma_i({\overbar{\X^*}}^{(i_3,\ldots,i_d)}) +\eta \sigma_1(\overbar{\nabla{}f(\mX^*)})\right)\right) \nonumber
\\ & \ \ \ -\frac{1}{N}\sum_{i_d=1}^{n_d}\cdots\sum_{i_3=1}^{n_3}\left(\#\sigma_{>{}r+1}^{(i_3,\dots,i_d)}(\mP^*)\cdot\eta\max_{i_3\in[n_3],\ldots,i_d\in[n_d]}\sigma_{r+1}(\overbar{\nabla{}f(\mX^*)}^{(i_3,\ldots,i_d)})\right) \nonumber
\\ & = \frac{1}{N}\sum_{i_d=1}^{n_d}\cdots\sum_{i_3=1}^{n_3}\left(\sum_{i=1}^{\rank({\overbar{\X^*}}^{(i_3,\ldots,i_d)})}\sigma_{i}({\overbar{\X^*}}^{(i_3,\ldots,i_d)})+\#\sigma_{>{}r+1}^{(i_3,\dots,i_d)}(\mP^*)\cdot\eta\delta(r)\right) \nonumber
\\ & = \Vert\mX^*\Vert_*+\eta\delta(r)\frac{1}{N}\sum_{i_d=1}^{n_d}\cdots\sum_{i_3=1}^{n_3}\#\sigma_{>{}r+1}^{(i_3,\dots,i_d)}(\mP^*) \nonumber 
\\ & \underset{(b)}{=} 1+\eta\delta(r)\frac{1}{N}\sum_{i_d=1}^{n_d}\cdots\sum_{i_3=1}^{n_3}\#\sigma_{>{}r+1}^{(i_3,\dots,i_d)}(\mP^*),
\end{align}
where (a) follows from plugging in \eqref{eq:svOpt_d_dim1_sketch} and (b) follows from the assumption that $\nabla{}f(\mX^*)\not=0$, which implies that $\Vert{\mX^*}\Vert_*=1$.

Thus, not only does $\mP^*$ satisfies the condition \eqref{ineq:conditionProofSketch} w.r.t. to the parameter $r$, but it actually satisfies it with an additional positive slack of 
\begin{align*} 
\frac{\eta\delta(r)}{N}\sum_{i_d=1}^{n_d}\cdots\sum_{i_3=1}^{n_3}\#\sigma_{>{}r+1}^{(i_3,\dots,i_d)}(\mP^*).
\end{align*}
To see why the sum over $\#\sigma_{>{}r+1}^{(i_3,\dots,i_d)}(\mP^*)$ is guaranteed to be positive, note that using 
 \eqref{eq:svOpt_d_dim1_sketch} together with the assumption that $r\ge{\#\sigma_1}^{\max}$, we have that
\begin{align*}
\sigma_1(\overbar{\P^*}) & = \sigma_1(\overbar{\X^*})+\eta \sigma_1(\overbar{\nabla{}f(\X^*)}) \geq \eta \sigma_1(\overbar{\nabla{}f(\X^*)}) 
 > \eta \max_{i_3\in[n_3],\ldots,i_d\in[n_d]}\sigma_{r+1}(\overbar{\nabla{}f(\mX^*)}^{(i_3,\ldots,i_d)}) 
 \\ &= \max_{i_3\in[n_3],\ldots,i_d\in[n_d]}\sigma_{r+1}(\overbar{\P^*}^{(i_3,\ldots,i_d)}) = \sigma_{r+1}^{\max}(\overbar{\P^*}).
\end{align*}
There exists at least one frontal slice of $\overbar{\mP}$, for some $i_3\in[n_3],\ldots,i_d\in[n_d]$, whose top singular value is also the top singular value of  $\overbar{\P^*}$, and thus, this frontal slice satisfies that %
$\#\sigma_{>{}r+1}^{(i_3,\dots,i_d)}(\mP^*)\ge1$, which implies that the sum over all frontal slices is positive.

The smoothness of $f$ implies that for any tensor $\mX$ close to $\mX^*$, the tensor $\mP=\mX-\eta\nabla{}f(\mX)$  is also close to $\mP^*$, up to an additional multiplicative factor. 
Therefore, by applying standard perturbation bounds for the singular values of all frontal slices of $\mP^*$ in Eq. \eqref{eq:PositiveSlackAtOpt}, and using the fact that the positive slack in the RHS of \eqref{eq:PositiveSlackAtOpt} allows to absorb sufficiently small errors (due to the use of these perturbation bounds),
we can establish that for such a tensor $\mX$, sufficiently close to $\mX^*$, the tensor $\mP = \mX-\eta\nabla{}f(\mX)$ also satisfies condition  \eqref{ineq:conditionProofSketch} w.r.t. the parameter $r$, meaning that $\rank_{\textnormal{t}}(\Pi_{\lbrace\Vert\mY\Vert_*\le1\rbrace}[\mP])\le r$.
\end{proof}

In connection with our discussion in Section \ref{sec:complexityAnalysis}, the result in  \cref{lemma:RadiusForProjectedGradient_d_dim} implies that projected-gradient methods for solving Problem \eqref{mainModel}, when initialized sufficiently close to an optimal solution, could be implemented using only low-rank SVDs to compute the projected gradient mapping w.r.t. the unit TNN ball, without changing their outputs. This is captured in the following corollary.

\begin{corollary}
Fix an optimal solution $\mX^*$ to Problem \eqref{mainModel}.
Let $r$ such that $\min\lbrace n_1,n_2\rbrace>r\ge{\#\sigma_1}^{\max}$, where ${\#\sigma_1}^{\max}:=\max_{i_3\in[n_3],\ldots,i_d\in[n_d]}\#\sigma_1(\overbar{\nabla{}f(\mX^*)}^{(i_3,\ldots,i_d)})$. 
Consider a projected gradient method, i.e., a method that relies on computing $\Pi_{\Vert{\mY}\Vert_*\leq 1}[\mX-\eta\nabla{}f(\mX)]$ for some input tensor $\mX$, initialized such that all points $\lbrace\mX_t\rbrace_{t\ge1}$ to which the projected-gradient mapping is applied  satisfy that $\Vert\mX_t-\mX^*\Vert_F \le R_0(r,\eta)$, where $R_0(r,\eta)$ is the RHS of \eqref{ineq:FrobRadiusBound}.
Then, throughout the run of the method, all projections onto the unit TNN ball could be replaced with their rank-$r$ counterparts (see \cref{def:lowRankProjection}) without changing their outputs 
\end{corollary}

\subsection{Some concrete algorithmic implications}
Our results obtained thus far regarding the quadratic growth bound and low-rank projections imply, mostly in a straightforward manner, the algorithmic results detailed in Table  \ref{table:bottomLine} regarding  the smooth case. A formal description of these results and full proofs are given for completeness in \cref{sec:algorithmicConsequencesSmooth}.

\subsubsection{Computing certificates for low tubal rank projections}
Since \cref{lemma:RadiusForProjectedGradient_d_dim} only applies in some neighborhood of an optimal solution, it is of interest to have a practical procedure for verifying if the rank-$r$ truncated projection of a given tensor onto the unit TNN ball indeed equals the exact Euclidean projection. In addition, from a practical point of view, it does not matter whether the conditions of  \cref{lemma:RadiusForProjectedGradient_d_dim} hold or not. As long as the rank-$r$ truncated projection equals its exact counterpart, we are guaranteed that the method converges with its original convergence guarantees (since there is no change to the sequence of iterates it produces), while only efficient low-rank matrix SVDs are required to compute the projections. Note that the condition in \cref{lemma:iffLowRankfCondiotionOrderP}, which characterizes the structure of the Euclidean projection onto the unit TNN ball, yields exactly such a verification procedure. By simply increasing the rank of the matrix SVDs of the frontal slices by one i.e., computing for each frontal slice of the Fourier transform of the tensor to project a rank-$(r+1)$ SVD, instead of only a rank-$r$ SVD, we can exactly check whether the condition \cref{lemma:iffLowRankfCondiotionOrderP} holds or not with respect to the tubal rank parameter $r$.

\section{The Nonsmooth Case}
\label{sec:nonsmoothCase}

In this section we turn to consider Problem \eqref{mainModel} in the important case that $f(\cdot)$ is nonsmooth. Our main goal here is to establish that the nonsmooth problem could also be solved via first-order methods that, at least in the proximity of an optimal solution which satisfies the appropriate strict complementarity condition, only require low-rank matrix SVDs to compute the projection onto the unit TNN ball. However, this ambition is complicated by the fact that in \cite{ourExtragradient}, which considered a related setting of nonsmooth low-rank matrix recovery problems, it was established that a result in the spirit of our \cref{lemma:RadiusForProjectedGradient_d_dim} cannot be obtained for the projected subgradient descent method, which is perhaps the simplest nonsmooth first-order method to consider. That is, in the matrix setting of  \cite{ourExtragradient}, it was established that in any proximity of a low-rank optimal solution which satisfies strict complementarity (when appropriately defined for the nonsmooth problem), the projected subgradient descent mapping could result in a matrix of higher rank than that of the optimal solution.  To circumvent this difficulty, \cite{ourExtragradient} proposed to consider nonsmooth objectives, such that the nonsmooth minimization problem could be written as a smooth saddle-point problem, and they considered the application of the projected Extragradient method for the saddle-point problem. They established that indeed in the proximity of low-rank optimal solutions which satisfy SC, the Extragradient method is guaranteed to produce low-rank iterates. This is also the approach we take here w.r.t. to our nonsmooth tensor optimization problem.

 \subsection{Generalized strict complementarity for nonsmooth problems}
As in \cref{sec:lowTubalRankProjectionsSmooth} which considered the smooth setting, here we also consider a generalized version of the strict complementarity condition.
We follow the definition of generalized strict complementarity established in \cite{ourExtragradient} for nonsmooth low-rank matrix problems. We begin by recalling the first-order optimality condition in case $f$ is nonsmooth.

\begin{lemma}[first-order optimality condition, see \cite{beckOptimizationBook}] \label{optCondition}
Let $f:\reals^{n_1\times\cdots\times n_d}\rightarrow\reals$ be a convex function.  $\mX^*\in\lbrace\mX\in\reals^{n_1\times\cdots\times n_d}\ \vert\ \Vert\mX\Vert_*\le1\rbrace$ minimizes $f$ over $\lbrace\mX\in\reals^{n_1\times\cdots\times n_d}\ \vert\ \Vert\mX\Vert_*\le1\rbrace$ if and only if there exists a subgradient $\mG^*\in\partial f(\mX^*)$ such that $\langle \mX-\mX^*,\mG^*\rangle\ge0$ for all $\mX\in\lbrace\mX\in\reals^{n_1\times\cdots\times n_d}\ \vert\ \Vert\mX\Vert_*\le1\rbrace$.
\end{lemma}

\begin{definition}[Generalized strict complementarity for nonsmooth problems] \label{lemma:SCnonsmooth}
Let $\mX^*\in\lbrace\mX\in\reals^{n_1\times\cdots\times n_d}\ \vert\ \Vert\mX\Vert_*\le1\rbrace$ be an optimal solution to Problem \eqref{mainModel}. $\mX^*$ satisfies the generalized strict complementarity assumption with measure $\delta(r)$ if there exists a subgradient $\mG^*\in\partial f(\mX^*)$ such that $\langle \mX-\mX^*,\mG^*\rangle\ge0$ for all $\mX\in\lbrace\mX\in\reals^{n_1\times\cdots\times n_d}\ \vert\ \Vert\mX\Vert_*\le1\rbrace$ and $r$ such that $\min\lbrace n_1,n_2\rbrace>r\ge\max_{i_3\in[n_3],\ldots,i_d\in[n_d]}{\#\sigma_1(\overbar{\mG^*}^{(i_3,\ldots,i_d)})}$, and
\begin{align*}
\delta(r) := \sigma_1(\overbar{\G^*}) - \max_{i_3\in[n_3],\ldots,i_d\in[n_d]}\sigma_{r+1}(\overbar{\G^*}^{(i_3,\ldots,i_d)}) > 0 ,
\end{align*}
where $\overbar{\G^*}:=\bdiag(\overbar{\mG^*})$. 
\end{definition}
Note that generalized strict complementarity in the nonsmooth setting takes the same form as the one in the corresponding smooth setting, when replacing the gradient direction at the optimal solution in \cref{def:GSC} with a subgradient which satisfies the first-order optimality condition in \cref{optCondition}.
 
\subsection{From nonsmooth to saddle-point formulation}

We assume the nonsmooth problem  \eqref{mainModel} can be written as a maximum of smooth functions,
i.e., $f(\mX) = \max_{\y\in\mK} F(\mX,\y)$, where $\mK\subset\mathbb{Y}$ is a convex and compact subset of a finite linear space $\mathbb{Y}$ over $\reals$. We assume there exists an efficient method of computing Euclidean projections onto $\mK$. We also assume $F(\cdot,\y)$ is convex for all $\y\in\mK$, and $F(\mX,\cdot)$ is concave for all $\mX\in\lbrace\mX\in\reals^{n_1\times\cdots\times n_d}\ \vert\ \Vert\mX\Vert_*\le1\rbrace$. Under these assumptions we can reformulate Problem \eqref{mainModel} as the following saddle-point problem:
\begin{align} \label{saddlePointProblem}
\min_{\Vert\mX\Vert_*\le\tau}\max_{\y\in\mK} F(\mX,\y).
\end{align}
We also assume that $F$ is smooth with respect to all components, that is, that there exists constants $\beta_{X},\beta_{y},\beta_{Xy},\beta_{yX}\ge0$ such that for any $\mX,\boldsymbol{\widetilde{\mX}}\in\lbrace\mX\in\reals^{n_1\times\cdots\times n_d}\ \vert\ \Vert\mX\Vert_*\le1\rbrace$ and  $\y,\widetilde{\y}\in\mK$, it holds that
\begin{align} \label{def:betas}
& \Vert\nabla_{\mX}F(\mX,\y)-\nabla_{\mX}F(\boldsymbol{\widetilde{\mX}},\y)\Vert_F \le \beta_{X}\Vert\mX-\boldsymbol{\widetilde{\mX}}\Vert_F, \nonumber
\\ & \Vert\nabla_{\y}F(\mX,\y)-\nabla_{\y}F(\mX,\widetilde{\y})\Vert_2 \le \beta_{y}\Vert\y-\widetilde{\y}\Vert_2, \nonumber
\\ & \Vert\nabla_{\mX}F(\mX,\y)-\nabla_{\mX}F(\mX,\widetilde{\y})\Vert_F \le \beta_{Xy}\Vert\y-\widetilde{\y}\Vert_2, \nonumber
\\ & \Vert\nabla_{\y}F(\mX,\y)-\nabla_{\y}F(\boldsymbol{\widetilde{\mX}},\y)\Vert_2 \le \beta_{yX}\Vert\mX-\boldsymbol{\widetilde{\mX}}\Vert_F, 
\end{align}
where $\nabla_{\mX}F=\frac{\partial F}{\partial\mX}$ and  $\nabla_{\y}F=\frac{\partial F}{\partial\y}$. 

Throughout the rest of this section we denote $\Vert\cdot\Vert$ to be the Euclidean norm over the product space $\reals^{n_1\times\cdots\times n_d}\times\mathbb{Y}$. 

Based on the definition of generalized strict complementarity for matrix saddle-point problems that has been established in \cite{ourExtragradient}, we define generalized strict complementarity for tensor saddle-point problems.
\begin{definition}[Generalized strict complementarity for saddle-point problems]
\label{def:GSCsaddlePoint}
Let $(\mX^*,\y^*)$ be a saddle point of Problem \eqref{saddlePointProblem}.  $(\mX^*,\y^*)$ satisfies the generalized strict complementarity assumption with measure $\delta(r)$, 
for  $r$ such that $\max_{i_3\in[n_3],\ldots,i_d\in[n_d]}{\#\sigma_1(\overbar{\nabla_{\mX}F(\mX^*,\y^*)}^{(i_3,\ldots,i_d)})}\le r  <\min\lbrace n_1,n_2\rbrace$, if
\begin{align*}
\delta(r):=\sigma_1(\overbar{\nabla_{\mX}F(\mX^*,\y^*)}) - \max_{i_3\in[n_3],\ldots,i_d\in[n_d]}\sigma_{r+1}(\overbar{\nabla_{\mX}F(\mX^*,\y^*)}^{(i_3,\ldots,i_d)})>0.
\end{align*}
\end{definition}

In \cite{ourExtragradient} the authors showed that under a fairly mild additional structural assumption on the objective function $f(\cdot)$, which holds for many nonsmooth functions of interest, generalized strict complementarity for Problem \eqref{saddlePointProblem} is equivalent to generalized strict complementarity for Problem \eqref{mainModel}. 

\begin{assumption} \label{lemma:structureOfNonsmoothAssumption}
$f(\mX)$ is of the form $f(\mX)= h(\mX) + \max_{\y\in\mK} \y^{\top}(\mathcal{A}(\mX)- \b)$, where $h(\cdot)$
is smooth and convex, and $\mathcal{A}$ is a linear map.
\end{assumption}

The following lemma for the tensor setting and its proof are identical to the matrix case in \cite{ourExtragradient}.

\begin{lemma}[Lemma 6 in \cite{ourExtragradient}] \label{lemma:connectionSubgradientNonsmoothAndSaddlePoint}
If $(\mX^*,\y^*)$ is a saddle-point of Problem \eqref{saddlePointProblem} then $\mX^*$ is an optimal solution to Problem \eqref{mainModel},  $\nabla_{\mX}F(\mX^*,\y^*)\in\partial f(\mX^*)$,  and for all $\mX\in\lbrace\mX\in\reals^{n_1\times\cdots\times n_d}\ \vert\ \Vert\mX\Vert_*\le1\rbrace$ it holds that $\langle\mX-\mX^*,\nabla_{\mX}F(\mX^*,\y^*)\rangle\ge0$.
Conversely, under \cref{lemma:structureOfNonsmoothAssumption}, if $\mX^*$ is an optimal solution to Problem \eqref{mainModel}, and $\mG^*\in\partial{}f(\mX^*)$ which satisfies $\langle\mX-\mX^*,\mG^*\rangle\ge0$  for all $\mX\in\lbrace\mX\in\reals^{n_1\times\cdots\times n_d}\ \vert\ \Vert\mX\Vert_*\le1\rbrace$, then there exists $\y^*\in\argmax_{\y\in\mathcal{K}}F(\mX^*,\y)$ such that $(\mX^*,\y^*)$ is a saddle-point of Problem \eqref{saddlePointProblem}, and $\nabla_{\mX}F(\mX^*,\y^*)=\mG^*$.
\end{lemma}

\begin{remark}
Under \cref{lemma:structureOfNonsmoothAssumption}, \cref{lemma:connectionSubgradientNonsmoothAndSaddlePoint} implies that generalized strict complementarity with measure $\delta(r)$ for some optimal solution $\mX^*$ to Problem \eqref{mainModel} implies that generalized strict complementarity with measure $\delta(r)$ holds for a corresponding saddle-point $(\mX^*,\y^*)$ of Problem \eqref{saddlePointProblem}. Nevertheless, our convergence results for Problem \eqref{saddlePointProblem}, which are stated directly in terms of generalized strict complementarity for saddle-point problems (\cref{def:GSCsaddlePoint}), are not dependent on whether \cref{lemma:structureOfNonsmoothAssumption} holds or not.
\end{remark}

\subsection{Low tubal rank of the extragradient mapping near low tubal rank saddle-points}

The projected extragradient method for solving saddle-point problem \eqref{saddlePointProblem} is given for convenience of the reader as  \cref{alg:EG}.

We are now ready to present our main result for this section which is an analogue of \cref{lemma:RadiusForProjectedGradient_d_dim} for the smooth setting. The theorem states that the primal updates of the extragradient method (the updates of $\mX_{t+1}$ and $\mZ_{t+1}$ in  \cref{alg:EG}) result in low tubal rank tensors at the proximity of saddle-points which satisfy generalized strict complementarity (\cref{def:GSCsaddlePoint}).

\begin{theorem} \label{lemma:RadiusForExtragradient_d_dim}
Assume $\nabla{}_{\mX}F$ is non-zero over the unit TNN ball and fix some saddle-point $(\mX^*,\y^*)$ of Problem \eqref{saddlePointProblem}. 
Denote $\nnzb(\overbar{\nabla_{\mX}F^*}):=\#\lbrace(i_3,\dots,i_d)~|~\overbar{\nabla_{\mX}F(\mX^*,\y^*)}^{(i_3,\ldots,i_d)}\not=0\rbrace$ and ${\#\sigma_1}^{\max}:=\max_{i_3\in[n_3],\ldots,i_d\in[n_d]}\#\sigma_1(\overbar{\nabla_{\mX}F(\mX^*,\y^*)}^{(i_3,\ldots,i_d)})$  and assume ${\#\sigma_1}^{\max}<\min\lbrace n_1,n_2\rbrace$. 
Then, for any $\eta\ge0$, $\min\lbrace n_1,n_2\rbrace>r\ge{\#\sigma_1}^{\max}$, and $(\mX,\y),(\mZ,\w)\in\reals^{n_1\times{}\cdots\times{}n_d}\times\mK$, if 
\begin{align} \label{ineq:boundNonsmooth}
& \max\lbrace\Vert\mX-\mX^*\Vert_F,\Vert(\mZ,\w)-(\mX^*,\y^*)\Vert\rbrace  \nonumber
\\ & \le \frac{\eta}{\sqrt{N}K}
\max\left\lbrace\frac{\delta(r)}{1+\frac{\sqrt{{\#\sigma_1}^{\max}\cdot\nnzb(\overbar{\nabla_{\mX}F^*})}}{\#\sigma_1(\overbar{\nabla_{\mX}F(\mX^*,\y^*)})}},\frac{\delta(r-{\#\sigma_1}^{\max}+1)}{\frac{1}{\sqrt{{\#\sigma_1}^{\max}}}+\frac{\sqrt{{\#\sigma_1}^{\max}\cdot\nnzb(\overbar{\nabla_{\mX}F^*})}}{\#\sigma_1(\overbar{\nabla_{\mX}F(\mX^*,\y^*)})}}\right\rbrace,
\end{align} 
where $K=1+\sqrt{2}\eta\max\lbrace\beta_{X},\beta_{Xy}\rbrace$, then $\rank_{\textnormal{t}}(\Pi_{\lbrace\Vert\mY\Vert_*\le1\rbrace}[\mX-\eta\nabla{}_{\mX}F(\mZ,\w)])\le r$.
\end{theorem}

The proof of \cref{lemma:RadiusForExtragradient_d_dim} follows similar arguments to those used in the proof of  \cref{lemma:RadiusForProjectedGradient_d_dim} for the smooth setting and is given in  \cref{sec:App:proofLemmaNonsmooth}.

\begin{algorithm}
	\caption{Projected extragradient descent for saddle-point problems}\label{alg:EG}
	\begin{algorithmic}
		\STATE \textbf{Input:} step-size $\eta\ge0$
		\STATE \textbf{Initialization:} $(\mX_1,\y_1)\in\lbrace\reals^{n_1\times\cdots\times n_d}\ \vert\ \Vert\mX\Vert_*\le1\rbrace\times\mathcal{K}$
		\FOR {$t=1,2,\ldots$}
			\STATE $\mZ_{t+1} = \Pi_{\lbrace\Vert{\mY}\Vert_*\leq 1\rbrace}[\mX_t-\eta\nabla_{\mX}F(\mX_t,\y_t)]$
			\STATE $\w_{t+1} = \Pi_{\mathcal{K}}[\y_t-\eta\nabla_{\y}F(\mX_t,\y_t)]$
			\STATE $\mX_{t+1} = \Pi_{\lbrace\Vert{\mY}\Vert_*\leq 1\rbrace}[\mX_t-\eta\nabla_{\mX}F(\mZ_{t+1},\w_{t+1})]$
			\STATE $\y_{t+1} = \Pi_{\mathcal{K}}[\y_t-\eta\nabla_{\y}F(\mZ_{t+1},\w_{t+1})]$
		\ENDFOR
	\end{algorithmic}
\end{algorithm}

Applying the result of \cref{lemma:RadiusForExtragradient_d_dim} to the projected extragradient decent method (\cref{alg:EG}), we have that t if the method is initialized within distance $R_0(r)/(1+\sqrt{2})$ from a saddle-point. where $R_0(r)$ is the bound in \eqref{ineq:boundNonsmooth}, then the method converges to a saddle-point with its well known rate: $$\frac{1}{T}\sum_{t=1}^{T}\max_{\y\in\mK}F(\mZ_{t+1},\y)-\frac{1}{T}\sum_{t=1}^{T}\min_{\mX\in\lbrace\mX\ \vert\ \Vert\mX\Vert_*\le1\rbrace}F(\mX,\w_{t+1})=\mathcal{O}(1/T),$$ while only requiring matrix SVDs of rank $r$ for computing the projections onto the unit TNN ball. 

Returning back to the original nonsmooth Problem \eqref{mainModel}, by using the relationship $f(\mX) = \max_{\y\in\mK} F(\mX,\y)$, we can observe that an approximated saddle-point of Problem \eqref{saddlePointProblem} translates back into an approximated optimal solution to Problem \eqref{mainModel}, in the sense that the following convegence rate holds w.r.t. Problem \eqref{mainModel}:
$$\min_{t\in[T]}f(\mZ_{t+1})-f(\mX^*)=\mathcal{O}(1/T).$$
The formal result and proof is given in \cref{sec:algorithmicConsequencesNonmooth}.

\section{Empirical Evidence}
\label{sec:experiments}
In this section we present some empirical evidence in support of our theoretical findings. We consider two tasks: low-rank tensor completion, which can be formulated as an instance of Problem \eqref{mainModel} with smooth $f$, and tensor robust principal component analysis, which can  be formulated as an instance of Problem \eqref{mainModel} with nonsmooth $f$. For both tasks we show that in plausible scenarios in which the relaxation \eqref{mainModel} indeed approximately recovers a low tubal rank ground-truth tensor with small error, strict complementarity  holds. Moreover, using simple initializations schemes is sufficient to initialize standard first-order methods so that only low-rank SVDs are required to compute exact projections onto the TNN ball, from very early stages of their run, which is in accordance with our theoretical findings in \cref{lemma:RadiusForProjectedGradient_d_dim} and \cref{lemma:RadiusForExtragradient_d_dim}. For the smooth tensor completion task we also demonstrate linear convergence rates which is in accordance with the quadratic growth result in \cref{thm:QG}.

For both tasks we denote by $\mM$ the ground-truth tensor to be recovered from a partially observed or noisy tensor. We measure the relative initialization error by $\Vert(\Vert\mM\Vert_*/\tau)\mX_1-\mM\Vert_F^2/\Vert\mM\Vert_F^2$, where $\mX_1$ is the initialization tensor and $\tau$ is the radius of the TNN ball, and the relative recovery error by $\Vert(\Vert\mM\Vert_*/\tau)\mX^*-\mM\Vert_F^2/\Vert\mM\Vert_F^2$, where $\mX^*$ is the best estimate for the optimal solution outputted by the optimization algorithm.

\subsection{Tensor completion}
We consider the low-rank tensor completion problem from \cite{tensor1_krank} which can be written as:
\begin{align*}
\min_{\Vert\mX\Vert_*\le\tau}\left\{f(\mX):=\frac{1}{2}\sum_{ (i_1,\ldots,i_d)\in\mathcal{S}}(\langle\mX,\mE_{i_1,\ldots,i_d}\rangle-\mM_{i_1,\ldots,i_d})^2\right\},
\end{align*}
where $\mathcal{S}\subset[i_1]\times\cdots\times[i_d]$ is the set of indexes of the known entries of $\mM$, and $\mE_{i_1,\ldots,i_d}\in\reals^{n_1\times\cdots\times n_d}$ denotes the tensor
whose components are all zero except for the entry $(i_1,\ldots,i_d)$ which is equal to one.

We set $\mM=\mG\times_1\A_1\times_2\cdots\times_{\textnormal{d}}\A_d$ where the core tensor $\mG\in\reals^{r\times\cdots\times r}$ and all $\A_j\in\reals^{n_j\times r}$ are chosen randomly with normal entries $\mathcal{N}(0,1)$, and the operator $\times_j$ is the $j$-mode product of a tensor \cite{kolda} which we calculate using the code of \cite{nModeProuctCode}. Using this construction we have that $\rank_{\textnormal{t}}(\mM)=r$  with probability 1. We denote by $\rho$ to be the probability of each entry to be observed. 

For the initialization we take $\mX_1$ to be the rank-$r$ truncated projection of the tensor $\mR$ onto the TNN ball of radius $\tau$ (\cref{def:lowRankProjection}), where $\mR\in\reals^{n_1\times\cdots\times n_d}$ is such that 
\begin{align*}
\mR_{i_1,\ldots,i_d} = \bigg\lbrace\begin{array}{ll}
 \mM_{i_1,\ldots,i_d}, & \textrm{if}\ (i_1,\ldots,i_d)\in\mathcal{S}
\\ 0, & \textrm{if}\ (i_1,\ldots,i_d)\not\in\mathcal{S}
\end{array}.
\end{align*}

We test the model using the FISTA algorithm \cite{fista} and we use the t-product toolbox \cite{tensorToolbox} for some of the tensor operations. We set the step-size to $\eta=1$, which is the theoretical step-size and also the step-size that performed best empirically, the number of iterations in each experiment to $T=800$, and $\tau=0.7 \Vert\mM\Vert_*$. For each value of $r$ and $\rho$ we average the measurements over $10$ i.i.d. runs.

To verify whether the obtained solution at each run $\mX^*$ is indeed close to optimal we compute the corresponding dual-gap, which due to the convexity of $f$ is an upper bound on the approximation error w.r.t. function value, and is given by 
\begin{align} \label{eq:dualGap}
\max_{\Vert\mZ\Vert_*\le\tau}\langle\mX^*-\mZ,\nabla{}f(\mX^*)\rangle=\langle\mX^*+\ifft(\bdiag^{-1}((\tau N/\#\sigma_1(\overbar{\nabla{}f(\mX^*)}))\barU\barV^{\top})),\nabla{}f(\mX^*)\rangle,
\end{align} 
 where $\bdiag^{-1}(\cdot)$ is the inverse operator of $\bdiag(\cdot)$ ,and  $\barU$,$\barV$ are the block diagonal matrices such that for all $i_3\in[n_3],\ldots,i_d\in[n_d]$, the frontal slices $\barU^{(i_3,\ldots,i_d)}$ and ${\barV^{(i_3,\ldots,i_d)}}$ which are placed as blocks on the diagonals of $\barU$ and $\barV$, respectively, are all zero except for the columns that correspond to $\sigma_1(\overbar{\nabla{}f(\mX^*)})$ in the SVD of $\overbar{\nabla{}f(\mX^*)}$, in which case we take the corresponding singular vectors of $\overbar{\nabla{}f(\mX^*)}$ (left singular vectors for $\overbar{\U}$ and right singular vectors for $\overbar{\V}$) . It can be seen that the tensor $\bdiag^{-1}((\tau N/\#\sigma_1(\overbar{\nabla{}f(\mX^*)}))\barU\barV^{\top})$ satisfies the conjugate-complex symmetry conditions \eqref{eq:conjugateComplexSymmetryCondition}, and so in particular, the tensor $\ifft(\bdiag^{-1}(\tau N/\#\sigma_1(\overbar{\nabla{}f(\mX^*)}))\barU\barV^{\top}))$ is a maximizer of the inner product which is also guaranteed to be real-valued.

As can be seen in \cref{table:TC}, the recovery error is indeed significantly lower than the initial error. Also, it can be seen that our simple initialization scheme is enough so that the low tubal rank projections, with tubal rank $\leq r = \rank_t(\mM)$, are equal to the corresponding full-rank projections starting from very early stages of the run, which we verified by checking in each iteration whether the condition
\eqref{ineq:conditionOfLowTubalRank} held or not. In addition, it can be seen that strict complementarity indeed seems to hold with a significant measure, which we calculated using: 
\begin{align} \label{eq:strictCompExp}
\min_{i_3\in[n_3],\ldots,i_d\in[n_d]}\sigma_{1}(\overbar{\nabla{}f(\mX^*)}^{(i_3,\ldots,i_d)})-\sigma_{r_{i_3,\ldots,i_d}+1}(\overbar{\nabla{}f(\mX^*)}^{(i_3,\ldots,i_d)}),
\end{align}
where $r_{i_3,\ldots,i_d}=\rank(\overbar{\mX^*}^{(i_3,\ldots,i_d)})$ (recall that according to \cref{lemma:svd_opt_grad_d_dim}, for an exact optimal solution $\sigma_1(\overbar{\nabla{}f(\mX^*)}^{(i_3,\ldots,i_d)})$ is the same for all nonzero frontal slices, and so,  Eq. \eqref{eq:strictCompExp} is a lower bound on the gap as defined in \cref{lemma:SCequivalence}).

In \cref{fig:tableTC} we plot the approximation error in function value w.r.t. the obtained solution $\mX^*$ and the recovery error (both in log scale). Since for all instances the convergence is very rapid, for clarity, we plot only the first 50 iterations.
We observe that in all cases FISTA indeed seems to converge with a linear rate w.r.t. function value, which is in accordance with our quadratic growth result from \cref{thm:QG}. 

\begin{table*}[!htb]\renewcommand{\arraystretch}{1.3}
{\footnotesize
\begin{center}
  \begin{tabular}{| l | c | c | c | c | c | c |} \hline 
& $r=2$, $ \rho=0.6$ & $r=8$, $ \rho=0.6$ &  $r=2$, $ \rho=0.3$ \\
& $\mM\in\reals^{50\times 50\times 50}$ & $\mM\in\reals^{50\times 50\times 50}$ & $\mM\in\reals^{50\times 50\times 50}$ \\ \hline
  initialization error & $0.1521$ & $0.2408$ & $0.5222$ \\ \hline
    recovery error & $0.0561$  & $0.0625$ & $0.0662$ \\ \hline
    dual gap & $1.9\times{10}^{-9}$ & $2.8\times{10}^{-8}$ & $1.4\times{10}^{-9}$ \\ \hline
  strict complementarity (Eq. \eqref{eq:strictCompExp}) & $4.4627$ & $5.2455$ & $1.8960$  \\ \hline 
%
    first iteration from which  all    & $2.8$ & $4.1$ & $8.8$ \\
    projections are of tubal rank $\leq r$ & & &   \\ \hline
& $r=4$, $ \rho=0.4$ & $r=4$, $ \rho=0.4$ &  $r=8$, $ \rho=0.6$ \\
& $\mM\in\reals^{50\times 50\times 50}$ & $\mM\in\reals^{50\times 50\times 50\times 50}$ & $\mM\in\reals^{50\times 50\times 50\times 50}$ \\ \hline
  initialization error & $0.4223$ & $0.4135$ & $0.2339$ \\ \hline
    recovery error & $0.0669$ & $0.0723$ & $0.0659$ \\ \hline
    dual gap & $1.0\times{10}^{-8}$  & $5.7\times{10}^{-6}$ & $4.7\times{10}^{-5}$ \\ \hline
  strict complementarity (Eq. \eqref{eq:strictCompExp}) & $2.9049$ & $1.2415$ & $3.1019$ \\ \hline  
%
    first iteration from which  all  & $7.4$ & $7.8$ & $4.1$ \\ 
   projections are of tubal rank $\leq r$ & & &   \\ \hline
   \end{tabular}
  \caption{Numerical results for the tensor completion problem. Each result is the average of 10 i.i.d. runs.
  }\label{table:TC}
\end{center}
}
\vskip -0.2in
\end{table*}\renewcommand{\arraystretch}{1}

\begin{figure}[H] 
\centering
\begin{subfigure}{0.46\textwidth}
  \begin{subfigure}{0.5\textwidth}
    \includegraphics[width=\textwidth]{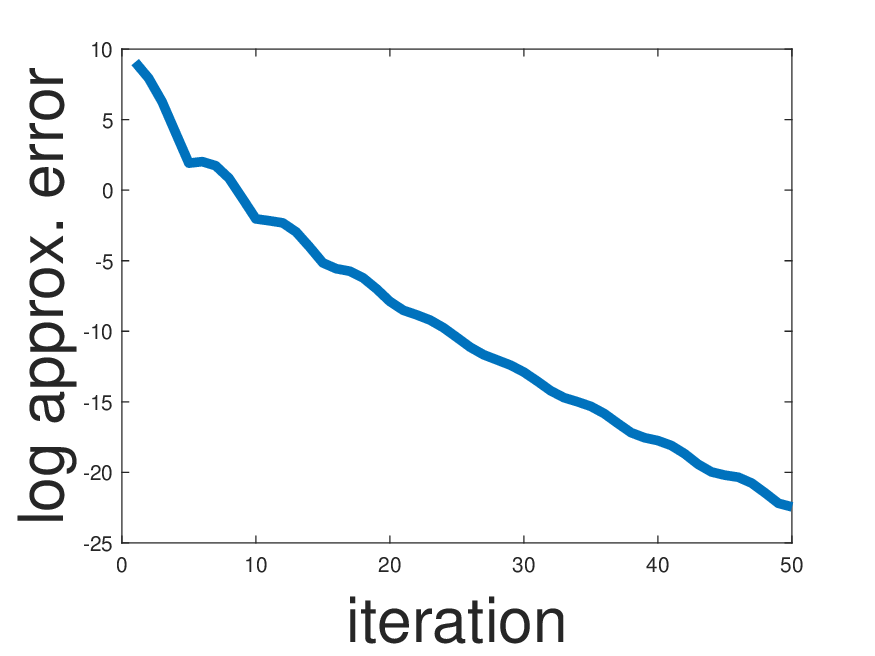}   
  \end{subfigure}\hfil
    \begin{subfigure}{0.5\textwidth}
    \includegraphics[width=\textwidth]{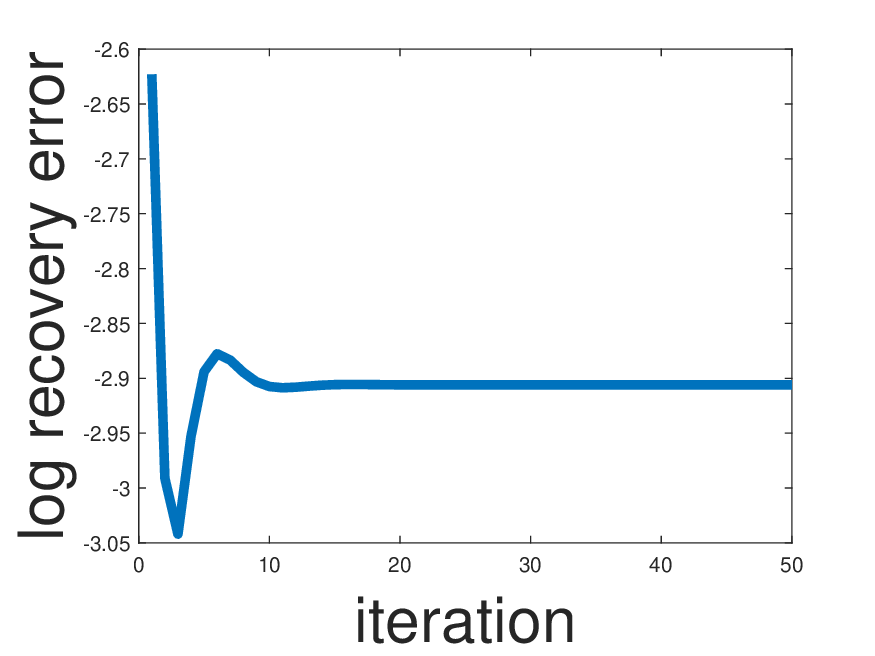} 
  \end{subfigure}
    \caption*{$r=2$, $ \rho=0.6$, $\mM\in\reals^{50\times 50\times 50}$}
\end{subfigure}\hfill
\begin{subfigure}{0.46\textwidth}
    \begin{subfigure}{0.5\textwidth}
    \includegraphics[width=\textwidth]{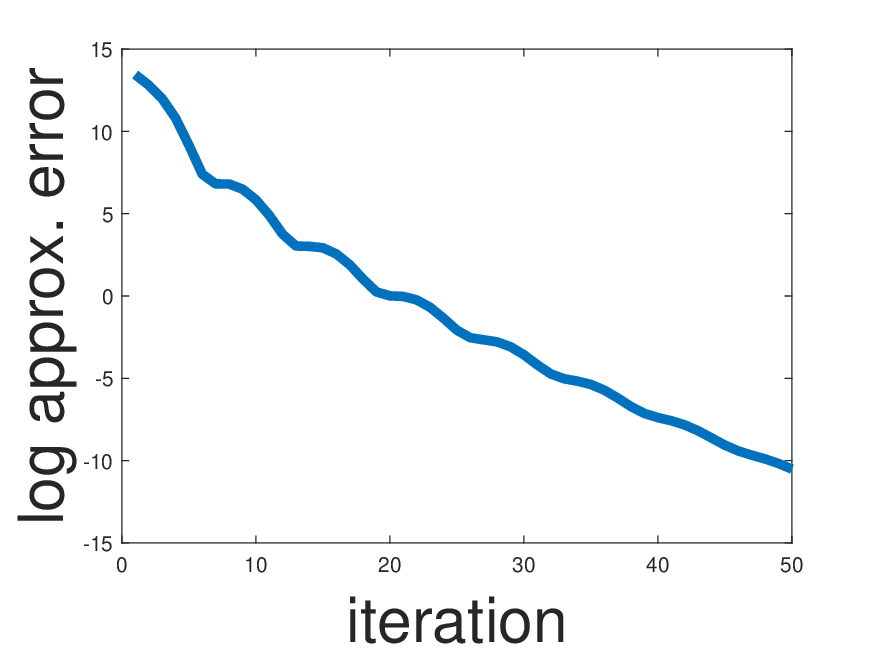}   
  \end{subfigure}\hfil
    \begin{subfigure}{0.5\textwidth}
    \includegraphics[width=\textwidth]{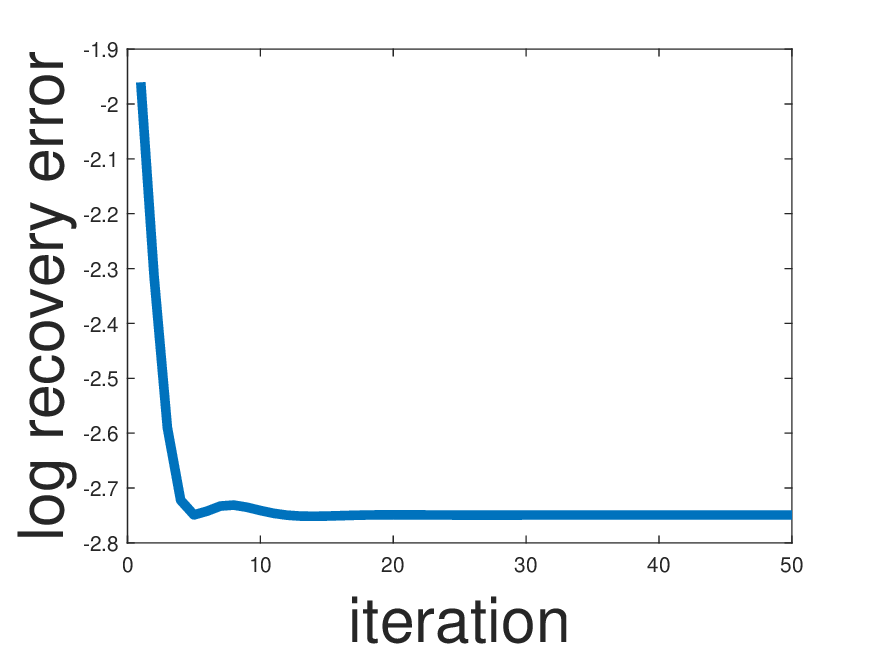} 
  \end{subfigure}
    \caption*{$r=8$, $ \rho=0.6$, $\mM\in\reals^{50\times 50\times 50}$}
     \end{subfigure}
%
    \medskip
    \begin{subfigure}{0.46\textwidth}
  \begin{subfigure}{0.5\textwidth}
    \includegraphics[width=\textwidth]{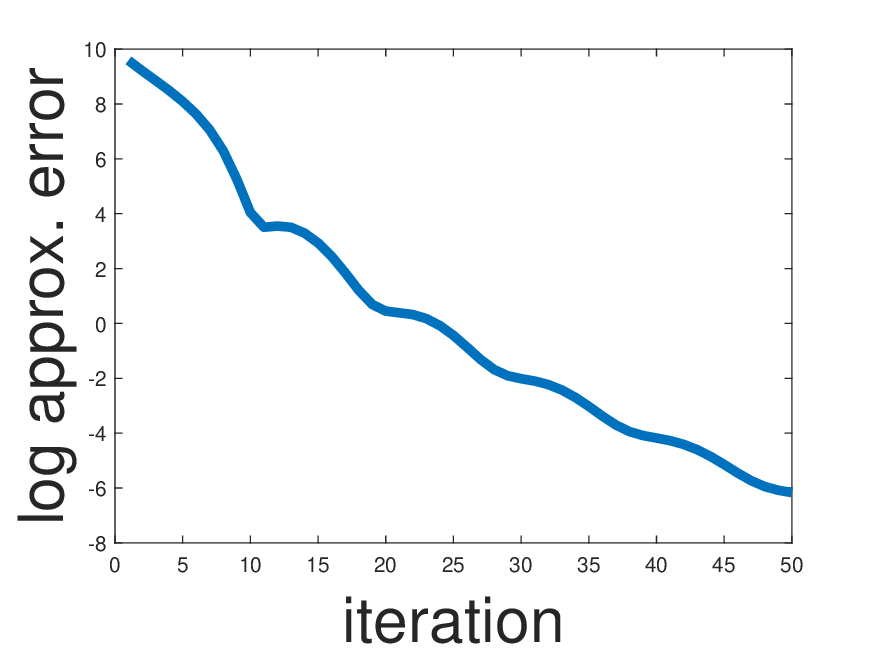}   
  \end{subfigure}\hfil
    \begin{subfigure}{0.5\textwidth}
    \includegraphics[width=\textwidth]{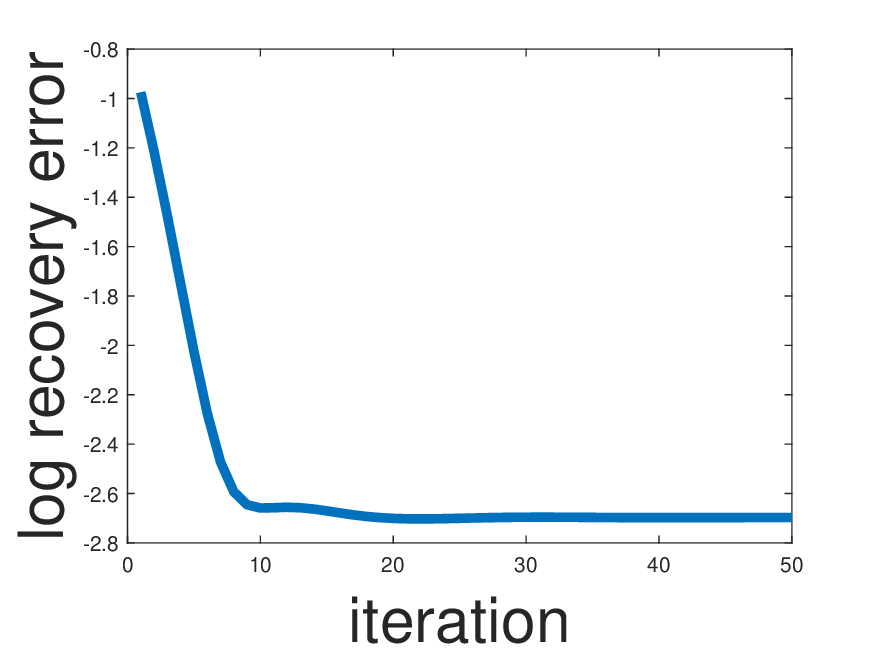} 
  \end{subfigure}
    \caption*{$r=2$, $ \rho=0.3$, $\mM\in\reals^{50\times 50\times 50}$}
\end{subfigure}\hfill
\begin{subfigure}{0.46\textwidth}
    \begin{subfigure}{0.5\textwidth}
    \includegraphics[width=\textwidth]{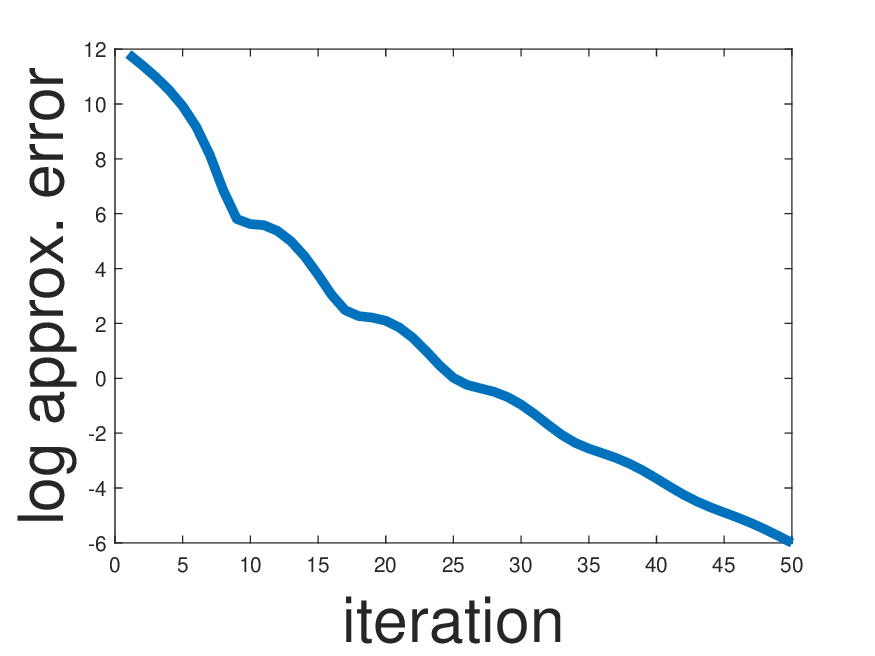}   
  \end{subfigure}\hfil
    \begin{subfigure}{0.5\textwidth}
    \includegraphics[width=\textwidth]{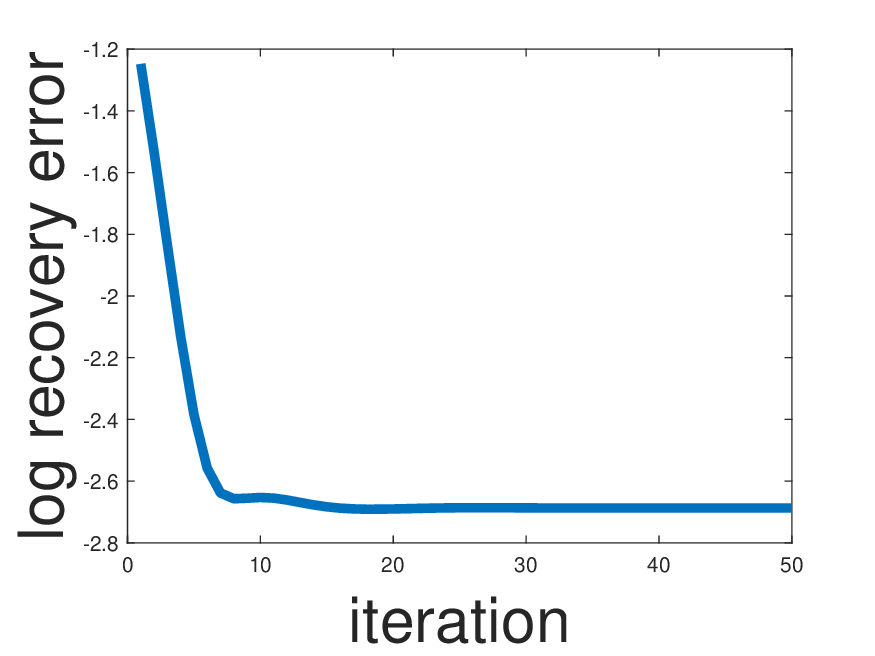} 
  \end{subfigure}
    \caption*{$r=4$, $ \rho=0.4$, $\mM\in\reals^{50\times 50\times 50}$}
     \end{subfigure}
%
    \medskip
    \begin{subfigure}{0.46\textwidth}
  \begin{subfigure}{0.5\textwidth}
    \includegraphics[width=\textwidth]{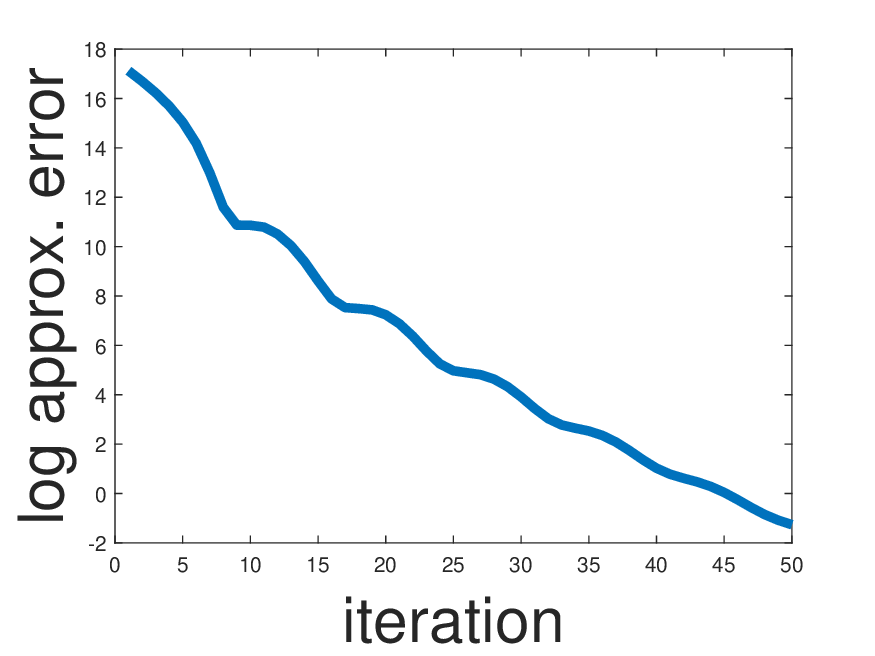}   
  \end{subfigure}\hfil
    \begin{subfigure}{0.5\textwidth}
    \includegraphics[width=\textwidth]{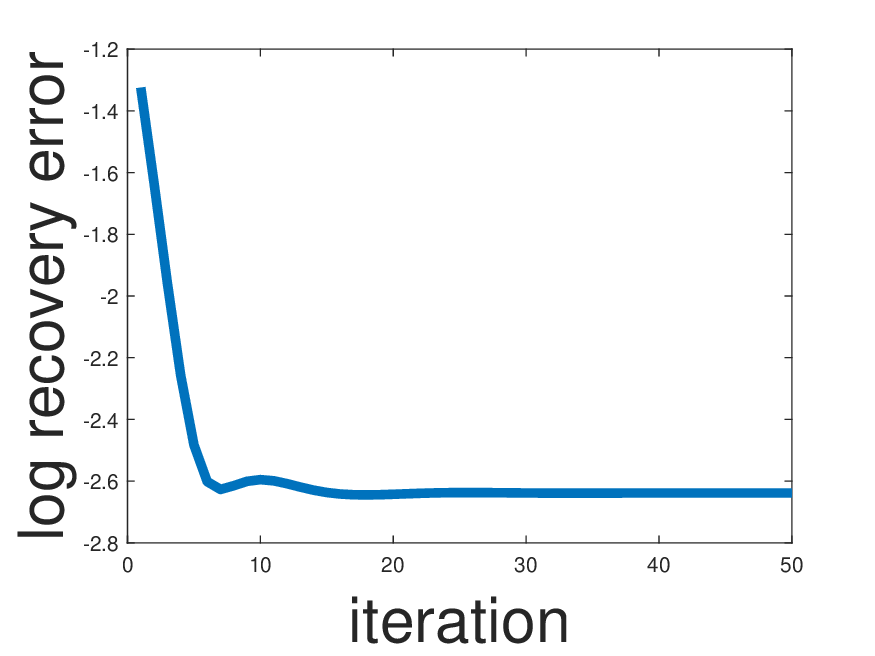} 
  \end{subfigure}
    \caption*{$r=4$, $ \rho=0.4$, $\mM\in\reals^{50\times 50\times 50\times 50}$}
\end{subfigure}\hfill
\begin{subfigure}{0.46\textwidth}
    \begin{subfigure}{0.5\textwidth}
    \includegraphics[width=\textwidth]{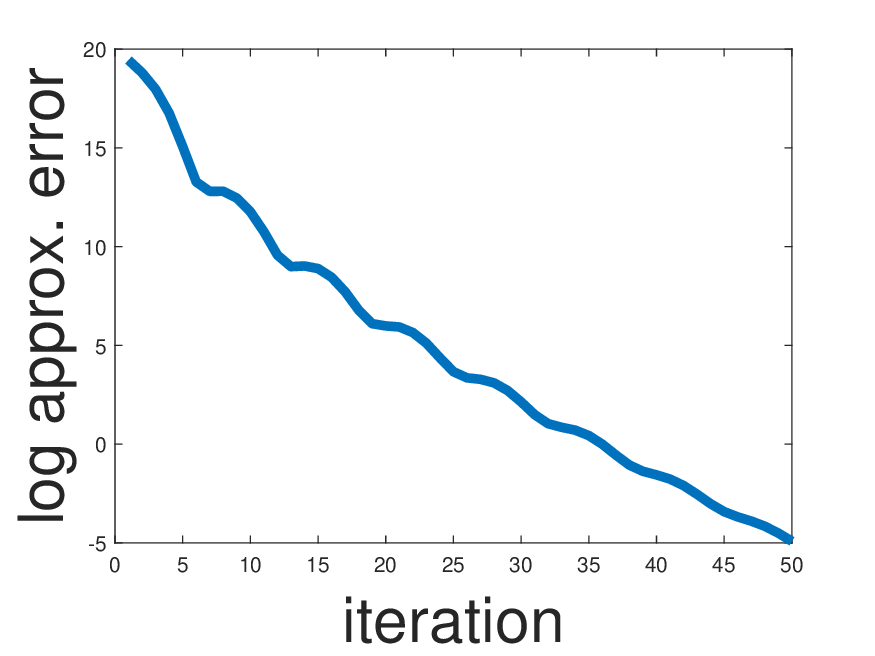}   
  \end{subfigure}\hfil
    \begin{subfigure}{0.5\textwidth}
    \includegraphics[width=\textwidth]{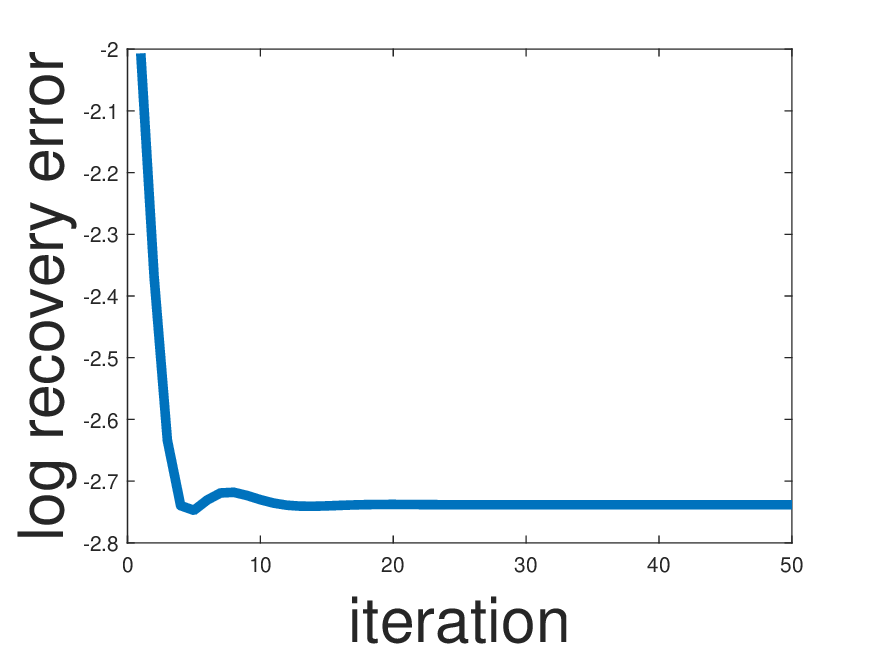} 
  \end{subfigure}
    \caption*{$r=8$, $ \rho=0.6$, $\mM\in\reals^{50\times 50\times 50\times 50}$}
     \end{subfigure}
  \caption{Approximation error w.r.t. function value and recovery error (in log scale) for the tensor completion problem. Each graph is the average of 10 i.i.d. runs.
  } \label{fig:tableTC}
\end{figure}

\subsection{Tensor robust PCA}
We consider the tensor robust PCA problem in the following formulation:
\begin{align*}
\min_{\Vert\mX\Vert_*\le\tau}\left\{f(\mX):=\Vert\mX-\widetilde{\mM}\Vert_1\right\} = \min_{\Vert\mX\Vert_*\le\tau}\max_{\Vert\mY\Vert_{\infty}\le1}\left\{F(\mX,\mY):=\langle\mX-\widetilde{\mM},\mY\rangle\right\}.
\end{align*}
Following the experiments  in \cite{tensor_tSVD} we set $\widetilde{\mM}=\mM+\mN$, where $\mM = \mP*\mQ^{\top}$ and $\mP,\mQ\in\reals^{n\times r\times n}$ are  such that all entries are chosen randomly from $\mathcal{N}(0,1/n)$, and $\mN\in\reals^{n\times n\times n}$ is  such that each entry is  0 with probability $1-m$ and otherwise it is    a Rademacher random variable (i.e., $\pm$1 with probability $1/2$). 

For the initialization we take $\mX_1$ to be the rank-$r$ truncated projection of the tensor $\widetilde{\mM}$ onto the TNN ball of radius $\tau$ (\cref{def:lowRankProjection}), and we set $\mY_1=\sign(\mX_1-\widetilde{\mM})$.

We test the model using the projected extragradient method (\cref{alg:EG}) and use the t-product toolbox \cite{tensorToolbox} for some of the tensor operations. We set the step-size to $\eta=1$, which gave the best empirical performance,  the number of iterations in each experiment to $T=10,000$, and $\tau=0.75 \Vert\mP*\mQ^{\top}\Vert_*$. For each value of $r=\rank_{\textnormal{t}}(\mP*\mQ^{\top})$ and $m$ we average the measurements over $10$ i.i.d. runs.

We choose our candidate for the optimal solution to be the iterate with the lowest dual-gap, which here also is an upper-bound on the approximation error. For saddle-point problems the dual-gap at a point $(\widehat{\mZ},\widehat{\mW})$ is calculated as $\max_{\Vert\mX\Vert_*\le\tau}\langle\widehat{\mZ}-\mX,\nabla_{\mX}F(\widehat{\mZ},\widehat{\mW})\rangle-\min_{{\Vert\mY\Vert_{\infty}\le1}}\langle\widehat{\mW}-\mY,\nabla_{\mY}F(\widehat{\mZ},\widehat{\mW})\rangle$ (see for instance Appendix E in \cite{ourExtragradient}). The maximizer of the first term over the TNN is computed as described in \eqref{eq:dualGap}, and the minimizer of the second term over the $\ell_{\infty}$ ball is the tensor such that $\mY(i_1,\ldots,i_d)=\sign(\nabla_{\mY}F(\widehat{\mZ},\widehat{\mW})(i_1,\ldots,i_d))$ for every $i_1\in[n_1],\ldots,i_d\in[n_d]$.

As can be seen in \cref{table:robustPCArank10}, the model returns a solution with significantly lower recovery error than that of the initialization. The dual gap is fairly large in comparison to the previous tensor completion task which is because the signal-to-noise ratio for the robust PCA task is significantly smaller due the magnitude of the noise and so the dual gap converges much slower. Nevertheless, Figure \ref{table:robustPCArank10} provides evidence for the convergence of the method. It can be seen in \cref{table:robustPCArank10} that the measure of strict complementarity for this task,  which was measured using 
\begin{align} \label{eq:strictCompExpNonsmooth}
\min_{i_3\in[n_3],\ldots,i_d\in[n_d]}\sigma_{1}(\overbar{\nabla_{\mX}F(\mX^*,\mY^*)}^{(i_3,\ldots,i_d)})-\sigma_{r_{i_3,\ldots,i_d}+1}(\overbar{\nabla_{\mX}F(\mX^*,\mY^*)}^{(i_3,\ldots,i_d)}),
\end{align}
where $r_{i_3,\ldots,i_d}=\rank(\overbar{\mX^*}^{(i_3,\ldots,i_d)})$, is positive for all instances and significantly larger than in the tensor completion task. Finally, we observed that in all instances the condition  \eqref{ineq:conditionOfLowTubalRank} held starting from the very first iteration w.r.t. to rank parameter $r$ for both types of primal projected gradient mappings applied in \cref{alg:EG} (i.e., the projections of
$\mX_t-\eta\nabla{}_{\mX}F(\mX_t,\mY_t)$ and $\mX_t-\eta\nabla{}_{\mX}F(\mZ_{t+1},\mW_{t+1})$), which implies 
that throughout the run w.r.t. all instances, all projections onto the TNN ball have tubal rank at most $r$.


In \cref{fig:table1} we plot the function value $f$ and the recovery error (in log scale) w.r.t. to the ergodic series $(1/t)\sum_{i=1}^{t}\Z_{t+1}, t\geq 1$. 

\begin{table*}[!htb]\renewcommand{\arraystretch}{1.3}
{\footnotesize
\begin{center}
  \begin{tabular}{| l | c | c | c | c | c | c | c |  c | } \hline 
   dimension (n) & $100$ & $200$ & $100$ & $200$  \\ \hline 
   & \multicolumn{2}{c|}{$r=0.05n$,} & \multicolumn{2}{c|}{$r=0.05n$,}  \\
   & \multicolumn{2}{c|}{$m=0.05n^3$} & \multicolumn{2}{c|}{$m=0.1n^3$}  \\ \hline
 
  initialization error & $3.6926$ & $5.0538$ & $4.2742$ & $5.9716$ \\ \hline
    recovery error & $0.0092$ & $0.0081$ & $0.0105$ & $0.0096$ \\ \hline
    dual gap & $6.0068$ & $64.1793$ & $5.6119$ & $60.6132$ \\ \hline
  strict complementarity (Eq. \eqref{eq:strictCompExpNonsmooth}) & $208.9066$ & $417.8775
$ & $179.4134$ & $357.6175$ \\ \hline 
    first iteration from which all primal    & $1$ & $1$ & $1$ & $1$ \\
    projections are of tubal rank $\leq r$ & & &  & \\ \hline


& \multicolumn{2}{c|}{$r=0.1n$,} & \multicolumn{2}{c|}{$r=0.1n$,} \\
& \multicolumn{2}{c|}{$m=0.1n^3$} & \multicolumn{2}{c|}{$m=0.2n^3$} \\ \hline

  initialization error & $5.3045$ & $7.3978$ & $6.2372$ & $8.9005$ \\ \hline
    recovery error & $0.0197$ & $0.0195$ & $0.0300$ & $0.0298$ \\ \hline
    dual gap & $4.8666$ & $53.8225$ & $4.3199$ & $47.9408$  \\ \hline
  strict complementarity (Eq. \eqref{eq:strictCompExpNonsmooth}) & $81.6385$ & $169.9079$ & $34.1120$ & $71.8532
$ \\ \hline 
    first iteration from which all primal    & $1$ & $1$ & $1$ & $1$ \\
    projections are of tubal rank $\leq r$ & & &  & \\ \hline

   \end{tabular}
  \caption{Numerical results for the tensor robust PCA problem. Each result is the average of 10 i.i.d. runs.
  }\label{table:robustPCArank10}
\end{center}
}
\vskip -0.2in
\end{table*}\renewcommand{\arraystretch}{1}

\begin{figure}[H] 
\centering
\begin{subfigure}{0.48\textwidth}
  \begin{subfigure}{0.5\textwidth}
    \includegraphics[width=\textwidth]{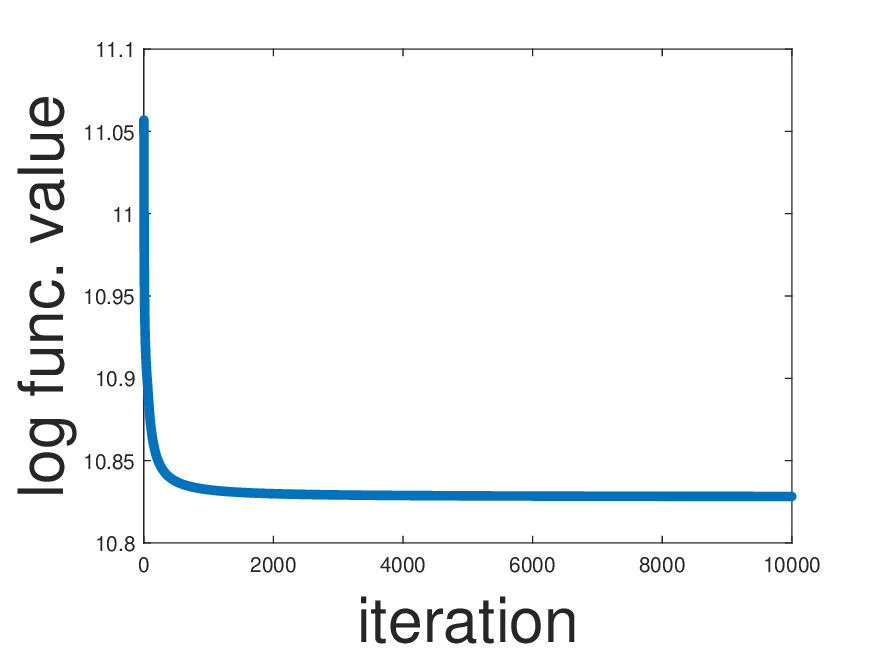}   
  \end{subfigure}\hfil
      \begin{subfigure}{0.5\textwidth}
    \includegraphics[width=\textwidth]{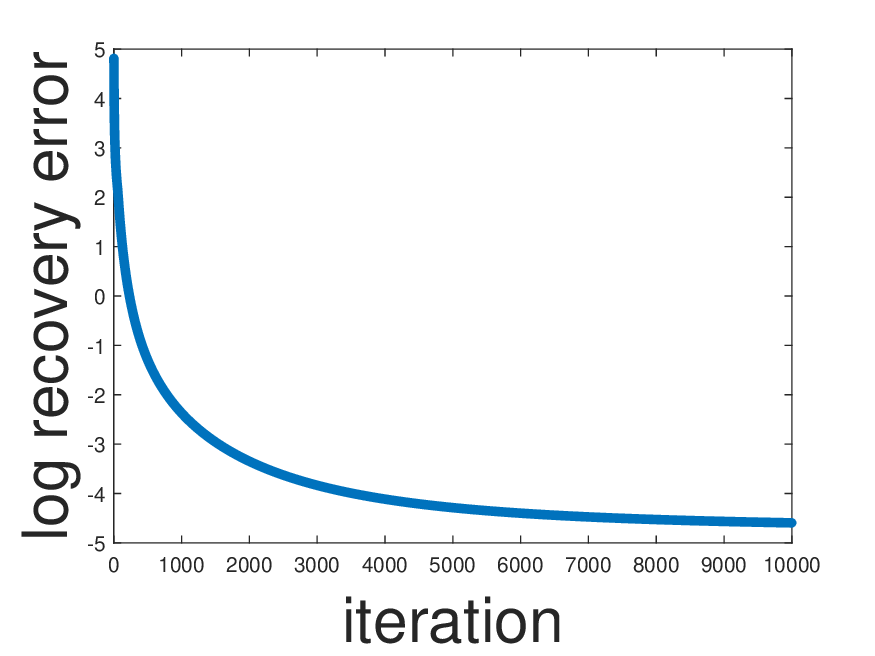} 
  \end{subfigure}\hfil
    \caption*{$n=100$}
\end{subfigure}\hfill
\begin{subfigure}{0.48\textwidth}
  \begin{subfigure}{0.5\textwidth}
    \includegraphics[width=\textwidth]{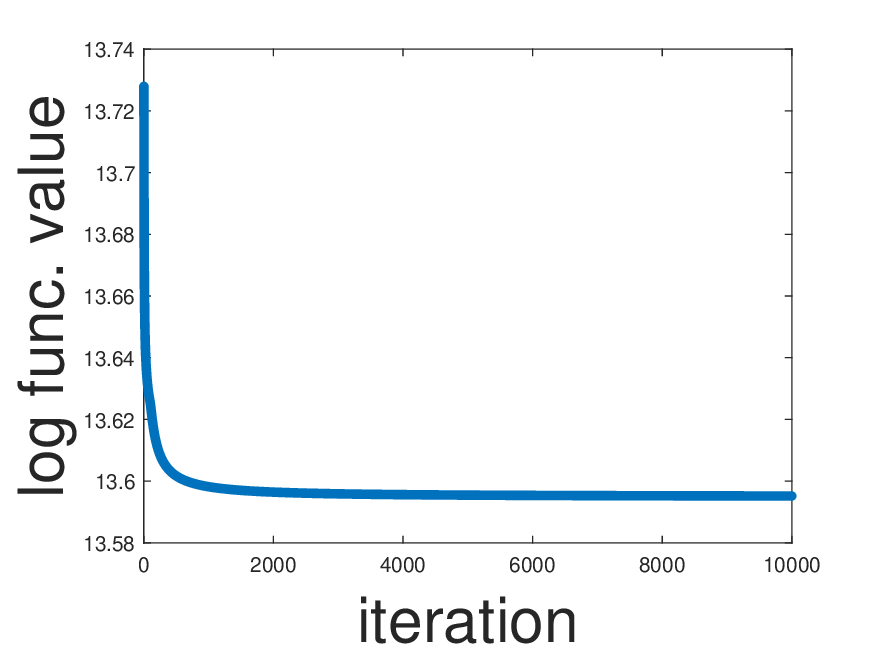}   
  \end{subfigure}\hfil
    \begin{subfigure}{0.5\textwidth}
    \includegraphics[width=\textwidth]{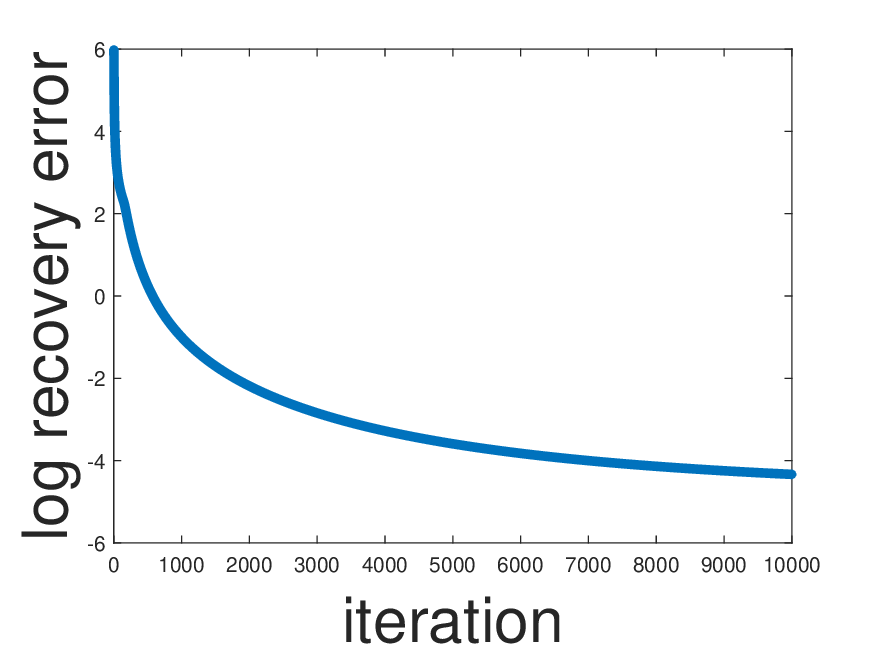} 
  \end{subfigure}\hfill
    \caption*{$n=200$}
     \end{subfigure}
         \caption*{$r=0.05n$, $m=0.05n^3$}
    \medskip
    \begin{subfigure}{0.48\textwidth}
  \begin{subfigure}{0.5\textwidth}
    \includegraphics[width=\textwidth]{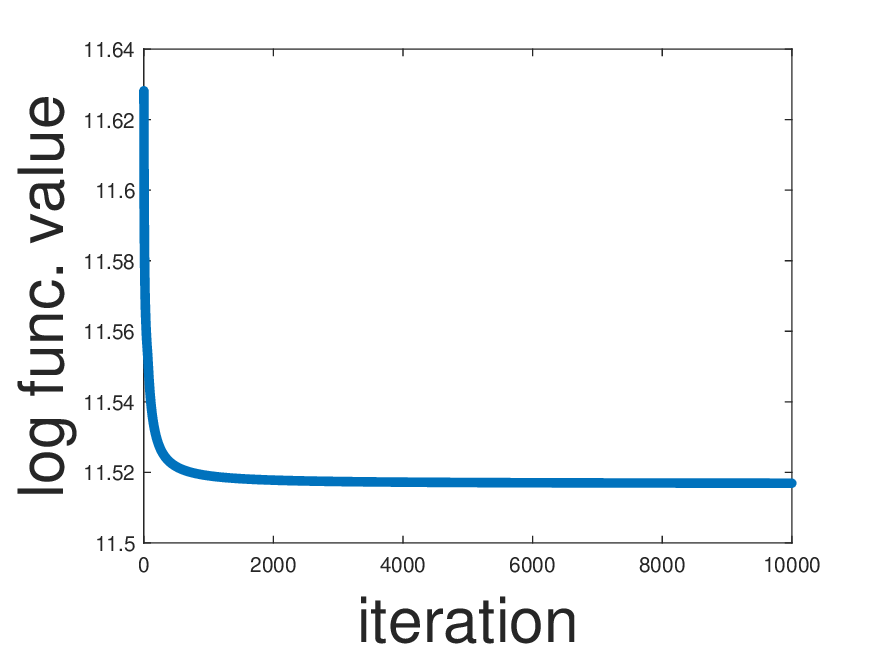} 
  \end{subfigure}\hfill
    \begin{subfigure}{0.5\textwidth}
    \includegraphics[width=\textwidth]{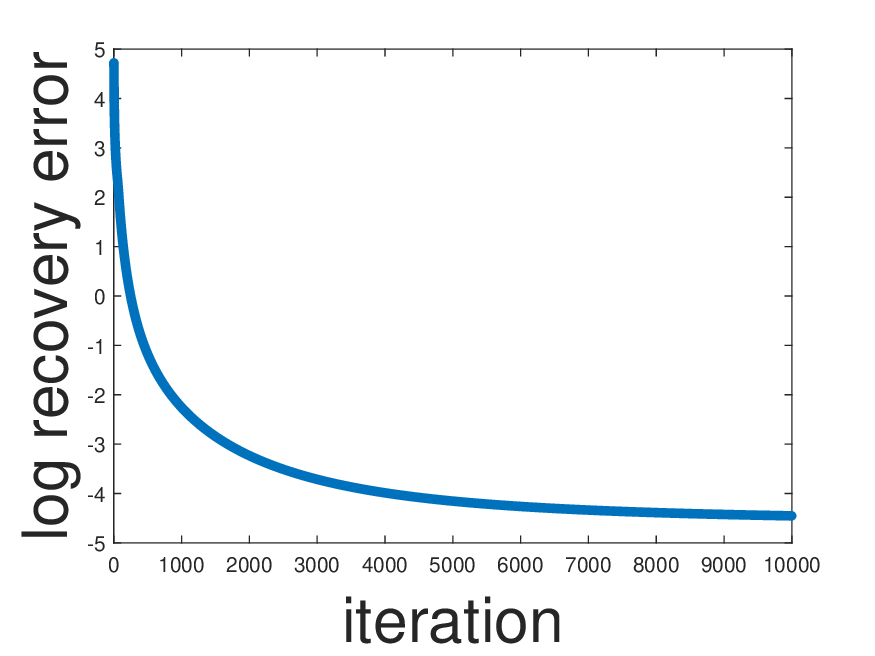} 
  \end{subfigure}\hfill
    \caption*{$n=100$}
    \end{subfigure}\hfill
    \begin{subfigure}{0.48\textwidth}
    \begin{subfigure}{0.5\textwidth}
    \includegraphics[width=\textwidth]{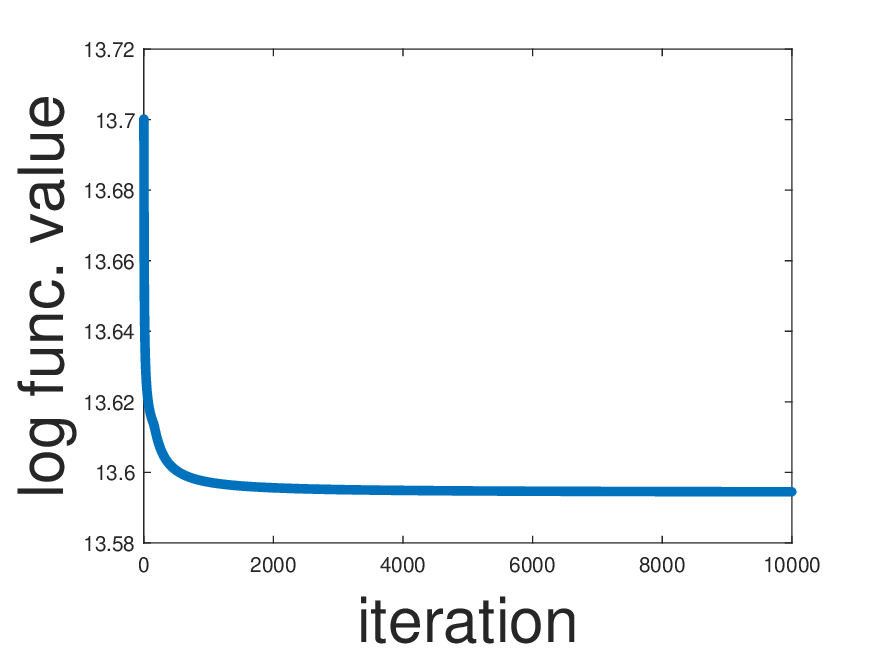} 
  \end{subfigure}\hfill
      \begin{subfigure}{0.5\textwidth}
    \includegraphics[width=\textwidth]{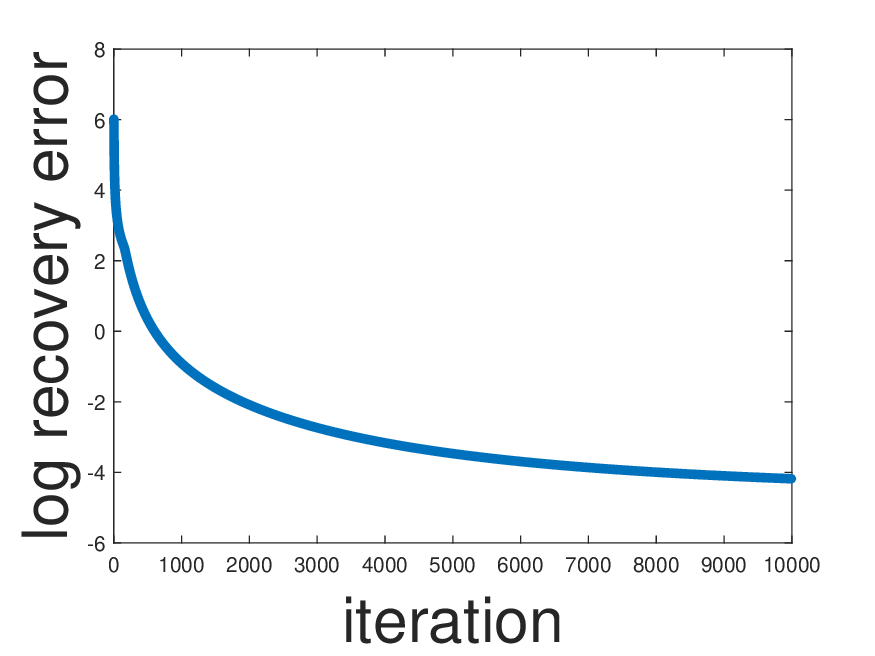} 
  \end{subfigure}\hfill
  \caption*{$n=200$}
  \end{subfigure}
      \caption*{$r=0.05n$, $m=0.1n^3$}
  
  \medskip
  \begin{subfigure}{0.48\textwidth}
   \begin{subfigure}{0.5\textwidth}
    \includegraphics[width=\textwidth]{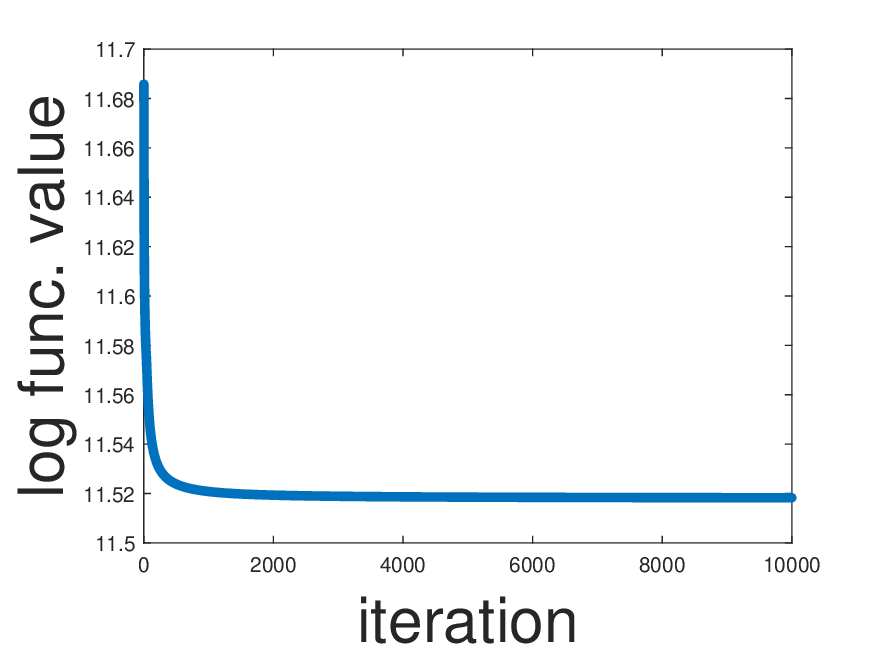} 
  \end{subfigure}\hfill
      \begin{subfigure}{0.5\textwidth}
    \includegraphics[width=\textwidth]{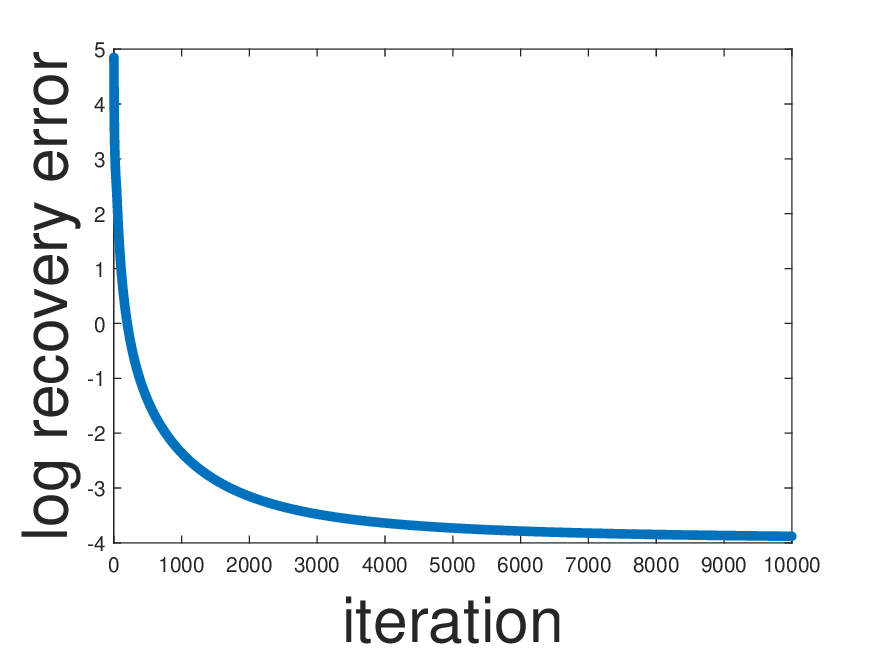} 
  \end{subfigure}\hfill
    \caption*{$n=100$}
\end{subfigure}\hfill
\begin{subfigure}{0.48\textwidth}
   \begin{subfigure}{0.5\textwidth}
    \includegraphics[width=\textwidth]{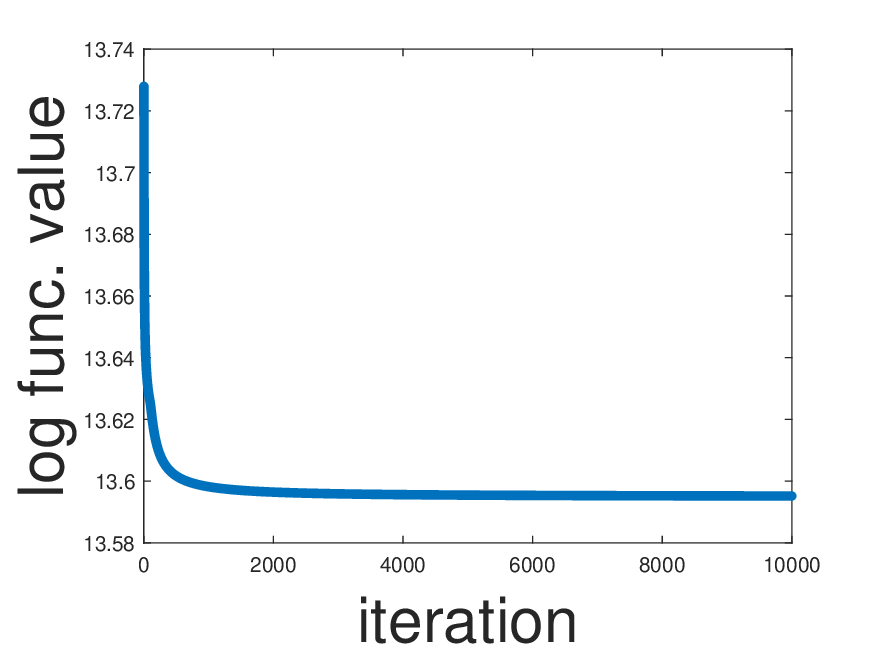} 
  \end{subfigure}\hfill
    \begin{subfigure}{0.5\textwidth}
    \includegraphics[width=\textwidth]{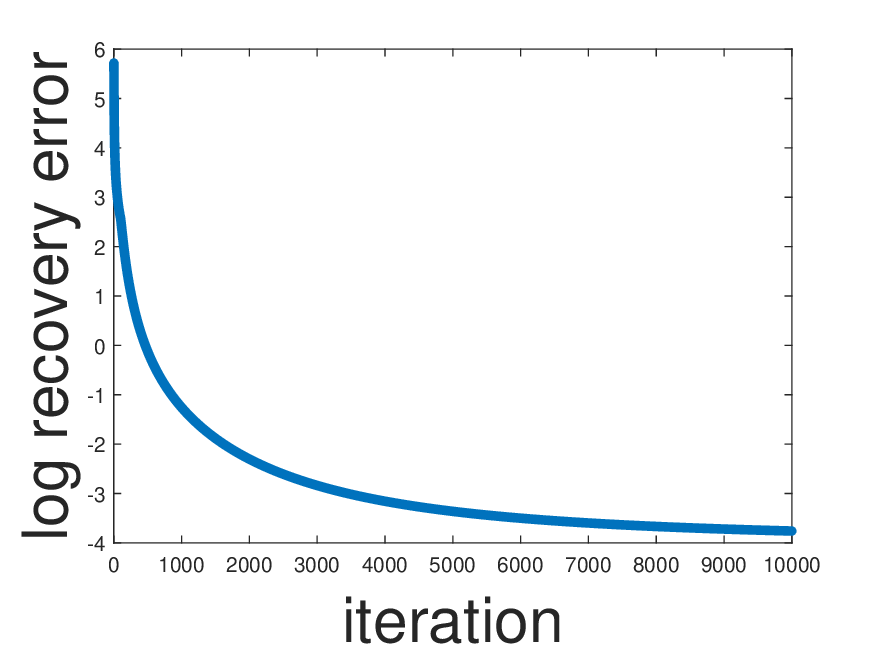} 
  \end{subfigure}\hfill
    \caption*{$n=200$}
    \end{subfigure}
        \caption*{$r=0.1n$, $m=0.1n^3$}    
      \medskip
      \begin{subfigure}{0.48\textwidth}
      \begin{subfigure}{0.5\textwidth}
    \includegraphics[width=\textwidth]{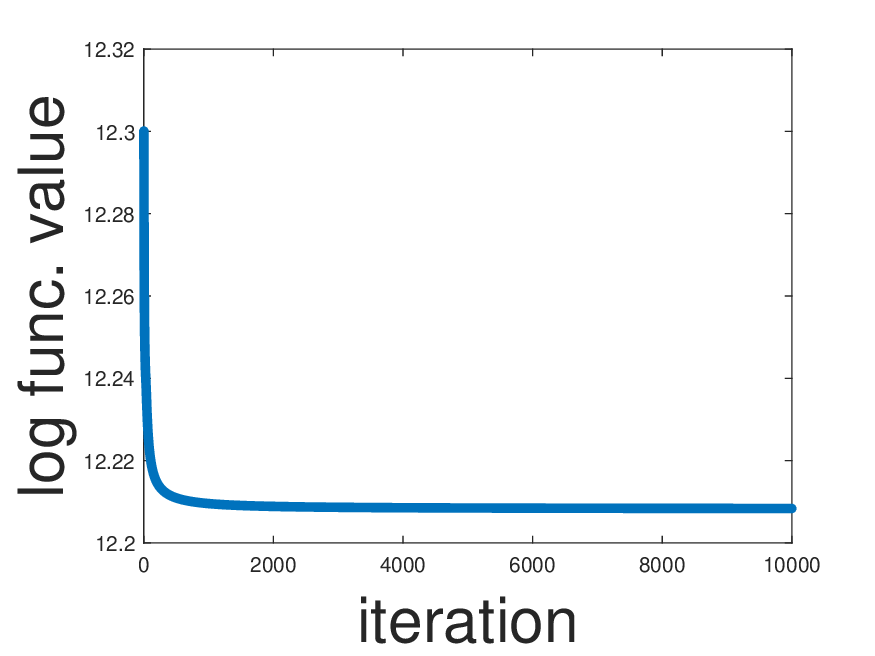} 
  \end{subfigure}\hfill
    \begin{subfigure}{0.5\textwidth}
    \includegraphics[width=\textwidth]{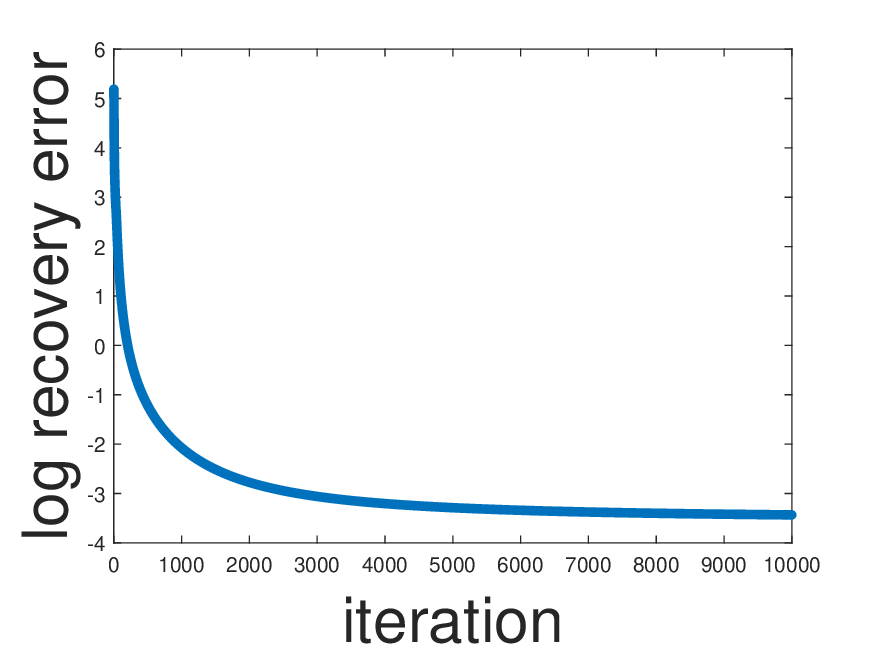} 
  \end{subfigure}\hfill
    \caption*{$n=100$}
\end{subfigure}\hfill
\begin{subfigure}{0.48\textwidth}
      \begin{subfigure}{0.5\textwidth}
    \includegraphics[width=\textwidth]{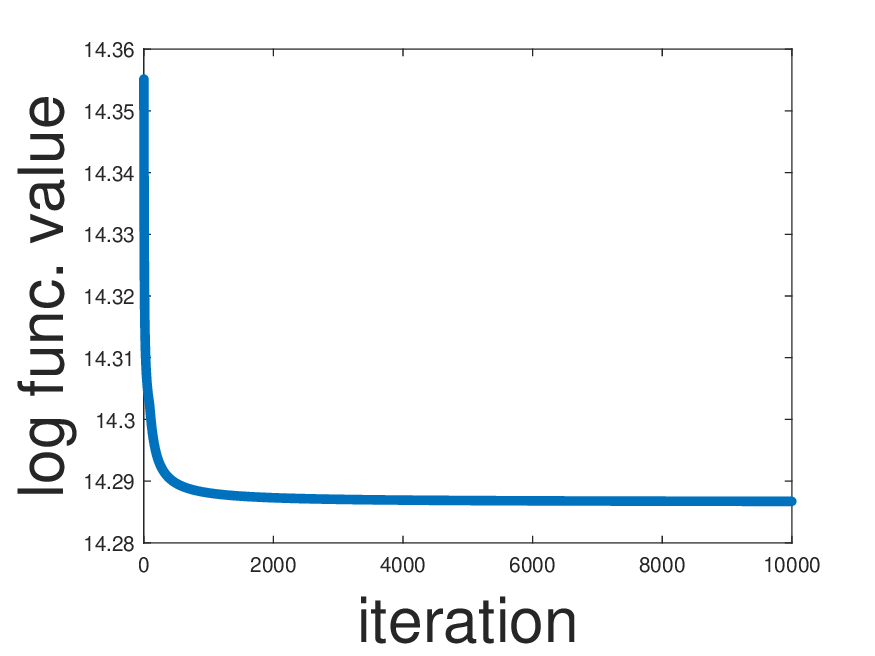} 
  \end{subfigure}\hfill
    \begin{subfigure}{0.5\textwidth}
    \includegraphics[width=\textwidth]{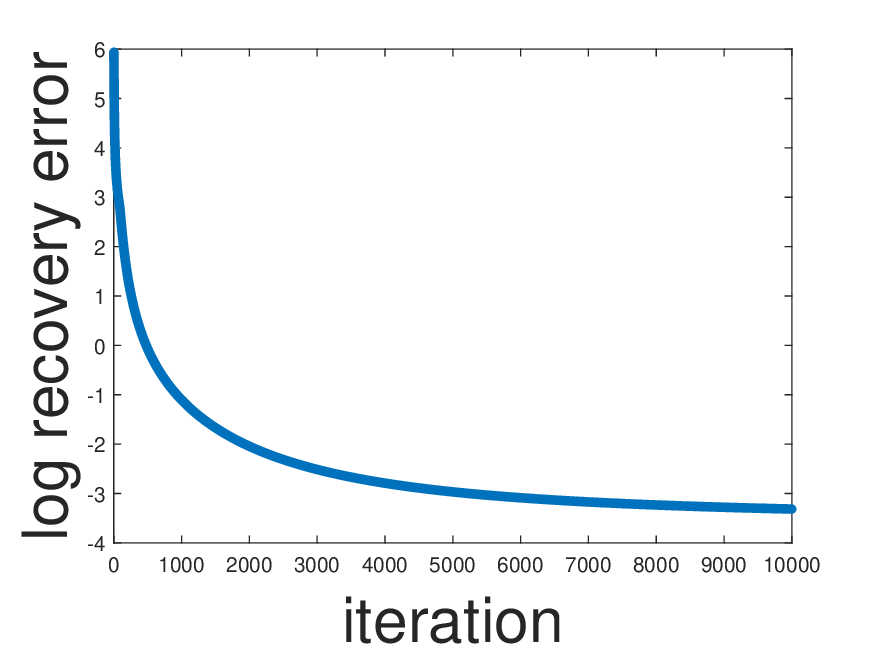} 
  \end{subfigure}\hfill
    \caption*{$n=200$}
    \end{subfigure}
    \caption*{$r=0.1n$, $m=0.2n^3$}
    
  \caption{Function value (w.r.t. $f$) and recovery error (in log scale) for the tensor robust PCA problem.  The plots are with respect to the ergodic series $(1/t)\sum_{i=1}^{t}\Z_{t+1}, t\geq 1$. Each graph is the average of 10 i.i.d. runs.
  } \label{fig:table1}
\end{figure}

\appendix
\appendixpage

\section{Proofs omitted from Section \ref{sec:tensorPreliminaries}}
\label{sec:App:tensorPreliminariesProofs}

For this section we will denote the matrices
\begin{align} \label{def:F_F1}
& \F :=\F_{n_d}\otimes\F_{n_{d-1}}\otimes\cdots\otimes\F_{n_3} \in \complex^{N\times N} \nonumber
\\ & \F^{-1} :=\F^{-1}_{n_d}\otimes\F^{-1}_{n_{d-1}}\otimes\cdots\otimes\F^{-1}_{n_3} \in \complex^{N\times N}. 
\end{align}
These notations will be used in many of the proofs.

For each lemma we will first restate the lemma and then prove it.

\subsection{Proof of \cref{lemma:bdiag_bcirc_equality}}
\label{sec:App:proofLemma1}

\begin{lemma}
Let $\mX\in\complex^{n_1\times \cdots\times n_d}$. 
Then, $\barX=\bdiag(\overbar{\mX})$ if and only if
\begin{align*}
\bcirc(\mX) = (\F^{-1}\otimes\I_{n_1}) \cdot \barX \cdot (\F\otimes\I_{n_2}).
\end{align*}
\end{lemma}

\begin{proof}
It it well known (see for instance Theorem 3.2.1 in \cite{davis1979circulant}) that for any $n\ge1$ the DFT matrix $\F_n$ diagonalizes the matrix periodic downward shift permutation matrix $\pi_n$ as defined in \eqref{def:shiftPermutationMatrix} through
\begin{align} \label{eq:piDiagonalization}
\pi_n=\F_{n}^{-1}\D_{n}\F_{n},
\end{align}
where $\D_{n}=\diag(1,\omega,\ldots,\omega^{n-1})\in\reals^{n\times n}$ and $\omega=\exp\left(-\frac{2\pi i}{n}\right)$.

Plugging \eqref{eq:piDiagonalization} into all shift permutation matrices in the definition of $\bcirc(\mX)$ in \eqref{def:bcirc}, we obtain by using the notation of $\F,\F^{-1}$ in \eqref{def:F_F1} that 
\begin{align}  \label{eq:inproof_bcirc_bdiag}
& \bcirc(\mX) \nonumber
\\ & = \sum_{i_d=1}^{n_d}\cdots\sum_{i_3=1}^{n_3} \left(\F_{n_d}^{-1}\D_{n_d}^{i_d-1}\F_{n_d}\right) \otimes \cdots \otimes \left(\F_{n_3}^{-1}\D_{n_3}^{i_3-1}\F_{n_3}\right) \otimes \mX{(:,:,i_3,\ldots,i_d)} \nonumber
\\ & = \left(\F^{-1}\otimes\I_{n_1}\right)\left(\sum_{i_d=1}^{n_d}\cdots\sum_{i_3=1}^{n_3}\D_{n_d}^{i_d-1} \otimes \cdots \otimes \D_{n_3}^{i_3-1} \otimes \mX{(:,:,i_3,\ldots,i_d)} \right)\left(\F\otimes\I_{n_2}\right),
\end{align}
where the second equality follows from many times applying the property of the Kronecker product that for any matrices $\A,\B,\C,\D$ it holds that $(\A\otimes\B)(\C\otimes\D)=(\A\C)\otimes(\B\D)$ if the matrices are of sizes such that the products $\A\C$ and $\B\D$ exist.

It remains to show that the middle term in the RHS of \eqref{eq:inproof_bcirc_bdiag} is equivalent to $\barX=\bdiag(\overbar{\mX})$. Indeed,

\begin{align*}
& \sum_{i_d=1}^{n_d}\cdots\sum_{i_3=1}^{n_3}\D_{n_d}^{i_d-1} \otimes \cdots \otimes \D_{n_3}^{i_3-1} \otimes \mX{(:,:,i_3,\ldots,i_d)}
\\ & =  \sum_{i_d=1}^{n_d}\cdots\sum_{i_4=1}^{n_4}\D_{n_d}^{i_d-1} \otimes \cdots \otimes \D_{n_4}^{i_4-1}
\otimes \sum_{i_3=1}^{n_3}\matbdiag\left(\begin{array}{l}\mX{(:,:,i_3,\ldots,i_d)},\\ \omega^{(i_3-1)}\mX{(:,:,i_3,\ldots,i_d)}, \\ \vdots\\ \omega^{(i_3-1)(n_3-1)}\mX{(:,:,i_3,\ldots,i_d)}\end{array}\right)
\\ & = \sum_{i_d=1}^{n_d}\cdots\sum_{i_4=1}^{n_4}\D_{n_d}^{i_d-1} \otimes \cdots \otimes \D_{n_4}^{i_4-1} 
\otimes \matbdiag\left(\begin{array}{l} \sum_{i_3=1}^{n_3}\mX{(:,:,i_3,\ldots,i_d)},
\\ \sum_{i_3=1}^{n_3}\omega^{(i_3-1)}\mX{(:,:,i_3,\ldots,i_d)},
\\  \vdots
\\ \sum_{i_3=1}^{n_3}\omega^{(i_3-1)(n_3-1)}\mX{(:,:,i_3,\ldots,i_d)}\end{array}\right)
\\ & = \cdots =
\\ & = \matbdiag\left(\left(\sum_{i_d=1}^{n_d}\cdots\sum_{i_3=1}^{n_3}\omega^{(i_d-1)(k_d-1)}\cdots\omega^{(i_3-1)(k_3-1)}\mX{(:,:,i_3,\ldots,i_d)}\right)_{k_d\in[n_d]
, \ldots , k_3\in[n_3]}\right)
\\ & = \matbdiag\left(\left(\overbar{\mX}(:,:,k_3,\ldots,k_d)\right)_{k_d\in[n_d]
, \ldots , k_3\in[n_3]}\right)=\barX,
\end{align*}
as desired. The second to last equality follows from the definition of the order-d Fourier transform.



\end{proof}

\subsection{Proof of \cref{lemma:conjegateComplexSymmetry_realTensors}}
\label{sec:App:proofLemma2}

\begin{lemma}
A tensor $\mX\in\reals^{n_1\times\cdots\times n_d}$ is real-valued if and only if $\overbar{\mX}=\fft_d(\cdots(\fft_4(\fft_3(\mX))))$ satisfies the conjugate-complex symmetry condition
\begin{align} \label{eq:conjugateComplexSymmetryCondition1}
\overbar{\mX}(:,:,i_3,\ldots,i_d) = \conj(\overbar{\mX}(:,:,i'_3,\ldots,i'_d))
\end{align}
for all $i_3\in\left\lbrace1,\ldots,\left\lceil\frac{n_3+1}{2}\right\rceil\right\rbrace,\ldots,i_d\in\left\lbrace1,\ldots,\left\lceil\frac{n_d+1}{2}\right\rceil\right\rbrace$, where for all $j\in\lbrace3,\ldots,d\rbrace$ 
\begin{align*} 
i'_j=\bigg\lbrace\begin{array}{ll}
1, & i_j=1
\\ n_j-i_j+2, & i_j\in\left\lbrace2,\ldots,\left\lceil\frac{n_j+1}{2}\right\rceil\right\rbrace
\end{array}.
\end{align*}
\end{lemma}

\begin{proof}

For every $j\in\lbrace3,\ldots,d\rbrace$ define $\omega_j=\exp\left(-\frac{2\pi i}{n_j}\right)$ and $\bar{\omega}_j=\exp\left(\frac{2\pi i}{n_j}\right)=\conj(\omega_j)$. For every $i_j\in\left\lbrace2,\ldots,\left\lceil\frac{n_j+1}{2}\right\rceil\right\rbrace$ and $k_j\in[n_j]$ it holds that
\begin{align} \label{eq:caseIndexGreaterThan1}
\omega_j^{(k_j-1)(i_j-1)} & =\exp\left(\frac{-2\pi i(k_j-1)(i_j-1)}{n_j}\right) \nonumber
\\ & = \exp\left(\frac{-2\pi i(k_j-1)(i_j-1)}{n_j}\right)\exp\left(2\pi i(k_j-1)\right) \nonumber
\\ & = \exp\left(\frac{2\pi i(k_j-1)(n_j-i_j+1)}{n_j}\right) = \bar{\omega}_j^{(k_j-1)(n_j-i_j+1)},
\end{align}
where the second equality holds since $\exp\left(2\pi i(k_j-1)\right)=1$.

In addition, for all $j\in\lbrace3,\ldots,d\rbrace$ and $i_j=1$ it holds that
\begin{align} \label{eq:caseIndex1}
\omega_j^{(k_j-1)(i_j-1)}=\exp\left(\frac{-2\pi i(k_j-1)(i_j-1)}{n_j}\right) & = 1.
\end{align}

By the definition of the Fourier transform along all but the first two dimensions (\cref{{def:fft}}), for every $i_1\in[n_1],\ldots,i_d\in[n_d]$ it holds that
\begin{align*}
\overbar{\mX}(i_1,i_2,i_3,\ldots,i_d) 
& = \sum_{k_d=1}^{n_d}\cdots\sum_{k_3=1}^{n_3}\omega_d^{(k_d-1)(i_d-1)}\cdots \omega_3^{(k_3-1)(i_3-1)}\mX(i_1,i_2,k_3,\ldots,k_d)
\\ & = \conj(\overbar{\mX}(i_1,i_2,i'_3,\ldots,i'_d)),
\end{align*}
where the last equality follows from \eqref{eq:caseIndexGreaterThan1} and \eqref{eq:caseIndex1} if $\mX$ is real-valued.

For the second direction, assume the conjugate-complex symmetry conditions in \eqref{eq:conjugateComplexSymmetryCondition1} hold. Then, by the inverse Fourier transform along all but the first two dimensions (\cref{{def:fft}}), for every $i_1\in[n_1],\ldots,i_d\in[n_d]$ it holds that
\begin{align} \label{eq:inverseFFTatPoint}
\mX(i_1,i_2,i_3,\ldots,i_d) 
& = \frac{1}{N}\sum_{k_d=1}^{n_d}\cdots\sum_{k_3=1}^{n_3}\bar{\omega}_d^{(k_d-1)(i_d-1)}\cdots \bar{\omega}_3^{(k_3-1)(i_3-1)}\overbar{\mX}(i_1,i_2,k_3,\ldots,k_d).
\end{align}
For $j\in\lbrace3,\ldots,d\rbrace$, using similar arguments to the ones in \eqref{eq:caseIndexGreaterThan1} and \eqref{eq:caseIndex1}, it can be seen that by denoting 
\begin{align*}
k'_j=\bigg\lbrace\begin{array}{ll}
1, & k_j=1
\\ n_j-k_j+2, & k_j\in\left\lbrace2,\ldots,\left\lceil\frac{n_j+1}{2}\right\rceil\right\rbrace, 
\end{array}
\end{align*}
it holds that $\omega_j^{(k_j-1)(i_j-1)}=\conj\left(\bar{\omega}_j^{(k'_j-1)(i_j-1)}\right)$.
Using this equality and \eqref{eq:conjugateComplexSymmetryCondition1}, it follows that
\begin{align*}
& \bar{\omega}_d^{(k_d-1)(i_d-1)}\cdots \bar{\omega}_3^{(k_3-1)(i_3-1)}\overbar{\mX}(i_1,i_2,k_3,\ldots,k_d)
\\ & = \conj\left(\omega_d^{(k_d-1)(i_d-1)}\right)\cdots \conj\left(\omega_3^{(k_3-1)(i_3-1)}\right)\conj(\overbar{\mX}(i_1,i_2,k'_3,\ldots,k'_d))
\\ & = \conj\left(\bar{\omega}_d^{(k'_d-1)(i_d-1)}\right)\cdots \conj\left(\bar{\omega}_3^{(k'_3-1)(i_3-1)}\right)\conj(\overbar{\mX}(i_1,i_2,k'_3,\ldots,k'_d))
\\ & = \conj\left(\bar{\omega}_d^{(k'_d-1)(i_d-1)}\cdots \bar{\omega}_3^{(k'_3-1)(i_3-1)}\overbar{\mX}(i_1,i_2,k'_3,\ldots,k'_d)\right).
\end{align*}
Therefore, each term in the sum in the RHS of \eqref{eq:inverseFFTatPoint} is either real-valued or has a pair which is its conjugate, and so together their sum is real-valued. All together, the sum in the RHS of \eqref{eq:inverseFFTatPoint} is real-valued.

\end{proof}

\subsection{Proof of \cref{lemma:orderPinnerProductEquivalence}}
\label{sec:App:proofLemma3}.
\begin{lemma}
Let $\mX,\mY\in\complex^{n_1\times\cdots\times n_d}$. Then,
\begin{align*}
(i) \  \langle\mX,\mY\rangle  = \frac{1}{n_3\cdots n_d}\langle\barX,\barY\rangle,
\qquad\qquad (ii) \  \Vert\mX\Vert_F = \frac{1}{\sqrt{n_3\cdots n_d}}\Vert\barX\Vert_F.
\end{align*}
\end{lemma}

\begin{proof}
Since $\F\F^{\tc}=\F^{\tc}\F= N\I$, we have that
\begin{align} \label{eq:inproofInnerproduct1}
(\F\otimes\I_{n_2})^{\tc} = \F^{\tc}\otimes\I_{n_2} =  N\F^{-1}\otimes\I_{n_2} =  N(\F^{-1}\otimes\I_{n_2}) =  N(\F\otimes\I_{n_2})^{-1}
\end{align}
and
\begin{align} \label{eq:inproofInnerproduct2}
(\F^{-1}\otimes\I_{n_1})^{\tc} = ({\F^{-1}})^{\tc}\otimes\I_{n_1} = \frac{1}{ N}\F\otimes\I_{n_1} = \frac{1}{ N}(\F\otimes\I_{n_1}) = \frac{1}{ N}(\F^{-1}\otimes\I_{n_1})^{-1}.
\end{align}
Therefore,
\begin{align*}
\langle\mX,\mY\rangle & = \frac{1}{ N}\langle\bcirc(\mX),\bcirc(\mY)\rangle \underset{(a)}{=} \frac{1}{ N}\langle(\F^{-1}\otimes\I_{n_1})\barX(\F\otimes\I_{n_2}),(\F^{-1}\otimes\I_{n_1})\barY(\F\otimes\I_{n_2})\rangle 
\\ & = \frac{1}{ N}\trace\left((\F\otimes\I_{n_2})^{\tc}\barX^{\tc}(\F^{-1}\otimes\I_{n_1})^{\tc}(\F^{-1}\otimes\I_{n_1})\barY(\F\otimes\I_{n_2})\right)
\underset{(b)}{=} \frac{1}{ N}\langle\barX,\barY\rangle,
\end{align*}
where (a) follows from \cref{lemma:bdiag_bcirc_equality} and (b) follows from \eqref{eq:inproofInnerproduct1} and \eqref{eq:inproofInnerproduct2}. This proves $(i)$.

$(ii)$ follows immediately from $(i)$, since $\Vert\mX\Vert_F=\sqrt{\langle\mX,\mX\rangle} = \frac{1}{\sqrt{ N}}\sqrt{\langle\barX,\barX\rangle} = \frac{1}{\sqrt{ N}}\Vert\barX\Vert_F$.


\end{proof}


\subsection{Proof of \cref{lemma:tproduct_bdiagMult}}
\label{sec:App:proofLemma5}

Before proving \cref{lemma:tproduct_bdiagMult} we first prove the following technical lemma which is a property of the t-product.
\begin{lemma} \label{lemma:bcirc_Tproduct}
Let $\mX\in\complex^{n_1\times n_2\times \cdots\times n_d}$ and $\mY\in\complex^{n_2\times \ell\times n_3\times \cdots\times n_d}$. Then,
\begin{align*}
\bcirc(\mX*\mY)=\bcirc(\mX)\bcirc(\mY).
\end{align*}.
\end{lemma}

\begin{proof}
Invoking \cref{lemma:bdiag_bcirc_equality} we can write
\begin{align*}
\bcirc(\mX)  = (\F^{-1}\otimes\I_{n_1})\barX(\F\otimes\I_{n_2}), \quad \bcirc(\mY) & = (\F^{-1}\otimes\I_{n_2})\barY(\F\otimes\I_{\ell}).
\end{align*}
Therefore,
\begin{align} \label{eq:inProofTproduct_viaFFT}
\bcirc(\mX)\bcirc(\mY) & = (\F^{-1}\otimes\I_{n_1})\barX(\F\otimes\I_{n_2})(\F^{-1}\otimes\I_{n_2})\barY(\F\otimes\I_{\ell}) \nonumber
\\ & =(\F^{-1}\otimes\I_{n_1})\barX\barY(\F\otimes\I_{\ell}).
\end{align}
Since $\barX$ and $\barY$ are block diagonal matrices,  $\barX\barY$ is also block diagonal, and thus there exists some tensor $\mB$ such that $\bdiag(\overbar{\mB})=\barX\barY$. Therefore, by \cref{lemma:bdiag_bcirc_equality}
\begin{align*}
\bcirc(\mX)\bcirc(\mY) & =(\F^{-1}\otimes\I_{n_1})\barX\barY(\F\otimes\I_{\ell})= \bcirc(\mB),
\end{align*}
and hence $\bcirc(\mX)\bcirc(\mY)$ is also a block circulant matrix.
Therefore,
\begin{align*}
\bcirc(\mX*\mY) & = \bcirc(\fold(\bcirc(\mX)\unfold(\mY))) \nonumber
\\ & = \bcirc(\fold(\bcirc(\mX)\bcirc(\mY)\C_1)) \nonumber
\\ & = \bcirc(\fold(\bcirc(\mB)\C_1)) \nonumber
\\ & = \bcirc(\fold(\unfold(\mB))) \nonumber 
\\ & = \bcirc(\mB)= \bcirc(\mX)\bcirc(\mY).
\end{align*}
\end{proof}

We now restate \cref{lemma:tproduct_bdiagMult} and then prove it.

\begin{lemma} 
Let $\mX\in\complex^{n_1\times n_2\times \cdots\times n_d}$ and $\mY\in\complex^{n_2\times \ell\times n_3\times \cdots\times n_d}$, and denote $\barX=\bdiag(\overbar{\mX})$ and $\barY=\bdiag(\overbar{\mY})$. Then, $\mZ=\mX*\mY$ if and only if $\barZ=\barX\barY$, where $\barZ=\bdiag(\overbar{\mZ})$.
\end{lemma}

\begin{proof}

Invoking \cref{lemma:bdiag_bcirc_equality} we can write
\begin{align*}
\barX  = (\F\otimes\I_{n_1})\bcirc(\mX)(\F^{-1}\otimes\I_{n_2}), \quad \barY & = (\F\otimes\I_{n_2})\bcirc(\mY)(\F^{-1}\otimes\I_{\ell}).
\end{align*}
Plugging these in, we have that
\begin{align*} 
\barX\barY & = (\F\otimes\I_{n_1})\bcirc(\mX)(\F^{-1}\otimes\I_{n_2})(\F\otimes\I_{n_2})\bcirc(\mY)(\F^{-1}\otimes\I_{\ell}) 
\\ & = (\F\otimes\I_{n_1})\bcirc(\mX)\bcirc(\mY)(\F^{-1}\otimes\I_{\ell}) 
\\ & = (\F\otimes\I_{n_1})\bcirc(\mX*\mY)(\F^{-1}\otimes\I_{\ell}),
\end{align*}
where the last equality follows from \cref{lemma:bcirc_Tproduct}.
Therefore, by \cref{lemma:bdiag_bcirc_equality}, $\mZ=\mX*\mY$ if and only if $\barZ=\barX\barY$, where $\barZ=\bdiag(\overbar{\mZ})$.
\end{proof}

\subsection{Proof of \cref{lemma:tSVD}}
\label{sec:App:proofLemma51}
\begin{lemma}[t-SVD]
Let $\mX\in\reals^{n_1\times \cdots\times n_d}$. Then, it can be factorized as
\begin{align*}
\mX = \mU*\mS*\mV^{\top}
\end{align*}
where $\mU\in\reals^{n_1\times n_1\times n_3\times \cdots\times n_d}$ and $\mV\in\reals^{n_2\times n_2\times n_3\times \cdots\times n_d}$ are orthogonal, and  $\mS\in\reals^{n_1\times \cdots\times n_d}$ is a f-diagonal tensor.
\end{lemma}

\begin{proof}
The proof is by construction. We begin by considering the SVD of each frontal slice $\barX^{(i_3,\ldots,i_d)}$, which we denote by $\barX^{(i_3,\ldots,i_d)}=\barU^{(i_3,\ldots,i_d)}\barS^{(i_3,\ldots,i_d)}{\barV^{(i_3,\ldots,i_d)}}^{\tc}$. Since $\mX$ is real-valued, by \cref{lemma:conjegateComplexSymmetry_realTensors} the frontal slices must satisfy the conjegate-complex symmetry conditions in \eqref{eq:conjugateComplexSymmetryCondition}. Therefore, by permuting over all indexes $i_3\in[n_3],\ldots,i_d\in[n_d]$, for any index $i_3\in[n_3],\ldots,i_d\in[n_d]$ for which the SVD of $\barX^{(i_3',\ldots,i_d')}$ has been already computed, where $i_j'$ is as defined in \eqref{def:i_j_tag}, we merely need to compute $\barX^{(i_3,\ldots,i_d)}=\conj(\barU^{(i_3',\ldots,i_d')})\barS^{(i_3',\ldots,i_d')}{\conj(\barV^{(i_3',\ldots,i_d')}}){}^{\tc}$ so that $\barX^{(i_3,\ldots,i_d)}$ will be the conjugate of the appropriate slice.


By constructing $\barU, \barV$, and $\barS$ as block diagonal matrices such that the frontal slices $\barU^{(i_3,\ldots,i_d)}$, ${\barV^{(i_3,\ldots,i_d)}}$, and $\barS^{(i_3,\ldots,i_d)}$ for all $i_3\in[n_3],\ldots,i_d\in[n_d]$ are placed as blocks on the diagonals of $\barU, \barV$, and $\barS$ respectively, we obtain that the full  SVD of $\barX$ can be written as $\barX=\barU\barS\barV^{\tc}$.

Denote by $\overbar{\mU},\overbar{\mS},\overbar{\mV}^{\tc}$ the tensors such that $\bdiag(\overbar{\mU})=\barU$, $\bdiag(\overbar{\mS})=\barS$, and $\bdiag(\overbar{\mV}^{\tc})=\barV^{\tc}$. Then, by our construction, $\overbar{\mU},\overbar{\mS},\overbar{\mV}^{\tc}$ all satisfy the conjugate-complex symmetry conditions of \eqref{eq:conjugateComplexSymmetryCondition}. Therefore, by \cref{lemma:conjegateComplexSymmetry_realTensors} the tensors generated from them by computing the inverse Fourier transforms $\mU,\mS,\mV^{\top}$, are all real-valued. It can be seen that $\mU$ is orthogonal since $\barU^{\tc}\barU=\barU\barU^{\tc}=\I$, which by \cref{lemma:tproduct_bdiagMult} implies that $\mU^{\top}*\mU=\mU*\mU^{\top}=\mI$. Similarly, $\mV$ is also orthogonal. Since $\barS$ is diagonal,  it can be seen that invoking the inverse Fourier transform as in \eqref{eq:inverseFFTatPoint}, $\mS$ must be f-diagonal.

In addition, from \cref{lemma:bdiag_bcirc_equality} we can write
\begin{align} \label{eq:inproofTSVD}
& \bcirc(\mU)  = (\F^{-1}\otimes\I_{n_1})\barU(\F\otimes\I_{n_1}), \nonumber
\\ &  \bcirc(\mS) = (\F^{-1}\otimes\I_{n_1})\barS(\F\otimes\I_{n_2}), \nonumber
\\ & \bcirc(\mV^{\top}) = (\F^{-1}\otimes\I_{n_2})\barV^{\tc}(\F\otimes\I_{n_2}),
\end{align}
where $\F:=\F_{n_d}\otimes\F_{n_{d-1}}\otimes\cdots\otimes\F_{n_3}$ and $\F^{-1}:=\F^{-1}_{n_d}\otimes\F^{-1}_{n_{d-1}}\otimes\cdots\otimes\F^{-1}_{n_3}$.

Also by \cref{lemma:bdiag_bcirc_equality}, we have that 
\begin{align*}
\bcirc(\mX) & = (\F^{-1}\otimes\I_{n_1})\barX(\F\otimes\I_{n_2}) = (\F^{-1}\otimes\I_{n_1})\barU\barS\barV^{\tc}(\F\otimes\I_{n_2})
\\ & = (\F^{-1}\otimes\I_{n_1})\barU(\F\otimes\I_{n_1})(\F^{-1}\otimes\I_{n_1})\barS(\F\otimes\I_{n_2})(\F^{-1}\otimes\I_{n_2})\barV^{\tc}(\F\otimes\I_{n_2})
\\ & \underset{(a)}{=}\bcirc(\mU)\bcirc(\mS)\bcirc(\mV^{\top}) \underset{(b)}{=} \bcirc(\mU*\mS*\mV^{\top}),
\end{align*}
where (a) follows from plugging in \eqref{eq:inproofTSVD}, and (b) follows from \cref{lemma:bcirc_Tproduct}.
Folding both sides of the equation back into tensors we obtain that indeed $\mX=\mU*\mS*\mV^{\top}$.
\end{proof}

\subsection{Proof of \cref{lemma:rankRelationships}}
\label{sec:App:proofLemma13}

\begin{lemma}
Let $\mX\in\reals^{n_1\times\cdots\times n_d}$. The following inequalities hold.
\begin{align*}
(i)\ & \rank_{\textnormal{a}}(\mX)\le\rank_{\textnormal{t}}(\mX)
\\ (ii)\ & \rank_{\textnormal{t}}(\mX)\le\rank_{\textnormal{cp}}(\mX)
\\ (iii)\ & \rank_{\textnormal{t}}(\mX)\le\min\left\lbrace\rank(\X^{\lbrace1\rbrace}),\rank(\X^{\lbrace2\rbrace})\right\rbrace,
\end{align*}
where $\rank_{\textnormal{cp}}(\mX)$ is the CP-rank of $\mX$, and $\X^{\lbrace j\rbrace}$ is the mode-$j$ matricization of $\mX$ \cite{kolda}.
\end{lemma}

\begin{proof}
Using the equivalent definition of the tensor average rank given in \eqref{eq:averageRankWithBar} we have that
\begin{align*}
\rank_{\textnormal{a}}(\mX) = \frac{1}{N}\rank(\barX) \le \max_{i_3\in[n_3],\ldots,i_d\in[n_d]} \rank(\barX^{(i_3,\ldots,i_d)}) = \rank_{\textnormal{t}}(\mX),
\end{align*}
which proves $(i)$.


To prove $(ii)$, let $\mX$ be of CP-rank $r$, and let $\mX=\sum_{k=1}^r\a^{(1,k)}\circ \a^{(2,k)}\circ\cdots\circ\a^{(d,k)}$ denote its CP-decomposition, where $\a^{(j,k)}\in\reals^{n_j}$ and $\circ$ denotes the outer product \cite{kolda}.
Then, every element can be written as
$\mX_{i_1,\ldots,i_d}=\sum_{k=1}^r \a^{(1,k)}_{i_1}\a^{(2,k)}_{i_2}\cdots\a^{(d,k)}_{i_d}$. Thus, computing a Fourier transformation of $\mX$ along the $j^{th}$ dimension can be done by multiplying just $\a^{(j,1)},\ldots,\a^{(j,r)}$ by $\F_{n_j}$, i.e., $$\fft_j(\mX)=\sum_{k=1}^r\a^{(1,k)}\circ \a^{(2,k)}\circ\cdots\circ\a^{(j-1,k)}\circ\bar{\a}^{(j,k)}\circ\a^{(j+1,k)}\circ\cdots\circ\a^{(d,k)},$$
where $\bar{\a}^{(j,k)}=\F_{n_j}\cdot{\a}^{(j,k)}$. Therefore, by computing a Fourier transformation of $\mX$ along all but the first two dimensions, we obtain that
$$\overbar{\mX}=\fft(\mX)=\sum_{k=1}^r\a^{(1,k)}\circ \a^{(2,k)}\circ \bar{\a}^{(3,k)}\circ\cdots\circ\bar{\a}^{(d,k)},$$
and therefore, $\rank_{\textnormal{cp}}(\overbar{\mX})=r$. This implies that each frontal slice of $\overbar{\mX}$ is the sum of $r$ rank-one matrices given as $\barX^{(i_3,\ldots,i_d)}= \sum_{k=1}^r (\bar{\a}^{(3,k)}_{i_3}\cdots\bar{\a}^{(d,k)}_{i_d})\cdot(\a^{(1,k)}\circ\a^{(2,k)})$, and hence of rank at most $r$. Therefore,  the tubal rank of $\mX$ is at most $r$.

To prove $(iii)$, denote the Tucker decomposition of a rank-$(r_1,\ldots,r_d)$ tensor $\mX$ as
\begin{align*}
\mX = \mG\times_1\U_1\times_2\U_2\cdots\times_{\textnormal{d}}\U_d,
\end{align*}
where $\mG\in\reals^{r_1,\ldots,r_d}$ is the core tensor, $\U_j\in\reals^{n_j\times r_j}$ are the left singular vectors from the SVD of $\mX^{\lbrace j\rbrace}$, and the operator $\times_n$ is the mode-$n$ product for tensors \cite{kolda}. Computing the Fourier transformation along the $j^{th}$ dimension of $\mX$ can also be written as
\begin{align*}
\fft_j(\mX)=\mX\times_j\F_{n_j}.
\end{align*}
Therefore, we obtain that
\begin{align*}
\overbar{\mX} = \fft(\mX) & = \mX\times_3\F_{n_3}\times_4\F_{n_4}\cdots\times_{\textnormal{d}}\F_{n_d}
\\ & = \mG\times_1\U_1\times_2\U_2\times_3(\F_{n_3}\U_3)\times_4(\F_{n_4}\U_4)\cdots\times_{\textnormal{d}}(\F_{n_d}\U_d),
\end{align*}
where the second equality follows from the properties of the $n$-mode product (see section 2.5 in \cite{kolda}). 

Since the dimensions of $\mG$ are also $r_1\times\cdots\times r_d$, we obtain that the Tucker rank of $\overbar{\mX}$ is also at most $(r_1,\ldots,r_d)$.  
In particular, it follows that $\rank(\X^{\lbrace1\rbrace})=r_1$, which by definition implies that there are a maximum of $r_1$ linear independent mode-$1$ fibers of $\overbar{\mX}$. 
By placing these mode-$1$ fibers as the rows of the block diagonal matrix $\barX=\bdiag(\overbar{\mX})$ we obtain that each block is of rank at most $r_1$. Similarly, since $\rank(\X^{\lbrace2\rbrace})=r_2$, it follows that there are a maximum of $r_2$ linear independent mode-$2$ fibers of $\overbar{\mX}$, and so placing the mode-$2$ fibers as the columns of the diagonal blocks of $\barX$ we obtain that each block is of rank at most $r_2$. Together we have that each block is of rank no larger than $\min\lbrace r_1,r_2\rbrace$, which proves $(iii)$. 
\end{proof}

\section{Proofs omitted from \cref{sec:TNN}}

We first restate each lemma and then prove it.

\subsection{Proof of \cref{lemma:dualNorms}}
\label{sec:App:proofLemma12}

\begin{lemma}
The tensor nuclear norm $\Vert\cdot\Vert_*$ is the dual norm of the spectral norm $\Vert\cdot\Vert_2$.
\end{lemma}

\begin{proof}
Denote $\phi(\mX)$ to be the dual norm to the spectral norm. We will show that $\phi(\mX)=\Vert\mX\Vert_*$.

By the definition of the dual norm, for any tensor $\mX\in\reals^{n_1\times\cdots\times n_d}$ it holds that
\begin{align} \label{ineq:dualNormProof}
\phi(\mX) & = \sup_{\Vert\mY\Vert_2\le1}\langle\mX,\mY\rangle \underset{(a)}{=} \frac{1}{ N}\sup_{\Vert\barY\Vert_2\le1}\langle\barX,\barY\rangle \nonumber
\\ & \le \frac{1}{ N}\sup_{\substack{\Vert\Y\Vert_2\le1
\\ \Y\in\complex^{n_1N\times n_2N}}}\vert\langle\barX,\Y\rangle\vert 
=\frac{1}{ N}\Vert\barX\Vert_*=\Vert\mX\Vert_*,
\end{align}
where (a) follows from \cref{lemma:orderPinnerProductEquivalence}.

Denote the t-SVD of $\mX$ as $\mX=\mU*\mS*\mV^{\top}$. We will show that for $\mY=\mU*\mV^{\top}$, \eqref{ineq:dualNormProof} holds with equality. Indeed,
\begin{align*}
\langle\mX,\mY\rangle & = \langle\mU*\mS*\mV^{\top},\mU*\mV^{\top}\rangle  \underset{(a)}{=} \frac{1}{ N}\langle\barU\boldsymbol{\bar{\S}}\barV^{\tc},\barU\barV^{\tc}\rangle = \frac{1}{ N}\trace(\boldsymbol{\bar{\S}}) = \frac{1}{ N}\Vert\barX\Vert_*=\Vert\mX\Vert_*,
\end{align*}
where (a) follows from \cref{lemma:orderPinnerProductEquivalence} and \cref{lemma:tproduct_bdiagMult}.
\end{proof}

\subsection{Proof of \cref{lemma:subdifferentialSet}}
\label{sec:App:proofLemma100}

\begin{lemma}
Let $\mX\in\reals^{n_1\times\cdots\times n_d}$ and let $\mX=\mU*\mS*\mV^{\top}$ denote its skinny t-SVD as defined in \cref{def:skinnytSVD}. The subdifferential set of the TNN at $\mX$ is given by 
\begin{align*}
\partial\Vert\mX\Vert_* = \lbrace\mU*\mV^{\top}+\mW\ \vert\ \mU^{\top}*\mW=\mathbf{0},\ \mW*\mV=\mathbf{0},\ \Vert\mW\Vert_2\le1\rbrace.
\end{align*}
\end{lemma}

\begin{proof}
It is well known (see for instance \cite{WATSON199233})  that 
$\mG\in\partial\Vert\mX\Vert_*$ if and only if the following two condition hold:
\begin{align*}
(i)   \ \Vert\mX\Vert_* = \langle \mX,\mG\rangle, \quad (ii) &  \ \Vert\mG\Vert_2\le1.
\end{align*}

Fix $\mX\in\reals^{n_1\times\cdots\times n_d}$. Denote $\barX=\barU\barS\barV^{\tc}$ as the skinny SVD of $\barX$, and denote
\begin{align*}
\mathcal{S}(\barX) := \lbrace\barU\barV^{\tc}+\barW\ \vert\ \barU^{\tc}\barW=\mathbf{0},\ \barW\barV=\mathbf{0},\ \Vert\barW\Vert_2\le1\rbrace.
\end{align*}
We will show that $\partial\Vert\mX\Vert_*= \lbrace\mG\in\reals^{n_1\times\cdots\times n_d}\ \vert\ \barG:=\bdiag(\overbar{\mG})\in\mathcal{S}(\barX)\rbrace$.

We  begin by proving that $\lbrace\mG\in\reals^{n_1\times\cdots\times n_d}\ \vert\ \barG:=\bdiag(\overbar{\mG})\in\mathcal{S}(\barX)\rbrace \subseteq\partial\Vert\mX\Vert_*$. Let $\mG$ such that $\barG=\barU\barV^{\tc}+\barW\in\mathcal{S}(\barX)$. Since $\Vert\barW\Vert_2\le1$, it follows that $\Vert\mG\Vert_2 = \Vert\barG\Vert_2 =1$, which proves $(ii)$. In addition, it holds that 
\begin{align*}
\langle\mX,\mG\rangle & = \frac{1}{ N}\langle\barX,\barG\rangle = \frac{1}{ N}\langle\barU\barS\barV^{\tc},\barU\barV^{\tc}+\barW\rangle = \frac{1}{ N}\trace(\barS) = \frac{1}{ N}\Vert\barX\Vert_* = \Vert\mX\Vert_*,
\end{align*}
which proves $(i)$. Therefore, $\mG\in\partial\Vert\mX\Vert_*$.

For the second direction, assume there exists $\mG\in\partial\Vert\mX\Vert_*$ for which $\barG\not\in\mathcal{S}(\barX)$.
Therefore, there exists at least one singular vector pair $\u_j,\v_j$ of $\barX$ that is either not a singular vector pair of $\barG$ all together or not a singular vector pair of $\barG$ corresponding to the leading singular value $\sigma_1(\barG)$. This implies that $\u_j^{\tc}\barG\v_j<\sigma_1(\barG)=1$, whereas for all $i\not=j$ it holds that $\u_i^{\tc}\barG\v_i\le\sigma_1(\barG)=1$. Therefore, in either case we have that
\begin{align*}
\langle\mX,\mG\rangle & = \frac{1}{ N}\langle\barX,\barG\rangle = \frac{1}{ N}\sum_{i=1}^{\rank(\barX)}\sigma_i(\barX)\u_i^{\tc}\barG\v_i < \frac{1}{ N}\sum_{i=1}^{\rank(\barX)}\sigma_i(\barX)
= \frac{1}{ N}\Vert\barX\Vert_* = \Vert\mX\Vert_*,
\end{align*}
which contradicts $(i)$.

Finally, notice that the subdifferential set can be written equivalently as
\begin{align*}
& \partial\Vert\mX\Vert_* 
 = \lbrace\mG\in\reals^{n_1\times\cdots\times n_d}\ \vert\ \barG:=\bdiag(\overbar{\mG})\in\mathcal{S}(\barX)\rbrace
\\ & = \lbrace\mG\in\reals^{n_1\times\cdots\times n_d}\ \vert\ \barG=\barU\barV^{\tc}+\barW,\ \barU^{\tc}\barW=\mathbf{0},\ \barW\barV=\mathbf{0},\ \Vert\barW\Vert_2\le1, \barG=\bdiag(\overbar{\mG})\rbrace
\\ & = \lbrace\mU*\mV^{\top}+\mW\ \vert\ \mU^{\top}*\mW=\mathbf{0},\ \mW*\mV=\mathbf{0},\ \Vert\mW\Vert_2\le1\rbrace,
\end{align*}
where $\mX=\mU*\mS*\mV^{\top}$ is the skinny SVD of $\mX$.
\end{proof}

\subsection{Proof of \cref{lemma:projectionOntoTNN}}
\label{sec:App:proofLemma11}

\begin{lemma}[Projection onto the tensor nuclear norm]
Let $\mX\in\reals^{n_1\times\cdots\times n_d}$. The Euclidean projection of $\mX$ onto the tensor nuclear norm ball of radius $\tau\ge0$ can be computed by the steps described in \cref{alg:projectionTNN}.

Moreover, if $\rank_{\textnormal{t}}(\Pi_{\lbrace\Vert\mY\Vert_*\le\tau\rbrace}[\mX])  \leq r$, then the t-SVD computation in the algorithm could be replaced with the rank-$r$ t-SVD (Definition \ref{def:rankRtSVD}) and the summation from $1$ to $\rank_{\textnormal{t}}(\mX)$ in the computation of $\sigma$ could be replaced with a summation from $1$ to $r$.
\end{lemma}

\begin{proof}
The problem of projecting a tensor $\mX\in\reals^{n_1\times\cdots\times n_d}$ onto the TNN ball of radius $\tau\ge0$ can be written
as the following optimization problem:
\begin{align} \label{eq:projOverTNN}
\Pi_{\lbrace\Vert{\mY}\Vert_*\leq \tau\rbrace}[\mX]=\argmin_{\Vert\mY\Vert_*\le\tau}\frac{1}{2}\Vert\mY-\mX\Vert_F^2.
\end{align}

Denote $\bdiag^{-1}(\cdot)$ to be the inverse operator of $\bdiag(\cdot)$ such that for any $\overbar{\mX}\in\complex^{n_1\times\cdots\times n_d}$ $\bdiag^{-1}(\bdiag(\overbar{\mX}))=\overbar{\mX}$.
Consider the following optimization problem over the block diagonal matrix domain
\begin{align} \label{optOverBarX}
\min_{\substack{\Vert\Y\Vert_*\le N\tau \\ \Y\in\mcM}}\frac{1}{2 N}\Vert\Y-\barX\Vert_F^2,
\end{align}
where $\barX=\bdiag(\overbar{\mX})$ and $\mcM\subset\complex^{n_1 N\times n_2 N}$ denotes the subset of $\complex^{n_1 N\times n_2 N}$ which contains all the block diagonal matrices $\Y$ such that $\bdiag^{-1}(\Y)$ satisfies the conjugate-complex symmetry condition in \eqref{eq:conjugateComplexSymmetryCondition}\footnote{note this set is closed and convex}. 

We will show that since $\mX$ is real-valued, the projection of $\barX$ onto the ball $\lbrace\Y\in\complex^{n_1\times\cdots\times n_d}\  \vert\ \Vert\Y\Vert_*\le N\tau\rbrace$, which we will denote by $\Y^*$, also satisfies that $\Y^*\in\mcM$, and thus, is the optimal solution to \eqref{optOverBarX}.


 It is well known that  the projection of a matrix $\barX$ onto the matrix nuclear norm ball of radius $\tau N$ can be written as $\Y^* = \barU\barS_{\tau}\barV^{\tc}$, where $\barX=\barU\barS\barV^{\tc}$ is the SVD of $\barX$ and $\barS_{\tau}$ is the matrix obtained by projecting the diagonal of $\barS$ onto the simplex of radius $\tau N$.

The t-SVD of $\mX$ can be written as $\mX=\mU*\mS*\mV^{\top}$, where $\mU=\ifft(\bdiag^{-1}(\barU))$, $\mS=\ifft(\bdiag^{-1}(\barS))$, and $\mV^{\top}=\ifft(\bdiag^{-1}(\barV^{\tc}))$. Since $\mX$ is real-valued, $\mU,\mS,\mV^{\top}$ are also all real-valued. Therefore, by \cref{lemma:conjegateComplexSymmetry_realTensors}, the tensors $\overbar{\mU},\overbar{\mS},\overbar{\mV}^{\tc}$ all satisfy the conjugate-complex symmetry condition in \eqref{eq:conjugateComplexSymmetryCondition}.
Therefore, the matrices $\barU,\barS,\barV$ are block diagonal and they satisfy that for any $i_3\in[n_3],\ldots,i_d\in[n_d]$, $\barU^{(i_3,\ldots,i_d)}=\conj(\barU^{(i_3',\ldots,i_d')})$,  $\barV^{(i_3,\ldots,i_d)}=\conj(\barV^{(i_3',\ldots,i_d')})$, and $\barS^{(i_3,\ldots,i_d)}=\barS^{(i_3',\ldots,i_d')}$ where $i_j'$ is as defined in \eqref{def:i_j_tag} for all $j\in{3,\ldots,d}$.

To project $\diag(\barS)$ onto the the simplex of radius $\tau N$ we need to find the unique $\sigma\ge0$ for which $\sum_{i=1}^{\rank_{\textnormal{t}}(\mX)}\sum_{i_3=1}^{n_3}\cdots\sum_{i_d=1}^{n_d} \max\lbrace0,\sigma_i(\barX^{(i_3,\ldots,i_d)})-\sigma\rbrace=\tau N$ \cite{projectOntoSimplex}. Since $\barS^{(i_3,\ldots,i_d)}=\barS^{(i_3',\ldots,i_d')}$, also $\max\lbrace0,\sigma_i(\barX^{(i_3,\ldots,i_d)})-\sigma\rbrace = \max\lbrace0,\sigma_i(\barX^{(i_3',\ldots,i_d')})-\sigma\rbrace$ for any $i_3\in[n_3],\ldots$, $i_d\in[n_d]$, and thus, $\barS_{\tau}^{(i_3,\ldots,i_d)}=\barS_{\tau}^{(i_3',\ldots,i_d')}$. Therefore, $\bdiag^{-1}(\Y^*)$ also satisfies the conjugate-complex symmetry condition in \eqref{eq:conjugateComplexSymmetryCondition}, and so, $\Y^*\in\mcM$ as desired. 

It remains to show that $\mY^*:=\ifft(\bdiag^{-1}(\Y^*))=\mU*\mS_{\tau}*\mV^{\top}$ where $\mS_{\tau}=\ifft(\bdiag^{-1}(\barS_{\tau}))$ is the optimal solution to \eqref{eq:projOverTNN}. First, since $\Y^*\in\mcM$,  by \cref{lemma:conjegateComplexSymmetry_realTensors} $\mY^*$ is real-valued. In addition, by \eqref{eq:nuclearNormEquivalence} we know that $\Vert\mY^*\Vert_*\le\tau$. Therefore $\mY^*$ is feasible for Problem \eqref{eq:projOverTNN}. Now, for every $\mY\in\lbrace\mY\in\reals^{n_1\times\cdots\times n_d}\  \vert\ \Vert\mY\Vert_*\le\tau\rbrace$ it holds that $\barY:=\bdiag(\overbar{\mY})\in\lbrace\barY\in\mcM\  \vert\ \Vert\barY\Vert_*\le\tau N\rbrace$. Therefore, we obtain that 
\begin{align*}
\frac{1}{2}\Vert\mX-\mY\Vert_F^2 \underset{(a)}{=} \frac{1}{2 N}\Vert\barX-\barY\Vert_F^2 \underset{(b)}{\ge} \frac{1}{2 N}\Vert\barX-\barY^*\Vert_F^2 \underset{(c)}{=} \frac{1}{2}\Vert\mX-\mY^*\Vert_F^2,
\end{align*}
where (a) and (c) follow from \cref{lemma:orderPinnerProductEquivalence}, and (b) follows from the optimality of $\barY^*$. Therefore, $\mY^*$ is optimal for  \eqref{eq:projOverTNN} and $\Pi_{\lbrace\Vert\mZ\Vert_*\le\tau\rbrace}[\mX]=\mY^*$.

For the second part of the lemma, notice that if we know that $\rank_{\textnormal{t}}(\Pi_{\lbrace\Vert\mY\Vert_*\le\tau\rbrace}[\mX])  \leq r$, then we know that for any $i_3\in[n_3],\ldots,i_d\in[n_d]$ and $r<i\le\rank_{\textnormal{t}}(\mX)$, the value of $\max\lbrace0,\sigma_i(\barX^{(i_3,\ldots,i_d)})-\sigma\rbrace$ will be zero. Therefore, the computation of any component of the t-SVD of $\mX$ that does not correspond to one of the leading $r$ components of one of its frontal slices is unnecessary and can be skipped.

\end{proof}


\section{Proofs Omitted from \cref{sec:SCsmooth}}
We first restate each lemma and then prove it.

\subsection{Proof of \cref{lemma:svd_opt_grad_d_dim}}
\label{sec:App:proofLemma199}

\begin{lemma}
Let $\mX^*$ be an optimal solution to Problem \eqref{mainModel} and let $\overbar{\mX^*}$ denote its Fourier transform as defined in \cref{def:fft}. For each $i_3\in[n_3],\ldots,i_d\in[n_d]$, denote the SVD of the frontal slice ${\overbar{\X^*}}^{(i_3,\ldots,i_d)}$ as ${\overbar{\X^*}}^{(i_3,\ldots,i_d)}=\sum_{i=1}^{r_{i_3,\ldots,i_d}}\sigma_i^{(i_3,\ldots,i_d)}\u_i^{(i_3,\ldots,i_d)}{\v_i^{(i_3,\ldots,i_d)}}^{\tc}$, where $r_{i_3,\ldots,i_d}=\rank({\overbar{\X^*}}^{(i_3,\ldots,i_d)})$.
Then, each frontal slice of the Fourier transform of the gradient vector $\overbar{\nabla{}f(\mX^*)}^{(i_3,\ldots,i_d)}$ admits a SVD such that the set of pairs of vectors $\lbrace -\u_i^{(i_3,\ldots,i_d)}, \v_i^{(i_3,\ldots,i_d)}\rbrace_{i=1}^{r_{i_3,\ldots,i_d}}$ is a
set of top singular-vector pairs of $\overbar{\nabla{}f(\mX^*)}^{(i_3,\ldots,i_d)}$ which corresponds to the largest singular value $\sigma_1(\overbar{\nabla{}f(\mX^*)}^{(i_3,\ldots,i_d)})$. Furthermore, the top singular values of all nonzero slices are equal, that is,
$$\sigma_1(\overbar{\nabla{}f(\mX^*)})=\sigma_1(\overbar{\nabla{}f(\mX^*)}^{(i_3,\ldots,i_d)}),$$
for all $i_3\in[n_3],\ldots,i_d\in[n_d]$ such that $\sigma_1(\overbar{\nabla{}f(\mX^*)}^{(i_3,\ldots,i_d)})\not=0$.
\end{lemma}

\begin{proof}
From the first-order optimality condition it holds that
\begin{align} \label{firstOrderOpt_d_dim}
\mathbf{0}\in \nabla{}f(\mX^*) + \mathcal{N}_{\lbrace\Vert\mX\Vert_*\le\tau\rbrace}(\mX^*).
\end{align}

Let $\mX^*=\mU*\mS*\mV^{\top}$ denote the skinny t-SVD of $\mX^*$. The normal cone of the tensor nuclear norm at $\mX^*$ can be written as the conic hull generated by the subdifferential of the tensor nuclear norm at $\mX^*$. Thus, using the characteristic of the subdifferential set in \cref{lemma:subdifferentialSet},
we have that
\begin{align*} 
\mathcal{N}_{\lbrace\Vert\mX\Vert_*\le1\rbrace}(\mX^*) = \lbrace\lambda(\mU*\mV^{\top}+\mW)\ \vert\ \mU^{\top}*\mW=\mathbf{0},\ \mW*\mV=\mathbf{0},\ \Vert\mW\Vert_2\le1,\ \lambda\ge0\rbrace.
\end{align*}
If we consider this set in the Fourier domain we have that
\begin{align} \label{eq:NNfft_d_dim}
\fft(\mathcal{N}_{\lbrace\Vert\mX\Vert_*\le1\rbrace}(\mX^*)) & = \lbrace\lambda(\barU\barV^{\tc}+\barW)\ \vert\ \barU^{\tc}\barW=\mathbf{0},\ \barW\barV=\mathbf{0},\ \Vert\barW\Vert_2\le1,\ \lambda\ge0\rbrace,
\end{align}
where $\barU$ and $\barV$ are  matrices such that for every $i_3\in[n_3],\ldots,i_d\in[n_d]$, it holds that ${\barU^{(i_3,\ldots,i_d)}}^{\tc}{\barU^{(i_3,\ldots,i_d)}}=\diag(\I_{r_{i_3,\ldots,i_d}},\mathbf{0})$ and ${\barV^{(i_3,\ldots,i_d)}}^{\tc}{\barV^{(i_3,\ldots,i_d)}}=\diag(\I_{r_{i_3,\ldots,i_d}},\mathbf{0})$.

For \eqref{firstOrderOpt_d_dim} to hold, in the Fourier domain it must hold that $\mathbf{0}\in \overbar{\nabla{}f(\mX^*)} + \fft(\mathcal{N}_{\lbrace\Vert\mX\Vert_*\le1\rbrace}(\mX^*))$, which implies that there must exist $\barW$ and $\lambda$ for which the conditions in the  RHS of \eqref{eq:NNfft_d_dim} hold and for which 
\begin{align*}
\overbar{\nabla{}f(\mX^*)} = \lambda(-\barU\barV^{\tc}-\barW).
\end{align*}
Writing this equality for each block separately, we obtain that for every $i_3\in[n_3],\ldots,i_d\in[n_d]$, it holds that
\begin{align*}
\overbar{\nabla{}f(\mX^*)}^{(i_3,\ldots,i_d)} = \lambda(-\barU^{(i_3,\ldots,i_d)}{\barV^{(i_3,\ldots,i_d)}}^{\tc}-\barW^{(i_3,\ldots,i_d)}).
\end{align*}
By the condition on $\barW$ in \eqref{eq:NNfft_d_dim} that $\Vert\barW\Vert_2\le1$, we have that $\Vert\barW^{(i_3,\ldots,i_d)}\Vert_2 \le \Vert\barW\Vert_2 \le1$. Additionally,  we know that $\Vert\barU^{(i_3,\ldots,i_d)}{\barV^{(i_3,\ldots,i_d)}}^{\tc}\Vert_2=1$. Since $\barW$ is orthogonal to $\barU\barV^{\tc}$, we have that the $r_{i_3,\ldots,i_d}$ nonzero columns of $-\barU^{(i_3,\ldots,i_d)}$ and $\barV^{(i_3,\ldots,i_d)}$ are top singular-vector pairs of  $\overbar{\nabla{}f(\mX^*)}^{(i_3,\ldots,i_d)}$ which correspond to the largest singular value $\sigma_1(\overbar{\nabla{}f(\mX^*)}^{(i_3,\ldots,i_d)})=\lambda$. Since this holds for all $i_3\in[n_3],\ldots,i_d\in[n_d]$, the top singular value of all frontal slices is equal to  $\lambda$, and so they are all equal to each other. In addition, they are also equal to  $\sigma_1(\overbar{\nabla{}f(\mX^*)})$ since the top singular value of $\overbar{\nabla{}f(\mX^*)}$ is the top singular value out of all the singular values of all the frontal slices.

\end{proof}

\subsection{Proof of \cref{lemma:SCequivalence}}
\label{sec:App:proofLemma14}

\begin{lemma} 
Let $\mX^*\in\reals^{n_1\times\cdots\times n_d}$ be an optimal solution to Problem \eqref{mainModel} for which $\rank_{\textnormal{a}}(\mX^*)=r<\min\lbrace n_1,n_2\rbrace$. $\mX^*$ satisfies the strict complementarity condition with some $\delta>0$ if and only if
\begin{align*}
\delta=\sigma_1(\overbar{\nabla{}f(\mX^*)}) - \sigma_{rN+1}(\overbar{\nabla{}f(\mX^*)}) >0.
\end{align*}  
\end{lemma}

\begin{proof}

Let $\mX^*=\mU*\mS*\mV^{\top}$ be the skinny t-SVD of $\mX^*$. Taking the Fourier transform of the relative interior of the normal cone of the TNN ball (see \eqref{eq:NNfft_d_dim}), we have that
\begin{align} \label{eq:riNNfft_d_dim}
\fft(\ri(\mathcal{N}_{\lbrace\Vert\mX\Vert_*\le1\rbrace}(\mX^*))) & = \lbrace\lambda(\barU\barV^{\tc}+\barW)\ \vert\ \barU^{\tc}\barW=\mathbf{0},\ \barW\barV=\mathbf{0},\ \Vert\barW\Vert_2<1,\ \lambda>0\rbrace.
\end{align}

For \eqref{strictComplementarityDef} to hold, in the Fourier domain it must hold that $\mathbf{0}\in \overbar{\nabla{}f(\mX^*)} + \fft(\ri(\mathcal{N}_{\lbrace\Vert\mX\Vert_*\le1\rbrace}(\mX^*)))$, which implies that there must exist $\barW$ and $\lambda>0$ for which the conditions in the  RHS of \eqref{eq:riNNfft_d_dim} hold, and for which 
\begin{align} \label{eq:optEquality_d_dim2}
\overbar{\nabla{}f(\mX^*)} = \lambda(-\barU\barV^{\tc}-\barW).
\end{align}

From \cref{lemma:svd_opt_grad_d_dim} we know that for every $i_3\in[n_3],\ldots,i_d\in[n_d]$, the SVD of $\overbar{\nabla{}f(\mX^*)}^{(i_3,\ldots,i_d)}$ can be written as
\begin{align*}
\overbar{\nabla{}f(\mX^*)}^{(i_3,\ldots,i_d)} = -\sigma_1(\overbar{\nabla{}f(\mX^*)})\barU^{(i_3,\ldots,i_d)}{\barV^{(i_3,\ldots,i_d)}}^{\tc} - \barU_{\perp}^{(i_3,\ldots,i_d)}\boldsymbol{\Sigma}^{(i_3,\ldots,i_d)}{\barV_{\perp}^{(i_3,\ldots,i_d)}}^{\tc},
\end{align*}
where $\barU^{(i_3,\ldots,i_d)}$ and $\barV^{(i_3,\ldots,i_d)}$ are orthogonal to $\barU_{\perp}^{(i_3,\ldots,i_d)}$ and $\barU_{\perp}^{(i_3,\ldots,i_d)}$ respectively. For this equation to satisfy \eqref{eq:optEquality_d_dim2} it must follow that $\lambda=\sigma_1(\overbar{\nabla{}f(\mX^*)})$  and $\barW^{(i_3,\ldots,i_d)}=\frac{1}{\sigma_1(\overbar{\nabla{}f(\mX^*)})}\barU_{\perp}^{(i_3,\ldots,i_d)}\boldsymbol{\Sigma}^{(i_3,\ldots,i_d)}{\barV_{\perp}^{(i_3,\ldots,i_d)}}^{\tc}$.

Therefore, for all $i_3\in[n_3],\ldots,i_d\in[n_d]$ we have that
\begin{align*}
\Vert\barW^{(i_3,\ldots,i_d)}\Vert_2 & =\frac{\sigma_{r_{i_3,\ldots,i_d}+1}(\overbar{\nabla{}f(\mX^*)}^{(i_3,\ldots,i_d)}))}{\sigma_1(\overbar{\nabla{}f(\mX^*)})} 
 = \frac{\sigma_{1}(\overbar{\nabla{}f(\mX^*)})-\delta_{i_3,\ldots,i_d}}{\sigma_1(\overbar{\nabla{}f(\mX^*)})} = 1-\frac{\delta_{i_3,\ldots,i_d}}{\lambda},
\end{align*}
where we denote $\delta_{i_3,\ldots,i_d}:=\sigma_1(\overbar{\nabla{}f(\mX^*)})-\sigma_{r_{i_3,\ldots,i_d}+1}(\overbar{\nabla{}f(\mX^*)}^{(i_3,\ldots,i_d)})$ and $r_{i_3,\ldots,i_d}=\rank(\overbar{\mX^*}^{(i_3,\ldots,i_d)})$.

Denote 
\begin{align*}
\delta_{\min}:&=\min_{i_3\in[n_3],\ldots,i_d\in[n_d]}\sigma_1(\overbar{\nabla{}f(\mX^*)})-\sigma_{r_{i_3,\ldots,i_d}+1}(\overbar{\nabla{}f(\mX^*)}^{(i_3,\ldots,i_d)})
 =\lambda-\sigma_{rN+1}(\overbar{\nabla{}f(\mX^*)}).
\end{align*}
Then, 
\begin{align*}
\Vert\barW\Vert_2 = \max_{i_3\in[n_3],\ldots,i_d\in[n_d]}\Vert\barW^{(i_3,\ldots,i_d)}\Vert_2 = 1-\frac{\delta_{\min}}{\lambda},
\end{align*}
and $\Vert\barW\Vert_2<1$ if and only if $\delta_{\min}>0$.

Finally, we will show that $\delta_{\min}$ is the complementarity measure as in the definition of \eqref{complementarityMeasure}, that is, that $\delta_{\min}=\delta$ holds.

First, note that
\begin{align*}
& -\langle{\overbar{\X^*}}^{(i_3,\ldots,i_d)},\overbar{\nabla{}f(\mX^*)}^{(i_3,\ldots,i_d)}\rangle 
\\ & = \langle\barU^{(i_3,\ldots,i_d)}\barS^{(i_3,\ldots,i_d)}{\barV^{(i_3,\ldots,i_d)}}^{\tc},\lambda\barU^{(i_3,\ldots,i_d)}{\barV^{(i_3,\ldots,i_d)}}^{\tc} - \barU_{\perp}^{(i_3,\ldots,i_d)}\boldsymbol{\Sigma}^{(i_3,\ldots,i_d)}{\barV_{\perp}^{(i_3,\ldots,i_d)}}^{\tc}\rangle
\\ & = \lambda\langle\barU^{(i_3,\ldots,i_d)}\barS^{(i_3,\ldots,i_d)}{\barV^{(i_3,\ldots,i_d)}}^{\tc},\barU^{(i_3,\ldots,i_d)}{\barV^{(i_3,\ldots,i_d)}}^{\tc}\rangle
\\ & = \lambda\trace(\barS^{(i_3,\ldots,i_d)}) \underset{(a)}{=}  N\lambda,
\end{align*}
where (a) holds since $\lambda>0$, which implies that ${\nabla}f(\mX^*)\not=0$ and so $1=\Vert\mX^*\Vert_*=(1/ N)\Vert\overbar{\X^*}\Vert_*$.

Therefore, we can write the complementarity measure  definition in \eqref{complementarityMeasure} as
\begin{align*} 
\delta & = \frac{1}{ N}\min\lbrace\langle\barZ-\overbar{\X^*},\overbar{\nabla{}f(\mX^*)}\rangle\ \vert\ \Vert\barZ\Vert_*\le  N,\ \barU^{\tc}\barZ=\mathbf{0},\ \barZ\barV=\mathbf{0}\rbrace \nonumber
\\ & = \frac{1}{ N}\min\left\lbrace\sum_{i_3=1}^{n_3}\cdots\sum_{i_d=1}^{n_d}\langle\barZ^{(i_3,\ldots,i_d)}-{\overbar{\X^*}}^{(i_3,\ldots,i_d)},\overbar{\nabla{}f(\mX^*)}^{(i_3,\ldots,i_d)}\rangle\ \bigg\vert\ \begin{array}{l}\Vert\barZ\Vert_*\le  N,\\ \barU^{\tc}\barZ=\mathbf{0},\ \barZ\barV=\mathbf{0}\end{array}\right\rbrace
\\ & = \lambda + \frac{1}{ N}\min\lbrace\langle\barZ,\overbar{\nabla{}f(\mX^*)}\rangle\ \vert\ \Vert\barZ\Vert_*\le  N,\ \barU^{\tc}\barZ=\mathbf{0},\ \barZ\barV=\mathbf{0}\rbrace
\\ & = \lambda - \sigma_{rN+1}(\overbar{\nabla{}f(\mX^*)}) = \delta_{\min},
\end{align*}
as desired.
\end{proof}

\subsection{Proof of \cref{lemma:motivateSCrobustness}}
\label{sec:App:proofLemma17}
\begin{lemma} 
Let $\mX^*\in\reals^{n_1\times\cdots\times n_d}$ be an optimal solution to Problem \eqref{mainModel} such that $\nabla{}f(\mX^*)\not=0$ and let $\varepsilon\ge0$. Then, for any step-size $\eta>0$, it holds that
$$\rank_{\textnormal{a}}(\Pi_{\lbrace\Vert\mY\Vert_*\le 1+\varepsilon\rbrace}[\mX^*-\eta\nabla{}f(\mX^*)])>r$$
if and only if $\varepsilon>\eta\left(\sigma_1(\overbar{\nabla{}f(\mX^*)})-\sigma_{rN+1}(\overbar{\nabla{}f(\mX^*)})\right)$.
\end{lemma}

\begin{proof}
Denote $\mP^*:=\mX^*-\eta\nabla{}f(\mX^*)$ and 
in the Fourier domain, for all $i_3\in[n_3],\ldots,i_d\in[n_d]$ denote ${\overbar{\P^*}}^{(i_3,\ldots,i_d)}:= {\overbar{\X^*}}^{(i_3,\ldots,i_d)} - \eta\overbar{\nabla{}f(\mX^*)}^{(i_3,\ldots,i_d)}$. Invoking \cref{lemma:svd_opt_grad_d_dim} we have that for all $i_3\in[n_3],\ldots,i_d\in[n_d]$ it holds that
\begin{align} \label{eq:svOpt_motivateSC_BySlice}
\forall i\le \rank({\overbar{\X^*}}^{(i_3,\ldots,i_d)}):\quad & \sigma_i({\overbar{\P^*}}^{(i_3,\ldots,i_d)}) = \sigma_i({\overbar{\X^*}}^{(i_3,\ldots,i_d)}) + \eta \sigma_1(\overbar{\nabla{}f(\mX^*)}) \nonumber
\\ \forall i>\rank({\overbar{\X^*}}^{(i_3,\ldots,i_d)}):\quad & \sigma_i({\overbar{\P^*}}^{(i_3,\ldots,i_d)}) = \eta \sigma_i(\overbar{\nabla{}f(\mX^*)}^{(i_3,\ldots,i_d)}).
\end{align}
Since $rN=\sum_{i_d=1}^{n_d}\cdots\sum_{i_3=1}^{n_3}\rank({\overbar{\X^*}}^{(i_3,\ldots,i_d)})$, it can be seen that \eqref{eq:svOpt_motivateSC_BySlice} can also be written as
\begin{align} \label{eq:svOpt_motivateSC_bigMatrix}
\forall i\le rN:\quad & \sigma_i({\overbar{\P^*}}) = \sigma_i({\overbar{\X^*}}) + \eta \sigma_1(\overbar{\nabla{}f(\mX^*)}) \nonumber
\\ \forall i>rN:\quad & \sigma_i({\overbar{\P^*}}) = \eta \sigma_i(\overbar{\nabla{}f(\mX^*)}).
\end{align}

Also, by the assumption in the lemma $\nabla{}f(\mX^*)\not=0$ and so it follows that $\Vert\mX^*\Vert_*=1$. 

To project $\mP^*$ onto the nuclear-norm ball of radius $1+\varepsilon$ it must hold  for some $\sigma\ge0$ that
$$\frac{1}{N}\sum_{i=1}^{\min\lbrace n_1,n_2\rbrace N}\max\lbrace0,\sigma_i({\overbar{\P^*}})-\sigma\rbrace =1+\varepsilon.$$ 


If $\rank_{\textnormal{a}}(\Pi_{\lbrace\Vert\mY\Vert_*\le1+\varepsilon\rbrace}[\mP^*])\le r$ then necessarily $\sigma\ge\sigma_{rN+1}(\overbar{\P^*}) = \eta\sigma_{rN+1}(\overbar{\nabla{}f(\mX^*)})$. Therefore it holds that,
\begin{align*}
1+\varepsilon & = \frac{1}{N}\sum_{i=1}^{\min\lbrace n_1,n_2\rbrace N}\max\lbrace0,\sigma_i({\overbar{\P^*}})-\sigma\rbrace = \frac{1}{N}\sum_{i=1}^{rN}\max\lbrace0,\sigma_i({\overbar{\P^*}})-\sigma\rbrace
\\ & = \frac{1}{N}\sum_{i=1}^{rN}\max\lbrace0,\sigma_i({\overbar{\X^*}})+\eta\sigma_1(\overbar{\nabla{}f(\mX^*)})-\sigma\rbrace
\\ & \le \frac{1}{N}\sum_{i=1}^{rN}\max\left\lbrace0,\sigma_i({\overbar{\X^*}})+\eta\left(\sigma_1(\overbar{\nabla{}f(\mX^*)})-\sigma_{rN+1}(\overbar{\nabla{}f(\mX^*)})\right)\right\rbrace
\\ & = \frac{1}{N}\sum_{i=1}^{rN}\left(\sigma_i({\overbar{\X^*}})+\eta\left(\sigma_1(\overbar{\nabla{}f(\mX^*)})-\sigma_{rN+1}(\overbar{\nabla{}f(\mX^*)})\right)\right)
\\ & = 1 + \eta\left(\sigma_1(\overbar{\nabla{}f(\mX^*)})-\sigma_{rN+1}(\overbar{\nabla{}f(\mX^*)})\right), 
\end{align*}
which implies that $\varepsilon \le \eta\left(\sigma_1(\overbar{\nabla{}f(\mX^*)})-\sigma_{rN+1}(\overbar{\nabla{}f(\mX^*)})\right)$.

On the other hand, if $\rank_{\textnormal{a}}(\Pi_{\lbrace\Vert\mY\Vert_*\le1+\varepsilon\rbrace}[\mP^*])>r$ then necessarily $\sigma<\sigma_{rN+1}(\overbar{\P^*}) = \eta\sigma_{rN+1}(\overbar{\nabla{}f(\mX^*)})$. In this case it follow that
\begin{align*}
1+\varepsilon & = \frac{1}{N}\sum_{i=1}^{\min\lbrace n_1,n_2\rbrace N}\max\lbrace0,\sigma_i({\overbar{\P^*}})-\sigma\rbrace = \frac{1}{N}\sum_{i=1}^{rN}\max\lbrace0,\sigma_i({\overbar{\P^*}})-\sigma\rbrace
\\ & = \frac{1}{N}\sum_{i=1}^{rN}\max\lbrace0,\sigma_i({\overbar{\X^*}})+\eta\sigma_1(\overbar{\nabla{}f(\mX^*)})-\sigma\rbrace
\\ & > \frac{1}{N}\sum_{i=1}^{rN}\left(\sigma_i({\overbar{\X^*}})+\eta\left(\sigma_1(\overbar{\nabla{}f(\mX^*)})-\sigma_{rN+1}(\overbar{\nabla{}f(\mX^*)})\right)\right)
\\ & = 1 + \eta\left(\sigma_1(\overbar{\nabla{}f(\mX^*)})-\sigma_{rN+1}(\overbar{\nabla{}f(\mX^*)})\right), 
\end{align*}
which implies that $\varepsilon > \eta\left(\sigma_1(\overbar{\nabla{}f(\mX^*)})-\sigma_{rN+1}(\overbar{\nabla{}f(\mX^*)})\right)$.

%
\end{proof}

\subsection{Proof of \cref{lemma:motivateSCalmostAll}}
\label{sec:App:proofLemma18}
\begin{lemma} 
Assume $f(\mX) = g(\mX) + \langle\mC,\mX\rangle$. Then, for almost all $\mC\in\reals^{n_1\times\cdots\times n_d}$, Problem \eqref{mainModel} admits a unique minimizer which furthermore satisfies strict complementarity.
\end{lemma}

\begin{proof}
Denote  $\psi_{\mC}(\mX) = g(\mX) + \langle\mC,\mX\rangle + \chi_{\lbrace\Vert\mX\Vert_*\le1\rbrace}(\mX)$, where $\chi_{\lbrace\Vert\mX\Vert_*\le1\rbrace}(\cdot)$ is the indicator function for the unit TNN ball. From Corollary 3.5 in \cite{nondegeneracy} for almost all $\mC$, since $\lbrace\Vert\mX\Vert_*\le1\rbrace$ is closed and bounded, $\psi_{\mC}$ admits a single minimizer $\mX^*$ and it satisfies that 
\begin{align*}
\mathbf{0}  \in \ri(\partial \psi_{\mC}(\mX^*)) 
 & \underset{(a)}{=} \ri\left(\nabla{}f(\mX^*) + \mathcal{N}_{\lbrace\Vert\mX\Vert_*\le1\rbrace}(\mX^*)\right)
\\ & \underset{(b)}{=} \nabla{}f(\mX^*) + \ri(\mathcal{N}_{\lbrace\Vert\mX\Vert_*\le1\rbrace}(\mX^*)),
\end{align*}
where (a) follows since the normal cone is the subdiffrential set of the indicator function, and (b) follows from the sum rule of relative interiors. 
This is precisely the condition for strict complementarity as defined in \eqref{strictComplementarityDef}.
\end{proof}

\subsection{Proof of \cref{thm:QG}}
\label{sec:App:proofTheoremQG}

In order to  prove \cref{thm:QG} we first need to prove several technical lemmas.

\begin{lemma} \label{lemma:spectrahedronLemmaQuadraticGrowth}
Let $\Z\in\mathbb{H}^{n}$ such that $\Z\succeq0$, $\trace(\Z)= N$, $\rank(\Z)=r$, and write its eigendecomposition as $\Z=\U_r\S_r\U_r^{\tc}$. Let $\M\in\mathbb{H}^{n}$ be a matrix such that the columns of $\U_r$ are the eigenvectors corresponding to the $r$ smallest eignevalues of $\M$, and assume its eigenvalues satisfy that $\lambda_n(\M)=\cdots=\lambda_{n-r+1}(\M)<\lambda_{n-r}(\M) \le \cdots\le\lambda_1(\M)$. Denote $\delta:=\lambda_{n-r}(\M)-\lambda_{n}(\M)>0$ and $\P:=\U_r\U_r^{\tc}$. Then, for any $\X\in\mathbb{H}^{n}$ such that $\X\succeq0$ and $\trace(\X)= N$, it holds that
\begin{align*}
\langle\X - \Z, \M \rangle \ge \frac{\delta}{2 N}\left\Vert\X-\frac{ N\P\X\P}{\trace(\P\X\P)}\right\Vert_F^2.
\end{align*}
\end{lemma}

\begin{proof}
Denote the eigendecomposition of $\M$ as $\M=\left[\begin{array}{cc}\U_{\perp} & \U_r\end{array}\right]\left[\begin{array}{cc}\boldsymbol{\Lambda}_{\perp} & \mathbf{0} \\ \mathbf{0} & \boldsymbol{\Lambda}_r \end{array}\right]\left[\begin{array}{c}\U_{\perp}^{\tc} \\ \U_r^{\tc}\end{array}\right]=\U\boldsymbol{\Lambda}\U^{\tc}=\sum_{i=1}^{n}\lambda_i\u_i\v_i^{\tc}$. Then,
\begin{align} \label{eq:inProof882}
\langle\Z, \M \rangle & = \langle\U_r\S_r\U_r^{\tc}, \sum_{i=1}^{n}\lambda_i\u_i\u_i^{\tc} \rangle = \langle\U_r\S_r\U_r^{\tc}, \U_r\boldsymbol{\Lambda}_r\U_r^{\tc} \rangle \nonumber
\\ & = \langle\S_r,\boldsymbol{\Lambda}_r\rangle = \lambda_n\trace(\S_r) = \lambda_n N.
\end{align}

In addition, since $\X\succeq0$ it holds that
\begin{align} \label{eq:inProof883}
\langle\X, \M \rangle & = \langle\X, \sum_{i=1}^{n}\lambda_i\u_i\u_i^{\tc} \rangle = \lambda_n\sum_{i=n-r+1}^{n}\u_i^{\tc}\X\u_i + \sum_{i=1}^{n-r}\lambda_i\u_i^{\tc}\X\u_i \nonumber
\\ &  = (\lambda_n-\lambda_{n-r})\sum_{i=n-r+1}^{n}\u_i^{\tc}\X\u_i + \left(\sum_{i=1}^{n-r}\lambda_i\u_i^{\tc}\X\u_i +\lambda_{n-r}\sum_{i=n-r+1}^{n}\u_i^{\tc}\X\u_i\right) \nonumber
\\ &  \ge (\lambda_n-\lambda_{n-r})\sum_{i=n-r+1}^{n}\u_i^{\tc}\X\u_i + \lambda_{n-r}\sum_{i=1}^{n}\u_i^{\tc}\X\u_i \nonumber
\\ &  = (\lambda_n-\lambda_{n-r})\sum_{i=n-r+1}^{n}\u_i^{\tc}\X\u_i + \lambda_{n-r}\trace(\X).
\end{align}

Subtracting \eqref{eq:inProof882} from \eqref{eq:inProof883} we obtain that
\begin{align} \label{ineq:inProof779}
\langle\X - \Z, \M \rangle \ge (\lambda_{n-r}-\lambda_{n})( N-\sum_{i=n-r+1}^{n}\u_i^{\tc}\X\u_i).
\end{align}

We now will upper bound the term $\left\Vert\X-\frac{ N\P\X\P}{\trace(\P\X\P)}\right\Vert_F^2$. Since $\left\langle \X-\P\X\P, \P\X\P-\frac{ N\P\X\P}{\trace(\P\X\P)}\right\rangle=0$, it follows that
\begin{align} \label{eq:inProof881}
\left\Vert\X-\frac{ N\P\X\P}{\trace(\P\X\P)}\right\Vert_F^2 & = \Vert\X-\P\X\P\Vert_F^2 + \left\Vert\P\X\P-\frac{ N\P\X\P}{\trace(\P\X\P)}\right\Vert_F^2.
\end{align}

We first note that
\begin{align*}
\trace(\P\X\P) \underset{(a)}{=}  \trace(\P\X) \underset{(b)}{\le} \Vert\P\Vert_2\trace(\X) \le  N,
\end{align*}
where (a) follows since $\P^2=\P$, and (b) follows from H\"{o}lder's inequality.

Therefore, the matrix $\frac{ N\P\X\P}{\trace(\P\X\P)}-\P\X\P$ is positive semidefinite, and hence,
the second term in the RHS of \eqref{eq:inProof881} can be bounded as
\begin{align} \label{ineq:inProof880}
\left\Vert\frac{ N\P\X\P}{\trace(\P\X\P)}-\P\X\P\right\Vert_F^2 & \le \left(\trace\left(\frac{ N\P\X\P}{\trace(\P\X\P)}-\P\X\P\right)\right)^2
 = ( N-\trace(\P\X\P))^2.
\end{align}

To bound the first term in the RHS of \eqref{eq:inProof881}, we first note that
\begin{align*}
\U^{\tc}\X\U = \left[\begin{array}{c}\U_{\perp}^{\tc} \\ \U_r^{\tc}\end{array}\right]\X
\left[\begin{array}{cc}\U_{\perp} & \U_r\end{array}\right] 
= \left[\begin{array}{cc}\U_{\perp}^{\tc}\X\U_{\perp} & \U_{\perp}^{\tc}\X\U_r
\\ \U_r^{\tc}\X\U_{\perp} & \U_r^{\tc}\X\U_r \end{array}\right],
\end{align*}
and since $\U_{\perp}^{\tc}\P=\P\U_{\perp}=0$ and $\P\U_r=\U_r$, it holds that
\begin{align} \label{eq:inProof772}
\U^{\tc}(\X-\P\X\P)\U = \left[\begin{array}{c}\U_{\perp}^{\tc} \\ \U_r^{\tc}\end{array}\right][\X-\P\X\P]
\left[\begin{array}{cc}\U_{\perp} & \U_r\end{array}\right] 
= \left[\begin{array}{cc}\U_{\perp}^{\tc}\X\U_{\perp} & \U_{\perp}^{\tc}\X\U_r
\\ \U_r^{\tc}\X\U_{\perp} & 0 \end{array}\right].
\end{align}
Therefore, 
\begin{align} \label{eq:inProof773}
\trace(\U_r^{\tc}\X\U_r) = \trace(\U^{\tc}\X\U) - \trace(\U^{\tc}(\X-\P\X\P)\U) = \trace(\X) - \trace(\X-\P\X\P),
\end{align}
where the last equality follows since $\U$ is orthogonal. 

In addition, also using the orthogonality of $\U$ and the structure of the matrix in \eqref{eq:inProof772}, it holds that
\begin{align} \label{eq:inProof777}
\trace(\X-\P\X\P) = \trace(\U^{\tc}(\X-\P\X\P)\U) = \trace(\U_{\perp}^{\tc}\X\U_{\perp}),
\end{align}
and
\begin{align} \label{eq:inProof776}
\Vert\X-\P\X\P\Vert_F^2 = \Vert\U^{\tc}(\X-\P\X\P)\U\Vert_F^2 = \Vert\U_{\perp}^{\tc}\X\U_{\perp}\Vert_F^2 + 2\Vert\U_{\perp}^{\tc}\X\U_r\Vert_F^2.
\end{align} 

It can be seen that $\U_{\perp}^{\tc}\X\U_{\perp}$ is positive semidefinite because $\X$ is positive semidefinite, and therefore,
\begin{align} \label{ineq:inProof775}
\Vert\U_{\perp}^{\tc}\X\U_{\perp}\Vert_F^2 \le (\trace(\U_{\perp}^{\tc}\X\U_{\perp}))^2 = (\trace(\X-\P\X\P))^2,
\end{align}
where the equality follows from \eqref{eq:inProof777}. 

In addition, since $\X$ is positive semidefinite, it can be written as $\X=\X^{1/2}\X^{1/2}$, and so
\begin{align} \label{ineq:inProof774}
\Vert\U_{\perp}^{\tc}\X\U_r\Vert_F^2 & = \Vert\U_{\perp}^{\tc}\X^{1/2}\X^{1/2}\U_r\Vert_F^2 \le \Vert\U_{\perp}^{\tc}\X^{1/2}\Vert_2^2\Vert\X^{1/2}\U_r\Vert_F^2 \nonumber
\\ & = \lambda_1(\U_{\perp}^{\tc}\X\U_{\perp})\trace(\U_r^{\tc}\X\U_r) \le \trace(\U_{\perp}^{\tc}\X\U_{\perp})\trace(\U_r^{\tc}\X\U_r) \nonumber
\\ & = \trace(\X-\P\X\P)(\trace(\X) - \trace(\X-\P\X\P)),
\end{align}
where the last equality follows from \eqref{eq:inProof777} and \eqref{eq:inProof773}. 

Plugging \eqref{ineq:inProof775} and \eqref{ineq:inProof774} into the RHS of \eqref{eq:inProof776} we obtain that the first term in the RHS of \eqref{eq:inProof881} can be bounded as
\begin{align} \label{ineq:inProof771}
\Vert\X-\P\X\P\Vert_F^2 & \le (\trace(\X-\P\X\P))^2+2\trace(\X-\P\X\P)(\trace(\X) - \trace(\X-\P\X\P)) \nonumber
\\ & = 2 N\trace(\X-\P\X\P)-(\trace(\X-\P\X\P))^2.
\end{align} 
Plugging \eqref{ineq:inProof880} and \eqref{ineq:inProof771} into the RHS of \eqref{eq:inProof881} we obtain that
\begin{align} \label{ineq:inProof770}
\left\Vert\X-\frac{ N\P\X\P}{\trace(\P\X\P)}\right\Vert_F^2 & \le 2 N\trace(\X-\P\X\P) = 2 N \left( N - \sum_{i=n-r+1}^{n}\u_i^{\tc}\X\u_i \right).
\end{align}
Finally, combining \eqref{ineq:inProof779} and \eqref{ineq:inProof770} we obtain that
\begin{align*}
\frac{\lambda_{n-r}-\lambda_n}{2 N}\left\Vert\X-\frac{ N\P\X\P}{\trace(\P\X\P)}\right\Vert_F^2 & \le \langle\X-\Z,\M\rangle.
\end{align*}
\end{proof}

For the rest of the analysis required for the proof of \cref{thm:QG} we need to introduce some notation. A summary of all the relevant notation can be found in \cref{table:notationsProofQG}.

Denote $\mcM\subset\complex^{n_1 N\times n_2 N}$ to be the subset of $\complex^{n_1 N\times n_2 N}$ which contains all the block diagonal matrices $\Y$ such that $\bdiag^{-1}(\Y)$ satisfies the conjugate-complex symmetry condition in \eqref{eq:conjugateComplexSymmetryCondition}, where $\bdiag^{-1}(\cdot)$ is the inverse operator of $\bdiag(\cdot)$ such that for any $\overbar{\mX}\in\complex^{n_1\times\cdots\times n_d}$, $\bdiag^{-1}(\bdiag(\overbar{\mX}))=\overbar{\mX}$. 

For any $\barX\in\lbrace\X\in\mcM\ \vert \ \Vert\X\Vert_*\le N\rbrace$ with SVD decomposition $\barX=\barU\barS\barV^{\tc}$, we denote its dilation as
\begin{align*}
\barX^{\sharp}: & = \frac{1}{2}\left[
\begin{array}{cc}
\barU(\barS+\xi)\barU^{\tc} & \barX
\\ \barX^{\tc} & \barV(\barS+\xi)\barV^{\tc}
\end{array}
\right] 
\\ & = \frac{1}{2}\left[
\begin{array}{c} \barU \\ \barV \end{array}
\right] 
\barS \left[ \begin{array}{cc} \barU^{\tc} & \barV^{\tc} \end{array}\right]
+\frac{\xi}{2}\left[
\begin{array}{cc}
\barU\barU^{\tc} & 0
\\ 0 & \barV\barV^{\tc}
\end{array}
\right] \in \mathbb{H}^{(n_1+n_2) N},
\end{align*}
where $\xi\ge0$ is chosen so that $\trace(\barX^{\sharp})= N$.

In addition, for any  for any $\widetilde{\X}\in\mathbb{H}^{(n_1+n_2) N}$ with block structure 
\begin{align*} 
\widetilde{\X}=\frac{1}{2}\left[\begin{array}{cc} \X_1 & \X
\\ \X^{\tc} & \X_2 \end{array}\right]
\in\mathbb{H}^{(n_1+n_2) N},
\end{align*}
where $\X_1\in\mathbb{H}^{n_1 N}$, $\X_2\in\mathbb{H}^{n_2 N}$, we denote the off diagonal block as $\widetilde{\X}_\flat := \X\in\complex^{n_1 N\times n_2 N}$.
Also, using the mapping of a complex matrix in the complex spectrahedron into a real-valued matrix in a larger real-valued spectrahedron suggested in  \cite{Complex2BigRealMatrix}, denote
\begin{align*}
\widetilde{\X}^{\lozenge}:=\left[\begin{array}{cc} \real(\widetilde{\X}) & -\im(\widetilde{\X})
\\ \im(\widetilde{\X}) & \real(\widetilde{\X}) \end{array}\right]\in\mathbb{S}^{2(n_1+n_2) N}.
\end{align*}
Since $\real(\widetilde{\X})$ is symmetric and $\im(\widetilde{\X})$ is antisymmetric, $\widetilde{\X}^{\lozenge}$ is symmetric.
Finally, for any $\X\in\mathbb{S}^{2(n_1+n_2) N}$ with block structure
\begin{align*}
\X=\left[\begin{array}{cc} \A & -\B
\\ \B & \A \end{array}\right]\in\mathbb{S}^{2(n_1+n_2) N},
\end{align*}
where $\A\in\mathbb{S}^{(n_1+n_2) N}$ is symmetric and $\B\in\reals^{(n_1+n_2) N\times(n_1+n_2) N}$ is antisymmetric, denote $\X_{\triangledown}:=\A+i\B\in\mathbb{H}^{(n_1+n_2) N}$.

If the eigendecomposition of $\widetilde{\X}\in\mathbb{H}^{(n_1+n_2) N}$ is $\widetilde{\X}=\U\boldsymbol{\Lambda}\U^{\tc}$, then the eigendecomposition of $\widetilde{\X}^{\lozenge}$ can be written as 
\begin{align} \label{eq:inProof9944}
\widetilde{\X}^{\lozenge}=\left[\begin{array}{cc} \real(\U) & \im(\U)
\\ \im(\U) & -\real(\U) \end{array}\right]
\left[\begin{array}{cc} \boldsymbol{\Lambda} & \mathbf{0}
\\ \mathbf{0} & \boldsymbol{\Lambda} \end{array}\right]
\left[\begin{array}{cc} \real(\U) & \im(\U)
\\ \im(\U) & -\real(\U) \end{array}\right]^{\top}.
\end{align}
This holds since for every eigenvector $\u_i$ of $\widetilde{\X}$ corresponding to an eigenvalue $\lambda$, both $\left[\real(\u_i)^{\top},\ \im(\u_i)^{\top}\right]^{\top}$ and $\left[\im(\u_i)^{\top},\ -\real(\u_i)^{\top}\right]^{\top}$ are eigenvectors of $\widetilde{\X}^{\lozenge}$ corresponding to an eigenvalue $\lambda$. It can be seen that the matrix $\left[\begin{array}{cc} \real(\U) & \im(\U)
\\ \im(\U) & -\real(\U) \end{array}\right]$ is orthogonal.

\begin{table*}[!htb]\renewcommand{\arraystretch}{1.3}
{\footnotesize
\begin{center}
  \begin{tabular}{| l | l | } \hline
\textrm{Notation} & \textrm{Description} \\ \hline
$\mX\in\reals^{n_1\times\cdots\times n_d}$ & real valued order-d tensor \\
$\overbar{\mX}=\fft(\mX)$ &  $\fft(\mX)=\fft_d(\cdots(\fft_4(\fft_3(\mX))))$ \\
$\ifft(\overbar{\mX})$ & $\ifft(\overbar{\mX})=\ifft_d(\cdots(\ifft_3(\overbar{\mX})))$ \\
$\barX\in\mcM$ & $\barX = \bdiag(\fft(\mX))$ \\
$\bdiag^{-1}(\barX)$ & inverse of $\bdiag(\cdot)$ such that $\bdiag^{-1}(\bdiag(\overbar{\mX}))=\overbar{\mX}$ \\
$\mcM\subset\complex^{n_1 N\times n_2 N}$ & all block diagonal matrices $\Y$ such that $\bdiag^{-1}(\Y)$ satisfies \eqref{eq:conjugateComplexSymmetryCondition} \\
$\barX^{\sharp}\in\mathbb{H}^{(n1+n2) N}$ & $\barX^{\sharp} = \frac{1}{2}\left[
\begin{array}{cc}
\barU(\barS+\xi)\barU^{\tc} & \barX
\\ \barX^{\tc} & \barV(\barS+\xi)\barV^{\tc}
\end{array}
\right]$, $\barX=\barU\barS\barV^{\tc}$, $\xi\ge0$ \\
$\widetilde{\X}\in\mathbb{H}^{(n_1+n_2) N}$ & $\widetilde{\X}=\frac{1}{2}\left[\begin{array}{cc} \X_1 & \X
\\ \X^{\tc} & \X_2 \end{array}\right]$, $\X\in\complex^{n_1 N\times n_2 N}$, $\X_1\in\mathbb{H}^{n_1 N}$, $\X_2\in\mathbb{H}^{n_2 N}$\\
$\widetilde{\X}_\flat\in\complex^{n_1 N\times n_2 N}$ & $\widetilde{\X}_\flat = \X$ \\
$\widetilde{\X}^{\lozenge}\in\mathbb{S}^{2(n_1+n_2) N}$ & $\widetilde{\X}^{\lozenge}=\left[\begin{array}{cc} \real(\widetilde{\X}) & -\im(\widetilde{\X})
\\ \im(\widetilde{\X}) & \real(\widetilde{\X}) \end{array}\right]$ \\
$\X\in\mathbb{S}^{2(n_1+n_2) N}$ & $\X=\left[\begin{array}{cc} \A & -\B
\\ \B & \A \end{array}\right]$, $\A\in\mathbb{S}^{(n_1+n_2) N}$, $\B\in\reals^{(n_1+n_2) N\times(n_1+n_2) N}$ antisymmetric \\
$\X_{\triangledown}\in\mathbb{H}^{(n_1+n_2) N}$ & $\X_{\triangledown}=\A+i\B$ \\
${{\X}_{\triangledown}}_{\flat}\in\complex^{n_1 N\times n_2 N}$ & ${\X_{\triangledown}}_{\flat}=(\X_{\triangledown})_{\flat}$ \\
${\barX^{\sharp}}^{\lozenge}\in\mathbb{S}^{2(n_1+n_2) N}$ & ${\barX^{\sharp}}^{\lozenge}=\left(\barX^{\sharp}\right)^{\lozenge}$ \\ \hline
\end{tabular}
  \caption{Notations for the proof of \cref{thm:QG}.}\label{table:notationsProofQG}
\end{center}
}
\vskip -0.2in
\end{table*}\renewcommand{\arraystretch}{1}

Next we define several linear operators that will be used. 

We denote the linear operator $\mcQ$ such that for any matrix $\X\in\mathbb{S}^{2(n_1+n_2) N}$, the equation $\mcQ(\X) = \mathbf{0}$ corresponds to the set of linear constraints that ensures ${\X_{\triangledown}}_{\flat}=(\A+i\B)_{\flat}\in\mcM$ by requiring all entries off the main diagonal blocks of $\A_{\flat}$ and $\B_{\flat}$ to be zero, and the constraints ensuring that the entries on the diagonal blocks satisfy that for all $i_3\in[n_3],\ldots,i_d\in[n_d]$, it holds that ${\A_{\flat}}^{(i_3,\ldots,i_d)}={\A_{\flat}}^{(i_3',\ldots,i_d')}$ and  ${\B_{\flat}}^{(i_3,\ldots,i_d)}=-{\B_{\flat}}^{(i_3',\ldots,i_d')}$, where $i_j'$ is defined as in \eqref{def:i_j_tag}.

In addition, we denote by $\mathcal{E}(\X)$ the linear operator upon the matrices in $\mathbb{S}^{2(n_1+n_2) N}$ such that the equation $\mathcal{E}(\X)=\textbf{0}$ corresponds to all the equalities
\begin{align*}
 & \left\langle\left[\begin{array}{cc}  \mathbf{E}_{ij} & \textnormal{\textbf{0}} \nonumber
\\ \textnormal{\textbf{0}} & -\mathbf{E}_{ij} \end{array}\right],\X\right\rangle=0,\ i,j\in[(n_1+n_2) N],\ i<j,
\\ & \left\langle\left[\begin{array}{cc}  \textnormal{\textbf{0}} & \mathbf{E}_{ij} 
\\  \mathbf{E}_{ij} & \textnormal{\textbf{0}} \end{array}\right],\X\right\rangle=0,\ i,j\in[(n_1+n_2) N],\ i<j.
\end{align*}
Here $\mathbf{E}_{ij}=\e_i\e_j^{\top}+\e_j\e_i^{\top}$, where $\e_i$ is the $i$th unit vector. Adding this as a constraint ensures that the solution is of the form $\X:=\left[\begin{array}{cc} \A & -\B
\\ \B & \A \end{array}\right]\in\mathbb{S}^{2(n_1+n_2)N}$ for some symmetric $\A\in\mathbb{S}^{(n_1+n_2)N}$ and antisymmetric $\B\in\reals^{(n_1+n_2)N\times(n_1+n_2)N}$.

Using these notation we will denote the operator $\mcP$ upon the vector space $\mathbb{S}^{2(n_1+n_2) N}$:
$$\mcP(\X):=[\trace(\X),\ \mcQ(\X)^{\top},\ \mathcal{E}(\X)^{\top},\ \mcB(\X)^{\top}]^{\top},$$
where $\mcB(\X) := \mA(\ifft(\bdiag^{-1}({\X_{\triangledown}}_{\flat}))$ and $\mA:\reals^{n_1\times\cdots\times n_d}\rightarrow\reals^m$ is some linear map.

The optimization problem we are interested in solving is of the form:
\begin{align} \label{eq:originalModel11}
\min_{\Vert\mX\Vert_*\le1}f(\mX):=g(\mA(\mX))+\langle\mC,\mX\rangle,
\end{align}
where $g:\reals^m\rightarrow\reals$ is convex. 

Let $\mX^*$ to be the optimal solution to Problem \eqref{eq:originalModel11} such that $\rank_{\textnormal{a}}(\mX^*)=r/N$ and denote $\overbar{\X^*}=\overbar{\U^*}_r\overbar{\S^*}_r{\overbar{\V^*}_r}^{\tc}$ to be the SVD of $\overbar{\X^*}=\bdiag(\overbar{\mX^*})$. Denote 
\begin{align} \label{eq:defVrtilde}
\widetilde{\V}_r:=\frac{1}{\sqrt{2}}\left[\begin{array}{c}\overbar{\U^*}_r \\ -\overbar{\V^*}_r \end{array}\right].
\end{align}
Using these notation we also define the mapping upon the vector space $\mathbb{S}^{2r}$:
\begin{align*}
\mcP_{V}(\S)=\mcP\left(\left[\begin{array}{cc} \real(\widetilde{\V}_r) & \im(\widetilde{\V}_r)
\\ \im(\widetilde{\V}_r) & -\real(\widetilde{\V}_r) \end{array}\right]\S\left[\begin{array}{cc} \real(\widetilde{\V}_r) & \im(\widetilde{\V}_r)
\\ \im(\widetilde{\V}_r) & -\real(\widetilde{\V}_r) \end{array}\right]^{\top}\right).
\end{align*}

In the following lemma we show that the tensor optimization problem over the unit TNN  ball as in Problem \eqref{mainModel}, is equivalent to a matrix optimization problem over a certain spectrahedron with additional linear constraints. 
\begin{lemma} \label{lemma:QGequivalenceOfProblems}
Consider the following two optimization problems: 
\begin{align} \label{eq:originalModel}
\min_{\Vert\mX\Vert_*\le1}f(\mX):=g(\mA(\mX))+\langle\mC,\mX\rangle
\end{align}
where $g:\reals^m\rightarrow\reals$ is convex and $\mA:\reals^{n_1\times\cdots\times n_d}\rightarrow\reals^m$ is a linear map,
and
\begin{align} \label{eq:spectrahedronModelUnComplexed}
\min_{\X\in\mathbb{S}^{2(n_1+n_2) N}}\ & {h(\X):=g(\mcB(\X))+\frac{1}{2 N}\langle\widetilde{\C}^{\lozenge},\X\rangle} 
\\ \textrm{s.t.}\ & \trace(\X)=2 N \nonumber
\\ & \X\succeq0 \nonumber
\\ & \mcQ(\X) = \mathbf{0} \nonumber
\\ & \mathcal{E}(\X) = \mathbf{0} \nonumber,
\end{align}
where $\mcB:\mathbb{S}^{2(n_1+n_2) N}\rightarrow\reals^{m}$ is defined as $\mcB(\X) := \mA(\ifft(\bdiag^{-1}({\X_{\triangledown}}_{\flat}))$ and $\widetilde{\C}^{\lozenge}=\left[\begin{array}{cc} \textnormal{\textbf{0}} & \barC
\\ \barC^{\tc} & \textnormal{\textbf{0}} \end{array}\right]^{\lozenge}$.

If \eqref{eq:originalModel} has a unique solution $\mX^*$ which also satisfies strict complementarity then also \eqref{eq:spectrahedronModelUnComplexed} has a unique solution ${{\overbar{\X^*}}^{\sharp}}^{\lozenge}$.
\end{lemma}

\begin{proof}
For any $\widetilde{\X},\widetilde{\Y}\in\mathbb{H}^{(n_1+n_2) N}$, since $\real(\widetilde{\X}),\real(\widetilde{\Y})$ are symmetric and $\im(\widetilde{\X}),\im(\widetilde{\Y})$ are antisymmetric, it holds that
\begin{align} \label{eq:InProofQGinnerProductEquivalence}
\langle\widetilde{\X}^{\lozenge},\widetilde{\Y}^{\lozenge}\rangle & = 
\trace\left(\left[\begin{array}{cc} \real(\widetilde{\X}) & -\im(\widetilde{\X})
\\ \im(\widetilde{\X}) & \real(\widetilde{\X}) \end{array}\right]\left[\begin{array}{cc} \real(\widetilde{\Y}) & -\im(\widetilde{\Y})
\\ \im(\widetilde{\Y}) & \real(\widetilde{\Y}) \end{array}\right]\right) \nonumber
\\ & =2\trace(\real(\widetilde{\X})\real(\widetilde{\Y})-\im(\widetilde{\X})\im(\widetilde{\Y}))=2\langle\widetilde{\X},\widetilde{\Y}\rangle.
\end{align}
Therefore, for every $\mX\in\reals^{n_1\times\cdots\times n_d}$ it holds that
\begin{align*}
\langle\mC,\mX\rangle \underset{(a)}{=} \frac{1}{ N}\langle\barC,\barX\rangle = \frac{1}{ N}\left\langle\left[\begin{array}{cc} \textbf{0} & \barC
\\ \barC^{\tc} & \textbf{0} \end{array}\right],\barX^{\sharp}\right\rangle \underset{(b)}{=} \frac{1}{2 N}\left\langle\left[\begin{array}{cc} \textbf{0} & \barC
\\ \barC^{\tc} & \textbf{0} \end{array}\right]^{\lozenge},{\barX^{\sharp}}^{\lozenge}\right\rangle,
\end{align*}
where (a) follows from \cref{lemma:orderPinnerProductEquivalence}, and (b) follows from \eqref{eq:InProofQGinnerProductEquivalence}. This, together with the fact that $\mcB({\barX^{\sharp}}^{\lozenge}) = \mA(\mX)$ implies that 
$h({\barX^{\sharp}}^{\lozenge})=f(\mX)$.

Since $\mX^*$ satisfies strict complementarity then $\nabla{}f(\mX^*)\not=0$, and therefore, $(1/ N)\Vert\overbar{\X^*}\Vert_*=\Vert\mX\Vert_*=1$. This implies that $\trace({\overbar{\X^*}}^{\sharp})= N$. By the definition of  ${\overbar{\X^*}}^{\sharp}$ it can be seen that ${\overbar{\X^*}}^{\sharp}\succeq0$ and $({\overbar{\X^*}}^{\sharp})_{\flat}=\overbar{\X^*}\in\mcM$. Since ${\overbar{\X^*}}^{\sharp}$ is a hermitian matrix, its diagonal is necessarily real and so $\trace({\overbar{\X^*}}^{\sharp})=\trace(\real({\overbar{\X^*}}^{\sharp}))= N$, which implies that $\trace({{\overbar{\X^*}}^{\sharp}}^{\lozenge})=2 N$. According to section 3 in \cite{Complex2BigRealMatrix}, ${\overbar{\X^*}}^{\sharp}\succeq0$ if and only if ${{\overbar{\X^*}}^{\sharp}}^{\lozenge}\succeq0$. The other constraints hold trivially from the construction. Therefore, ${{\overbar{\X^*}}^{\sharp}}^{\lozenge}$ is a feasible solution to \eqref{eq:spectrahedronModelUnComplexed}.

We next show that ${{\overbar{\X^*}}^{\sharp}}^{\lozenge}$ is an optimal solution to \eqref{eq:spectrahedronModelUnComplexed}. Let $\X\in\mathbb{S}^{2(n_1+n_2) N}$ such that $\trace(\X)=2 N$, $\X\succeq0$, $\mcQ(\X) = \mathbf{0}$, and for which $\mathcal{E}(\X)=\mathbf{0}$. Then, it can be written as $\X=\left[\begin{array}{cc} \A & -\B
\\ \B & \A \end{array}\right]\in\mathbb{S}^{2(n_1+n_2) N}$ for some symmetric $\A\in\mathbb{S}^{(n_1+n_2) N}$ and antisymmetric $\B\in\reals^{(n_1+n_2) N\times(n_1+n_2) N}$, where $\A+i\B\succeq0$, $\trace(\A)= N$, and $(\A+i\B)_{\flat}\in\mcM$. Therefore, $\A+i\B$ can be written as 
$\A+i\B=\frac{1}{2}\left[\begin{array}{cc} \X_1 & (\A+i\B)_{\flat}
\\ {(\A+i\B)}_{\flat}^{\tc} & \X_2 \end{array}\right]\in\mathbb{H}^{(n_1+n_2) N}$, for some $\X_1\in\mathbb{H}^{n_1 N}$, $\X_2\in\mathbb{H}^{n_2 N}$ such that $\trace(\X_1)+\trace(\X_2)=2 N$. Then, by Lemma 1 in \cite{NuclearNorm2Spectrahedron}\footnote{\cite{NuclearNorm2Spectrahedron}  states the lemma for real matrices however, the proof holds also for the complex case by replacing symmetric matrices with hermitian matrices.}, it follows that $\Vert(\A+i\B)_{\flat}\Vert_*\le N$, and hence, $\ifft(\bdiag^{-1}((\A+i\B)_{\flat}))\in\reals^{n_1\times n_d}$ and $\Vert\ifft(\bdiag^{-1}((\A+i\B)_{\flat}))\Vert_*\le1$, and therefore, $\ifft(\bdiag^{-1}((\A+i\B)_{\flat}))$ is a feasible solution to \eqref{eq:originalModel}. Therefore, since $\mX^*$ is optimal to \eqref{eq:originalModel}, it follows that
\begin{align*}
h(\widetilde{\X}^{\lozenge})=f(\ifft(\bdiag^{-1}((\A+i\B)_{\flat}))) \ge f(\mX^*)=h({{\overbar{\X^*}}^{\sharp}}^{\lozenge}),
\end{align*}
which implies that ${{\overbar{\X^*}}^{\sharp}}^{\lozenge}$ is an optimal solution to  \eqref{eq:spectrahedronModelUnComplexed}.

We will now show that ${{\overbar{\X^*}}^{\sharp}}^{\lozenge}$ is a unique solution. Assume there exists a different optimal solution $\Y^*$. Then, it can be written as $\Y^*=\left[\begin{array}{cc} \A & -\B
\\ \B & \A \end{array}\right]\succeq0$
for some $\A\in\mathbb{S}^{(n_1+n_2) N}$ and some anti-symmetric $\B\in\reals^{(n_1+n_2) N\times(n_1+n_2) N}$, and it satisfies that $\trace(\Y^*)=2 N$, $\Y^*\succeq0$, $\mcQ(\Y^*) = \mathbf{0}$, $\mathcal{E}(\Y^*)=\mathbf{0}$. Since both ${{\overbar{\X^*}}^{\sharp}}^{\lozenge}$ and $\Y^*$ are optimal solutions to \eqref{eq:spectrahedronModelUnComplexed}, it follows that
\begin{align*}
f(\mX^*)=h({{\overbar{\X^*}}^{\sharp}}^{\lozenge})=h(\Y^*)
=f(\ifft(\bdiag^{-1}((\A+i\B)_{\flat}))),
\end{align*}
which by the uniqueness of $\mX^*$ implies that
$\mX^*=\ifft(\bdiag^{-1}((\A+i\B)_{\flat}))$. 

Since the Fourier transform and the $\bdiag(\cdot)$ operator are invertible, it follows that $\overbar{\X^*}=\bdiag(\overbar{\mX^*})=(\A+i\B)_{\flat}$. Now, since $\Vert\overbar{\X^*}\Vert_*= N$, by Lemma 3 in \cite{Simplicity} it follows that ${\overbar{\X^*}}^{\sharp}$ is the unique positive semidefinite matrix of the form $\frac{1}{2}\left[\begin{array}{cc} \X_1 & \widetilde{\X^*}_{\flat}
\\ \widetilde{\X^*}_{\flat}^{\tc} & \X_2 \end{array}\right]$ for some $\X_1\in\mathbb{H}^{n_1 N}$, $\X_2\in\mathbb{H}^{n_2 N}$ such that $\trace(\X_1)+\trace(\X_2)=2 N$. Therefore, ${\overbar{\X^*}}^{\sharp}=\A+i\B$, and so, $\A=\real({\overbar{\X^*}}^{\sharp})$ and $\B=\im({\overbar{\X^*}}^{\sharp})$, which implies that $\Y^*={{\overbar{\X^*}}^{\sharp}}^{\lozenge}$. 

\end{proof}

%

The following lemma bounds the distance, under the application of $\mP$, between ${{\overbar{\X^*}}^{\sharp}}^{\lozenge}$, which is the mapping of the optimal solution $\mX^*$ to a real-valued matrix in a spectrahedron, and a matrix whose eigenvectors are related to the eigenvectors as ${{\overbar{\X^*}}^{\sharp}}^{\lozenge}$.

\begin{lemma}
\label{lemma:technicalLemma3QG}
Let $\mX^*$ be a unique optimal solution to Problem \eqref{eq:originalModel11} such that $\rank_{\textnormal{a}}(\mX)=r/N$, and assume $g$ is $\alpha$-strongly convex.
Denote $\overbar{\X^*}=\overbar{\U^*}_r\overbar{\S^*}_r{\overbar{\V^*}_r}^{\tc}$ to be the SVD of $\overbar{\X^*}=\bdiag(\overbar{\mX^*})$, and denote 
\begin{align*}
\widetilde{\V}_r:=\frac{1}{\sqrt{2}}\left[\begin{array}{c}\overbar{\U^*}_r \\ -\overbar{\V^*}_r \end{array}\right].
\end{align*}
Then, for any $\S\in\mathbb{S}^{2r}$ it holds that
\begin{align*} 
&\left\Vert\left[\begin{array}{cc} \real(\widetilde{\V}_r) & \im(\widetilde{\V}_r)
\\ \im(\widetilde{\V}_r) & -\real(\widetilde{\V}_r) \end{array}\right]\S\left[\begin{array}{cc} \real(\widetilde{\V}_r) & \im(\widetilde{\V}_r)
\\ \im(\widetilde{\V}_r) & -\real(\widetilde{\V}_r) \end{array}\right]^{\top}-{{\overbar{\X^*}}^{\sharp}}^{\lozenge}\right\Vert_F \nonumber
\\ & \le \frac{1}{\sigma_{\min}(\mcP_{V})}\left\Vert\mcP\left(\left[\begin{array}{cc} \real(\widetilde{\V}_r) & \im(\widetilde{\V}_r)
\\ \im(\widetilde{\V}_r) & -\real(\widetilde{\V}_r) \end{array}\right]\S\left[\begin{array}{cc} \real(\widetilde{\V}_r) & \im(\widetilde{\V}_r)
\\ \im(\widetilde{\V}_r) & -\real(\widetilde{\V}_r) \end{array}\right]^{\top}\right)-\mcP({{\overbar{\X^*}}^{\sharp}}^{\lozenge})\right\Vert_2,
\end{align*}
where $\sigma_{\min}(\mcP_{V}):= \min_{\Vert\S\Vert_F=1}\Vert\mcP_{V}(\S)\Vert_2>0$.
\end{lemma}

\begin{proof}
Invoking \cref{lemma:QGequivalenceOfProblems}, we have that since $\mX^*$ is a unique optimal solution to Problem \eqref{eq:originalModel}, then
${{\overbar{\X^*}}^{\sharp}}^{\lozenge}$ is the unique optimal solution of  \eqref{eq:spectrahedronModelUnComplexed}. Therefore, by the KKT conditions for Problem \eqref{eq:spectrahedronModelUnComplexed}, it must hold that
\begin{align*}
& {\mcB}^{\top}\nabla{}g(\mcB({{\overbar{\X^*}}^{\sharp}}^{\lozenge}))+\frac{1}{2 N}\widetilde{\C}^{\lozenge}-\Z^*-s^*\I-{\v^*}^{\top}\mcQ^{\top}\mcQ({{\overbar{\X^*}}^{\sharp}}^{\lozenge})-{\w^*}^{\top}\mathcal{E}^{\top}\mathcal{E}({{\overbar{\X^*}}^{\sharp}}^{\lozenge})=0
\\ & \trace({{\overbar{\X^*}}^{\sharp}}^{\lozenge})=2 N
\\ & {{\overbar{\X^*}}^{\sharp}}^{\lozenge}\succeq0
\\ & \mcQ({{\overbar{\X^*}}^{\sharp}}^{\lozenge})=\mathbf{0}
\\ & \mathcal{E}({{\overbar{\X^*}}^{\sharp}}^{\lozenge})=\mathbf{0}
\\ & \langle\Z^*,{{\overbar{\X^*}}^{\sharp}}^{\lozenge}\rangle=0,
\end{align*}
where $\Z^*$ is some optimal dual solution. In addition, since $g$ is strongly convex, it follows that $\mcB(\X)$ is constant over the optimal set. We will denote this constant $\d:=\mcB({{\overbar{\X^*}}^{\sharp}}^{\lozenge})$.

Thus, if the following system holds for some $\X\in \mathbb{S}^{2(n_1+n_2) N}$: 
\begin{align} \label{system:linear}
& \mcP(\X)=\c \nonumber
\\ & \X\succeq0 \nonumber
\\ & \langle\Z^*,\X\rangle=0,
\end{align}
where $\mcP(\X):=[\trace(\X),\ \mcQ(\X)^{\top},\ \mathcal{E}(\X)^{\top},\ \mcB(\X)^{\top}]^{\top}$, and $\c:=[2 N,\ \mathbf{0},\ \mathbf{0},\ \d^{\top}]^{\top}$, then the KKT conditions hold, which implies that $\X$ is an optimal solution to Problem \eqref{eq:spectrahedronModelUnComplexed} since problem \eqref{eq:spectrahedronModelUnComplexed} is convex. 
Therefore, since the optimal solution to problem \eqref{eq:spectrahedronModelUnComplexed} is unique, system \eqref{system:linear} has a unique solution ${{\overbar{\X^*}}^{\sharp}}^{\lozenge}$.

Denote the SVD of $\overbar{\nabla{}f(\mX^*)}$ as $\overbar{\nabla{}f(\mX^*)}=\overbar{\U^*}\overbar{\S^*}{\overbar{\V^*}}^{\tc}$. Define
\begin{align*}
\widetilde{\boldsymbol{\nabla}}:=\left[\begin{array}{cc} \textbf{0} & \overbar{\nabla{}f(\mX^*)} \\ \overbar{\nabla{}f(\mX^*)}^{\tc} & \textbf{0} \end{array}\right]=\frac{1}{2}\left[\begin{array}{cc} \overbar{\U^*} & \overbar{\U^*} \\ \overbar{\V^*} & -\overbar{\V^*} \end{array}\right]\left[\begin{array}{cc} \overbar{\S^*} & \textbf{0} \\ \textbf{0} & -\overbar{\S^*} \end{array}\right]\left[\begin{array}{cc} \overbar{\U^*}^{\tc} & \overbar{\V^*}^{\tc} \\ \overbar{\U^*}^{\tc} & -\overbar{\V^*}^{\tc} \end{array}\right] ,
\end{align*}
and $r=\sum_{i_d=1}^{n_d}\cdots\sum_{i_3=1}^{n_3}\rank(\overbar{\X^*}^{(i_3,\ldots,i_d)})$. By \cref{lemma:svd_opt_grad_d_dim}, it follows that $\overbar{\S^*}(1,1)=\cdots=\overbar{\S^*}(r,r)$.
Therefore, the bottom $r+1$ eigenvalues of $\widetilde{\boldsymbol{\nabla}}$ are $-\sigma_1(\overbar{\nabla{}f(\mX^*)}),\ldots,-\sigma_{r+1}(\overbar{\nabla{}f(\mX^*)})$, and the matrix containing the eigenvectors corresponding to the bottom $r$ eigenvalues of $\widetilde{\boldsymbol{\nabla}}$ is 
\begin{align*}
\widetilde{\V}_r:=\frac{1}{\sqrt{2}}\left[\begin{array}{c}\overbar{\U^*}_r \\ -\overbar{\V^*}_r \end{array}\right] ,
\end{align*}
where $\overbar{\U^*}_r$ and $\overbar{\V^*}_r$ contain the $r$ singular vectors of $\overbar{\U^*}$ and $\overbar{\V^*}$ corresponding to the top $r$ singular values of $\overbar{\nabla{}f(\mX^*)}$, and $\overbar{\U^*}_r$ and $-\overbar{\V^*}_r$ are also the singular vectors of $\overbar{\X^*}$.

Since the SVD of $\overbar{\X^*}^{\sharp}$ can be written as $\overbar{\X^*}^{\sharp}=\frac{1}{2}\left[
\begin{array}{c} \overbar{\U^*}_r \\ -\overbar{\V^*}_r \end{array}
\right] 
\boldsymbol{\Lambda} \left[ \begin{array}{cc} \overbar{\U^*}_r^{\tc} & -\overbar{\V^*}_r^{\tc} \end{array}\right]=\widetilde{\V}_r\boldsymbol{\Lambda}\widetilde{\V}_r^{\tc}$, for some $\boldsymbol{\Lambda}\in\mathbb{S}^{r}$ such that $\boldsymbol{\Lambda}\succeq0$, by \eqref{eq:inProof9944} we know that the solution ${{\overbar{\X^*}}^{\sharp}}^{\lozenge}$ can be written in the form
\begin{align*}
{{\overbar{\X^*}}^{\sharp}}^{\lozenge} = \left[\begin{array}{cc} \real(\widetilde{\V}_r) & \im(\widetilde{\V}_r)
\\ \im(\widetilde{\V}_r) & -\real(\widetilde{\V}_r) \end{array}\right]
\left[\begin{array}{cc} \boldsymbol{\Lambda} & \mathbf{0}
\\ \mathbf{0} & \boldsymbol{\Lambda} \end{array}\right]
\left[\begin{array}{cc} \real(\widetilde{\V}_r) & \im(\widetilde{\V}_r)
\\ \im(\widetilde{\V}_r) & -\real(\widetilde{\V}_r) \end{array}\right]^{\top}.
\end{align*}

Consider now the mapping upon the vector space $\mathbb{S}^{2r}$:
\begin{align*}
\mcP_{V}(\S)=\mcP\left(\left[\begin{array}{cc} \real(\widetilde{\V}_r) & \im(\widetilde{\V}_r)
\\ \im(\widetilde{\V}_r) & -\real(\widetilde{\V}_r) \end{array}\right]\S\left[\begin{array}{cc} \real(\widetilde{\V}_r) & \im(\widetilde{\V}_r)
\\ \im(\widetilde{\V}_r) & -\real(\widetilde{\V}_r) \end{array}\right]^{\top}\right).
\end{align*}

Assume some $\S_0\in\mathbb{S}^{2r}$ satisfies that $\mcP_{V}(\S_0)=0$. Then, for a small enough $\alpha\ge0$ it holds that 
$$\left[\begin{array}{cc} \real(\widetilde{\V}_r) & \im(\widetilde{\V}_r)
\\ \im(\widetilde{\V}_r) & -\real(\widetilde{\V}_r) \end{array}\right]\left(\alpha\S_0+\left[\begin{array}{cc} \boldsymbol{\Lambda} & \mathbf{0}
\\ \mathbf{0} & \boldsymbol{\Lambda} \end{array}\right]\right)\left[\begin{array}{cc} \real(\widetilde{\V}_r) & \im(\widetilde{\V}_r)
\\ \im(\widetilde{\V}_r) & -\real(\widetilde{\V}_r) \end{array}\right]^{\top}$$ satisfies the system \eqref{system:linear}. Therefore, $\S_0=0$, and hence, the mapping $\mcP_{V}$ is injective.

Since $\mcP_{V}$ is injective, it follows that $\sigma_{\min}(\mcP_{V}) := \min_{\Vert\S\Vert_F=1}\Vert\mcP_{V}(\S)\Vert_2>0$. Therefore, for any $\S\in\mathbb{S}^{2r}$ it holds that
\begin{align*} 
&\left\Vert\left[\begin{array}{cc} \real(\widetilde{\V}_r) & \im(\widetilde{\V}_r)
\\ \im(\widetilde{\V}_r) & -\real(\widetilde{\V}_r) \end{array}\right]\S\left[\begin{array}{cc} \real(\widetilde{\V}_r) & \im(\widetilde{\V}_r)
\\ \im(\widetilde{\V}_r) & -\real(\widetilde{\V}_r) \end{array}\right]^{\top}-{{\overbar{\X^*}}^{\sharp}}^{\lozenge}\right\Vert_F \nonumber
\\ & \le \frac{1}{\sigma_{\min}(\mcP_{V})}\left\Vert\mcP\left(\left[\begin{array}{cc} \real(\widetilde{\V}_r) & \im(\widetilde{\V}_r)
\\ \im(\widetilde{\V}_r) & -\real(\widetilde{\V}_r) \end{array}\right]\S\left[\begin{array}{cc} \real(\widetilde{\V}_r) & \im(\widetilde{\V}_r)
\\ \im(\widetilde{\V}_r) & -\real(\widetilde{\V}_r) \end{array}\right]^{\top}\right)-\mcP({{\overbar{\X^*}}^{\sharp}}^{\lozenge})\right\Vert_2.
\end{align*}

\end{proof}

We now restate \cref{thm:QG} and then prove it.
\begin{theorem}[quadratic growth]
Let $f(\mX)=g(\mA(\mX))+\langle\mC,\mX\rangle$, where $g$ is $\alpha$-strongly convex and $\mA:\reals^{n_1\times\cdots\times n_d}\rightarrow\reals^m$ is a linear map. Assume there exist a unique optimal solution $\mX^*\in\lbrace\mX\in\reals^{n_1\times\cdots\times n_d}\ \vert\ \Vert\mX\Vert_*\le1\rbrace$ to Problem \eqref{mainModel}, and that it satisfies strict complementarity. Then, there exists a constant $\gamma>0$ such that for every $\mX\in\lbrace\mX\in\reals^{n_1\times\cdots\times n_d}\ \vert\ \Vert\mX\Vert_*\le1\rbrace$, it holds that
\begin{align*}
f(\mX)-f(\mX^*) \ge \gamma\Vert\mX-\mX^*\Vert_F^2.
\end{align*}
\end{theorem}

\begin{proof}

Let $\mX\in\lbrace\mX\in\reals^{n_1\times\cdots\times n_d}\ \vert\ \Vert\mX\Vert_*\le1\rbrace$, and denote $\barX^{\sharp}=\bdiag(\overbar{\mX})^{\sharp}$. 
Let $\widetilde{\V}_r$ be as defined in \eqref{eq:defVrtilde}.
Denote $\M:=\frac{ N\widetilde{\V}_r^{\tc}\barX^{\sharp}\widetilde{\V}_r}{\trace(\P\barX^{\sharp}\P)}\in\mathbb{H}^{r}$ and $\W:=\frac{ N\P\barX^{\sharp}\P}{\trace(\P\barX^{\sharp}\P)}=\widetilde{\V}_r\M\widetilde{\V}_r^{\tc}$, where $\P=\widetilde{\V}_r\widetilde{\V}_r^{\tc}$. It can be seen that
\begin{align*}
\W^{\lozenge}& = \left[\begin{array}{cc} \real(\W) & -\im(\W)
\\ \im(\W) & \real(\W) \end{array}\right]
\\ & =\left[\begin{array}{cc} \real(\widetilde{\V}_r) & \im(\widetilde{\V}_r)
\\ \im(\widetilde{\V}_r) & -\real(\widetilde{\V}_r) \end{array}\right]
\left[\begin{array}{cc} \real(\M) & \im(\M)
\\ -\im(\M) & \real(\M) \end{array}\right]
\left[\begin{array}{cc} \real(\widetilde{\V}_r) & \im(\widetilde{\V}_r)
\\ \im(\widetilde{\V}_r) & -\real(\widetilde{\V}_r) \end{array}\right]^{\top}.
\end{align*}
Therefore, by \cref{lemma:technicalLemma3QG} it holds that 
\begin{align} \label{ineq:inProof979}
\Vert\W^{\lozenge}-{{\overbar{\X^*}}^{\sharp}}^{\lozenge}\Vert_F
& \le \frac{1}{\sigma_{\min}(\mcP_{V})}\Vert\mcP(\W^{\lozenge})-\mcP({{\overbar{\X^*}}^{\sharp}}^{\lozenge})\Vert_2,
\end{align}
where $\sigma_{\min}(\mcP_{V}):= \min_{\Vert\S\Vert_F=1}\Vert\mcP_{V}(\S)\Vert_2>0$.

Therefore, for any $\mX\in\lbrace\mX\in\reals^{n_1\times\cdots\times n_d}\ \vert\ \Vert\mX\Vert_*\le1\rbrace$ we have that
\begin{align} \label{ineq:inProof9999}
2\Vert\W-{{\overbar{\X^*}}^{\sharp}}\Vert_F^2 & \underset{(a)}{=} \Vert\W^{\lozenge}-{{\overbar{\X^*}}^{\sharp}}^{\lozenge}\Vert_F^2  \underset{(b)}{\le} \frac{1}{\sigma_{\min}^2(\mcP_{V})}\Vert\mcP(\W^{\lozenge})-\mcP({{\overbar{\X^*}}^{\sharp}}^{\lozenge})\Vert_2^2 \nonumber
\\ & \le \frac{1}{\sigma_{\min}^2(\mcP_{V})} \left(\Vert\mcP(\W^{\lozenge})-\mcP({\barX^{\sharp}}^{\lozenge})\Vert_2+\Vert\mcP({\barX^{\sharp}}^{\lozenge})-\mcP({{\overbar{\X^*}}^{\sharp}}^{\lozenge})\Vert_2\right)^2 \nonumber
\\ & \le 2\frac{\sigma_{\max}^2(\mcP)}{\sigma_{\min}^2(\mcP_{V})} \Vert\W^{\lozenge}-{\barX^{\sharp}}^{\lozenge}\Vert_F^2+ \frac{2}{\sigma_{\min}^2(\mcP_{V})}\Vert\mcP({\barX^{\sharp}}^{\lozenge})-\mcP({{\overbar{\X^*}}^{\sharp}}^{\lozenge})\Vert_2^2 \nonumber
\\ & \underset{(c)}{=} 2\frac{\sigma_{\max}^2(\mcP)}{\sigma_{\min}^2(\mcP_{V})} \Vert\W^{\lozenge}-{\barX^{\sharp}}^{\lozenge}\Vert_F^2+ \frac{2}{\sigma_{\min}^2(\mcP_{V})}\Vert\mcB({\barX^{\sharp}}^{\lozenge})-\mcB({\overbar{\X^*}^{\sharp}}^{\lozenge})\Vert_2^2 \nonumber
\\ & \underset{(d)}{=} 4\frac{\sigma_{\max}^2(\mcP)}{\sigma_{\min}^2(\mcP_{V})} \Vert\W-\barX^{\sharp}\Vert_F^2+ \frac{2}{\sigma_{\min}^2(\mcP_{V})}\Vert\mA(\mX)-\mA(\mX^*)\Vert_2^2,
\end{align}
where $\sigma_{\max}(\mcP):=\max_{\Vert\S\Vert_F=1}\Vert\mcP(\S)\Vert_2$. Here (a) and (d) follow since as we saw in \eqref{eq:InProofQGinnerProductEquivalence}, for any $\widetilde{\A},\widetilde{\B}\in\mathbb{H}^{(n_1+n_2) N}$ it holds that $\langle\widetilde{\A}^{\lozenge},\widetilde{\B}^{\lozenge}\rangle=2\langle\widetilde{\A},\widetilde{\B}\rangle$, which implies that  $\Vert\widetilde{\A}^{\lozenge}\Vert_F^2= 2\Vert\widetilde{\A}\Vert_F^2$. (b) follows from \eqref{ineq:inProof979}, (c) holds since $\trace({\barX^{\sharp}}^{\lozenge})=\trace({\overbar{\X^*}^{\sharp}}^{\lozenge})$ and $\mcQ({\barX^{\sharp}}^{\lozenge})=\mcQ({\overbar{\X^*}^{\sharp}}^{\lozenge})$, and (d) also follows since $\ifft(\bdiag^{-1}({{\overbar{\X^*}^{\sharp}}_{\triangledown}^{\lozenge}}_{\flat}))=\mX^*$.

Therefore, we obtain that
\begin{align} \label{ineq:inProof996}
\Vert\barX^{\sharp}-{\overbar{\X^*}}^{\sharp}\Vert_F^2 & \le 2\Vert\barX^{\sharp}-\W\Vert_F^2 + 2\Vert\W-{\overbar{\X^*}}^{\sharp}\Vert_F^2 \nonumber
\\ & \underset{(a)}{\le} \left(2+\frac{4\sigma_{\max}^2(\mcP)}{\sigma_{\min}^2(\mcP_{V})}\right)\Vert\barX^{\sharp}-\W\Vert_F^2 + \frac{2}{\sigma_{\min}
^2(\mcP_{V})}\Vert\mA(\mX)-\mA(\mX^*)\Vert_2^2 \nonumber
\\ & \underset{(b)}{\le} \frac{4 N}{\delta}
\left(1+\frac{2\sigma_{\max}^2(\mcP)}{\sigma_{\min}^2(\mcP_{V})}\right)\langle\barX^{\sharp}-{\overbar{\X^*}}^{\sharp}, \widetilde{\boldsymbol{\nabla}}\rangle + \frac{2}{\sigma_{\min}^2(\mcP_{V})}\Vert\mA(\mX)-\mA(\mX^*)\Vert_2^2,
\end{align}
where (a) follows from plugging-in \eqref{ineq:inProof9999}, and (b) follows by invoking \cref{lemma:spectrahedronLemmaQuadraticGrowth} and letting $\delta$ denote the complementarity measure of $\mX^*$.

It can be seen that 
\begin{align} \label{ineq:inProof998}
 N\Vert\mX-\mX^*\Vert_F^2 \underset{(a)}{=} \Vert\barX-\overbar{\X^*}\Vert_F^2=\Vert\barX^{\sharp}_{\flat}-{\overbar{\X^*}}^{\sharp}_{\flat}\Vert_F^2\le2\Vert\barX^{\sharp}-{\overbar{\X^*}}^{\sharp}\Vert_F^2,
\end{align}
and 
\begin{align} \label{ineq:inProof997}
\langle\barX^{\sharp}-{\overbar{\X^*}}^{\sharp}, \widetilde{\boldsymbol{\nabla}}  \rangle = \langle\barX^{\sharp}_{\flat}-{\overbar{\X^*}}^{\sharp}_{\flat}, \overbar{\nabla{}f(\mX^*)}\rangle 
= \langle\barX-\overbar{\X^*}, \overbar{\nabla{}f(\mX^*)}\rangle
\underset{(b)}{=}  N\langle\mX-\mX^*,\nabla{}f(\mX^*)\rangle,
\end{align}
where both (a) and (b) follow from \cref{lemma:orderPinnerProductEquivalence}.

Plugging-in \eqref{ineq:inProof998} into the LHS of \eqref{ineq:inProof996} and \eqref{ineq:inProof997} into the RHS of 
\eqref{ineq:inProof996}, we obtain that
\begin{align} \label{ineq:inProof995}
& \Vert\mX-\mX^*\Vert_F^2 \nonumber
\\ & \le \frac{8 N}{\delta}
\left(1+\frac{2\sigma_{\max}^2(\mcP)}{\sigma_{\min}^2(\mcP_{V})}\right)\langle\mX-\mX^*, \nabla{}f(\mX^*)\rangle + \frac{4}{\sigma_{\min}^2(\mcP_{V}) N}\Vert\mA(\mX)-\mA(\mX^*)\Vert_2^2 \nonumber
\\ & \le \max\left\lbrace\frac{8 N}{\delta}
\left(1+\frac{2\sigma_{\max}^2(\mcP)}{\sigma_{\min}^2(\mcP_{V})}\right), \frac{8}{\sigma_{\min}^2(\mcP_{V}) N\alpha}\right\rbrace \nonumber
\\ & \ \ \ \left(\langle\nabla{}f(\mX^*),\mX-\mX^*\rangle + \frac{\alpha}{2}\Vert\mA(\mX)-\mA(\mX^*)\Vert_2^2\right).
\end{align}
Therefore, we have that
\begin{align*}
f(\mX)-f(\mX^*) & = g(\mA(\mX))-g(\mA(\mX^*))+\langle\mC,\mX-\mX^*\rangle
\\ & \underset{(a)}{\ge} \langle\nabla{}g(\mA(\mX^*)),\mA(\mX-\mX^*)\rangle + \frac{\alpha}{2}\Vert\mA(\mX)-\mA(\mX^*)\Vert_2^2+\langle\mC,\mX-\mX^*\rangle
\\ & = \langle\mA^{\top}\nabla{}g(\mA\mX^*)+\mC,\mX-\mX^*\rangle + \frac{\alpha}{2}\Vert\mA(\mX)-\mA(\mX^*)\Vert_2^2
\\ & = \langle\nabla{}f(\mX^*),\mX-\mX^*\rangle + \frac{\alpha}{2}\Vert\mA(\mX)-\mA(\mX^*)\Vert_2^2
\\ & \underset{(b)}{\ge} \min\left\lbrace\frac{\delta}{8 N}
\left(1+\frac{2\sigma_{\max}^2(\mcP)}{\sigma_{\min}^2(\mcP_{V})}\right)^{-1}, \frac{\sigma_{\min}^2(\mcP_{V}) N\alpha}{8}\right\rbrace\Vert\mX-\mX^*\Vert_F^2,
\end{align*}
where (a) follows from the strong convexity of $g$, and (b) follows from \eqref{ineq:inProof995}.
\end{proof}

\subsection{Proof of \cref{lemma:iffLowRankfCondiotionOrderP}}
\label{sec:App:proofLemma122}

\begin{lemma} 
Let $\mX\in\reals^{n_1\times\cdots\times n_d}$ and let $\mX=\mU*\mS*\mV^{\top}$ denote its t-SVD. For every $j\in\{1,\dots,\min\{n_1,n_2\}\}$, denote $\sigma_j^{\max}(\mX) = \max_{k_3\in[n_3],\ldots,k_d\in[n_d]}\sigma_{j}(\barX^{(k_3,\ldots,k_d)})$. Also, for every $j\in\{1,\dots,\min\{n_1,n_2\}\}$ and $i_3\in[n_3],\ldots,i_d\in[n_d]$, denote $\#\sigma_{>{}j}^{(i_3,\dots,i_d)}(\mX) =\#\left\lbrace i\ \bigg\vert\ \sigma_i(\barX^{(i_3,\ldots,i_d)})>\sigma_{j}^{\max}(\barX)\right\rbrace\le j-1$. Let $\min\lbrace n_1,n_2\rbrace>r\ge0$.
Then, $\rank_{\textnormal{t}}(\Pi_{\lbrace\Vert\mY\Vert_*\le\tau\rbrace}[\mX])\le r$ if and only if
\begin{align*} 
\frac{1}{N}\sum_{i_d=1}^{n_d}\cdots\sum_{i_3=1}^{n_3}\left(\sum_{i=1}^{\#\sigma_{>r+1}^{(i_3,\dots,i_d)}(\mX)}\sigma_{i}(\barX^{(i_3,\ldots,i_d)})-\#\sigma_{>{}r+1}^{(i_3,\dots,i_d)}(\mX) \cdot \sigma_{r+1}^{\max}(\barX)\right)\ge\tau.
\end{align*}
\end{lemma}

\begin{proof}
Invoking \cref{lemma:projectionOntoTNN}, we know that the projection of $\mX$ onto the TNN ball of radius $\tau$ must satisfy  $$\frac{1}{ N}\sum_{i_d=1}^{n_d}\cdots\sum_{i_3=1}^{n_3}\sum_{i=1}^{\min\lbrace n_1,n_2\rbrace} \max\lbrace0,\sigma_i(\barX^{(i_3,\ldots,i_d)})-\sigma\rbrace=\tau,$$ for some $\sigma\ge0$.
Therefore, 
\begin{align*}
& 0  = \frac{1}{ N}\sum_{i_d=1}^{n_d}\cdots\sum_{i_3=1}^{n_3}\sum_{i=1}^{\min\lbrace n_1,n_2\rbrace} \max\lbrace0,\sigma_i(\barX^{(i_3,\ldots,i_d)})-\sigma\rbrace-\tau
\\ & \ge \frac{1}{ N}\sum_{i_d=1}^{n_d}\cdots\sum_{i_3=1}^{n_3}\sum_{i=1}^{\#\sigma_{>r+1}^{(i_3,\dots,i_d)}(\mX)} (\sigma_i(\barX^{(i_3,\ldots,i_d)})-\sigma)-\tau
\\ & \ge \left(\frac{1}{ N}\sum_{i_d=1}^{n_d}\cdots\sum_{i_3=1}^{n_3}\#\sigma_{>{}r+1}^{(i_3,\dots,i_d)}(\mX)\right)\sigma_{r+1}^{\max}(\barX)-\left(\frac{1}{ N}\sum_{i_d=1}^{n_d}\cdots\sum_{i_3=1}^{n_3}\#\sigma_{>{}r+1}^{(i_3,\dots,i_d)}(\mX)\right)\sigma,
\end{align*}
which implies that
$\max_{k_3\in[n_3],\ldots,k_d\in[n_d]}\sigma_{r+1}(\barX^{(k_3,\ldots,k_d)})\le \sigma$,
and so, $\rank_{\textnormal{t}}(\Pi_{\lbrace\Vert\mY\Vert_*\le\tau\rbrace}[\mX]) \le r$.

On the other hand, if $\rank_{\textnormal{t}}(\Pi_{\lbrace\Vert\mY\Vert_*\le\tau\rbrace}[\mX]) \le r$, then for all $i_3\in[n_3],\ldots,i_d\in[n_d]$, there exist $r_{i_3,\ldots,i_d}'\le r$ and $\tau_{i_3,\ldots,i_d}\ge0$ such that $\frac{1}{N}\sum_{i_d=1}^{n_d}\cdots\sum_{i_3=1}^{n_3}\tau_{i_3,\ldots,i_d}=\tau$, for which it holds that $\rank(\Pi_{\lbrace\Vert\Y\Vert_*\le\tau_{i_3,\ldots,i_d}\rbrace}[\barX^{(i_3,\ldots,i_d)}]) = r_{i_3,\ldots,i_d}' \le \#\sigma_{>{}r+1}^{(i_3,\dots,i_d)}(\mX) \le r$. 
In addition, it implies that for all $i_3\in[n_3],\ldots,i_d\in[n_d]$, it holds that $\sum_{i=1}^{r_{i_3,\ldots,i_d}'}(\sigma_i(\barX^{(i_3,\ldots,i_d)})-\sigma)=\tau_{i_3,\ldots,i_d}$ and $\sigma_{r+1}^{\max}(\barX)\le \sigma$. 
Therefore, for all $i_3\in[n_3],\ldots,i_d\in[n_d]$, we have that
\begin{align*}
\tau_{i_3,\ldots,i_d} & = \sum_{i=1}^{r_{i_3,\ldots,i_d}'}\sigma_{i}(\barX^{(i_3,\ldots,i_d)})-r_{i_3,\ldots,i_d}'\sigma
\le \sum_{i=1}^{r_{i_3,\ldots,i_d}'}\sigma_{i}(\barX^{(i_3,\ldots,i_d)})-r_{i_3,\ldots,i_d}'\sigma_{r+1}^{\max}(\barX)
\\ & \le \sum_{i=1}^{\#\sigma_{>r+1}^{(i_3,\dots,i_d)}(\mX)}\sigma_{i}(\barX^{(i_3,\ldots,i_d)})-\#\sigma_{>{}r+1}^{(i_3,\dots,i_d)}(\mX)\cdot\sigma_{r+1}^{\max}(\barX).
\end{align*}
Averaging over all $i_3\in[n_3],\ldots,i_d\in[n_d]$, we obtain that
\begin{align*}
\frac{1}{ N}\sum_{i_d=1}^{n_d}\cdots\sum_{i_3=1}^{n_3}\left(\sum_{i=1}^{\#\sigma_{>r+1}^{(i_3,\dots,i_d)}(\mX)}\sigma_i(\barX^{({i_3,\ldots,i_d})}) - \#\sigma_{>{}r+1}^{(i_3,\dots,i_d)}(\mX)\cdot\sigma_{r+1}^{\max}(\barX)\right)\ge\tau.
\end{align*}

\end{proof}

\subsection{Proof of \cref{lemma:RadiusForProjectedGradient_d_dim}}
\label{sec:App:proofTheoremSmoothRadius}

\begin{theorem}
Assume $\nabla{}f$ is non-zero over the unit TNN ball and fix some optimal solution $\mX^*$ to Problem \eqref{mainModel}. Denote $ {\#\sigma_1}^{\max}:=\max_{i_3\in[n_3],\ldots,i_d\in[n_d]}\#\sigma_1(\overbar{\nabla{}f(\mX^*)}^{(i_3,\ldots,i_d)})$ and assume ${\#\sigma_1}^{\max}<\min\lbrace n_1,n_2\rbrace$.
Then, for any $\eta\ge0$, $\min\lbrace n_1,n_2\rbrace> r\ge{\#\sigma_1}^{\max}$, and $\mX\in\reals^{n_1\times{}\cdots\times{}n_d}$, if 
\begin{align*} 
& \Vert\mX-\mX^*\Vert_F \le \nonumber
\\ &  \frac{\eta}{\sqrt{N}(1+\eta\beta)}\max\left\lbrace\frac{\delta(r)}{1+\frac{\sqrt{{\#\sigma_1}^{\max}\cdot\nnzb(\overbar{\nabla{}f(\mX^*)})}}{\#\sigma_1(\overbar{\nabla{}f(\mX^*)})}},\frac{\delta(r-{\#\sigma_1}^{\max}+1)}{\frac{1}{\sqrt{{\#\sigma_1}^{\max}}}+\frac{\sqrt{{\#\sigma_1}^{\max}\cdot\nnzb(\overbar{\nabla{}f(\mX^*)})}}{\#\sigma_1(\overbar{\nabla{}f(\mX^*)})}}\right\rbrace,
\end{align*}
or  
\begin{align*} 
\Vert\mX-\mX^*\Vert_2 \le \frac{\eta}{2(1+\eta\beta_2)}\delta(r),
\end{align*}
then $\rank_{\textnormal{t}}(\Pi_{\lbrace\Vert\mY\Vert_*\le1\rbrace}[\mX-\eta\nabla{}f(\mX)])\le r$.
\end{theorem}

\begin{proof}
For all $i_3\in[n_3],\ldots,i_d\in[n_d]$ denote ${\overbar{\P^*}}^{(i_3,\ldots,i_d)}:= {\overbar{\X^*}}^{(i_3,\ldots,i_d)} - \eta\overbar{\nabla{}f(\mX^*)}^{(i_3,\ldots,i_d)}$. Invoking \cref{lemma:svd_opt_grad_d_dim} we have that
\begin{align} \label{eq:svOpt_d_dim}
\forall i\le\rank({\overbar{\X^*}}^{(i_3,\ldots,i_d)}):\quad & \sigma_i({\overbar{\P^*}}^{(i_3,\ldots,i_d)}) = \sigma_i({\overbar{\X^*}}^{(i_3,\ldots,i_d)}) + \eta \sigma_1(\overbar{\nabla{}f(\mX^*)}) \nonumber
\\ \forall i>\rank({\overbar{\X^*}}^{(i_3,\ldots,i_d)}):\quad & \sigma_i({\overbar{\P^*}}^{(i_3,\ldots,i_d)}) = \eta \sigma_i(\overbar{\nabla{}f(\mX^*)}^{(i_3,\ldots,i_d)}).
\end{align}

Fix some $i_3\in[n_3],\ldots,i_d\in[n_d]$. Let $\#\sigma_1^{(i_3,\dots,i_d)}$ denote the multiplicity of $\sigma_1(\overbar{\nabla{}f(\mX^*)}^{(i_3,\ldots,i_d)})$. From \eqref{eq:svOpt_d_dim} we have that
\begin{align} \label{ineq:inProof1_d_dim}
\sum_{i=1}^{\#\sigma_1^{(i_3,\dots,i_d)}}\sigma_i({\overbar{\P^*}}^{(i_3,\ldots,i_d)}) 
& = \sum_{i=1}^{\#\sigma_1^{(i_3,\dots,i_d)}} \sigma_i({\overbar{\X^*}}^{(i_3,\ldots,i_d)} - \eta\overbar{\nabla{}f(\mX^*)}^{(i_3,\ldots,i_d)})  \nonumber
\\ & = \sum_{i=1}^{\#\sigma_1^{(i_3,\dots,i_d)}} \sigma_i({\overbar{\X^*}}^{(i_3,\ldots,i_d)}) + \eta\sum_{i=1}^{\#\sigma_1^{(i_3,\dots,i_d)}}\sigma_i(\overbar{\nabla{}f(\mX^*)}^{(i_3,\ldots,i_d)}) \nonumber
\\ & = \sum_{i=1}^{\rank({\overbar{\X^*}}^{(i_3,\ldots,i_d)})} \sigma_i({\overbar{\X^*}}^{(i_3,\ldots,i_d)}) + \eta\sum_{i=1}^{\#\sigma_1^{(i_3,\dots,i_d)}}\sigma_i(\overbar{\nabla{}f(\mX^*)}^{(i_3,\ldots,i_d)}) \nonumber
\\ &  = \Vert{\overbar{\X^*}}^{(i_3,\ldots,i_d)}\Vert_* + \eta\cdot\#\sigma_1^{(i_3,\dots,i_d)}\sigma_1(\overbar{\nabla{}f(\mX^*)}).
\end{align}

Let $\mP\in\reals^{n_1\times{}\cdots\times{}n_d}$. From \cref{lemma:iffLowRankfCondiotionOrderP}, it follows that a sufficient condition so that 
$\rank_{\textnormal{t}}(\Pi_{\lbrace\Vert\mX\Vert_*\le1\rbrace}[\mP]) \le r$ is that 
\begin{align} \label{ineq:inProof13_d_dim}
\frac{1}{N}\sum_{i_d=1}^{n_d}\cdots\sum_{i_3=1}^{n_3}\left(\sum_{i=1}^{\#\sigma_{>r+1}^{(i_3,\dots,i_d)}(\mP)}\sigma_{i}(\barP^{(i_3,\ldots,i_d)})-\#\sigma_{>r+1}^{(i_3,\dots,i_d)}(\mP) \cdot \sigma_{r+1}^{\max}(\barP)\right)\ge1,
\end{align}
where $\barP=\bdiag(\overbar{\mP})$ and for every $j\in\{1,\dots,\min\{n_1,n_2\}\}$ we denote $\sigma_{j}^{\max}(\barP) = \max_{k_3\in[n_3],\ldots,k_d\in[n_d]}\sigma_{j}(\barP^{(k_3,\ldots,k_d)})$ and $\#\sigma_{>{}(j)}^{(i_3,\dots,i_d)}(\mP) =\#\left\lbrace i\ \bigg\vert\ \sigma_i(\barP^{(i_3,\ldots,i_d)})>\sigma_{j}^{\max}(\barP)\right\rbrace$. We will lower-bound the LHS of  \eqref{ineq:inProof13_d_dim}.

Fix $i_3\in[n_3],\ldots,i_d\in[n_d]$. We first show that
\begin{align} \label{ineq:inProof2_d_dim}
& \sum_{i=1}^{\#\sigma_1^{(i_3,\dots,i_d)}}\sigma_i(\barP^{(i_3,\ldots,i_d)})  \underset{(a)}{\ge} \sum_{i=1}^{\#\sigma_1^{(i_3,\dots,i_d)}}\sigma_i({\overbar{\P^*}}^{(i_3,\ldots,i_d)}) - \sum_{i=1}^{\#\sigma_1^{(i_3,\dots,i_d)}}\sigma_i(\barP^{(i_3,\ldots,i_d)}-{\overbar{\P^*}}^{(i_3,\ldots,i_d)}) \nonumber
\\ & \underset{(b)}{=} \Vert{\overbar{\X^*}}^{(i_3,\ldots,i_d)}\Vert_* + \eta\cdot\#\sigma_1^{(i_3,\dots,i_d)}\sigma_1(\overbar{\nabla{}f(\mX^*)}) - \sum_{i=1}^{\#\sigma_1^{(i_3,\dots,i_d)}}\sigma_i(\barP^{(i_3,\ldots,i_d)}-{\overbar{\P^*}}^{(i_3,\ldots,i_d)}),
\end{align}
where (a) follows from Ky Fan's inequality for singular values and (b) follows from \eqref{ineq:inProof1_d_dim}.

Averaging \eqref{ineq:inProof2_d_dim} over all $i_3\in[n_3],\ldots,i_d\in[n_d]$, we obtain that
\begin{align} 
& \frac{1}{N}\sum_{i_3=1}^{n_3}\cdots\sum_{i_d=1}^{n_d}\sum_{i=1}^{\#\sigma_1^{(i_3,\dots,i_d)}}\sigma_i(\barP^{(i_3,\ldots,i_d)}) \nonumber
\\ & \ge   1 +\frac{\eta}{N}\sigma_1(\overbar{\nabla{}f(\mX^*)})\sum_{i_3=1}^{n_3}\cdots\sum_{i_d=1}^{n_d}\#\sigma_1^{(i_3,\dots,i_d)} \nonumber
\\ & \ \ \ - \frac{1}{N}\sum_{i_3=1}^{n_3}\cdots\sum_{i_d=1}^{n_d}\sum_{i=1}^{\#\sigma_1^{(i_3,\dots,i_d)}}\sigma_i(\barP^{(i_3,\ldots,i_d)}-{\overbar{\P^*}}^{(i_3,\ldots,i_d)}) \label{ineq:inProof33333}
\\ & \ge 1 +\frac{\eta}{N}\sigma_1(\overbar{\nabla{}f(\mX^*)})\sum_{i_3=1}^{n_3}\cdots\sum_{i_d=1}^{n_d}\#\sigma_1^{(i_3,\dots,i_d)} \nonumber
\\ & \ \ \ - \frac{1}{N}\sum_{\substack{(i_3,\dots,i_d)\in[n_3]\times\cdots\times[n_d] \\ \#\sigma_1^{(i_3,\dots,i_d)}\not=0}}\sum_{i=1}^{{\#\sigma_1}^{\max}}\sigma_i(\barP^{(i_3,\ldots,i_d)}-{\overbar{\P^*}}^{(i_3,\ldots,i_d)}) \nonumber
\\ & \ge 1 +\frac{\eta}{N}\sigma_1(\overbar{\nabla{}f(\mX^*)})\sum_{i_3=1}^{n_3}\cdots\sum_{i_d=1}^{n_d}\#\sigma_1^{(i_3,\dots,i_d)} \nonumber
\\ & \ \ \ - \frac{1}{N}\sqrt{{\#\sigma_1}^{\max}\cdot\nnzb(\overbar{\nabla{}f(\mX^*)})\sum_{\substack{(i_3,\dots,i_d)\in[n_3]\times\cdots\times[n_d] \\ \#\sigma_1^{(i_3,\dots,i_d)}\not=0}}\sum_{i=1}^{{\#\sigma_1}^{\max}}\sigma_i^2(\barP^{(i_3,\ldots,i_d)}-{\overbar{\P^*}}^{(i_3,\ldots,i_d)})} \nonumber
\\ & \ge 1 +\frac{\eta}{N}\sigma_1(\overbar{\nabla{}f(\mX^*)})\sum_{i_3=1}^{n_3}\cdots\sum_{i_d=1}^{n_d}\#\sigma_1^{(i_3,\dots,i_d)} \nonumber
\\ & \ \ \ - \frac{1}{N}\sqrt{{\#\sigma_1}^{\max}\cdot\nnzb(\overbar{\nabla{}f(\mX^*)})\sum_{i_3=1}^{n_3}\cdots\sum_{i_d=1}^{n_d}\sum_{i=1}^{\min\lbrace n_1,n_2\rbrace}\sigma_i^2(\barP^{(i_3,\ldots,i_d)}-{\overbar{\P^*}}^{(i_3,\ldots,i_d)})} \nonumber
\\ & = 1 +\frac{\eta}{N}\sigma_1(\overbar{\nabla{}f(\mX^*)})\sum_{i_3=1}^{n_3}\cdots\sum_{i_d=1}^{n_d}\#\sigma_1^{(i_3,\dots,i_d)}- \frac{\sqrt{{\#\sigma_1}^{\max}\cdot\nnzb(\overbar{\nabla{}f(\mX^*)})}}{N}\Vert\barP-{\overbar{\P^*}}\Vert_F. \label{ineq:inProof3_d_dim}
\end{align} 

In addition, for all $i_3\in[n_3],\ldots,i_d\in[n_d]$ and any $r\ge\#\sigma_1^{(i_3,\dots,i_d)}$, using Weyl's inequality and \eqref{eq:svOpt_d_dim} we have that
\begin{align}
\sigma_{r+1}(\barP^{(i_3,\ldots,i_d)}) & \le \sigma_{r+1}({\overbar{\P^*}}^{(i_3,\ldots,i_d)}) + \sigma_{1}(\barP^{(i_3,\ldots,i_d)}-{\overbar{\P^*}}^{(i_3,\ldots,i_d)}) \label{ineq:inProof3334}
\\ & \le \sigma_{r+1}({\overbar{\P^*}}^{(i_3,\ldots,i_d)})+\Vert\barP-{\overbar{\P^*}}\Vert_F \nonumber
\\ & = \eta \sigma_{r+1}(\overbar{\nabla{}f(\mX^*)}^{(i_3,\ldots,i_d)})+\Vert\barP-{\overbar{\P^*}}\Vert_F. \nonumber
\end{align} 

Taking the maximum over all $k_3\in[n_3],\ldots,k_d\in[n_d]$, we obtain that 
\begin{align} \label{ineq:inProof7_d_dim}
& \sigma_{r+1}^{\max}(\barP)  = \max_{k_3\in[n_3],\ldots,k_d\in[n_d]}\sigma_{r+1}(\barP^{(k_3,\ldots,k_d)}) \nonumber
\\ & \le \eta \max_{k_3\in[n_3],\ldots,k_d\in[n_d]}\sigma_{r+1}(\overbar{\nabla{}f(\mX^*)}^{(k_3,\ldots,k_d)})+\Vert\barP-{\overbar{\P^*}}\Vert_F = \eta\sigma_{r+1}^{\max}(\overbar{\nabla{}f(\mX^*)})+\Vert\barP-{\overbar{\P^*}}\Vert_F.
\end{align} 

For every $i_3\in[n_3],\ldots,i_d\in[n_d]$ it holds that
\begin{align} \label{ineq:inProof9_d_dim}
& \sum_{i=1}^{\#\sigma_{>r+1}^{(i_3,\dots,i_d)}(\mP)}\sigma_i(\barP^{(i_3,\ldots,i_d)})-\#\sigma_{>r+1}^{(i_3,\dots,i_d)}(\mP)\cdot\sigma_{r+1}^{\max}(\barP) \nonumber
\\ & \ge \sum_{i=1}^{\#\sigma_1^{(i_3,\dots,i_d)}}\sigma_i(\barP^{(i_3,\ldots,i_d)})-\#\sigma_1^{(i_3,\dots,i_d)}\sigma_{r+1}^{\max}(\barP).
\end{align}
Note that this holds whether $\#\sigma_{>r+1}^{(i_3,\dots,i_d)}(\mP)\le\#\sigma_1^{(i_3,\dots,i_d)}\le r$ or whether $\#\sigma_1^{(i_3,\dots,i_d)}\le \#\sigma_{>r+1}^{(i_3,\dots,i_d)}(\mP)\le r$.
Averaging \eqref{ineq:inProof9_d_dim} over all $i_3\in[n_3],\ldots,i_d\in[n_d]$ we obtain that
\begin{align} \label{ineq:inProof10_d_dim}
& \frac{1}{N}\sum_{i_3=1}^{n_3}\cdots\sum_{i_d=1}^{n_d}\sum_{i=1}^{\#\sigma_{>r+1}^{(i_3,\dots,i_d)}(\mP)}\sigma_i(\barP^{(i_3,\ldots,i_d)}) - \frac{1}{N}\left(\sum_{i_3=1}^{n_3}\cdots\sum_{i_d=1}^{n_d}\#\sigma_{>r+1}^{(i_3,\dots,i_d)}(\mP)\right)\sigma_{r+1}^{\max}(\barP) \nonumber
\\ & \ge \frac{1}{N}\sum_{i_3=1}^{n_3}\cdots\sum_{i_d=1}^{n_d}\sum_{i=1}^{\#\sigma_1^{(i_3,\dots,i_d)}}\sigma_i(\barP^{(i_3,\ldots,i_d)}) - \frac{1}{N}\left(\sum_{i_3=1}^{n_3}\cdots\sum_{i_d=1}^{n_d}\#\sigma_1^{(i_3,\dots,i_d)}\right)\sigma_{r+1}^{\max}(\barP).
\end{align}

Plugging \eqref{ineq:inProof3_d_dim} and \eqref{ineq:inProof7_d_dim} into the RHS of \eqref{ineq:inProof10_d_dim} we obtain that
\begin{align} \label{ineq:inProof8_d_dim}
& \frac{1}{N}\sum_{i_3=1}^{n_3}\cdots\sum_{i_d=1}^{n_d}\sum_{i=1}^{\#\sigma_{>r+1}^{(i_3,\dots,i_d)}(\mP)}\sigma_i(\barP^{(i_3,\ldots,i_d)}) \nonumber
- \frac{1}{N}\left(\sum_{i_3=1}^{n_3}\cdots\sum_{i_d=1}^{n_d}\#\sigma_{>r+1}^{(i_3,\dots,i_d)}(\mP)\right)\sigma_{r+1}^{\max}(\barP) \nonumber
\\ & \ge 1 +\frac{\eta}{N}\left(\sigma_1(\overbar{\nabla{}f(\mX^*)})-\sigma_{r+1}^{\max}(\overbar{\nabla{}f(\mX^*)})\right)\sum_{i_3=1}^{n_3}\cdots\sum_{i_d=1}^{n_d}\#\sigma_1^{(i_3,\dots,i_d)} \nonumber
\\ & \ \ \ - \left(\frac{1}{N}\sum_{i_3=1}^{n_3}\cdots\sum_{i_d=1}^{n_d}\#\sigma_1^{(i_3,\dots,i_d)}+\frac{\sqrt{{\#\sigma_1}^{\max}\cdot\nnzb(\overbar{\nabla{}f(\mX^*)})}}{N}\right)\Vert\barP-{\overbar{\P^*}}\Vert_F.
\end{align}

Taking $\barP:=\barX-\eta\overbar{\nabla{}f(\mX)}$, invoking \cref{lemma:orderPinnerProductEquivalence}, and using the $\beta$-smoothness of $f$, we have that
\begin{align} \label{ineq:inProof12_d_dim}
\Vert\barP-{\overbar{\P^*}}\Vert_F & = \Vert\barX-\eta\overbar{\nabla{}f(\mX)}-\overbar{\X^*}-\eta\overbar{\nabla{}f(\mX^*)}\Vert_F \nonumber
 \\
&\le \Vert\barX-\overbar{\X^*}\Vert_F+\eta\Vert\overbar{\nabla{}f(\mX)}-\overbar{\nabla{}f(\mX^*)}\Vert_F \nonumber
\\ & = \sqrt{N}(\Vert\mX-\mX^*\Vert_F+\eta\Vert\nabla{}f(\mX)-\nabla{}f(\mX^*)\Vert_F) \nonumber\\
& \le \sqrt{N}(1+\eta\beta)\Vert\mX-\mX^*\Vert_F.
\end{align}

Plugging \eqref{ineq:inProof12_d_dim} into \eqref{ineq:inProof8_d_dim}, we obtain that
\begin{align*} 
& \frac{1}{N}\sum_{i_3=1}^{n_3}\cdots\sum_{i_d=1}^{n_d}\sum_{i=1}^{\#\sigma_{>r+1}^{(i_3,\dots,i_d)}(\mP)}\sigma_i(\barP^{(i_3,\ldots,i_d)})  - \frac{1}{N}\left(\sum_{i_3=1}^{n_3}\cdots\sum_{i_d=1}^{n_d}\#\sigma_{>r+1}^{(i_3,\dots,i_d)}(\mP)\right)\sigma_{r+1}^{\max}(\barP) 
\\ & \ge 1 +\frac{\eta}{N}\left(\sigma_1(\overbar{\nabla{}f(\mX^*)})-\sigma_{r+1}^{\max}(\overbar{\nabla{}f(\mX^*)})\right)\sum_{i_3=1}^{n_3}\cdots\sum_{i_d=1}^{n_d}\#\sigma_1^{(i_3,\dots,i_d)}
\\ & \ \ \ - \left(\frac{1}{\sqrt{N}}\sum_{i_3=1}^{n_3}\cdots\sum_{i_d=1}^{n_d}\#\sigma_1^{(i_3,\dots,i_d)}+\sqrt{\frac{{\#\sigma_1}^{\max}\cdot\nnzb(\overbar{\nabla{}f(\mX^*)})}{N}}\right)(1+\eta\beta)\Vert\mX-\mX^*\Vert_F.
\end{align*}

Rearranging, we obtain that the condition in \eqref{ineq:inProof13_d_dim} holds if
\begin{align} \label{ineq:inProof4_d_dim}
\Vert\mX-\mX^*\Vert_F \le \frac{\eta\left(\sigma_1(\overbar{\nabla{}f(\mX^*)})-\sigma_{r+1}^{\max}(\overbar{\nabla{}f(\mX^*)})\right)}{\sqrt{N}\left(1+\frac{\sqrt{{\#\sigma_1}^{\max}\cdot\nnzb(\overbar{\nabla{}f(\mX^*)})}}{\sum_{i_3=1}^{n_3}\cdots\sum_{i_d=1}^{n_d}\#\sigma_1^{(i_3,\dots,i_d)}}\right)(1+\eta\beta)}.
\end{align}

Alternatively, for all $i_3\in[n_3],\ldots,i_d\in[n_d]$, if $r\ge 2\cdot{\#\sigma_1}^{\max}-1$, then using the general Weyl inequality and \eqref{eq:svOpt_d_dim} and denoting $\tilde{r}_{\max}={\#\sigma_1}^{\max}$ for clarity of notation we have that
\begin{align*} 
\sigma_{r+1}(\barP^{(i_3,\ldots,i_d)}) & \le \sigma_{r-\tilde{r}_{\max}+2}({\overbar{\P^*}}^{(i_3,\ldots,i_d)}) + \sigma_{\tilde{r}_{\max}}(\barP^{(i_3,\ldots,i_d)}-{\overbar{\P^*}}^{(i_3,\ldots,i_d)})
\\ & = \sigma_{r-\tilde{r}_{\max}+2}({\overbar{\P^*}}^{(i_3,\ldots,i_d)})+\sqrt{\sigma_{\tilde{r}_{\max}}^2(\barP^{(i_3,\ldots,i_d)}-{\overbar{\P^*}}^{(i_3,\ldots,i_d)})}
\\ & \le \sigma_{r-\tilde{r}_{\max}+2}({\overbar{\P^*}}^{(i_3,\ldots,i_d)})+\sqrt{\sigma_{\tilde{r}_{\max}}^2(\barP-{\overbar{\P^*}})}
\\ & \le \sigma_{r-\tilde{r}_{\max}+2}({\overbar{\P^*}}^{(i_3,\ldots,i_d)}) +\frac{1}{\sqrt{\tilde{r}_{\max}}}\Vert\barP-{\overbar{\P^*}}\Vert_F
\\ & = \eta \sigma_{r-\tilde{r}_{\max}+2}(\overbar{\nabla{}f(\mX^*)}^{(i_3,\ldots,i_d)})+\frac{1}{\sqrt{\tilde{r}_{\max}}}\Vert\barP-{\overbar{\P^*}}\Vert_F.
\end{align*} 

Therefore, taking the maximum over all $k_3\in[n_3],\ldots,k_d\in[n_d]$, we obtain that if $r\ge2\cdot{\#\sigma_1}^{\max}-1$ then,
\begin{align} \label{ineq:inProof11_d_dim}
\sigma_{r+1}^{\max}(\barP) & = \max_{k_3\in[n_3],\ldots,k_d\in[n_d]}\sigma_{r+1}(\barP^{(k_3,\ldots,k_d)}) \nonumber
\\ & \le \eta \max_{k_3\in[n_3],\ldots,k_d\in[n_d]}\sigma_{r-\tilde{r}_{\max}+2}(\overbar{\nabla{}f(\mX^*)}^{(k_3,\ldots,k_d)})+\frac{1}{\sqrt{\tilde{r}_{\max}}}\Vert\barP-{\overbar{\P^*}}\Vert_F \nonumber
\\ & = \eta \sigma_{r-\tilde{r}_{\max}+2}^{\max}(\overbar{\nabla{}f(\mX^*)})+\frac{1}{\sqrt{\tilde{r}_{\max}}}\Vert\barP-{\overbar{\P^*}}\Vert_F.
\end{align} 

Plugging \eqref{ineq:inProof3_d_dim} and \eqref{ineq:inProof11_d_dim} into the RHS of \eqref{ineq:inProof10_d_dim} we obtain that
\begin{align} 
& \frac{1}{N}\sum_{i_3=1}^{n_3}\cdots\sum_{i_d=1}^{n_d}\sum_{i=1}^{\#\sigma_{>r+1}^{(i_3,\dots,i_d)}(\mP)}\sigma_i(\barP^{(i_3,\ldots,i_d)})  - \frac{1}{N}\left(\sum_{i_3=1}^{n_3}\cdots\sum_{i_d=1}^{n_d}\#\sigma_{>r+1}^{(i_3,\dots,i_d)}(\mP)\right)\sigma_{r+1}^{\max}(\barP) \nonumber
\\ & \ge 1 +\frac{\eta}{N}\left(\sigma_1(\overbar{\nabla{}f(\mX^*)})-\sigma_{r-\tilde{r}_{\max}+2}^{\max}(\overbar{\nabla{}f(\mX^*)})\right)\sum_{i_3=1}^{n_3}\cdots\sum_{i_d=1}^{n_d}\#\sigma_1^{(i_3,\dots,i_d)} \nonumber
\\ & \ \ \ - \left(\frac{1}{N\sqrt{\tilde{r}_{\max}}}\sum_{i_3=1}^{n_3}\cdots\sum_{i_d=1}^{n_d}\#\sigma_1^{(i_3,\dots,i_d)}+\frac{\sqrt{\tilde{r}_{\max}\cdot\nnzb(\overbar{\nabla{}f(\mX^*)})}}{N}\right)\Vert\barP-{\overbar{\P^*}}\Vert_F  \label{ineq:inProof15_d_dim}
\\ & \underset{(a)}{\ge} 1+\frac{\eta}{N}\left(\sigma_1(\overbar{\nabla{}f(\mX^*)})-\sigma_{r-\tilde{r}_{\max}+2}^{\max}(\overbar{\nabla{}f(\mX^*)})\right)\sum_{i_3=1}^{n_3}\cdots\sum_{i_d=1}^{n_d}\#\sigma_1^{(i_3,\dots,i_d)} \nonumber
\\ & \ \ \ - \left(\frac{1}{\sqrt{N\tilde{r}_{\max}}}\sum_{i_3=1}^{n_3}\cdots\sum_{i_d=1}^{n_d}\#\sigma_1^{(i_3,\dots,i_d)}+\sqrt{\frac{\tilde{r}_{\max}\cdot\nnzb(\overbar{\nabla{}f(\mX^*)})}{N}}\right)(1+\eta\beta)\Vert\mX-\mX^*\Vert_F, \label{ineq:inProof6_d_dim}
\end{align}
where (a) follows from \eqref{ineq:inProof12_d_dim}.

Therefore, we obtain that the condition in \eqref{ineq:inProof13_d_dim} holds if
\begin{align} \label{ineq:inProof5_d_dim}
\Vert\mX-\mX^*\Vert_F \le \frac{\eta\left(\sigma_1(\overbar{\nabla{}f(\mX^*)})-\sigma_{r-\tilde{r}_{\max}+2}^{\max}(\overbar{\nabla{}f(\mX^*)})\right)}{\sqrt{N}\left(\frac{1}{\sqrt{\tilde{r}_{\max}}}+\frac{\sqrt{\tilde{r}_{\max}\cdot\nnzb(\overbar{\nabla{}f(\mX^*)})}}{\sum_{i_3=1}^{n_3}\cdots\sum_{i_d=1}^{n_d}\#\sigma_1^{(i_3,\dots,i_d)}}\right)(1+\eta\beta)}.
\end{align}

Invoking \cref{lemma:svd_opt_grad_d_dim} which implies that $\sum_{i_3=1}^{n_3}\cdots\sum_{i_d=1}^{n_d}\#\sigma_1^{(i_3,\dots,i_d)}=\#\sigma_1(\overbar{\nabla{}f(\mX^*)})$, taking the maximum between the radius in \eqref{ineq:inProof4_d_dim} and  \eqref{ineq:inProof5_d_dim}, and returning to the original notation ${\#\sigma_1}^{\max}=\tilde{r}_{\max}$, we obtain the radius with respect to the Frobenius norm stated in the lemma.

We proceed to obtaining the bound of the spectral radius around the optimal solution.
Bounding the RHS of \eqref{ineq:inProof33333} we have that
\begin{align} \label{ineq:inProof3_d_dim_spectral}
& \frac{1}{N}\sum_{i_3=1}^{n_3}\cdots\sum_{i_d=1}^{n_d}\sum_{i=1}^{\#\sigma_1^{(i_3,\dots,i_d)}}\sigma_i(\barP^{(i_3,\ldots,i_d)}) \nonumber
\\ & \ge 1 +\frac{\eta}{N}\sigma_1(\overbar{\nabla{}f(\mX^*)})\sum_{i_3=1}^{n_3}\cdots\sum_{i_d=1}^{n_d}\#\sigma_1^{(i_3,\dots,i_d)} - \frac{1}{N}\sum_{i_3=1}^{n_3}\cdots\sum_{i_d=1}^{n_d}\#\sigma_1^{(i_3,\dots,i_d)}\Vert\barP-{\overbar{\P^*}}\Vert_2.
\end{align}  

In addition, bounding the RHS of \eqref{ineq:inProof3334} we have that
\begin{align*}
\sigma_{r+1}(\barP^{(i_3,\ldots,i_d)}) & \le
\eta \sigma_{r+1}(\overbar{\nabla{}f(\mX^*)}^{(i_3,\ldots,i_d)})+\Vert\barP-{\overbar{\P^*}}\Vert_2.
\end{align*} 
Taking the maximum over all $k_3\in[n_3],\ldots,k_d\in[n_d]$, we obtain that 
\begin{align} \label{ineq:inProof7_d_dim_spectral}
& \sigma_{r+1}^{\max}(\barP)  = \max_{k_3\in[n_3],\ldots,k_d\in[n_d]}\sigma_{r+1}(\barP^{(k_3,\ldots,k_d)}) \nonumber
\\ & \le \eta \max_{k_3\in[n_3],\ldots,k_d\in[n_d]}\sigma_{r+1}(\overbar{\nabla{}f(\mX^*)}^{(k_3,\ldots,k_d)})+\Vert\barP-{\overbar{\P^*}}\Vert_2 = \eta\sigma_{r+1}^{\max}(\overbar{\nabla{}f(\mX^*)})+\Vert\barP-{\overbar{\P^*}}\Vert_2.
\end{align} 

Plugging \eqref{ineq:inProof3_d_dim_spectral} and \eqref{ineq:inProof7_d_dim_spectral}
into the RHS of \eqref{ineq:inProof10_d_dim} we obtain that
\begin{align} \label{ineq:inProof8_d_dim_spectral}
& \frac{1}{N}\sum_{i_3=1}^{n_3}\cdots\sum_{i_d=1}^{n_d}\sum_{i=1}^{\#\sigma_{>r+1}^{(i_3,\dots,i_d)}(\mP)}\sigma_i(\barP^{(i_3,\ldots,i_d)}) - \frac{1}{N}\left(\sum_{i_3=1}^{n_3}\cdots\sum_{i_d=1}^{n_d}\#\sigma_{>r+1}^{(i_3,\dots,i_d)}(\mP)\right)\sigma_{r+1}^{\max}(\barP) \nonumber
\\ & \ge 1 +\frac{\eta}{N}\left(\sigma_1(\overbar{\nabla{}f(\mX^*)})-\sigma_{r+1}^{\max}(\overbar{\nabla{}f(\mX^*)})\right)\sum_{i_3=1}^{n_3}\cdots\sum_{i_d=1}^{n_d}\#\sigma_1^{(i_3,\dots,i_d)} \nonumber
\\ & \ \ \ - \frac{2}{N}\sum_{i_3=1}^{n_3}\cdots\sum_{i_d=1}^{n_d}\#\sigma_1^{(i_3,\dots,i_d)}\Vert\barP-{\overbar{\P^*}}\Vert_2.
\end{align}

Taking $\barP:=\barX-\eta\overbar{\nabla{}f(\mX)}$ and using the $\beta_2$-smoothness of $f$ with respect to the spectral norm, we have that
\begin{align} \label{ineq:inProof12_d_dim_spectral}
\Vert\barP-{\overbar{\P^*}}\Vert_2 & = \Vert\barX-\eta\overbar{\nabla{}f(\mX)}-\overbar{\X^*}-\eta\overbar{\nabla{}f(\mX^*)}\Vert_2 \nonumber
\le \Vert\barX-\overbar{\X^*}\Vert_2+\eta\Vert\overbar{\nabla{}f(\mX)}-\overbar{\nabla{}f(\mX^*)}\Vert_2
\\ & = \Vert\mX-\mX^*\Vert_2+\eta\Vert\nabla{}f(\mX)-\nabla{}f(\mX^*)\Vert_2 \le (1+\eta\beta_2)\Vert\mX-\mX^*\Vert_2.
\end{align}

Plugging \eqref{ineq:inProof12_d_dim_spectral} into \eqref{ineq:inProof8_d_dim_spectral}, we obtain that
\begin{align*} 
& \frac{1}{N}\sum_{i_3=1}^{n_3}\cdots\sum_{i_d=1}^{n_d}\sum_{i=1}^{\#\sigma_{>r+1}^{(i_3,\dots,i_d)}(\mP)}\sigma_i(\barP^{(i_3,\ldots,i_d)}) - \frac{1}{N}\left(\sum_{i_3=1}^{n_3}\cdots\sum_{i_d=1}^{n_d}\#\sigma_{>r+1}^{(i_3,\dots,i_d)}(\mP)\right)\sigma_{r+1}^{\max}(\barP)
\\ & \ge 1 +\frac{\eta}{N}\left(\sigma_1(\overbar{\nabla{}f(\mX^*)})-\sigma_{r+1}^{\max}(\overbar{\nabla{}f(\mX^*)})\right)\sum_{i_3=1}^{n_3}\cdots\sum_{i_d=1}^{n_d}\#\sigma_1^{(i_3,\dots,i_d)}
\\ & \ \ \ - \frac{2(1+\eta\beta_2)}{N}\sum_{i_3=1}^{n_3}\cdots\sum_{i_d=1}^{n_d}\#\sigma_1^{(i_3,\dots,i_d)}\Vert\mX-\mX^*\Vert_2.
\end{align*}

Rearranging, we finally obtain that the condition in \eqref{ineq:inProof13_d_dim} holds if
\begin{align*}
\Vert\mX-\mX^*\Vert_2 \le \frac{\eta\left(\sigma_1(\overbar{\nabla{}f(\mX^*)})-\sigma_{r+1}^{\max}(\overbar{\nabla{}f(\mX^*)})\right)}{2(1+\eta\beta_2)}.
\end{align*}

\end{proof}

\section{Proof omitted from \cref{sec:nonsmoothCase}}

\subsection{Proof of \cref{lemma:RadiusForExtragradient_d_dim}}
\label{sec:App:proofLemmaNonsmooth}
We first restate the lemma and then prove it.

\begin{theorem}
Assume $\nabla{}_{\mX}F$ is non-zero over the unit TNN ball and fix some saddle-point $(\mX^*,\y^*)$ of Problem \eqref{saddlePointProblem}. 
Denote $\nnzb(\overbar{\nabla_{\mX}F^*}):=\#\lbrace(i_3,\dots,i_d)~|~\overbar{\nabla_{\mX}F(\mX^*,\y^*)}^{(i_3,\ldots,i_d)}\not=0\rbrace$ and ${\#\sigma_1}^{\max}:=\max_{i_3\in[n_3],\ldots,i_d\in[n_d]}\#\sigma_1(\overbar{\nabla_{\mX}F(\mX^*,\y^*)}^{(i_3,\ldots,i_d)})$  and assume ${\#\sigma_1}^{\max}<\min\lbrace n_1,n_2\rbrace$. 
Then, for any $\eta\ge0$, $\min\lbrace n_1,n_2\rbrace>r\ge{\#\sigma_1}^{\max}$, and $(\mX,\y),(\mZ,\w)\in\reals^{n_1\times{}\cdots\times{}n_d}\times\mK$, if 
\begin{align*} 
& \max\lbrace\Vert\mX-\mX^*\Vert_F,\Vert(\mZ,\w)-(\mX^*,\y^*)\Vert\rbrace \nonumber
\\ & \le  \frac{\eta}{\sqrt{N}K}
\max\left\lbrace\frac{\delta(r)}{1+\frac{\sqrt{{\#\sigma_1}^{\max}\cdot\nnzb(\overbar{\nabla_{\mX}F^*})}}{\#\sigma_1(\overbar{\nabla_{\mX}F(\mX^*,\y^*)})}},\frac{\delta(r-{\#\sigma_1}^{\max}+1)}{\frac{1}{\sqrt{{\#\sigma_1}^{\max}}}+\frac{\sqrt{{\#\sigma_1}^{\max}\cdot\nnzb(\overbar{\nabla_{\mX}F^*})}}{\#\sigma_1(\overbar{\nabla_{\mX}F(\mX^*,\y^*)})}}\right\rbrace,
\end{align*} 
where $K=1+\sqrt{2}\eta\max\lbrace\beta_{X},\beta_{Xy}\rbrace$, then $\rank_{\textnormal{t}}(\Pi_{\lbrace\Vert\mY\Vert_*\le1\rbrace}[\mX-\eta\nabla{}_{\mX}F(\mZ,\w)])\le r$.
\end{theorem}

\begin{proof}

For any $(\mZ,\w)\in\reals^{n_1\times{}\cdots\times{}n_d}\times\mK$ it holds that 
\begin{align} \label{ineq:smoothnessEGproof_d_dim}
& \Vert\overbar{\nabla_{\mX}F(\mZ,\w)}-\overbar{\nabla_{\mX}F(\mX^*,\y^*)}\Vert_F \nonumber
\\ & \le \Vert\overbar{\nabla_{\mX}F(\mZ,\w)}-\overbar{\nabla_{\mX}F(\mX^*,\w)}\Vert_F + \Vert\overbar{\nabla_{\mX}F(\mX^*,\w)}-\overbar{\nabla_{\mX}F(\mX^*,\y^*)}\Vert_F \nonumber
\\ & \underset{(a)}{=} \sqrt{N}(\Vert\nabla{}_{\mX}F(\mZ,\w)-\nabla{}_{\mX}F(\mX^*,\w)\Vert_F + \Vert\nabla{}_{\mX}F(\mX^*,\w)-\nabla{}_{\mX}F(\mX^*,\y^*)\Vert_F) \nonumber
\\ & \underset{(b)}{\le} \beta_{X}\sqrt{N}\Vert\mZ-\mX^*\Vert_F+\beta_{Xy}\sqrt{N}\Vert\w-\y^*\Vert_2 \nonumber
\\ & \le \max\lbrace\beta_{X},\beta_{Xy}\rbrace\sqrt{N}(\Vert\mZ-\mX^*\Vert_F+\Vert\w-\y^*\Vert_2) \nonumber
\\ & \le \sqrt{2N}\max\lbrace\beta_{X},\beta_{Xy}\rbrace\Vert(\mZ,\w)-(\mX^*,\y^*)\Vert,
\end{align}
where (a) follows from \cref{lemma:orderPinnerProductEquivalence} and (b) follows from the smoothness of $f$.

Denote $\barP:=\barX-\eta\overbar{\nabla_{\mX}F(\mZ,\w)}$ and $\overbar{\P^*}:=\overbar{\X^*}-\eta\overbar{\nabla_{\mX}F(\mX^*,\y^*)}$. Then, we have that
\begin{align} \label{ineq:PNormBound1_d_dim}
& \Vert\barP-\overbar{\P^*}\Vert_F \nonumber
\\ & \le \Vert\barX-\overbar{\X^*}\Vert_F + \eta\Vert\overbar{\nabla_{\mX}F(\mZ,\w)}-\overbar{\nabla_{\mX}F(\mX^*,\y^*)}\Vert_F \nonumber
\\ & \underset{(a)}{=} \sqrt{N}\Vert\mX-\mX^*\Vert_F +\eta\Vert\overbar{\nabla_{\mX}F(\mZ,\w)}-\overbar{\nabla_{\mX}F(\mX^*,\y^*)}\Vert_F \nonumber
\\ & \underset{(b)}{\le} \sqrt{N}\Vert\mX-\mX^*\Vert_F +\sqrt{2N}\eta\max\lbrace\beta_{X},\beta_{Xy}\rbrace\Vert(\mZ,\w)-(\mX^*,\y^*)\Vert \nonumber
\\ & \le \sqrt{N}(1+\sqrt{2}\eta\max\lbrace\beta_{X},\beta_{Xy}\rbrace)\max\lbrace\Vert\mX-\mX^*\Vert_F,\Vert(\mZ,\w)-(\mX^*,\y^*)\Vert\rbrace,
\end{align}
where (a) follows from \cref{lemma:orderPinnerProductEquivalence} and (b) follows from \eqref{ineq:smoothnessEGproof_d_dim}.

For every $j\in\{1,\dots,\min\{n_1,n_2\}\}$ we denote $\sigma_{j}^{\max}(\barP) = \max_{k_3\in[n_3],\ldots,k_d\in[n_d]}\sigma_{j}(\barP^{(k_3,\ldots,k_d)})$ and $\#\sigma_{>{}j}^{(i_3,\dots,i_d)}(\mP) =\#\left\lbrace i\ \bigg\vert\ \sigma_i(\barP^{(i_3,\ldots,i_d)})>\sigma_{j}^{\max}(\barP)\right\rbrace$.
Then, plugging \eqref{ineq:PNormBound1_d_dim} into the RHS of \eqref{ineq:inProof8_d_dim} and replacing $\overbar{\nabla{}f(\mX^*)}^{(i_3,\ldots,i_d)}$ with $\overbar{\nabla_{\mX}F(\mX^*,\y^*)}^{(i_3,\ldots,i_d)}$ we have that
\begin{align}  \label{ineq:EG1ineq_d_dim}
& \frac{1}{N}\sum_{i_d=1}^{n_d}\cdots\sum_{i_3=1}^{n_3}\sum_{i=1}^{\#\sigma_{>r+1}^{(i_3,\dots,i_d)}(\mP)}\sigma_i(\barP^{(i_3,\ldots,i_d)})-\frac{1}{N}\left(\sum_{i_d=1}^{n_d}\cdots\sum_{i_3=1}^{n_3}\#\sigma_{>r+1}^{(i_3,\dots,i_d)}(\mP)\right)\sigma_{r+1}^{\max}(\barP) \nonumber
\\ & \ge 1 +\frac{\eta}{N}\left(\sigma_1(\overbar{\nabla_{\mX}f(\mX^*,\y^*)})-\sigma_{r+1}^{\max}(\overbar{\nabla_{\mX}f(\mX^*,\y^*)})\right)\sum_{i_d=1}^{n_d}\cdots\sum_{i_3=1}^{n_3}\#\sigma_1^{(i_3,\dots,i_d)} \nonumber
\\ & \ \ \ - \left(\frac{1}{\sqrt{N}}\sum_{i_d=1}^{n_d}\cdots\sum_{i_3=1}^{n_3}\#\sigma_1^{(i_3,\dots,i_d)}+\sqrt{\frac{{\#\sigma_1}^{\max}\cdot\nnzb(\overbar{\nabla_{\mX}F^*})}{N}}\right)(1+\sqrt{2}\eta\max\lbrace\beta_{X},\beta_{Xy}\rbrace) \nonumber
\\ & \quad\quad\ \max\lbrace\Vert\mX-\mX^*\Vert_F,\Vert(\mZ,\w)-(\mX^*,\y^*)\Vert\rbrace,
\end{align}
where $\#\sigma_1^{(i_3,\dots,i_d)}$ denotes the multiplicity of $\sigma_1(\overbar{\nabla_{\mX}f(\mX^*,\y^*)}^{(i_3,\ldots,i_d)})$.

Alternatively, plugging \eqref{ineq:PNormBound1_d_dim} into the RHS of \eqref{ineq:inProof15_d_dim} and replacing $\overbar{\nabla{}f(\mX^*)}^{(i_3,\ldots,i_d)}$ with $\overbar{\nabla_{\mX}F(\mX^*,\y^*)}^{(i_3,\ldots,i_d)}$ we have that 
\begin{align} \label{ineq:EG2ineq_d_dim}
& \frac{1}{N}\sum_{i_d=1}^{n_d}\cdots\sum_{i_3=1}^{n_3}\sum_{i=1}^{\#\sigma_{>r+1}^{(i_3,\dots,i_d)}(\mP)}\sigma_i(\barP^{(i_3,\ldots,i_d)})-\frac{1}{N}\left(\sum_{i_d=1}^{n_d}\cdots\sum_{i_3=1}^{n_3}\#\sigma_{>r+1}^{(i_3,\dots,i_d)}(\mP)\right)\sigma_{r+1}^{\max}(\barP) \nonumber
\\ & \ge 1 +\frac{\eta}{N}\left(\sigma_1(\overbar{\nabla_{\mX}f(\mX^*,\y^*)})-\sigma_{r-{\#\sigma_1}^{\max}+2}^{\max}(\overbar{\nabla_{\mX}f(\mX^*,\y^*)})\right)\sum_{i_d=1}^{n_d}\cdots\sum_{i_3=1}^{n_3}\#\sigma_1^{(i_3,\dots,i_d)} \nonumber
\\ & \ \ \ - \left(\frac{1}{\sqrt{N\cdot{\#\sigma_1}^{\max}}}\sum_{i_d=1}^{n_d}\cdots\sum_{i_3=1}^{n_3}\#\sigma_1^{(i_3,\dots,i_d)}+\sqrt{\frac{{\#\sigma_1}^{\max}\cdot\nnzb(\overbar{\nabla_{\mX}F^*})}{N}}\right) \nonumber
\\ & \quad\quad\ (1+\sqrt{2}\eta\max\lbrace\beta_{X},\beta_{Xy}\rbrace)\max\lbrace\Vert\mX-\mX^*\Vert_F,\Vert(\mZ,\w)-(\mX^*,\y^*)\Vert\rbrace.
\end{align}

Therefore, from \eqref{ineq:EG1ineq_d_dim} and \eqref{ineq:EG2ineq_d_dim} we obtain that the condition $$\frac{1}{N}\sum_{i_d=1}^{n_d}\cdots\sum_{i_3=1}^{n_3}\left(\sum_{i=1}^{\#\sigma_{>r+1}^{(i_3,\dots,i_d)}(\mP)}\sigma_i(\barP^{(i_3,\ldots,i_d)})-\#\sigma_{>r+1}^{(i_3,\dots,i_d)}(\mP)\cdot\sigma_{r+1}^{\max}(\barP)\right)\ge1$$ holds if 
\begin{align*}
& \max\lbrace\Vert\mX-\mX^*\Vert_F,\Vert(\mZ,\w)-(\mX^*,\y^*)\Vert\rbrace 
\\ & \le \frac{\eta}{\sqrt{N}K}
\max\left\lbrace\frac{\delta(r)}{1+\frac{\sqrt{{\#\sigma_1}^{\max}\cdot\nnzb(\overbar{\nabla_{\mX}F^*})}}{\sum_{i_d=1}^{n_d}\cdots\sum_{i_3=1}^{n_3}\#\sigma_1^{(i_3,\dots,i_d)}}},\frac{\delta(r-{\#\sigma_1}^{\max}+1)}{\frac{1}{\sqrt{{\#\sigma_1}^{\max}}}+\frac{\sqrt{{\#\sigma_1}^{\max}\cdot\nnzb(\overbar{\nabla_{\mX}F^*})}}{\sum_{i_d=1}^{n_d}\cdots\sum_{i_3=1}^{n_3}\#\sigma_1^{(i_3,\dots,i_d)}}}\right\rbrace
\end{align*} 
which by \cref{lemma:iffLowRankfCondiotionOrderP} implies that under this condition $\rank_{\textnormal{t}}(\Pi_{\lbrace\Vert\mY\Vert_*\le1\rbrace}[\mX-\eta\nabla{}_{\mX}F(\mZ,\w)])\le r$.
\end{proof}

\section{Some Algorithmic Consequences}
\label{sec:algorithmicConsequences}

In this section we consider several first-order methods for solving Problem \eqref{mainModel} and the nonsmooth case in Problem \eqref{mainModel}. Using the results we obtained in \cref{lemma:RadiusForProjectedGradient_d_dim} and \cref{lemma:RadiusForExtragradient_d_dim} we show that for projected gradient decent, Nesterov's accelerated gradient method, and projected extragradient for saddle-point problems all converge with their standard convergence rates to the optimal solution when initializing with some ``warm  start", while only requiring low-rank gradient mappings. Furthermore, we show that under a quadratic growth condition, in addition to the low-rank mappings, projected gradient decent obtains a linear convergence rate.


\subsection{Smooth setting}
\label{sec:algorithmicConsequencesSmooth}

In this section we use the following notations.
For any optimal solution $\mX^*\in\lbrace\mX\ \vert\ \Vert\mX\Vert_*\le1\rbrace$ of Problem \eqref{mainModel}, we denote ${\#\sigma_1}^{\max}:=\max_{i_3\in[n_3],\ldots,i_d\in[n_d]}\#\sigma_1(\overbar{\nabla{}f(\mX^*)}^{(i_3,\ldots,i_d)})$ and $\nnzb(\overbar{\nabla{}f(\mX^*)}):=\#\lbrace(i_3,\dots,i_d)~|~\overbar{\nabla{}f(\mX^*)}^{(i_3,\ldots,i_d)}\not=0\rbrace$. For any $\min\lbrace n_1,n_2\rbrace>r\ge{\#\sigma_1}^{\max}$
the term $\delta(r)$ is as defined in \cref{def:GSC}.

\begin{theorem}[local convergence of projected gradient decent] 
\label{thm:PGD}
Fix an optimal solution $\mX^*$ to Problem \eqref{mainModel}. Let $\lbrace\mX_t\rbrace_{t\ge1}$ be the sequence of iterates produced by the Projected Gradient Decent method:
\begin{align*}
& \mX_1\in\lbrace\mX\in\reals^{n_1\times\cdots\times n_d}\ \vert\ \Vert\mX\Vert_*\le1\rbrace, 
\\ & \forall t\ge1:\ \mX_{t+1}=\Pi_{\lbrace\mX\in\reals^{n_1\times\cdots\times n_d}\ \vert\ \Vert\mX\Vert_*\le1\rbrace}[\mX_t-\frac{1}{\beta}\nabla{}f(\mX_t)].
\end{align*}
Let $\min\lbrace n_1,n_2\rbrace>r\ge{\#\sigma_1}^{\max}$. Assume the initialization $\mX_1$ satisfies  $\Vert\mX_1-\mX^*\Vert_F \le R_0(r)$, where
\begin{align*}
& R_0(r):=
 \frac{\eta}{2\beta\sqrt{N}}\max\left\lbrace\frac{\delta(r)}{1+\frac{\sqrt{{\#\sigma_1}^{\max}\cdot\nnzb(\overbar{\nabla{}f(\mX^*)})}}{\#\sigma_1(\overbar{\nabla{}f(\mX^*)})}},\frac{\delta(r-{\#\sigma_1}^{\max}+1)}{\frac{1}{\sqrt{{\#\sigma_1}^{\max}}}+\frac{\sqrt{{\#\sigma_1}^{\max}\cdot\nnzb(\overbar{\nabla{}f(\mX^*)})}}{\#\sigma_1(\overbar{\nabla{}f(\mX^*)})}}\right\rbrace.
\end{align*}
Then, for all $t\ge1$, the projections $\Pi_{\lbrace\mX\in\reals^{n_1\times\cdots\times n_d}\ \vert\ \Vert\mX\Vert_*\le1\rbrace}[\cdot]$ throughout the run of the algorithm, could be replaced with rank-r truncated projections (see \cref{def:lowRankProjection}) without changing the sequence $\lbrace\mX_t\rbrace_{t\ge1}$. In particular, for all $t\ge1$ it holds that
\begin{align*}
f(\mX_t)-f(\mX^*)\le \frac{\beta\Vert\mX_1-\mX^*\Vert_F^2}{2(t-1)}.
\end{align*}
Furthermore, if $f$ is of the form $f(\mX)=g(\mA(\mX))+\langle\mC,\mX\rangle$, where  $g$ is $\alpha$-strongly convex and $\beta$-smooth and  $\mA:\reals^{n_1\times\cdots\times n_d}\rightarrow\reals^m$ is a linear map, and the optimal solution $\mX^*$ is unique and satisfies strict complementarity (Definition in \cref{lemma:SCequivalence}), then there exists $\gamma > 0$ (see  \cref{remark:QGparameter}) such that for all $t\ge1$,
 \begin{align*}
f(\mX_{t})-f(\mX^*)\le \frac{\beta\Vert\mX_1-\mX^*\Vert^2}{2}\left(\frac{1}{1+\frac{2\gamma}{\beta}}\right)^{t-1}.
\end{align*}
\end{theorem}

\begin{proof}
If for all $t\ge1$ it holds that $\Vert\mX_t-\mX^*\Vert_F \le R_0(r)$, then by \cref{lemma:RadiusForProjectedGradient_d_dim} it follows that $\rank_{\textnormal{t}}(\Pi_{\lbrace\Vert\mY\Vert_*\le1\rbrace}[\mX-\eta\nabla{}f(\mX)])\le r$, and therefore, the projection $\Pi_{\lbrace\Vert\mY\Vert_*\le1\rbrace}[\mX-\eta\nabla{}f(\mX)])$ could be replaced with its rank-$r$ truncated counterpart, without any change to the result. Thus, the standard convergence rate result for the projected gradient decent method, which is known to be 
\begin{align*}
f(\mX_t)-f(\mX^*)\le \frac{\beta\Vert\mX_1-\mX^*\Vert_F^2}{2(t-1)}
\end{align*} 
(see for instance Theorem 9.16 in \cite{beckIntroduction}), still holds.

Since the initialization $\mX_1$ satisfies that $\Vert\mX_1-\mX^*\Vert_F \le R_0(r)$, and using the known result 
 that the iterates of the projected gradient decent method satisfy that
$\Vert\mX_{t+1}-\mX^*\Vert_F\le\Vert\mX_{t}-\mX^*\Vert_F$ for all $t\ge1$ (see for instance lemma 9.17 in \cite{beckIntroduction}), it follows that indeed for or all $t\ge1$, $\Vert\mX_t-\mX^*\Vert_F \le R_0(r)$

For the second part of the theorem, under strict complementarity and assuming there exists a unique optimal solution, \cref{thm:QG} implies that quadratic growth holds with some parameter $\gamma>0$. Therefore, as discussed for instance in Section 5.1 in \cite{nesterov_LCwithQG},  the projected gradient decent method converges with the  linear rate specified in the theorem.
\end{proof}

\begin{algorithm}
	\caption{Nesterov's restarted fast gradient method \cite{nesterov_LCwithQG}}\label{alg:RFGM}
	\begin{algorithmic}
	    \STATE \textbf{Input:} $K\in\mathbb{N}$
		\STATE \textbf{Initialization:} $\mX_1\in\lbrace\mX\in\reals^{n_1\times\cdots\times n_d}\ \vert\ \Vert\mX\Vert_*\le1\rbrace$
		\FOR {$t=1,2,\ldots$}
		\STATE $\mX_{t,1}=\mY_{t,1}=\mX_t$
		\STATE $\theta_1=1$
		\FOR {$s=1,2,\ldots,K$}
			\STATE $\mX_{t,s+1}=\Pi_{\lbrace\mX\in\reals^{n_1\times\cdots\times n_d}\ \vert\ \Vert\mX\Vert_*\le1\rbrace}[\mY_{t,s}-\frac{1}{\beta}\nabla{}f(\mY_{t,s})]$
			\STATE $\theta_{s+1} =\frac{1+\sqrt{1+4\theta_s^2}}{2}$
			\STATE $\mY_{t,s+1}=\mX_{t,s+1}+\frac{\theta_{s}-1}{\theta_{s+1}}(\mX_{t,s+1}-\mX_{t,s})$
		\ENDFOR
		\STATE $\mX_{t+1}=\mX_{t,K+1}$
		\ENDFOR
	\end{algorithmic}
\end{algorithm}

\begin{theorem}[local convergence of Nesterov's restarted fast gradient method method under quadratic growth] 
\label{thm:acceleratedWithQG}
Assume $f$ in Problem \eqref{mainModel} is of the form $f(\mX)=g(\mA(\mX))+\langle\mC,\mX\rangle$, where $g$ is $\alpha$-strongly convex and {$\beta$-smooth},  $\mA:\reals^{n_1\times\cdots\times n_d}\rightarrow\reals^m$ is a linear map, and that the optimal solution $\mX^*$ is unique and satisfies strict complementarity (Definition in \cref{lemma:SCequivalence}). Let $\lbrace\mX_t\rbrace_{t\ge1}$ be the sequence of iterates produced by Nesterov's restart fast gradient method, \cref{alg:RFGM}, with $K=\lceil(\sqrt{2\beta}e)/\sqrt{\gamma}\rceil$, where $\gamma>0$ is as discussed in \cref{remark:QGparameter}.
Let $\min\lbrace n_1,n_2\rbrace>r\ge{\#\sigma_1}^{\max}$. If the initialization $\mX_1$ satisfies that $\Vert\mX_1-\mX^*\Vert_F \le \frac{\sqrt{2}\min\left\lbrace\gamma,\sqrt{\gamma}\right\rbrace}{3\sqrt{\beta}} R_0(r)$, where
\begin{align*}
& R_0(r):=
 \frac{\eta}{2\beta\sqrt{N}}\max\left\lbrace\frac{\delta(r)}{1+\frac{\sqrt{{\#\sigma_1}^{\max}\cdot\nnzb(\overbar{\nabla{}f(\mX^*)})}}{\#\sigma_1(\overbar{\nabla{}f(\mX^*)})}},\frac{\delta(r-{\#\sigma_1}^{\max}+1)}{\frac{1}{\sqrt{{\#\sigma_1}^{\max}}}+\frac{\sqrt{{\#\sigma_1}^{\max}\cdot\nnzb(\overbar{\nabla{}f(\mX^*)})}}{\#\sigma_1(\overbar{\nabla{}f(\mX^*)})}}\right\rbrace
\end{align*}
then, for all $t\ge1$, the projections $\Pi_{\lbrace\mX\in\reals^{n_1\times\cdots\times n_d}\ \vert\ \Vert\mX\Vert_*\le1\rbrace}[\cdot]$ could be replaced with their rank-r truncated counterparts (see \cref{def:lowRankProjection}) without changing the sequence $\lbrace\mX_t\rbrace_{t\ge1}$, and for all $t\ge1$ it holds that
\begin{align*}
f(\mX_t)-f(\mX^*)\le \exp\left(-\frac{\sqrt{2\gamma}}{\sqrt{\beta}e}(t-1)\right)(f(\mX_1)-f(\mX^*)).
\end{align*}
\end{theorem}

\begin{proof}
By the update rule for $\mY_{t,s+1}$, for any $t\ge1$ and $s\in[K]$, we have that
\begin{align*}
\Vert\mY_{t,s+1}-\mX_{t,s+1}\Vert_F = \frac{\theta_{s}-1}{\theta_{s+1}}\Vert\mX_{t,s+1}-\mX_{t,s}\Vert_F \le \Vert\mX_{t,s+1}-\mX_{t,s}\Vert_F,
\end{align*}
where the last inequality holds using the update of $\theta_s$, which implies that $\theta_{s+1}^2-\theta_{s+1}=\theta_{s}^2$, and therefore, since for all $s\in[K]$ it holds that $\theta_{s+1}\ge1$ (see Lemma 10.33 in \cite{beckOptimizationBook}), it follows that $\frac{\theta_{s}-1}{\theta_{s+1}}=\sqrt{1-\frac{1}{\theta_{s+1}}}-\frac{1}{\theta_{s+1}}\le1$.

Thus, we have that
\begin{align*}
\Vert\mY_{t,s+1}-\mX^*\Vert_F & \le \Vert\mY_{t,s+1}-\mX_{t,s+1}\Vert_F+\Vert\mX_{t,s+1}-\mX^*\Vert_F
\\ & \le \Vert\mX_{t,s+1}-\mX_{t,s}\Vert_F+\Vert\mX_{t,s+1}-\mX^*\Vert_F
\\ & \le \Vert\mX_{t,s}-\mX^*\Vert_F+2\Vert\mX_{t,s+1}-\mX^*\Vert_F.
\end{align*} 

We will prove by induction that for every $t\ge1$ and any $s\in[K]$ it holds that $\Vert\mX_{t,s}-\mX^*\Vert_F\le R_0(r)/3$. Together with the inequality above, this will imply that for every $t\ge1$ and any $s\in[K]$, $\Vert\mY_{t,s+1}-\mX^*\Vert_F\le R_0(r)$, which by \cref{lemma:RadiusForProjectedGradient_d_dim} implies that all projections onto the unit TNN ball could be replaced with their rank-r truncated counterparts, without any change to the outcome. Thus, the original convergence rate of the restart fast gradient method will be kept, which as established in Section 5.2.2 in \cite{nesterov_LCwithQG}, is
\begin{align*}
f(\mX_t)-f(\mX^*)\le \exp\left(-\frac{\sqrt{2\gamma}}{\sqrt{\beta}e}(t-1)\right)(f(\mX_1)-f(\mX^*)).
\end{align*} 
The base case of the induction $t=s=1$ holds directly by our initialization choice. Assume that the claim holds up to some $t\ge1$ and some $s\in[K]$. Then, all the iterates computed up to $\mX_{s,t}$ and $\mY_{t,s}$ are identical to the iterates that would have been computed if using full-rank SVD computations for the projections, and so, the original convergence rate of the fast gradient method with our choice of $\theta_s$ (see theorem 10.34 in \cite{beckOptimizationBook}) of each epoch is maintained, i.e., it holds that
\begin{align}
f(\mX_{t,s+1})-f(\mX^*) & \le \frac{2\beta\Vert\mX_{t,1}-\mX^*\Vert_F^2}{(s+1)^2} = \frac{2\beta\Vert\mX_{t}-\mX^*\Vert_F^2}{(s+1)^2}   \label{CR:FGM}
\\ & \le \frac{\beta}{2}\Vert\mX_{t}-\mX^*\Vert_F^2. \label{CR:FGM_l2}
\end{align}
In particular, for $t=1$ we have that for any $s\in[K]$ it holds that
\begin{align} \label{ineq:inProof553}
\Vert\mX_{1,s+1}-\mX^*\Vert_F^2 \underset{(a)}{\le} \frac{1}{\gamma}(f(\mX_{1,s+1})-f(\mX^*)) & \underset{(b)}{\le} \frac{\beta}{2\gamma}\Vert\mX_{1}-\mX^*\Vert_F^2,
\end{align}
where (a) follows from the quadratic growth property, and (b) follows from \eqref{CR:FGM_l2}. Otherwise, if $t>1$ then we have that
\begin{align} \label{ineq:inProof557}
\Vert\mX_{t,s+1}-\mX^*\Vert_F^2 \underset{(a)}{\le} \frac{1}{\gamma}(f(\mX_{t,s+1})-f(\mX^*)) & \underset{(b)}{\le} \frac{\beta}{2\gamma}\Vert\mX_{t}-\mX^*\Vert_F^2 \underset{(c)}{\le} \frac{\beta}{2\gamma^2}(f(\mX_{t})-f(\mX^*)),
\end{align}
where both (a) and (c) follow from the quadratic growth property, and (b) follows from \eqref{CR:FGM_l2}. 

In addition,
\begin{align*}
f(\mX_{t})-f(\mX^*) & = f(\mX_{t-1,K+1})-f(\mX^*) \underset{(a)}{\le} \frac{2\beta\Vert\mX_{t-1}-\mX^*\Vert_F^2}{(K+1)^2}\le \frac{2\beta\Vert\mX_{t-1}-\mX^*\Vert_F^2}{K^2}
\\ & \underset{(b)}{=} \frac{\gamma}{e^2}\Vert\mX_{t-1}-\mX^*\Vert_F^2 \underset{(c)}{\le} \frac{1}{e^2}(f(\mX_{t-1})-f(\mX^*)) \le f(\mX_{t-1})-f(\mX^*),
\end{align*}
where (a) follows from  \eqref{CR:FGM}, (b) follows from our choice of $K$, and (c) follows from the quadratic growth property. Unrolling the  recursion, we have that
\begin{align} \label{ineq:inProof552}
f(\mX_{t})-f(\mX^*) &  \le \cdots \le f(\mX_{2})-f(\mX^*) = f(\mX_{1,K+1})-f(\mX^*) \nonumber
\\ & \underset{(a)}{\le} \frac{2\beta\Vert\mX_{1}-\mX^*\Vert_F^2}{(K+1)^2}
\underset{(b)}{\le} \Vert\mX_{1}-\mX^*\Vert_F^2,
\end{align}
where here too (a) follows from  \eqref{CR:FGM}, and (b) follows from our choice of $K$. 

Plugging-in \eqref{ineq:inProof552} into the RHS of \eqref{ineq:inProof557}, we obtain that for $t>1$, it holds that
\begin{align} \label{ineq:inProof558}
\Vert\mX_{t,s+1}-\mX^*\Vert_F^2 \le \frac{\beta}{2\gamma^2}\Vert\mX_{1}-\mX^*\Vert_F^2.
\end{align}
Therefore, choosing $\mX_1$ such that $\Vert\mX_{1}-\mX^*\Vert_F \le \frac{\sqrt{2}\min\left\lbrace\gamma,\sqrt{\gamma}\right\rbrace}{3\sqrt{\beta}} R_0(r)$, it follows by taking the maximum  between \eqref{ineq:inProof553} and \eqref{ineq:inProof558}, that for any $t\ge1$ and $s\in[K]$, indeed $\Vert\mX_{t,s+1}-\mX^*\Vert_F \le R_0(r)/3$ as desired.

\end{proof}

For accelerated gradient decent with a unique optimal solution $\mX^*$, but in case the quadratic growth result does not hold (e.g., when $f$ does not admit the structure $f(\mX) = g(\mA(\mX))+\langle{\mC,\mX}\rangle$ with strongly convex $g$ etc.), 
a stronger initialization condition is required to ensure that throughout the run, all iterates remain inside a certain ball so that \cref{lemma:RadiusForProjectedGradient_d_dim} could be applied. In this case, it is also sufficient to use the standard accelerated gradient method, i.e., without restarts. The sufficient  condition so that all the projections $\Pi_{\lbrace\mX\in\reals^{n_1\times\cdots\times n_d}\ \vert\ \Vert\mX\Vert_*\le1\rbrace}[\cdot]$ could be replaced with their rank-r counterparts, is to initialize it with $\mX_1$ such that 
\begin{align*}
& \max\lbrace\Vert\mX_1-\mX^*\Vert_F,3R(\mX_1)\rbrace
\\ & \le  \frac{\eta}{2\beta\sqrt{N}}\max\left\lbrace\frac{\delta(r)}{1+\frac{\sqrt{{\#\sigma_1}^{\max}\cdot\nnzb(\overbar{\nabla{}f(\mX^*)})}}{\#\sigma_1(\overbar{\nabla{}f(\mX^*)})}},\frac{\delta(r-{\#\sigma_1}^{\max}+1)}{\frac{1}{\sqrt{{\#\sigma_1}^{\max}}}+\frac{\sqrt{{\#\sigma_1}^{\max}\cdot\nnzb(\overbar{\nabla{}f(\mX^*)})}}{\#\sigma_1(\overbar{\nabla{}f(\mX^*)})}}\right\rbrace,
\end{align*}
where
\begin{align*}
R(\mX) := \sup_{\Vert\mZ\Vert_*\le1:\ f(\mZ)\le f(\mX^*)+2\beta\Vert\mX-\mX^*\Vert_F^2}\Vert\mZ-\mX^*\Vert_F,
\end{align*}
and the rest of the parameters are as defined in \cref{thm:acceleratedWithQG}.
See Theorem 5 in \cite{garberNuclearNormMatrices} for a complete proof.

\subsection{Nonsmooth setting}
\label{sec:algorithmicConsequencesNonmooth}

In this section we introduce the following notation.
For any optimal solution $\mX^*\in\lbrace\mX\in\reals^{n_1\times\cdots\times n_d}\ \vert\ \Vert\mX\Vert_*\le1\rbrace$ of Problem \eqref{mainModel} and $\mG^*\in\partial{}f(\mX^*)$, we denote  $\nnzb(\G^*):=\#\lbrace(i_3,\dots,i_d)~|~\overbar{\G^*}^{(i_3,\ldots,i_d)}\not=0\rbrace$ and ${\#\sigma_1}^{\max}:=\max_{i_3\in[n_3],\ldots,i_d\in[n_d]}\#\sigma_1(\overbar{\G^*})$, where $\overbar{\G^*}:=\bdiag(\overbar{\mG^*})$, and for any $\min\lbrace n_1,n_2\rbrace>r\ge{\#\sigma_1}^{\max}$
 the term $\delta(r)$ is as defined in \cref{lemma:SCnonsmooth}.


\begin{theorem}[local convergence of projected extragradient] 
\label{thm:nonsmoothEG}
Fix an optimal solution $\mX^*\in\lbrace\mX\in\reals^{n_1\times\cdots\times n_d}\ \vert\ \Vert\mX\Vert_*\le1\rbrace$ to Problem \eqref{mainModel} and assume \cref{lemma:structureOfNonsmoothAssumption} holds. Let $\mG^*\in\partial{}f(\mX^*)$ which satisfies that $\langle\mX-\mX^*,\mG^*\rangle\ge0$ for all $\mX\in\lbrace\mX\in\reals^{n_1\times\cdots\times n_d}\ \vert\ \Vert\mX\Vert_*\le1\rbrace$. 
Define $F$ as in Problem \eqref{saddlePointProblem} and 
%
let $\lbrace(\mX_t,\y_t)\rbrace_{t\ge1}$ and $\lbrace(\mZ_t,\w_t)\rbrace_{t\ge2}$ be the sequences of iterates produced by the projected extragradient method, \cref{alg:EG}, with a fixed step-size:
\begin{align*}
\eta = \min\left\lbrace \frac{1}{2\sqrt{\beta_X^2+\beta_{yX}^2}}, \frac{1}{2\sqrt{\beta_y^2+\beta_{Xy}^2}} , \frac{1}{\beta_X+\beta_{Xy}} , \frac{1}{\beta_y+\beta_{yX}}  \right\rbrace,
\end{align*}
where $\beta_X,\beta_y,\beta_{Xy},\beta_{yX}$ are as defined in \eqref{def:betas}.
Let $\min\lbrace n_1,n_2\rbrace>r\ge{\#\sigma_1}^{\max}$. Assume the initialization $(\mX_1,\y_1)$ satisfies that $\Vert(\mX_1,\y_1)-(\mX^*,\y^*)\Vert_F \le R_0(r),$ where
\begin{align*}
& R_0(r):=
\\ & \frac{(1+\sqrt{2})^{-1}\eta}{\sqrt{N}\left(1+\sqrt{2}\eta\max\lbrace\beta_{X},\beta_{Xy}\rbrace\right)}
\max\left\lbrace\frac{\delta(r)}{1+\frac{\sqrt{{\#\sigma_1}^{\max}\cdot\nnzb(\G^*)}}{\#\sigma_1(\overbar{\G^*})}},\frac{\delta(r-{\#\sigma_1}^{\max}+1)}{\frac{1}{\sqrt{{\#\sigma_1}^{\max}}}+\frac{\sqrt{{\#\sigma_1}^{\max}\cdot\nnzb(\G^*)}}{\#\sigma_1(\overbar{\G^*})}}\right\rbrace .
\end{align*}
Then, for all $t\ge1$, the projections $\Pi_{\lbrace\mX\in\reals^{n_1\times\cdots\times n_d}\ \vert\ \Vert\mX\Vert_*\le1\rbrace}[\cdot]$ could be replaced with their rank-$r$ truncated counterparts (see \cref{def:lowRankProjection}) without changing the sequences $\lbrace(\mX_t,\y_t)\rbrace_{t\ge1}$ and $\lbrace(\mZ_t,\w_t)\rbrace_{t\ge2}$, and for all $T\ge0$ it holds that
\begin{align*}
& \min_{t\in[T]}f(\mZ_{t+1})-f(\mX^*)
\le \frac{D^2\max\left\lbrace\sqrt{\beta_X^2+\beta_y^2}, \sqrt{\beta_y^2+\beta_{Xy}^2}, \frac{1}{2}(\beta_X+\beta_{Xy}), \frac{1}{2}(\beta_y+\beta_{yX}) \right\rbrace}{T},
\end{align*}
where $D:=\sup_{(\mX,\y),(\mZ,\w)\in \lbrace\mX\in\reals^{n_1\times\cdots\times n_d}\ \vert\ \Vert\mX\Vert_*\le1\rbrace\times\mK}\Vert(\mX,\y)-(\mZ,\w)\Vert$.
\end{theorem}


\begin{proof} 
Denote $\beta=\sqrt{2}\max\lbrace\sqrt{\beta_X^2+\beta_{yX}^2},\sqrt{\beta_y^2+\beta_{Xy}^2}\rbrace$.
By Lemma 8 in \cite{ourExtragradient} (which is not unique for the matrix case but holds for any finite Euclidean space), for all $t\ge2$, the iterates of the projected extragradient method satisfy that
\begin{align*}
\Vert(\mX_{t},\y_{t})-(\mX^*,\y^*)\Vert \le \Vert(\mX_{t-1},\y_{t-1})-(\mX^*,\y^*)\Vert,
\end{align*}
and for all $t\ge1$ they satisfy that
\begin{align*}
\Vert(\mZ_{t+1},\w_{t+1})-(\mX^*,\y^*)\Vert & \le  \left(1+\frac{1}{\sqrt{1-\eta^2\beta^2}}\right)\Vert(\mX_{t},\y_{t})-(\mX^*,\y^*)\Vert
\\ &  \le (1+\sqrt{2})\Vert(\mX_{t},\y_{t})-(\mX^*,\y^*)\Vert,
\end{align*}
where the last inequality follows from our choice of $\eta$.
Therefore, unrolling the recursion and using our initialization choice of $(\mX_1,\y_1)$, we obtain that for all $t\ge1$, 
\begin{align*}
& \max\left\lbrace\Vert(\mX_{t},\y_{t})-(\mX^*,\y^*)\Vert,\Vert(\mZ_{t+1},\w_{t+1})-(\mX^*,\y^*)\Vert\right\rbrace 
\\ & \le(1+\sqrt{2})\Vert(\mX_{t-1},\y_{t-1})-(\mX^*,\y^*)\Vert \le \ldots \le (1+\sqrt{2})\Vert(\mX_{1},\y_{1})-(\mX^*,\y^*)\Vert 
\\ & \le (1+\sqrt{2})R_0(r).
\end{align*}
Since for all $t\ge1$ it holds that 
\begin{align*}
\Vert\mX_{t}-\mX^*\Vert & \le \Vert(\mX_{t},\y_{t})-(\mX^*,\y^*)\Vert 
\\ & \le\max\left\lbrace\Vert(\mX_{t},\y_{t})-(\mX^*,\y^*)\Vert,\Vert(\mZ_{t+1},\w_{t+1})-(\mX^*,\y^*)\Vert\right\rbrace,
\end{align*}
we have that for all $t\ge1$ the condition in \cref{lemma:RadiusForExtragradient_d_dim} holds for $(\mX_{t},\y_{t})$ and $(\mZ_{t+1},\w_{t+1})$, and so for all $t\ge1$ it follows that 
\begin{align*}
& \rank_{\textnormal{t}}(\Pi_{\lbrace\Vert\mY\Vert_*\le1\rbrace}[\mX_t-\eta\nabla{}_{\mX}F(\mX_t,\y_t)])\le r
\\ & \rank_{\textnormal{t}}(\Pi_{\lbrace\Vert\mY\Vert_*\le1\rbrace}[\mX_t-\eta\nabla{}_{\mX}F(\mZ_{t+1},\w_{t+1})])\le r.
\end{align*}
Hence, the the iterates of projected extragradient method will remain unchanged when replacing all projections onto the unit TNN ball with their rank-$r$ truncated counterparts, and so, the method will also maintain its original convergence rate stated in \cite{ourExtragradient}, i.e.,
\begin{align} \label{ineq:convergenceRateSP}
& \frac{1}{T}\sum_{t=1}^{T}\max_{\y\in\mK}F(\mZ_{t+1},\y)-\frac{1}{T}\sum_{t=1}^{T}\min_{\mX\in\lbrace\mX\in\reals^{n_1\times\cdots\times n_d}\ \vert\ \Vert\mX\Vert_*\le1\rbrace}F(\mX,\w_{t+1}) \nonumber
\\ & \le \frac{D^2\max\left\lbrace\sqrt{\beta_X^2+\beta_y^2}, \sqrt{\beta_y^2+\beta_{Xy}^2}, \frac{1}{2}(\beta_X+\beta_{Xy}), \frac{1}{2}(\beta_y+\beta_{yX}) \right\rbrace}{T}.
\end{align}
Since we assume \cref{lemma:structureOfNonsmoothAssumption} holds, invoking \cref{lemma:connectionSubgradientNonsmoothAndSaddlePoint}
we know that there exists a point $\y^*\in\argmax_{\y\in\mathcal{K}}F(\mX^*,\y)$ such that $(\mX^*,\y^*)$ is a saddle-point of Problem \eqref{saddlePointProblem}, and $\nabla_{\mX}F(\mX^*,\y^*)=\mG^*$. Therefore, we can replace $\nabla_{\mX}F(\mX^*,\y^*)$ with $\mG^*$ in the assumptions and radius in \cref{lemma:RadiusForExtragradient_d_dim} to obtain the radius written in the statement of this theorem.

We can now use the relationship $f(\mX) = \max_{\y\in\mK} F(\mX,\y)$ to bound
\begin{align} \label{ineq:convergenceRateNS_1}
\min_{t\in[T]}f(\mZ_{t+1}) = \min_{t\in[T]}\max_{\y\in\mK} F(\mZ_{t+1},\y) \le \frac{1}{T}\sum_{t=1}^{T}\max_{\y\in\mK}F(\mZ_{t+1},\y)
\end{align}
and
\begin{align} \label{ineq:convergenceRateNS_2}
\frac{1}{T}\sum_{t=1}^{T}\min_{\mX\in\lbrace\mX\ \vert\ \Vert\mX\Vert_*\le1\rbrace}F(\mX,\w_{t+1}) \le \frac{1}{T}\sum_{t=1}^{T}F(\mX^*,\w_{t+1}) \le \max_{\y\in\mK} F(\mX^*,\y) = f(\mX^*).
\end{align}
Plugging \eqref{ineq:convergenceRateNS_1} and \eqref{ineq:convergenceRateNS_2} into the RHS of \eqref{ineq:convergenceRateSP}, we obtain the convergence rate for the nonsmooth problem in the theorem.

\end{proof}

\bibliographystyle{plain}
\bibliography{bibs}

\end{document}

%% file: defs.tex
\newtheorem{theorem} {Theorem}
\newtheorem{lemma} {Lemma}
\newtheorem{definition} {Definition}
\newtheorem{corollary} {Corollary}

\newtheorem{assumption} {Assumption}

\newtheorem{remark} {Remark}

\def\a{{\mathbf{a}}}
\def\b{{\mathbf{b}}}
\def\c{{\mathbf{c}}}
\def\d{{\mathbf{d}}}
\def\x{{\mathbf{x}}}

\def\e{{\mathbf{e}}}
\def\u{{\mathbf{u}}}
\def\v{{\mathbf{v}}}
\def\z{{\mathbf{z}}}
\def\w{{\mathbf{w}}}

\def\y{{\mathbf{y}}}

\def\b{{\mathbf{b}}}
\def\X{{\mathbf{X}}}
\def\Y{{\mathbf{Y}}}
\def\C{{\mathbf{C}}}
\def\A{{\mathbf{A}}}
\def\D{{\mathbf{D}}}
\def\M{{\mathbf{M}}}

\def\I{{\mathbf{I}}}
\def\B{{\mathbf{B}}}
\def\C{{\mathbf{C}}}
\def\V{{\mathbf{V}}}

\def\F{{\mathbf{F}}}
\def\S{{\mathbf{S}}}
\def\Z{{\mathbf{Z}}}
\def\W{{\mathbf{W}}}
\def\U{{\mathbf{U}}}

\def\P{{\mathbf{P}}}

\def\G{{\mathbf{G}}}

\DeclareMathOperator*{\argmin}{arg\,min}
\DeclareMathOperator*{\argmax}{arg\,max}

\newcommand{\overbar}[1]{\mkern 1.5mu\overline{\mkern-1.5mu#1\mkern-1.5mu}\mkern 1.5mu}

\newcommand{\barC}{\overbar{\C}}
\newcommand{\barG}{\overbar{\G}}
\newcommand{\barU}{\overbar{\U}}
\newcommand{\barV}{\overbar{\V}}
\newcommand{\barW}{\overbar{\W}}
\newcommand{\barX}{\overbar{\X}}
\newcommand{\barY}{\overbar{\Y}}
\newcommand{\barZ}{\overbar{\Z}}
\newcommand{\barP}{\overbar{\P}}
\newcommand{\barS}{\overbar{\S}}
\newcommand{\barI}{\overbar{\I}}

\newcommand{\mK}{\mathcal{K}}
\newcommand{\mZ}{\boldsymbol{\mathcal{Z}}}
\newcommand{\mP}{\boldsymbol{\mathcal{P}}}
\newcommand{\mX}{\boldsymbol{\mathcal{X}}}
\newcommand{\mU}{\boldsymbol{\mathcal{U}}}
\newcommand{\mV}{\boldsymbol{\mathcal{V}}}
\newcommand{\mS}{\boldsymbol{\mathcal{S}}}
\newcommand{\mY}{\boldsymbol{\mathcal{Y}}}
\newcommand{\mA}{\boldsymbol{\mathcal{A}}}
\newcommand{\mI}{\boldsymbol{\mathcal{I}}}
\newcommand{\mB}{\boldsymbol{\mathcal{B}}}
\newcommand{\mW}{\boldsymbol{\mathcal{W}}}
\newcommand{\mM}{\boldsymbol{\mathcal{M}}}
\newcommand{\mQ}{\boldsymbol{\mathcal{Q}}}
\newcommand{\mN}{\boldsymbol{\mathcal{N}}}
\newcommand{\mG}{\boldsymbol{\mathcal{G}}}\newcommand{\mC}{\boldsymbol{\mathcal{C}}}

\newcommand{\mbS}{\mathbb{S}}
\newcommand{\mE}{\boldsymbol{\mathcal{E}}}
\newcommand{\mR}{\boldsymbol{\mathcal{R}}}

\newcommand{\mcB}{\mathcal{B}}
\newcommand{\mcM}{\mathcal{M}}

\newcommand{\mcP}{\mathcal{P}}
\newcommand{\mcQ}{\mathcal{Q}}

\newcommand{\trace}{\textnormal{\textrm{Tr}}}
\newcommand{\rank}{\textnormal{\textrm{rank}}}
\newcommand{\bcirc}{\textnormal{\textrm{bcirc}}}
\newcommand{\bdiag}{\textnormal{\textrm{bdiag}}}
\newcommand{\matbdiag}{\textnormal{\textrm{mat-bdiag}}}
\newcommand{\nnzb}{\textnormal{\textrm{nnzb}}}
\newcommand{\conj}{\textnormal{\textrm{conj}}}
\newcommand{\ri}{\textnormal{\textrm{ri}}}

\newcommand{\diag}{\textnormal{\textrm{diag}}}
\newcommand{\fold}{\textnormal{\textrm{fold}}}
\newcommand{\unfold}{\textnormal{\textrm{unfold}}}

\newcommand{\reals}{\mathbb{R}}
\newcommand{\complex}{\mathbb{C}}
\newcommand{\sign}{\textrm{sign}}

\newcommand{\fft}{\textnormal{\textrm{fft}}}
\newcommand{\ifft}{\textnormal{\textrm{ifft}}}
\newcommand{\tc}{\textnormal{\textrm{H}}}
\newcommand{\real}{\textnormal{\textrm{Re}}}
\newcommand{\im}{\textnormal{\textrm{Im}}}

\DeclarePairedDelimiterX\setc[2]{\{}{\}}{\,#1 \;\delimsize\vert\; #2\,}